\documentclass[letterpaper,12pt]{amsart}

\usepackage[hdivide={1in,*,1in},divide={1in,*,1in}]{geometry}
\usepackage{enumitem}
\usepackage{hhline}
\usepackage[font={small},margin=0.5in]{caption}

\usepackage{amsfonts,amssymb,amsmath,amsthm,amstext}
\let\cites\cite		
\usepackage{bm} 
\usepackage{mathtools}
\usepackage{stmaryrd}
\usepackage{accents}

\usepackage{pstricks,pst-node}
\usepackage{centerpict}
\usepackage{cobordisms}
\usepackage{diagps,diagps-arrows}
\usepackage{movies}

\usepackage{array}
\usepackage{graphics}
\usepackage{mywrap}

\usepackage[
		pdfauthor={Krzysztof K. Putyra},
		pdftitle={A 2-category of chronological cobordisms and odd Khovanov homology},
		colorlinks,
		unicode,
		bookmarksopen,
		bookmarksdepth=3,
		bookmarksnumbered=true]{hyperref}
\usepackage[all]{hypcap}		
\usepackage[savepos]{zref}

\newgray{blackRegion}{0.625}
\def\diagArrowLeft{<-}
\def\diagArrowRight{->}

\def\diagArrowGen#1#2/#3,#4/(#5,#6)(#7,#8)(#9){%
	\edef\testArrow{#2}%
	\dimen20=#5\psxunit	\dimen21=#6\psyunit
	\dimen22=#7\psxunit	\dimen23=#8\psyunit
	\dimen24=#3\psxunit \dimen25=#4\psyunit
	\ifx\testArrow\diagArrowLeft
		\advance\dimen20 by \dimen24
		\advance\dimen21 by \dimen25
	\else\ifx\testArrow\diagArrowRight
		\advance\dimen22 by -\dimen24
		\advance\dimen23 by -\dimen25
	\fi\fi
	\psline{#2}(\the\dimen20,\the\dimen21)(\the\dimen22,\the\dimen23)%
	\rput[c](#9){$\scriptstyle{#1}$}
\endgroup\ignorespaces}

\def\diagArrow{\begingroup\psset{linewidth=2pt}\diagArrowGen}
\def\diagDblArrow{\begingroup\psset{linewidth=0.5pt,doubleline=true,doublesep=2pt}\diagArrowGen}

\newcommand*\diagConnMM[4]{%
	\psset{linewidth=0.5pt,dimen=outer,arrowlength=0.8,arrowsize=6pt}%
	\pscircle(-1.2,0){0.4}
	\pscircle( 0.0,0){0.4}
	\pscircle( 1.2,0){0.4}
	\diagArrow{#1}{#2}/0,0/(-0.8,0)(-0.4,0)(-0.6,0.2)
	\diagArrow{#3}{#4}/0,0/( 0.4,0)( 0.8,0)( 0.6,0.2)
}

\newcommand*\diagConnX[4]{%
	\psset{linewidth=0.5pt,arrowlength=0.8,arrowsize=6pt,dimen=outer,fillstyle=solid,fillcolor=white}
	\diagArrow{#1}{#2}/0.35,0/(-0.7, 0.15)(0.7, 0.15)(0, 0.35)
	\diagArrow{#3}{#4}/0.35,0/(-0.7,-0.15)(0.7,-0.15)(0,-0.35)
	\pscircle(-0.7,0){0.4}
	\pscircle( 0.7,0){0.4}
}

\newcommand*\diagConnT[4]{%
	\psset{linewidth=0.5pt,arrowlength=0.8,arrowsize=6pt,dimen=outer}
	\edef\testArrow{#4}%
	\ifx\testArrow\diagArrowLeft
		\psbezier[linewidth=2pt]{<-}(0.4,0)(0.8,0)(0.8,0.6)(0,0.6)
	\else
		\psbezier[linewidth=2pt,arrowlength=0.8,arrowsize=6pt]{-}(0.4,0)(0.8,0)(0.8,0.6)(0,0.6)
	\fi
	\ifx\testArrow\diagArrowRight
		\psbezier[linewidth=2pt]{<-}(-0.4,0)(-0.8,0)(-0.8,0.6)(0,0.6)
	\else
		\psbezier[linewidth=2pt,arrowlength=0.8,arrowsize=6pt]{-}(-0.4,0)(-0.8,0)(-0.8,0.6)(0,0.6)
	\fi
	\rput[c](0,0.8){$\scriptstyle{#3}$}
	\psclip{\pscircle(0,0){0.4}}%
		\diagArrow{#1}{#2}/0,0.02/(0,-0.4)(0,0.4)(0.15,0)
	\endpsclip
}

\newcommand*\diagConnTwb[4]{%
	\psset{linewidth=0.5pt,arrowlength=0.8,arrowsize=6pt,dimen=outer}
	\edef\testArrow{#4}%
	\ifx\testArrow\diagArrowLeft
		\psbezier[linewidth=2pt]{<-}(0.4,0)(0.8,0)(0.8,0.6)(0,0.6)
	\else
		\psbezier[linewidth=2pt,arrowlength=0.8,arrowsize=6pt]{-}(0.4,0)(0.8,0)(0.8,0.6)(0,0.6)
	\fi
	\ifx\testArrow\diagArrowRight
		\psbezier[linewidth=2pt]{<-}(-0.4,0)(-0.8,0)(-0.8,0.6)(0,0.6)
	\else
		\psbezier[linewidth=2pt,arrowlength=0.8,arrowsize=6pt]{-}(-0.4,0)(-0.8,0)(-0.8,0.6)(0,0.6)
	\fi
	\rput[c](0,0.8){$\scriptstyle{#3}$}
	\diagDblArrow{#1}{#2}/0,0.05/(0,-0.4)(0,0.4)(0.15,0)
	\pscircle(0,0){0.4}
}

\newcommand*\diagConnTbw[4]{%
	\psframe[linestyle=none,fillstyle=solid,fillcolor=blackRegion](-1,-0.5)(1,0.9)
	\psset{arrowlength=0.8,arrowsize=6pt,dimen=outer,linewidth=0.5pt,doubleline=true,doublesep=2pt}%
	\edef\testArrow{#4}%
	\ifx\testArrow\diagArrowLeft
		\psbezier{<-}(0.4,0)(0.8,0)(0.8,0.6)(0,0.6)
	\else
		\psbezier{-}(0.4,0)(0.8,0)(0.8,0.6)(0,0.6)
	\fi
	\ifx\testArrow\diagArrowRight
		\psbezier{<-}(-0.4,0)(-0.8,0)(-0.8,0.6)(0,0.6)
	\else
		\psbezier{-}(-0.4,0)(-0.8,0)(-0.8,0.6)(0,0.6)
	\fi
	\rput[c](0,0.8){$\scriptstyle{#3}$}
	\psset{doubleline=false}%
	\begin{psclip}{\pscircle(0,0){0.4}}%
		\diagArrow{#1}{#2}/0,0.05/(0,-0.4)(0,0.4)(0.15,0)
	\end{psclip}
}

\makeatletter
\def\diagToPictOne#1(#2,#3)(#4,#5){%
	\expandafter\edef\csname pictLowered#1\endcsname##1{\noexpand\begingroup
		\dimen20=#3\psyunit
		\advance\dimen20 by 0.2\baselineskip
		\noexpand\pspicture[shift=\noexpand\the\dimen20](#2,#3)(#4,#5)
			\expandafter\noexpand\csname diag#1\endcsname{##1}%
		\noexpand\endpspicture
	\noexpand\endgroup}%
	\expandafter\edef\csname pictCentered#1\endcsname##1{%
		\noexpand\centerpict(#2,#3)(#4,#5)
			\expandafter\noexpand\csname diag#1\endcsname{##1}%
		\noexpand\endcenterpict
	}%
	\expandafter\edef\csname pict#1\endcsname{%
		\noexpand\@ifstar
			{\expandafter\noexpand\csname pictCentered#1\endcsname}%
			{\expandafter\noexpand\csname pictLowered#1\endcsname}%
	}
}
\def\diagToPictTwo#1(#2,#3)(#4,#5){%
	\expandafter\edef\csname pictLowered#1\endcsname##1##2{\noexpand\begingroup
		\dimen20=#3\psyunit
		\advance\dimen20 by 0.2\baselineskip
		\noexpand\pspicture[shift=\noexpand\the\dimen20](#2,#3)(#4,#5)
			\expandafter\noexpand\csname diag#1\endcsname{##1}{##2}%
		\noexpand\endpspicture
	\noexpand\endgroup}%
	\expandafter\edef\csname pictCentered#1\endcsname##1##2{%
		\noexpand\centerpict(#2,#3)(#4,#5)
			\expandafter\noexpand\csname diag#1\endcsname{##1}{##2}%
		\noexpand\endcenterpict
	}%
	\expandafter\edef\csname pict#1\endcsname{%
		\noexpand\@ifstar
			{\expandafter\noexpand\csname pictCentered#1\endcsname}%
			{\expandafter\noexpand\csname pictLowered#1\endcsname}%
	}
}
\def\diagToPictThree#1(#2,#3)(#4,#5){%
	\expandafter\edef\csname pictLowered#1\endcsname##1##2##3{\noexpand\begingroup
		\dimen20=#3\psyunit
		\advance\dimen20 by 0.2\baselineskip
		\noexpand\pspicture[shift=\noexpand\the\dimen20](#2,#3)(#4,#5)
			\expandafter\noexpand\csname diag#1\endcsname{##1}{##2}{##3}%
		\noexpand\endpspicture
	\noexpand\endgroup}%
	\expandafter\edef\csname pictCentered#1\endcsname##1##2##3{%
		\noexpand\centerpict(#2,#3)(#4,#5)
			\expandafter\noexpand\csname diag#1\endcsname{##1}{##2}{##3}%
		\noexpand\endcenterpict
	}%
	\expandafter\edef\csname pict#1\endcsname{%
		\noexpand\@ifstar
			{\expandafter\noexpand\csname pictCentered#1\endcsname}%
			{\expandafter\noexpand\csname pictLowered#1\endcsname}%
	}
}
\def\diagToPictFour#1(#2,#3)(#4,#5){%
	\expandafter\edef\csname pictLowered#1\endcsname##1##2##3##4{\noexpand\begingroup
		\dimen20=#3\psyunit
		\advance\dimen20 by 0.2\baselineskip
		\noexpand\pspicture[shift=\noexpand\the\dimen20](#2,#3)(#4,#5)
			\expandafter\noexpand\csname diag#1\endcsname{##1}{##2}{##3}{##4}%
		\noexpand\endpspicture
	\noexpand\endgroup}%
	\expandafter\edef\csname pictCentered#1\endcsname##1##2##3##4{%
		\noexpand\centerpict(#2,#3)(#4,#5)
			\expandafter\noexpand\csname diag#1\endcsname{##1}{##2}{##3}{##4}%
		\noexpand\endcenterpict
	}%
	\expandafter\edef\csname pict#1\endcsname{%
		\noexpand\@ifstar
			{\expandafter\noexpand\csname pictCentered#1\endcsname}%
			{\expandafter\noexpand\csname pictLowered#1\endcsname}%
	}
}
\makeatother

\diagToPictTwo{M}(-1.1,-0.7)(1.1,0.7)
\diagToPictTwo{S}(-0.7,-0.7)(0.7,0.7)
\diagToPictOne{B}(-0.7,-0.7)(0.5,0.5)
\diagToPictOne{D}(-0.7,-0.7)(0.5,0.5)

\diagToPictOne{CreateBirth}(-0.7,-0.7)(1.3,0.7)
\diagToPictOne{CreateDeath}(-0.7,-0.7)(1.3,0.7)

\diagToPictTwo{DisBB}(-0.9,-0.7)(0.9,0.5)
\diagToPictThree{DisBM}(-0.9,-0.7)(2.3,0.7)
\diagToPictThree{DisBS}(-0.9,-0.7)(1.5,0.7)
\diagToPictTwo{DisBD}(-0.9,-0.7)(0.9,0.5)

\diagToPictTwo{DisDB}(-0.9,-0.7)(0.9,0.5)
\diagToPictThree{DisDM}(-0.9,-0.7)(2.3,0.7)
\diagToPictThree{DisDS}(-0.9,-0.7)(1.5,0.7)
\diagToPictTwo{DisDD}(-0.9,-0.7)(0.9,0.5)

\diagToPictThree{DisMB}(-2.3,-0.7)(0.9,0.7)
\diagToPictFour{DisMM}(-2.3,-0.7)(2.3,0.7)
\diagToPictFour{DisMS}(-2.3,-0.7)(1.5,0.7)
\diagToPictThree{DisMD}(-2.3,-0.7)(0.9,0.7)

\diagToPictThree{DisSB}(-1.5,-0.7)(0.9,0.7)
\diagToPictFour{DisSM}(-1.5,-0.7)(2.3,0.7)
\diagToPictFour{DisSS}(-1.5,-0.7)(1.5,0.7)
\diagToPictThree{DisSD}(-1.5,-0.7)(0.9,0.7)

\diagToPictFour{ConnMM}(-1.7,-0.7)(1.7,0.7)
\diagToPictFour{ConnMS}(-1.1,-0.7)(1.5,0.7)
\diagToPictFour{ConnSM}(-1.5,-0.7)(1.1,0.7)
\diagToPictFour{ConnSS}(-0.9,-0.7)(0.9,0.7)
\diagToPictFour{ConnX}(-1.2,-0.7)(1.2,0.7)

\diagToPictFour{ConnT}(-0.8,-0.7)(0.8,1.0)
\diagToPictFour{ConnTpos}(-0.8,-0.7)(0.8,1.0)
\diagToPictFour{ConnTneg}(-0.8,-0.7)(0.8,1.0)
\diagToPictFour{ConnTwb}(-0.8,-0.7)(0.8,1.0)
\diagToPictFour{ConnTwbpos}(-0.8,-0.7)(0.8,1.0)
\diagToPictFour{ConnTwbneg}(-0.8,-0.7)(0.8,1.0)
\diagToPictFour{ConnTbw}(-1,-0.7)(1,1.0)
\diagToPictFour{ConnTbwpos}(-1,-0.7)(1,1.0)
\diagToPictFour{ConnTbwneg}(-1,-0.7)(1,1.0)

\input{images/diag-rels.tex}
\def\trefoil#1#2#3{%
	\begin{centerpict}(-1,-0.8)(1,1.1)%
		\psbezier(0,1)( 0.5,1)( 0.59641,0.23302)( 0.34641,-0.2)			
		\psbezier(-0.86603,-0.5)(-1.11503,-0.06698)(-0.5,0.4)(0,0.4)		
		\psbezier( 0.86603,-0.5)( 0.61603,-0.93302)(-0.09641,-0.63302)(-0.34641,-0.2)		
		\psbezier[border=4pt](0,1)(-0.5,1)(-0.59641,0.23302)(-0.34641,-0.2)			
		\psbezier[border=4pt]( 0.86603,-0.5)( 1.11503,-0.06698)( 0.5,0.4)(0,0.4)		
		\psbezier[border=4pt](-0.86603,-0.5)(-0.61603,-0.93302)( 0.09641,-0.63302)( 0.34641,-0.2)		
		\psline[linewidth=0.5pt,arrowsize=5pt,arrowinset=0.4]{->}(0.625,0.01519)(0.225,0.70801)
		\psline[linewidth=0.5pt,arrowsize=5pt,arrowinset=0.4]{->}(0.3,-0.55)(-0.5,-0.55)
		\psline[linewidth=0.5pt,arrowsize=5pt,arrowinset=0.4]{->}(-0.625,0.01519)(-0.225,0.70801)
		\uput[ 30]( 0.475, 0.275){$1$}
		\uput[-90]( 0.000,-0.550){$2$}
		\uput[150](-0.475, 0.275){$3$}
	\end{centerpict}%
}

\def\trefoilshadow{%
	\psbezier(0,1)( 0.5,1)( 0.59641,0.23302)( 0.34641,-0.2)
	\psbezier(-0.86603,-0.5)(-1.11503,-0.06698)(-0.5,0.4)(0,0.4)
	\psbezier( 0.86603,-0.5)( 0.61603,-0.93302)(-0.09641,-0.63302)(-0.34641,-0.2)
	\psbezier(0,1)(-0.5,1)(-0.59641,0.23302)(-0.34641,-0.2)
	\psbezier( 0.86603,-0.5)( 1.11503,-0.06698)( 0.5,0.4)(0,0.4)
	\psbezier(-0.86603,-0.5)(-0.61603,-0.93302)( 0.09641,-0.63302)( 0.34641,-0.2)
}

\def\trefoilarcdiag#1#2#3{\begingroup
	\trefoilshadow
	\psset{linestyle=solid,linewidth=0.5pt,fillstyle=solid,fillcolor=white}%
	\pscircle( 0.475,0.275){0.2}\rput( 0.475, 0.275){\tiny#1}%
	\pscircle( 0    ,-0.55){0.2}\rput( 0.000,-0.550){\tiny#2}%
	\pscircle(-0.475,0.275){0.2}\rput(-0.475, 0.275){\tiny#3}%
\endgroup}

\def\trefoilresolution#1#2#3{%
	\trefoilshadow
	\begingroup
		\psset{linestyle=none,fillstyle=solid,fillcolor=white}%
		\pscircle( 0    ,-0.55){0.2}
		\pscircle( 0.475,0.275){0.2}
		\pscircle(-0.475,0.275){0.2}
	\endgroup
	\ifx1#1\relax
		\psbezier(0.47105,0.46555)(0.48679,0.33682)(0.53343,0.254  )(0.63713,0.17613)
		\psbezier(0.45367,0.0879 )(0.48132,0.21841)(0.42964,0.30766)(0.30286,0.34897)
	\else
		\psbezier(0.47105,0.46555)(0.48679,0.33682)(0.42964,0.30766)(0.30286,0.34897)
		\psbezier(0.45367,0.0879 )(0.48132,0.21841)(0.53343,0.254  )(0.63713,0.17613)
		\ifx0#1\relax\else
			\psline[linewidth=0.5pt,arrowsize=3pt,arrowinset=0.4]{->}(0.625,0.01519)(0.225,0.70801)
		\fi
	\fi
	\ifx1#2\relax
		\psbezier(-0.16765,-0.64072)(-0.04831,-0.58998)( 0.04688,-0.58917)( 0.1662 ,-0.6401 )
		\psbezier( 0.15071,-0.43686)( 0.05151,-0.52606)(-0.05151,-0.52606)(-0.15071,-0.43686)
	\else
		\psbezier(-0.16765,-0.64072)(-0.04831,-0.58998)(-0.05151,-0.52606)(-0.15071,-0.43686)
		\psbezier( 0.15071,-0.43686)( 0.05151,-0.52606)( 0.04688,-0.58917)( 0.1662 ,-0.6401 )
		\ifx0#2\relax\else
			\psline[linewidth=0.5pt,arrowsize=3pt,arrowinset=0.4]{->}(0.4,-0.55)(-0.5,-0.55)
		\fi
	\fi
	\ifx1#3\relax
		\psbezier(-0.47105,0.46555)(-0.48679,0.33682)(-0.53343,0.254  )(-0.63713,0.17613)
		\psbezier(-0.45367,0.0879 )(-0.48132,0.21841)(-0.42964,0.30766)(-0.30286,0.34897)
	\else
		\psbezier(-0.47105,0.46555)(-0.48679,0.33682)(-0.42964,0.30766)(-0.30286,0.34897)
		\psbezier(-0.45367,0.0879 )(-0.48132,0.21841)(-0.53343,0.254  )(-0.63713,0.17613)
		\ifx0#3\relax\else
			\psline[linewidth=0.5pt,arrowsize=3pt,arrowinset=0.4]{->}(-0.625,0.01519)(-0.225,0.70801)
		\fi
	\fi
}

\input{images/pics-tangles.tex}
\def\defPicSmallCob#1#2{%
	\expandafter\def\csname fntCob#1\endcsname{%
		\begin{centerpict}(-0.1,0)(1.1,1.2)
			\psset{linewidth=0.15pt,dotsep=1pt,dash=1pt 1.5pt,linestyle=dotted,dimen=mid}
			\psellipse(0.5,0.2)(0.5,0.2)\psline(0,0.2)(0,1)
			\psellipse(0.5,1.0)(0.5,0.2)\psline(1,0.2)(1,1)
			\psset{linewidth=0.3pt,linestyle=solid}#2
		\end{centerpict}}%
}

\def\defPicCob#1#2{\expandafter\def\csname cobInsert@#1\endcsname{#2}}%

\def\psdash#1{#1[linestyle=dashed,linewidth=0.15pt]}

\def\insertCob#1{\begin{centerpict}(-0.1,0)(1.3,2)
	\psset{xunit=1.2cm,linewidth=0.15pt,dotsep=1pt,dash=1pt 1pt,linestyle=dotted,dimen=mid}
	\psellipse(0.5,0.2)(0.5,0.2)\psline(0,0.2)(0,1.8)
	\psellipse(0.5,1.8)(0.5,0.2)\psline(1,0.2)(1,1.8)
	\psset{linewidth=0.4pt,linestyle=solid}
	\csname cobInsert@#1\endcsname
\end{centerpict}}%

\def\fntCob#1{\csname fntCob#1\endcsname}

\def\cobRIhtop{%
	\psellipticarc(0.65,0)(0.15,0.1){-90}{90}
	\psbezier(0.13523, 0.13679)(0.45,-0.10)(0.3 , 0.1)(0.65, 0.1)
	\psbezier(0.27639,-0.17889)(0.35, 0.05)(0.45,-0.1)(0.65,-0.1)
}
\def\cobRIvtop{%
	\psellipse(0.65,0)(0.15,0.1)
	\psbezier(0.27639,-0.17889)(0.4,0)(0.4,0)(0.13523,0.13679)
}
\def\cobRIhbottom{%
	\psellipticarc(0.65,0)(0.15,0.1){-90}{0}
	\psbezier(0.27639,-0.17889)(0.35,0.05)(0.45,-0.1)(0.65,-0.1)
	\psclip{\psframe[linestyle=none](0,-0.2)(0.277,1)}
		\psbezier(0.13523,0.13679)(0.45,-0.1)(0.3,0.1)(0.65,0.1)
	\endpsclip
	\psdash\psellipticarc(0.65,0)(0.15,0.1){0}{90}
	\psclip{\psframe[linestyle=none](1,-0.2)(0.276,1)}
		\psdash\psbezier(0.13523,0.13679)(0.45,-0.1)(0.3,0.1)(0.65,0.1)
	\endpsclip
}
\def\cobRIvbottom{%
	\psellipticarc(0.65,0)(0.15,0.1){-180}{0}
	\psdash\psellipticarc(0.65,0)(0.15,0.1){0}{180}
	\psclip{\pspolygon[linestyle=none](0,-0.2)(0.4,-0.2)(0.4,-0.028)(0.277,-0.028)(0.277,0.2)(0,0.2)}
		\psbezier(0.27639,-0.17889)(0.4,0)(0.4,0)(0.13523,0.13679)
	\endpsclip
	\psclip{\psframe[linestyle=none](0.2763,-0.0284)(0.3564,0.2)}
		\psdash\psbezier(0.27639,-0.17889)(0.4,0)(0.4,0)(0.13523,0.13679)
	\endpsclip
}

\defPicSmallCob{RId}{%
	\psline(0.27639,0.02111)(0.27639,0.82111)
	\psline(0.13523,0.33679)(0.13523,1.13679)
	\psline(0.80000,0.20000)(0.80000,1.00000)
	\rput(0,0.2)\cobRIvbottom
	\rput(0,1.0)\cobRIhtop
	\psbezier(0.35637,0.17164)(0.37,0.7)(0.49,0.7)(0.5,0.2)
}

\defPicSmallCob{RIh}{%
	\psline(0.27639,0.02111)(0.27639,0.82111)
	\psline(0.13523,0.33679)(0.13523,1.13679)
	\psbezier(0.35637,0.97164)(0.37,0.4)(0.8,0.4)(0.8,0.2)
	\rput(0,0.2)\cobRIhbottom
	\rput(0,1.0)\cobRIvtop
	\psbezier(0.5,1)(0.5,0.6)(0.8,0.6)(0.8,1)
}

\defPicSmallCob{RIg}{%
	\psline(0.27639,0.02111)(0.27639,0.82111)
	\psline(0.13523,0.33679)(0.13523,1.13679)
	\psbezier(0.35637,0.17164)(0.37,0.8)(0.8,0.8)(0.8,1)
	\rput(0,0.2)\cobRIvbottom
	\rput(0,1.0)\cobRIhtop
	\psbezier(0.5,0.2)(0.5,0.6)(0.8,0.6)(0.8,0.2)
}

\defPicSmallCob{RIfa}{%
	\psline(0.27639,0.02111)(0.27639,0.82111)
	\psline(0.13523,0.33679)(0.13523,1.13679)
	\psbezier(0.35637,0.97164)(0.37,0.1)(0.8,0.5)(0.8,0.2)
	\rput(0,0.2)\cobRIhbottom
	\rput(0,1.0)\cobRIvtop
	\psellipticarc(0.65,0.75)(0.2,0.2){-180}{0}
	\psbezier(0.5,1)(0.5,0.9)(0.45,0.85)(0.45,0.75)
	\psbezier(0.8,1)(0.8,0.9)(0.85,0.85)(0.85,0.75)
	\psellipticarc(0.65,0.78)(0.1,0.1){-170}{-10}
	\psellipticarc(0.65,0.65)(0.1,0.1){40}{140}
}

\defPicSmallCob{RIfb}{%
	\psline(0.27639,0.02111)(0.27639,0.82111)
	\psline(0.13523,0.33679)(0.13523,1.13679)
	\psline(0.80000,0.20000)(0.80000,1.00000)
	\rput(0,0.2)\cobRIhbottom
	\rput(0,1.0)\cobRIvtop
	\psbezier(0.35637,0.97164)(0.37,0.5)(0.49,0.5)(0.5,1)
}

\defPicCob{R1-df-0}{%
	\rput(0,0.2)\cobRIhbottom
	\rput(0,1.0)\cobRIvbottom
	\rput(0,1.8)\cobRIhtop
	\psline(0.27639,0.02111)(0.27639,1.62111)
	\psline(0.13523,0.33679)(0.13523,1.93679)
	\psline(0.80000,1.00000)(0.80000,1.80000)
	\psbezier(0.35637,0.97164)(0.37,0.1)(0.8,0.5)(0.8,0.2)
	\psbezier(0.35637,0.97164)(0.37,1.5)(0.49,1.5)(0.5,1)
	\psellipticarc(0.65,0.75)(0.2,0.2){-180}{0}
	\psbezier(0.5,1)(0.5,0.9)(0.45,0.85)(0.45,0.75)
	\psbezier(0.8,1)(0.8,0.9)(0.85,0.85)(0.85,0.75)
	\psellipticarc(0.65,0.78)(0.1,0.1){-170}{-10}
	\psellipticarc(0.65,0.65)(0.1,0.1){40}{140}
}
\defPicCob{R1-df-1}{%
	\rput(0,0.2)\cobRIhbottom
	\rput(0,1.0)\cobRIvbottom
	\rput(0,1.8)\cobRIhtop
	\psline(0.27639,0.02111)(0.27639,1.62111)
	\psline(0.13523,0.33679)(0.13523,1.93679)
	\psline(0.80000,0.20000)(0.80000,1.80000)
	\psbezier(0.35637,0.97164)(0.37,1.5)(0.49,1.5)(0.5,1)
	\psbezier(0.35637,0.97164)(0.37,0.5)(0.49,0.5)(0.5,1)
}

\defPicCob{R1-df-2}{%
	\rput(0,0.2)\cobRIhbottom
	\rput(0,1.8)\cobRIhtop
	\psline(0.27639,0.02111)(0.27639,1.62111)
	\psline(0.13523,0.33679)(0.13523,1.93679)
	\psbezier(0.35637,0.97164)(0.37,0.1)(0.8,0.5)(0.8,0.2)
	\rput(0,1){%
		\psclip{\pspolygon[linestyle=none](0,-0.2)(0.4,-0.2)(0.4,-0.028)(0.277,-0.028)(0.277,0.2)(0,0.2)}
			\psbezier(0.27639,-0.17889)(0.4,0)(0.4,0)(0.13523,0.13679)
		\endpsclip
		\psclip{\psframe[linestyle=none](0.2763,-0.0284)(0.3564,0.2)}
			\psdash\psbezier(0.27639,-0.17889)(0.4,0)(0.4,0)(0.13523,0.13679)
		\endpsclip
		\psellipticarc(0.55,0)(0.1,0.06){-180}{0}
		\psellipticarc(0.85,0)(0.1,0.06){-180}{0}
		\psdash\psellipticarc(0.55,0)(0.1,0.06){0}{180}
		\psdash\psellipticarc(0.85,0)(0.1,0.06){0}{180}
	}
	\psbezier(0.35637,0.97164)(0.37,1.3)(0.44,1.3)(0.45,1)
	\psellipticarc(0.7,1)(0.25,0.5){-180}{0}
	\psellipticarc(0.7,1)(0.05,0.3){-180}{0}
	\psellipticarc(0.7,1)(0.05,0.4){0}{180}
	\psbezier(0.95,1)(0.95,1.5)(0.8,1.5)(0.8,1.8)
}

\defPicCob{R1-df-3}{%
	\rput(0,0.2)\cobRIhbottom
	\rput(0,1.8)\cobRIhtop
	\psline(0.27639,0.02111)(0.27639,1.62111)
	\psline(0.13523,0.33679)(0.13523,1.93679)
	\psline(0.8,0.7)(0.8,1.8)
	\rput(0,1){%
		\psclip{\pspolygon[linestyle=none](0,-0.2)(0.5,-0.2)(0.5,0)(0.277,0)(0.277,0.2)(0,0.2)}
			\psbezier(0.27639,-0.17889)(0.49,0)(0.49,0)(0.13523,0.13679)
		\endpsclip
		\psclip{\psframe[linestyle=none](0.2763,0)(0.5,0.2)}
			\psdash\psbezier(0.27639,-0.17889)(0.49,0)(0.49,0)(0.13523,0.13679)
		\endpsclip
		\psellipticarc(0.7,0)(0.1,0.06){-180}{0}
		\psdash\psellipticarc(0.7,0)(0.1,0.06){0}{180}
	}
	\psellipse(0.5,1.2)(0.12,0.3)
	\psarc(0.72,0.7){0.08}{-180}{0}
	\psarc(0.54,0.7){0.1}{0}{180}
	\psbezier(0.44,0.7)(0.44,0.4)(0.8,0.5)(0.8,0.2)
}

\defPicCob{R1-GF-0}{%
	\rput(0,0.2)\cobRIhbottom
	\rput(0,1.0)\cobRIvbottom
	\rput(0,1.8)\cobRIhtop
	\psline(0.27639,0.02111)(0.27639,1.62111)
	\psline(0.13523,0.33679)(0.13523,1.93679)
	\psbezier(0.35637,0.97164)(0.37,0.1)(0.8,0.5)(0.8,0.2)
	\psbezier(0.35637,0.97164)(0.37,1.6)(0.8,1.6)(0.8,1.8)
	\psellipticarc(0.65,0.75)(0.2,0.2){-180}{0}
	\psbezier(0.5,1)(0.5,0.9)(0.45,0.85)(0.45,0.75)
	\psbezier(0.8,1)(0.8,0.9)(0.85,0.85)(0.85,0.75)
	\psellipticarc(0.65,0.78)(0.1,0.1){-170}{-10}
	\psellipticarc(0.65,0.65)(0.1,0.1){40}{140}
	\psbezier(0.5,1)(0.5,1.4)(0.8,1.4)(0.8,1)
}

\defPicCob{R1-GF-1}{%
	\rput(0,0.2)\cobRIhbottom
	\rput(0,1.0)\cobRIvbottom
	\rput(0,1.8)\cobRIhtop
	\psline(0.27639,0.02111)(0.27639,1.62111)
	\psline(0.13523,0.33679)(0.13523,1.93679)
	\psbezier(0.35637,0.97164)(0.37,1.6)(0.8,1.6)(0.8,1.8)
	\psline(0.8,0.2)(0.8,1.0)
	\psbezier(0.35637,0.97164)(0.37,0.5)(0.49,0.5)(0.5,1)
	\psbezier(0.5,1)(0.5,1.4)(0.8,1.4)(0.8,1)
}

\defPicCob{R1-GF-2}{%
	\rput(0,0.2)\cobRIhbottom
	\rput(0,1.8)\cobRIhtop
	\rput(0,1.0){%
		\psclip{\pspolygon[linestyle=none](0,-0.2)(0.4,-0.2)(0.4,-0.028)(0.277,-0.028)(0.277,0.2)(0,0.2)}
			\psbezier(0.27639,-0.17889)(0.4,0)(0.4,0)(0.13523,0.13679)
		\endpsclip
		\psclip{\psframe[linestyle=none](0.2763,-0.0284)(0.3564,0.2)}
			\psdash\psbezier(0.27639,-0.17889)(0.4,0)(0.4,0)(0.13523,0.13679)
		\endpsclip
	}
	\psline(0.27639,0.02111)(0.27639,1.62111)
	\psline(0.13523,0.33679)(0.13523,1.93679)
	\psbezier(0.35637,0.97164)(0.37,0.1)(0.8,0.5)(0.8,0.2)
	\psbezier(0.35637,0.97164)(0.37,1.6)(0.8,1.6)(0.8,1.8)
}

\defPicCob{R1-GF-3}{%
	\rput(0,0.2)\cobRIhbottom
	\rput(0,1.0)\cobRIhbottom
	\rput(0,1.8)\cobRIhtop
	\psline(0.27639,0.02111)(0.27639,1.62111)
	\psline(0.13523,0.33679)(0.13523,1.93679)
	\psline(0.80000,0.20000)(0.80000,1.80000)
}

\def\cobRIFourTubes{%
	\rput(0,0.2)\cobRIvbottom
	\rput(0,1.8)\cobRIvtop
	\psline(0.27639,0.02111)(0.27639,1.62111)
	\psline(0.13523,0.33679)(0.13523,1.93679)
	\psline(0.35637,1.77164)(0.35637,0.6)
	\psellipticarc(0.85,0.8)(0.08,0.05){-180}{0}
	\psellipticarc(0.55,0.8)(0.08,0.05){-180}{0}
	\psdash\psellipticarc(0.85,0.8)(0.08,0.05){0}{180}
	\psdash\psellipticarc(0.55,0.8)(0.08,0.05){0}{180}
	\psbezier(0.35637,0.6)(0.35637,0.5)(0.47,0.6)(0.47,0.8)
	\psbezier(0.35637,0.2)(0.35637,0.3)(0.63,0.4)(0.63,0.8)
	\psbezier(0.5,0.2)(0.5,0.3)(0.77,0.4)(0.77,0.8)
	\psbezier(0.8,0.2)(0.8,0.3)(0.93,0.4)(0.93,0.8)
	\psellipticarc(0.85,1.4)(0.08,0.05){-180}{0}
	\psellipticarc(0.55,1.4)(0.08,0.05){-180}{0}
	\psdash\psellipticarc(0.85,1.4)(0.08,0.05){0}{180}
	\psdash\psellipticarc(0.55,1.4)(0.08,0.05){0}{180}
	\psbezier(0.5,1.8)(0.5,1.6)(0.47,1.6)(0.47,1.4)
	\psbezier(0.8,1.8)(0.8,1.6)(0.93,1.6)(0.93,1.4)
	\psellipticarc(0.7,1.4)(0.07,0.15){0}{180}
}

\defPicCob{R1-FG-0}{%
	\cobRIFourTubes
	\psline(0.47,0.8)(0.47,1.4)
	\psline(0.63,0.8)(0.63,1.4)
	\psbezier(0.77,0.8)(0.77,1.1)(0.93,1.1)(0.93,0.8)
	\psbezier(0.77,1.4)(0.77,1.1)(0.93,1.1)(0.93,1.4)
}

\defPicCob{R1-FG-1}{%
	\cobRIFourTubes
	\psbezier(0.47,0.8)(0.47,1.1)(0.63,1.1)(0.63,0.8)
	\psbezier(0.47,1.4)(0.47,1.1)(0.63,1.1)(0.63,1.4)
	\psline(0.77,0.8)(0.77,1.4)
	\psline(0.93,0.8)(0.93,1.4)
}

\defPicCob{R1-FG-2}{%
	\cobRIFourTubes
	\psbezier(0.47,1.4)(0.47,1.1)(0.63,1.1)(0.63,1.4)
	\psbezier(0.77,1.4)(0.77,1.2)(0.93,1.2)(0.93,1.4)
	\psbezier(0.47,0.8)(0.47,1.2)(0.93,1.2)(0.93,0.8)
	\psbezier(0.63,0.8)(0.63,1.0)(0.77,1.0)(0.77,0.8)
}

\defPicCob{R1-FG-3}{%
	\cobRIFourTubes
	\psbezier(0.47,1.4)(0.47,1.0)(0.93,1.0)(0.93,1.4)
	\psbezier(0.63,1.4)(0.63,1.2)(0.77,1.2)(0.77,1.4)
	\psbezier(0.47,0.8)(0.47,1.1)(0.63,1.1)(0.63,0.8)
	\psbezier(0.77,0.8)(0.77,1.0)(0.93,1.0)(0.93,0.8)
}

\defPicCob{R1-FG-4}{%
	\rput(0,0.2)\cobRIvbottom
	\rput(0,1.0)\cobRIhbottom
	\rput(0,1.8)\cobRIvtop
	\psline(0.27639,0.02111)(0.27639,1.62111)
	\psline(0.13523,0.33679)(0.13523,1.93679)
	\psbezier(0.35637,0.17164)(0.37,0.8)(0.8,0.8)(0.8,1)
	\psline(0.8,1)(0.8,1.8)
	\psbezier(0.5,0.2)(0.5,0.6)(0.8,0.6)(0.8,0.2)
	\psbezier(0.35637,1.77164)(0.37,1.3)(0.49,1.3)(0.5,1.8)
}

\defPicCob{R1-FG-5}{%
	\rput(0,0.2)\cobRIvbottom
	\rput(0,1.0)\cobRIvbottom
	\rput(0,1.8)\cobRIvtop
	\psline(0.27639,0.02111)(0.27639,1.62111)
	\psline(0.13523,0.33679)(0.13523,1.93679)
	\psline(0.35637,0.17164)(0.35637,1.77164)
	\psline(0.5,0.2)(0.5,1.8)
	\psline(0.8,0.2)(0.8,1.8)
}

\defPicCob{R1-FG-6}{%
	\rput(0,0.2)\cobRIvbottom
	\rput(0,1.0)\cobRIhbottom
	\rput(0,1.8)\cobRIvtop
	\psline(0.27639,0.02111)(0.27639,1.62111)
	\psline(0.13523,0.33679)(0.13523,1.93679)
	\psline(0.8,0.2)(0.8,1)
	\psbezier(0.35637,1.77164)(0.37,1.2)(0.8,1.2)(0.8,1)
	\psbezier(0.35637,0.17164)(0.37,0.7)(0.49,0.7)(0.5,0.2)
	\psbezier(0.5,1.8)(0.5,1.4)(0.8,1.4)(0.8,1.8)
}

\defPicCob{R1-FG-7}{%
	\rput(0,0.2)\cobRIvbottom
	\rput(0,1.0)\cobRIhbottom
	\rput(0,1.8)\cobRIvtop
	\psline(0.27639,0.02111)(0.27639,1.62111)
	\psline(0.13523,0.33679)(0.13523,1.93679)
	\psbezier(0.35637,1.77164)(0.37,0.9)(0.8,1.3)(0.8,1)
	\psbezier(0.35637,0.17164)(0.37,0.8)(0.8,0.8)(0.8,1)
	\psellipticarc(0.65,1.55)(0.2,0.2){-180}{0}
	\psbezier(0.5,1.8)(0.5,1.7)(0.45,1.65)(0.45,1.55)
	\psbezier(0.8,1.8)(0.8,1.7)(0.85,1.65)(0.85,1.55)
	\psellipticarc(0.65,1.58)(0.1,0.1){-170}{-10}
	\psellipticarc(0.65,1.45)(0.1,0.1){40}{140}
	\psbezier(0.5,0.2)(0.5,0.6)(0.8,0.6)(0.8,0.2)
}

\def\cobRIIhT{%
	\psbezier(0.20,-0.160)(0.52,-0.03)(0.52,-0.03)(0.95,-0.087)
	\psbezier(0.05, 0.087)(0.48, 0.03)(0.48, 0.03)(0.80, 0.160)
}

\def\cobRIIvvT{%
	\psbezier(0.2,-0.16)(0.3,-0.05)(0.25,0)(0.05, 0.087)
	\psbezier(0.8, 0.16)(0.7, 0.05)(0.75,0)(0.95,-0.087)
	\psbezier(0.45,0.07)(0.61,0.10)(0.71,-0.04)(0.55,-0.07)
	\psbezier(0.45,0.07)(0.29,0.04)(0.39,-0.10)(0.55,-0.07)
}

\def\cobRIIvhT{%
	\psbezier(0.2,-0.16)(0.3,-0.05)(0.25,0)(0.05, 0.087)
	\psbezier(0.45,0.07)(0.29,0.04)(0.39,-0.10)(0.55,-0.07)
	\psbezier(0.80, 0.160)(0.75,0.00)(0.61, 0.10)(0.45, 0.07)
	\psbezier(0.95,-0.087)(0.73,0.06)(0.71,-0.04)(0.55,-0.07)
}

\def\cobRIIhvT{%
	\psbezier(0.8, 0.16)(0.7, 0.05)(0.75,0)(0.95,-0.087)
	\psbezier(0.45,0.07)(0.61,0.10)(0.71,-0.04)(0.55,-0.07)
	\psbezier(0.20,-0.160)(0.25, 0.00)(0.39,-0.10)(0.55,-0.07)
	\psbezier(0.05, 0.087)(0.27,-0.06)(0.29, 0.04)(0.45, 0.07)
}

\def\cobRIIhhT{%
	\psbezier(0.80, 0.160)(0.75,0.00)(0.61, 0.10)(0.45, 0.07)
	\psbezier(0.95,-0.087)(0.73,0.06)(0.71,-0.04)(0.55,-0.07)
	\psbezier(0.20,-0.160)(0.25, 0.00)(0.39,-0.10)(0.55,-0.07)
	\psbezier(0.05, 0.087)(0.27,-0.06)(0.29, 0.04)(0.45, 0.07)
}

\def\cobRIIhB{%
	\psbezier(0.20,-0.160)(0.52,-0.03)(0.52,-0.03)(0.95,-0.087)
	\psclip{\psframe[linestyle=none](0,-0.2)(0.2,0.2)}%
		\psbezier(0.05, 0.087)(0.48, 0.03)(0.48, 0.03)(0.80, 0.160)
	\endpsclip
	\psclip{\psframe[linestyle=none](0.2,-0.2)(1,0.2)}%
		\psdash\psbezier(0.05, 0.087)(0.48, 0.03)(0.48, 0.03)(0.80, 0.160)
	\endpsclip
}

\def\cobRIIvvB{%
	\psclip{\pspolygon[linestyle=none](0,-0.2)(0,0.2)(0.2,0.2)(0.2,-0.0664)(0.25,-0.0664)(0.368,0)(0.631,0)(0.75,0.0664)(1,0.0664)(1,-0.2)}%
		\psbezier(0.2,-0.16)(0.3,-0.05)(0.25,0)(0.05, 0.087)
		\psbezier(0.8, 0.16)(0.7, 0.05)(0.75,0)(0.95,-0.087)
		\psbezier(0.45,0.07)(0.61,0.10)(0.71,-0.04)(0.55,-0.07)
		\psbezier(0.45,0.07)(0.29,0.04)(0.39,-0.10)(0.55,-0.07)
	\endpsclip
	\psclip{\pspolygon[linestyle=none](0.2,0.2)(0.2,-0.0664)(0.25,-0.0664)(0.368,0)(0.631,0)(0.75,0.0664)(1,0.0664)(1,0.2)}
		\psdash\psbezier(0.2,-0.16)(0.3,-0.05)(0.25,0)(0.05, 0.087)
		\psdash\psbezier(0.8, 0.16)(0.7, 0.05)(0.75,0)(0.95,-0.087)
		\psdash\psbezier(0.45,0.07)(0.61,0.10)(0.71,-0.04)(0.55,-0.07)
		\psdash\psbezier(0.45,0.07)(0.29,0.04)(0.39,-0.10)(0.55,-0.07)
	\endpsclip
}

\def\cobRIIvhB{%
	\psclip{\pspolygon[linestyle=none](0,0.2)(0.2,0.2)(0.2,-0.0664)(0.25,-0.0664)(0.368,0)(1,0.11)(1,-0.2)(0,-0.2)}
		\psbezier(0.2,-0.16)(0.3,-0.05)(0.25,0)(0.05, 0.087)
		\psbezier(0.45,0.07)(0.29,0.04)(0.39,-0.10)(0.55,-0.07)
		\psbezier(0.95,-0.087)(0.73,0.06)(0.71,-0.04)(0.55,-0.07)
	\endpsclip
	\psclip{\pspolygon[linestyle=none](0.2,0.2)(0.2,-0.0664)(0.25,-0.0664)(0.368,0)(1,0.11)(1,0.2)}
		\psdash\psbezier(0.2,-0.16)(0.3,-0.05)(0.25,0)(0.05, 0.087)
		\psdash\psbezier(0.45,0.07)(0.29,0.04)(0.39,-0.10)(0.55,-0.07)
		\psdash\psbezier(0.80, 0.160)(0.75,0.00)(0.61, 0.10)(0.45, 0.07)
	\endpsclip
}

\def\cobRIIhvB{%
	\psclip{\pspolygon[linestyle=none](1,-0.2)(1,0.0664)(0.75,0.0664)(0.631,0)(0.2,-0.02)(0.2,0.2)(0,0.2)(0,-0.2)}
		\psbezier(0.8, 0.16)(0.7, 0.05)(0.75,0)(0.95,-0.087)
		\psbezier(0.45,0.07)(0.61,0.10)(0.71,-0.04)(0.55,-0.07)
		\psbezier(0.20,-0.160)(0.25, 0.00)(0.39,-0.10)(0.55,-0.07)
		\psbezier(0.05, 0.087)(0.27,-0.06)(0.29, 0.04)(0.45, 0.07)
	\endpsclip
	\psclip{\pspolygon[linestyle=none](1,0.2)(1,0.0664)(0.75,0.0664)(0.631,0)(0.2,-0.11)(0.2,0.2)}
		\psdash\psbezier(0.8, 0.16)(0.7, 0.05)(0.75,0)(0.95,-0.087)
		\psdash\psbezier(0.45,0.07)(0.61,0.10)(0.71,-0.04)(0.55,-0.07)
		\psdash\psbezier(0.05, 0.087)(0.27,-0.06)(0.29, 0.04)(0.45, 0.07)
	\endpsclip
}

\def\cobRIIhhB{%
	\psbezier(0.95,-0.087)(0.73,0.06)(0.71,-0.04)(0.55,-0.07)
	\psbezier(0.20,-0.160)(0.25, 0.00)(0.39,-0.10)(0.55,-0.07)
	\psclip{\psframe[linestyle=none](0,0.2)(0.2,-0.2)}
		\psbezier(0.05, 0.087)(0.27,-0.06)(0.29, 0.04)(0.45, 0.07)
	\endpsclip
	\psdash\psbezier(0.80, 0.160)(0.75,0.00)(0.61, 0.10)(0.45, 0.07)
	\psclip{\psframe[linestyle=none](0.2,0.2)(1,-0.2)}
		\psdash\psbezier(0.05, 0.087)(0.27,-0.06)(0.29, 0.04)(0.45, 0.07)
	\endpsclip
}

\defPicSmallCob{R2-d0*}{%
	\psline(0.05, 0.287)(0.05, 1.087)
	\psline(0.20, 0.040)(0.20, 0.840)
	\psdash\psline(0.80, 0.360)(0.80, 0.997)
	\psline(0.80, 0.997)(0.80, 1.160)
	\psline(0.95, 0.113)(0.95, 0.913)
	\psline(0.369,0.2)(0.369,1)
	\psline(0.25,0.1336)(0.25,0.9336)
	\rput(0,0.2)\cobRIIvhB
	\rput(0,1.0)\cobRIIvvT
	\psbezier(0.631,1)(0.631,0.3)(0.75,0.3)(0.75,1.0664)
}

\defPicSmallCob{R2-d*0}{%
	\psline(0.05, 0.287)(0.05, 1.087)
	\psline(0.20, 0.040)(0.20, 0.840)
	\psdash\psline(0.80, 0.360)(0.80, 0.987)
	\psline(0.80, 0.987)(0.80, 1.160)
	\psline(0.95, 0.113)(0.95, 0.913)
	\rput(0,0.2)\cobRIIvhB
	\rput(0,1.0)\cobRIIhhT
	\psbezier(0.369,0.2)(0.369,0.9)(0.25,0.9)(0.25,0.1336)
}

\defPicSmallCob{R2-d1*}{%
	\psline(0.05, 0.287)(0.05, 1.087)
	\psline(0.20, 0.040)(0.20, 0.840)
	\psdash\psline(0.80, 0.360)(0.80, 0.997)
	\psline(0.80, 0.997)(0.80, 1.160)
	\psline(0.95, 0.113)(0.95, 0.913)
	\rput(0,0.2)\cobRIIhhB
	\rput(0,1.0)\cobRIIhvT
	\psbezier(0.631,1)(0.631,0.3)(0.75,0.3)(0.75,1.0664)
}

\defPicSmallCob{R2-d*1}{%
	\psline(0.05, 0.287)(0.05, 1.087)
	\psline(0.20, 0.040)(0.20, 0.840)
	\psdash\psline(0.80, 0.360)(0.80, 0.997)
	\psline(0.80, 0.997)(0.80, 1.160)
	\psline(0.95, 0.113)(0.95, 0.913)
	\psline(0.631,0.2)(0.631,1)
	\psline(0.75,0.2664)(0.75,1.0664)
	\rput(0,0.2)\cobRIIvvB
	\rput(0,1.0)\cobRIIhvT
	\psbezier(0.369,0.2)(0.369,0.9)(0.25,0.9)(0.25,0.1336)
}

\defPicSmallCob{R2-h0*}{%
	\psline(0.05, 0.287)(0.05, 1.087)
	\psline(0.20, 0.040)(0.20, 0.840)
	\psdash\psline(0.80, 0.360)(0.80, 0.987)
	\psline(0.80, 0.987)(0.80, 1.160)
	\psline(0.95, 0.113)(0.95, 0.913)
	\psline(0.25,0.1336)(0.25,0.9336)
	\rput(0,0.2)\cobRIIvvB
	\rput(0,1.0)\cobRIIvhT
	\psbezier(0.75,0.2664)(0.75,0.9)(0.369,0.8)(0.369,1)
	\psbezier(0.369,0.2)(0.369,0.6)(0.631,0.6)(0.631,0.2)
}

\defPicSmallCob{R2-h*1}{%
	\psline(0.05, 0.287)(0.05, 1.087)
	\psline(0.20, 0.040)(0.20, 0.840)
	\psdash\psline(0.80, 0.360)(0.80, 0.997)
	\psline(0.80, 0.997)(0.80, 1.160)
	\psline(0.95, 0.113)(0.95, 0.913)
	\psline(0.75,0.2664)(0.75,1.0664)
	\rput(0,0.2)\cobRIIhvB
	\rput(0,1.0)\cobRIIvvT
	\psbezier(0.25,0.9336)(0.25,0.3)(0.631,0.4)(0.631,0.2)
	\psbezier(0.369,1)(0.369,0.6)(0.631,0.6)(0.631,1)
}

\defPicSmallCob{R2-F}{%
	\psline(0.05, 0.287)(0.05, 1.087)
	\psline(0.20, 0.040)(0.20, 0.840)
	\psdash\psline(0.80, 0.360)(0.80, 0.997)
	\psline(0.80, 0.997)(0.80, 1.160)
	\psline(0.95, 0.113)(0.95, 0.913)
	\rput(0,0.2)\cobRIIhB
	\rput(0,1.0)\cobRIIvvT
	\psbezier(0.369,1)(0.369,0.6)(0.631,0.6)(0.631,1)
	\psbezier(0.25,0.9336)(0.25,0.2336)(0.75,0.1664)(0.75,1.0664)
	\psbezier(0.534,0.148)(0.534,0.3)(0.52,0.4)(0.4923,0.4)
	\psdash\psbezier(0.466,0.252)(0.466,0.35)(0.475,0.4)(0.4923,0.4)
}

\defPicSmallCob{R2-G}{%
	\psline(0.05, 0.287)(0.05, 1.087)
	\psline(0.20, 0.040)(0.20, 0.840)
	\psdash\psline(0.80, 0.360)(0.80, 0.933)
	\psline(0.80, 0.933)(0.80, 1.160)
	\psline(0.95, 0.113)(0.95, 0.913)
	\rput(0,0.2)\cobRIIvvB
	\rput(0,1.0)\cobRIIhT
	\psbezier(0.369,0.2)(0.369,0.6)(0.631,0.6)(0.631,0.2)
	\psbezier(0.25,0.1336)(0.25,1.0336)(0.75,0.9664)(0.75,0.2664)
	\psbezier(0.534,0.948)(0.534,0.9)(0.52,0.8)(0.4923,0.8)
	\psclip{\psframe[linestyle=none](0.4,1.1)(0.5,0.935)}
		\psbezier(0.466,1.052)(0.466,0.85)(0.475,0.8)(0.4923,0.8)
	\endpsclip
	\psclip{\psframe[linestyle=none](0.4,0.7)(0.5,0.935)}
		\psdash\psbezier(0.466,1.052)(0.466,0.85)(0.475,0.8)(0.4923,0.8)
	\endpsclip
}

\def\cobRIIFourTubes{%
	\rput(0,0.2)\cobRIIvvB
	\rput(0,1.8)\cobRIIvvT
	\psline(0.05, 0.287)(0.05, 1.887)
	\psline(0.20, 0.040)(0.20, 1.640)
	\psdash\psline(0.80, 0.360)(0.80, 1.797)
	\psline(0.80, 1.797)(0.80, 1.960)
	\psline(0.95, 0.113)(0.95, 1.713)
	\psline(0.25,1.7336)(0.25,0.8)
	\psline(0.75,0.2664)(0.75,1.3)
	\psellipticarc(0.39,1.3)(0.06,0.04){-180}{0}
	\psellipticarc(0.61,1.3)(0.06,0.04){-180}{0}
	\psdash\psellipticarc(0.39,1.3)(0.06,0.04){0}{180}
	\psdash\psellipticarc(0.61,1.3)(0.06,0.04){0}{180}
	\psbezier(0.361,1.8)(0.361,1.5)(0.33,1.6)(0.33,1.3)
	\psbezier(0.639,1.8)(0.639,1.6)(0.45,1.5)(0.45,1.3)
	\psbezier(0.75,1.8664)(0.75,1.5)(0.55,1.5)(0.55,1.3)
	\psbezier(0.75,1.3)(0.75,1.6)(0.67,1.4)(0.67,1.3)
	\psellipticarc(0.39,0.8)(0.06,0.04){-180}{0}
	\psellipticarc(0.61,0.8)(0.06,0.04){-180}{0}
	\psdash\psellipticarc(0.39,0.8)(0.06,0.04){0}{180}
	\psdash\psellipticarc(0.61,0.8)(0.06,0.04){0}{180}
	\psbezier(0.361,0.2)(0.361,0.6)(0.33,0.5)(0.33,0.8)
	\psbezier(0.45,0.8)(0.45,0.5)(0.639,0.4)(0.639,0.2)
	\psclip{\pspolygon[linestyle=none](0.36,0.3)(0.35,0.7)(0.42,0.7)(0.62,0.3)}
		\psdash\psbezier(0.25,0.1336)(0.25,0.4)(0.67,0.4)(0.67,0.8)
		\psdash\psbezier(0.25,0.8)(0.25,0.2)(0.55,0.6)(0.55,0.8)
	\endpsclip
	\psclip{\pspolygon[linestyle=none](0,-0.2)(0,1)(0.35,1)(0.35,0.7)(0.36,0.3)(0.62,0.3)(0.42,0.7)(0.42,1)(1,1)(1,-0.2)}
		\psbezier(0.25,0.1336)(0.25,0.4)(0.67,0.4)(0.67,0.8)
		\psbezier(0.25,0.8)(0.25,0.2)(0.55,0.6)(0.55,0.8)
	\endpsclip
}

\defPicCob{R2-FG-0a}{%
	\cobRIIFourTubes
	\psline(0.33,0.8)(0.33,1.3)
	\psline(0.45,0.8)(0.45,1.3)
	\psbezier(0.55,0.8)(0.55,1)(0.67,1)(0.67,0.8)
	\psbezier(0.55,1.3)(0.55,1.05)(0.67,1.05)(0.67,1.3)
}

\defPicCob{R2-FG-0b}{%
	\cobRIIFourTubes
	\psbezier(0.33,0.8)(0.33,1)(0.45,1)(0.45,0.8)
	\psbezier(0.33,1.3)(0.33,1.05)(0.45,1.05)(0.45,1.3)
	\psline(0.55,0.8)(0.55,1.3)
	\psline(0.67,0.8)(0.67,1.3)
}

\defPicCob{R2-FG-0c}{%
	\cobRIIFourTubes
	\psbezier(0.33,0.8)(0.33,1.1)(0.67,1.1)(0.67,0.8)
	\psbezier(0.45,0.8)(0.45,0.9)(0.55,0.9)(0.55,0.8)
	\psbezier(0.33,1.3)(0.33,1.0)(0.45,1.0)(0.45,1.3)
	\psbezier(0.55,1.3)(0.55,1.1)(0.67,1.1)(0.67,1.3)
}

\defPicCob{R2-FG-0d}{%
	\cobRIIFourTubes
	\psbezier(0.33,1.3)(0.33,1.0)(0.67,1.0)(0.67,1.3)
	\psbezier(0.45,1.3)(0.45,1.2)(0.55,1.2)(0.55,1.3)
	\psbezier(0.33,0.8)(0.33,1.05)(0.45,1.05)(0.45,0.8)
	\psbezier(0.55,0.8)(0.55,0.95)(0.67,0.95)(0.67,0.8)
}

\defPicCob{R2-FG-1a}{%
	\rput(0,0.2)\cobRIIvvB
	\rput(0,1.0)\cobRIIvvB
	\rput(0,1.8)\cobRIIvvT
	\psline(0.05, 0.287)(0.05, 1.887)
	\psline(0.20, 0.040)(0.20, 1.640)
	\psdash\psline(0.80, 0.360)(0.80, 1.797)
	\psline(0.80, 1.797)(0.80, 1.960)
	\psline(0.95, 0.113)(0.95, 1.713)
	\psline(0.25,0.1336)(0.25,1.7336)
	\psline(0.369,0.2)(0.369,1.8)
	\psline(0.631,0.2)(0.631,1.8)
	\psline(0.75,0.2664)(0.75,1.8664)
}

\defPicCob{R2-FG-1b}{%
	\rput(0,0.2)\cobRIIvvB
	\rput(0,1.8)\cobRIIvvT
	\psline(0.05, 0.287)(0.05, 1.887)
	\psline(0.20, 0.040)(0.20, 1.640)
	\psdash\psline(0.80, 0.360)(0.80, 1.797)
	\psline(0.80, 1.797)(0.80, 1.960)
	\psline(0.95, 0.113)(0.95, 1.713)
	\rput(0,0.75)\cobRIIhvB
	\rput(0,1.25)\cobRIIvhB
	\psbezier(0.25,1.1836)(0.25,0.65)(0.369,0.75)(0.369,1.25)
	\psbezier(0.75,0.8164)(0.75,1.25)(0.631,1.15)(0.631,0.75)
	\psbezier(0.361,0.2)(0.361,0.5)(0.639,0.5)(0.639,0.2)
	\psbezier(0.361,1.8)(0.361,1.5)(0.639,1.5)(0.639,1.8)
	\psline(0.25,1.7336)(0.25,1.1836)
	\psline(0.75,0.2664)(0.75,0.8164)
	\psbezier(0.25,0.1336)(0.25,0.55)(0.631,0.55)(0.631,0.75)
	\psbezier(0.75,1.8664)(0.75,1.4)(0.369,1.45)(0.369,1.25)
}

\defPicCob{R2-FG-1c}{%
	\rput(0,0.2)\cobRIIvvB
	\rput(0,1.0)\cobRIIhvB
	\rput(0,1.8)\cobRIIvvT
	\psline(0.05, 0.287)(0.05, 1.887)
	\psline(0.20, 0.040)(0.20, 1.640)
	\psdash\psline(0.80, 0.360)(0.80, 1.797)
	\psline(0.80, 1.797)(0.80, 1.960)
	\psline(0.95, 0.113)(0.95, 1.713)
	\psline(0.75,0.2664)(0.75,1.8664)
	\psbezier(0.369,0.2)(0.369,0.9)(0.25,0.9)(0.25,0.1336)
	\psline(0.631,0.2)(0.631,1.0)
	\psbezier(0.25,1.7336)(0.25,1.1)(0.631,1.2)(0.631,1.0)
	\psbezier(0.369,1.8)(0.369,1.4)(0.631,1.4)(0.631,1.8)
}

\defPicCob{R2-FG-1d}{%
	\rput(0,0.2)\cobRIIvvB
	\rput(0,1.0)\cobRIIvhB
	\rput(0,1.8)\cobRIIvvT
	\psline(0.05, 0.287)(0.05, 1.887)
	\psline(0.20, 0.040)(0.20, 1.640)
	\psdash\psline(0.80, 0.360)(0.80, 1.797)
	\psline(0.80, 1.797)(0.80, 1.960)
	\psline(0.95, 0.113)(0.95, 1.713)
	\psline(0.25,0.1336)(0.25,1.7336)
	\psbezier(0.75,0.2664)(0.75,0.9)(0.369,0.8)(0.369,1)
	\psbezier(0.369,0.2)(0.369,0.6)(0.631,0.6)(0.631,0.2)
	\psbezier(0.631,1.8)(0.631,1.1)(0.75,1.1)(0.75,1.80664)
	\psline(0.361,1.0)(0.361,1.8)
}

\defPicCob{R2-FG-2}{%
	\rput(0,0.2)\cobRIIvvB
	\rput(0,1.0)\cobRIIhB
	\rput(0,1.8)\cobRIIvvT
	\psline(0.05, 0.287)(0.05, 1.887)
	\psline(0.20, 0.040)(0.20, 1.640)
	\psdash\psline(0.80, 0.360)(0.80, 1.797)
	\psline(0.80, 1.797)(0.80, 1.960)
	\psline(0.95, 0.113)(0.95, 1.713)
	\psbezier(0.369,0.2)(0.369,0.6)(0.631,0.6)(0.631,0.2)
	\psbezier(0.25,0.1336)(0.25,1.0336)(0.75,0.9664)(0.75,0.2664)
	\psbezier(0.534,0.948)(0.534,0.9)(0.52,0.8)(0.4923,0.8)
	\psdash\psbezier(0.466,1.052)(0.466,0.85)(0.475,0.8)(0.4923,0.8)
	\psbezier(0.369,1.8)(0.369,1.4)(0.631,1.4)(0.631,1.8)
	\psbezier(0.25,1.7336)(0.25,1.0336)(0.75,0.9664)(0.75,1.8664)
	\psbezier(0.534,0.948)(0.534,1.1)(0.52,1.2)(0.4923,1.2)
	\psdash\psbezier(0.466,1.052)(0.466,1.15)(0.475,1.2)(0.4923,1.2)
}

\def\pictRelS{\begin{centerpict}(-0.6,-0.5)(0.6,0.5)
	\psset{linewidth=0.5pt,dash=1pt 1.5pt,dimen=middle}%
	\pscircle(0,0){0.5}
	\psellipticarc(0,0)(0.5,0.15){-180}{0}
	\psdash\psellipticarc(0,0)(0.5,0.15){0}{180}
\end{centerpict}}

\def\pictRelT{\begin{centerpict}(-0.1,-0.7)(1.3,0.7)
	\psset{linewidth=0.5pt,dash=1pt 1.5pt,dimen=middle}
	\psdash\psellipticarc(0.2,0)(0.2,0.1){0}{180}
	\psdash\psellipticarc(1,0)(0.2,0.1){0}{180}
	\psellipticarc(0.2,0)(0.2,0.1){-180}{0}
	\psellipticarc(1,0)(0.2,0.1){-180}{0}
	\psbezier(0.0,0)(0.00,0.9)(1.20,0.9)(1.2,0)
	\psbezier(0.4,0)(0.35,0.4)(0.85,0.4)(0.8,0)
	\psline[linewidth=0.3pt,arrowsize=3pt]{<-}(0.45,0.1)(0.75,0.1)
	\psbezier(0.0,0)(0.00,-0.9)(1.20,-0.9)(1.2,0)
	\psbezier(0.4,0)(0.35,-0.4)(0.85,-0.4)(0.8,0)
	\psline[linewidth=0.3pt,border=1pt,arrowsize=3pt]{->}(0.40,-0.4)(0.70,-0.1)
\end{centerpict}}

\def\pictRelTuCircles{%
	\psset{linewidth=0.5pt,dash=1pt 1.5pt,dimen=middle}
	\psellipticarc(0.2,0.1)(0.2,0.07\psxunit){-180}{0}
	\psellipticarc(1.0,0.1)(0.2,0.07\psxunit){-180}{0}
	\psdash\psellipticarc(0.2,0.1)(0.2,0.07\psxunit){0}{180}
	\psdash\psellipticarc(1.0,0.1)(0.2,0.07\psxunit){0}{180}
	\psellipse(0.2,1.4)(0.2,0.07\psxunit)
	\psellipse(1.0,1.4)(0.2,0.07\psxunit)
}

\def\pictRelTuL{\begin{centerpict}(-0.1,0)(1.3,1.5)
	\pictRelTuCircles
	\psline(0.0,0.1)(0.0,1.4)\psline(0.4,0.1)(0.4,1.4)
	\psbezier(0.8,0.1)(0.8,0.7)(1.2,0.7)(1.2,0.1)
	\psbezier(0.8,1.4)(0.8,0.8)(1.2,0.8)(1.2,1.4)
\end{centerpict}}

\def\pictRelTuR{\begin{centerpict}(-0.1,0)(1.3,1.5)
	\pictRelTuCircles
	\psbezier(0.0,0.1)(0.0,0.7)(0.4,0.7)(0.4,0.1)
	\psbezier(0.0,1.4)(0.0,0.8)(0.4,0.8)(0.4,1.4)
	\psline(0.8,0.1)(0.8,1.4)\psline(1.2,0.1)(1.2,1.4)
\end{centerpict}}

\def\pictRelTuB{\begin{centerpict}(-0.1,0)(1.3,1.5)
	\pictRelTuCircles
	\psbezier(0.0,0.1)(0.0,1.0)(1.2,1.0)(1.2,0.1)
	\psbezier(0.4,0.1)(0.35,0.5)(0.85,0.5)(0.8,0.1)
	\psline[linewidth=0.3pt,arrowsize=3pt]{<-}(0.45,0.2)(0.75,0.2)
	\psbezier(0.0,1.4)(0.0,1.0)(0.4,1.0)(0.4,1.4)
	\psbezier(0.8,1.4)(0.8,0.8)(1.2,0.8)(1.2,1.4)
\end{centerpict}}

\def\pictRelTuT{\begin{centerpict}(-0.1,0)(1.3,1.5)
	\pictRelTuCircles
	\psbezier(0.0,1.4)(0.0,0.5)(1.2,0.5)(1.2,1.4)
	\psbezier(0.4,1.4)(0.35,1.0)(0.85,1.0)(0.8,1.4)
	\psline[linewidth=0.3pt,border=1pt,arrowsize=3pt]{->}(0.40,1.0)(0.70,1.3)
	\psbezier(0.0,0.1)(0.0,0.5)(0.4,0.5)(0.4,0.1)
	\psbezier(0.8,0.1)(0.8,0.7)(1.2,0.7)(1.2,0.1)
\end{centerpict}}

\def\pictRelD{\begin{centerpict}(-0.6,-0.5)(0.6,0.5)
	\psset{linewidth=0.5pt,dash=1pt 1.5pt,dimen=middle,dotsize=5pt}%
	\pscircle(0,0){0.5}
	\psellipticarc(0,0)(0.5,0.15){-180}{0}
	\psdash\psellipticarc(0,0)(0.5,0.15){0}{180}
	\psdot(0,0.15)
\end{centerpict}}

\def\pictRelNeckI{\begin{centerpict}(-0.6,0)(0.6,2)
	\psset{linewidth=0.5pt,dash=1pt 1.5pt,dimen=middle,dotsize=4pt}%
	\psellipticarc(0,0.1)(0.4,0.1){-180}{0}
	\psdash\psellipticarc(0,0.1)(0.4,0.1){0}{180}
	\psellipse(0,1.9)(0.4,0.1)
	\psline(-0.4,0.1)(-0.4,1.9)
	\psline( 0.4,0.1)( 0.4,1.9)
\end{centerpict}}

\def\pictRelNeckB{\begin{centerpict}(-0.6,0)(0.6,2)
	\psset{linewidth=0.5pt,dash=1pt 1.5pt,dimen=middle,dotsize=4pt}%
	\psellipticarc(0,0.1)(0.4,0.1){-180}{0}
	\psdash\psellipticarc(0,0.1)(0.4,0.1){0}{180}
	\psellipse(0,1.9)(0.4,0.1)
	\psbezier(-0.4,0.1)(-0.4,0.9)(0.4,0.9)(0.4,0.1)
	\psbezier(-0.4,1.9)(-0.4,1.1)(0.4,1.1)(0.4,1.9)
	\psdot(0,0.4)
\end{centerpict}}

\def\pictRelNeckT{\begin{centerpict}(-0.6,0)(0.6,2)
	\psset{linewidth=0.5pt,dash=1pt 1.5pt,dimen=middle,dotsize=4pt}%
	\psellipticarc(0,0.1)(0.4,0.1){-180}{0}
	\psdash\psellipticarc(0,0.1)(0.4,0.1){0}{180}
	\psellipse(0,1.9)(0.4,0.1)
	\psbezier(-0.4,0.1)(-0.4,0.9)(0.4,0.9)(0.4,0.1)
	\psbezier(-0.4,1.9)(-0.4,1.1)(0.4,1.1)(0.4,1.9)
	\psdot(0,1.6)
\end{centerpict}}

\def\pictRelNeckM{\begin{centerpict}(-0.6,0)(0.6,2)
	\psset{linewidth=0.5pt,dash=1pt 1.5pt,dimen=middle,dotsize=4pt}%
	\psellipticarc(0,0.1)(0.4,0.1){-180}{0}
	\psdash\psellipticarc(0,0.1)(0.4,0.1){0}{180}
	\psellipticarc(0,1.0)(0.4,0.1){-180}{0}
	\psdash\psellipticarc(0,1.0)(0.4,0.1){0}{180}
	\psellipse(0,1.9)(0.4,0.1)
	\psbezier(-0.4,0.1)(-0.4,0.6)(0.4,0.6)(0.4,0.1)
	\psbezier(-0.4,1.9)(-0.4,1.4)(0.4,1.4)(0.4,1.9)
	\pscircle(0,1){0.4}
	\psdot(0,1.15)
	\psdot(0,0.75)
\end{centerpict}}

\def\pictRelTNscheme{%
	\psset{linewidth=0.5pt,dash=1pt 1.5pt,dimen=middle,dotsize=3pt}
	\rput(0,0){
		\psellipticarc(0.2,0)(0.2,0.07\psxunit){-180}{0}
		\psellipticarc(1.0,0)(0.2,0.07\psxunit){-180}{0}
		\psdash\psellipticarc(0.2,0)(0.2,0.07\psxunit){0}{180}
		\psdash\psellipticarc(1.0,0)(0.2,0.07\psxunit){0}{180}
	}\rput(0,1){
		\psellipticarc(0.2,0)(0.2,0.07\psxunit){-180}{0}
		\psellipticarc(1.0,0)(0.2,0.07\psxunit){-180}{0}
		\psdash\psellipticarc(0.2,0)(0.2,0.07\psxunit){0}{180}
		\psdash\psellipticarc(1.0,0)(0.2,0.07\psxunit){0}{180}
	}\rput(0,0){
		\psbezier(0.0,0)(0.00,-0.9)(1.20,-0.9)(1.2,0)
		\psbezier(0.4,0)(0.35,-0.4)(0.85,-0.4)(0.8,0)
		\psline[linewidth=0.3pt,border=1pt,arrowsize=3pt]{->}(0.40,-0.4)(0.70,-0.1)
	}\rput(0,1){
		\psbezier(0.0,0)(0.00,0.9)(1.20,0.9)(1.2,0)
		\psbezier(0.4,0)(0.35,0.4)(0.85,0.4)(0.8,0)
		\psline[linewidth=0.3pt,arrowsize=3pt]{<-}(0.45,0.1)(0.75,0.1)
	}
	\psline(0.8,0)(0.8,1)\psline(1.2,0)(1.2,1)
}

\def\pictRelTfromNi{\begin{centerpict}(-0.2,-0.7)(1.4,1.7)
	\pictRelTNscheme
	\psline(0,0)(0,1)\psline(0.4,0)(0.4,1)
	\psline[linestyle=dashed,dash=2pt 2pt,linewidth=1pt,linecolor=red](-0.1,0.5)(0.5,0.5)
\end{centerpict}}

\def\pictRelTfromNb{\begin{centerpict}(-0.2,-0.7)(1.4,1.7)
	\pictRelTNscheme
	\psbezier(0,0)(0,0.5)(0.4,0.5)(0.4,0)
	\psbezier(0,1)(0,0.5)(0.4,0.5)(0.4,1)
	\psdot(0.2,0.2)
\end{centerpict}}

\def\pictRelTfromNt{\begin{centerpict}(-0.2,-0.7)(1.4,1.7)
	\pictRelTNscheme
	\psbezier(0,0)(0,0.5)(0.4,0.5)(0.4,0)
	\psbezier(0,1)(0,0.5)(0.4,0.5)(0.4,1)
	\psdot(0.2,0.8)
\end{centerpict}}

\def\pictRelTfromNm{\begin{centerpict}(-0.2,-0.7)(1.4,1.7)
	\pictRelTNscheme
	\psbezier(0,0)(0,0.25)(0.4,0.25)(0.4,0)
	\psbezier(0,1)(0,0.75)(0.4,0.75)(0.4,1)
	\pscircle(0.2,0.5){0.2}
	\psdot(0.2,0.42)
	\psdot(0.2,0.58)
\end{centerpict}}

\def\pictRelTfromNd{\begin{centerpict}(-0.5,-0.5)(0.5,0.5)
	\psset{linewidth=0.5pt,dash=1pt 1.5pt,dimen=middle,dotsize=3pt}%
	\pscircle(0,0){0.3}
	\psellipticarc(0,0)(0.3,0.1){-180}{0}
	\psdash\psellipticarc(0,0)(0.3,0.1){0}{180}
	\psdot(0,0.1)
\end{centerpict}}

\def\pictRelTuNcaps{%
	\psset{yunit=1.2\psyunit}%
	\pictRelTuCircles
	\psbezier(0.0,0.1)(0.0,0.5)(0.4,0.5)(0.4,0.1)
	\psbezier(0.0,1.4)(0.0,1.0)(0.4,1.0)(0.4,1.4)
	\psline(0.8,0.1)(0.8,0.3)
	\psline(1.2,0.1)(1.2,0.3)
	\psbezier(0.8,0.3)(0.8,0.7)(1.2,0.7)(1.2,0.3)
	\psline(0.8,1.4)(0.8,1.2)
	\psline(1.2,1.4)(1.2,1.2)
	\psbezier(0.8,1.2)(0.8,0.8)(1.2,0.8)(1.2,1.2)
}

\def\pictRelTuNtl{\begin{centerpict}(-0.1,0)(1.3,1.8)
	\pictRelTuNcaps
	\psdot(0.2,1.23)
\end{centerpict}}

\def\pictRelTuNtr{\begin{centerpict}(-0.1,0)(1.3,1.8)
	\pictRelTuNcaps
	\psdot(1.0,1.02)
\end{centerpict}}

\def\pictRelTuNbl{\begin{centerpict}(-0.1,0)(1.3,1.8)
	\pictRelTuNcaps
	\psdot(0.2,0.25)
\end{centerpict}}

\def\pictRelTuNbr{\begin{centerpict}(-0.1,0)(1.3,1.8)
	\pictRelTuNcaps
	\psdot(1.0,0.48)
\end{centerpict}}

\def\pictRelTuNmid{\begin{centerpict}(-0.1,0)(1.3,1.8)
	\psset{yunit=1.2\psyunit}\pictRelTuCircles
	\psbezier(0.0,0.1)(0.0,0.40)(0.4,0.40)(0.4,0.1)
	\psbezier(0.8,0.1)(0.8,0.55)(1.2,0.55)(1.2,0.1)
	\psbezier(0.0,1.4)(0.0,1.05)(0.4,1.05)(0.4,1.4)
	\psbezier(0.8,1.4)(0.8,0.90)(1.2,0.90)(1.2,1.4)
	\pscircle(0.6,0.75){0.23}
	\psdot(0.6,0.67)\psdot(0.6,0.83)
\end{centerpict}}

\COBnewobject deathNoBirth|D-B[1,0]<0,1>{%
	\COButurn(0,1,0){0.5}
}

\COBnewobject deathDotNoBirth|D*-B[1,0]<0,1>{%
	\COButurn(0,1,0){0.5}
	\psdot[dotsize=3pt](0.5,0.2)
}

\COBnewobject deathDotsNoBirth|D**-B[1,0]<0,1>{%
	\COButurn(0,1,0){0.3}
	\pscircle(0.5,0.5){0.25}
	\psdot[dotsize=3pt](0.5,0.58)
	\psdot[dotsize=3pt](0.5,0.42)
}

\COBnewobject deathBirth|DB[1,1]<0,1>{%
	\COButurn(0,1,0){0.5}
	\COButurn(0,1,1){0.5}
}

\COBnewobject deathDotBirth|D*B[1,1]<0,1>{%
	\COButurn(0,1,0){0.5}
	\COButurn(0,1,1){0.5}
	\psdot[dotsize=3pt](0.5,0.2)
}

\COBnewobject deathBirthDot|DB*[1,1]<0,1>{%
	\COButurn(0,1,0){0.5}
	\COButurn(0,1,1){0.5}
	\psdot[dotsize=3pt](0.5,0.8)
}

\COBnewobject deathDotsBirth|D**B[1,1]<0,1>{%
	\COButurn(0,1,0){0.3}
	\COButurn(0,1,1){0.7}
	\pscircle(0.5,0.5){0.25}
	\psdot[dotsize=3pt](0.5,0.58)
	\psdot[dotsize=3pt](0.5,0.42)
}

\COBnewobject mergeDot|M*[2,1]<1,1>{%
	\COBsline(0,0)(1,1)%
	\COBsline(3,0)(2,1)%
	\COButurn(1,2,0){0.5}%
	\rput[c]{*0}(1.5,0.48){$\scriptstyle#1$}
	\psdot[dotsize=3pt](1.5,0.7)
}
\COBnewobject mergeDotFrRight|M*-R[2,1]<1,1>{%
	\COBsline(0,0)(1,1)%
	\COBsline(3,0)(2,1)%
	\COButurn(1,2,0){0.5}%
	\rput[c]{*0}(1.5,0.48){$\scriptstyle#1$}
	\psarrow(1.125,0.125)(1.875,0.125)%
	\psdot[dotsize=3pt](1.5,0.7)
}
\COBnewobject mergeDotFrLeft|M*-L[2,1]<1,1>{%
	\COBsline(0,0)(1,1)%
	\COBsline(3,0)(2,1)%
	\COButurn(1,2,0){0.5}%
	\rput[c]{*0}(1.5,0.48){$\scriptstyle#1$}
	\psarrow(1.875,0.125)(1.125,0.125)%
	\psdot[dotsize=3pt](1.5,0.7)
}
\COBnewobject mergeDotFrBack|M*-B[2,1]<1,1>{%
	\COBsline(0,0)(1,1)%
	\COBsline(3,0)(2,1)%
	\COButurn(1,2,0){0.5}%
	\pscarrow(1.1,0.5)(1.2,0.1)(1.6,0.1)(1.75,0.25)
	\rput[c]{*0}(1.5,0.48){$\scriptstyle#1$}
	\psdot[dotsize=3pt](1.5,0.7)
}
\COBnewobject mergeDotFrFront|M*-F[2,1]<1,1>{%
	\COBsline(0,0)(1,1)%
	\COBsline(3,0)(2,1)%
	\COButurn(1,2,0){0.5}%
	\pscarrow(1.75,0.25)(1.6,0.1)(1.2,0.1)(1.1,0.5)
	\rput[c]{*0}(1.5,0.48){$\scriptstyle#1$}
	\psdot[dotsize=3pt](1.5,0.7)
}

\COBnewobject dotCylinderVert|I*[1,1]<0,1>{%
	\psline(0,0)(0,1)%
	\psline(1,0)(1,1)%
	\psdot[dotsize=3pt](0.5,0.5)
}

\COBnewobject shortDotCylinderVert|sI*[1,1]<0,0.5>{%
	\psline(0,0)(0,0.5)%
	\psline(1,0)(1,0.5)%
	\psdot[dotsize=3pt](0.5,0.25)
}

\COBnewobject cutCyliderVert|cI[1,1]<0,1>{%
	\psline(0,0)(0,1)%
	\psline(1,0)(1,1)%
	\psline[linestyle=dashed,dash=2pt 2pt,linewidth=1pt,linecolor=red](-0.3,0.5)(1.3,0.5)%
}

\def\pictDeloopTR{\begin{pspicture}(0,0)(0.4,0.6)
	\psellipse(0.2,0.5)(0.2,0.07)
	\psbezier(0,0.5)(0,0)(0.4,0)(0.4,0.5)
	\psdot(0.2,0.3)
\end{pspicture}}
\def\pictDeloopTL{\begin{pspicture}(0,0)(0.4,0.6)
	\psellipticarc(0.2,0.1)(0.2,0.07){-180}{0}
	\psdash\psellipticarc(0.2,0.1)(0.2,0.07){0}{180}
	\psbezier(0.4,0.1)(0.4,0.6)(0,0.6)(0,0.1)
\end{pspicture}}
\def\pictDeloopBR{\begin{pspicture}(0,0)(0.4,0.6)
	\psellipse(0.2,0.5)(0.2,0.07)
	\psbezier(0,0.5)(0,0)(0.4,0)(0.4,0.5)
\end{pspicture}}
\def\pictDeloopBLa{\begin{centerpict}(0,0)(0.4,0.7)
	\psellipticarc(0.2,0.1)(0.2,0.07){-180}{0}
	\psdash\psellipticarc(0.2,0.1)(0.2,0.07){0}{180}
	\psbezier(0.4,0.1)(0.4,0.6)(0,0.6)(0,0.1)
	\psdot(0.2,0.3)
\end{centerpict}}
\def\pictDeloopBLb{\begin{centerpict}(0,0)(0.4,0.7)
	\psellipticarc(0.2,0.1)(0.2,0.07){-180}{0}
	\psdash\psellipticarc(0.2,0.1)(0.2,0.07){0}{180}
	\psbezier(0.4,0.1)(0.4,0.4)(0,0.4)(0,0.1)
	\pscircle(0.2,0.6){0.2}
	\psdot(0.2,0.52)
	\psdot(0.2,0.68)
\end{centerpict}}

\COBembeddedtrue
\def\COBxsize{0.4}
\def\COBysize{1.2}



\def\locusCanvas(#1,#2){\begingroup
	\psset{linewidth=0.5pt,arrowsize=0 8}%
	\rput(0,#2){\psline{->}(0,0)(0,0.5)\uput[l](0,0.5){$f_t$}}%
	\rput(#1,0){\psline{->}(0,0)(0.5,0)\uput[d](0.5,0){$t$}}%
	\psline{-}(0,#2)(0,0)(#1,0)
	\psline[linewidth=0.5pt,linestyle=dashed]{-}(#1,0)(#1,#2)(0,#2)
\endgroup}

\def\drawSaddle{%
	\begin{centerpict}(-1pt,-2pt)(21pt,18pt)%
		\psset{linewidth=0.5pt}%
		\psline(17pt,16pt)(17pt,5pt)
		\psbezier( 0pt, 5pt)( 9pt, 3pt)( 9pt, 3pt)(17pt, 5pt)
		\psline[border=1pt]( 3pt,11pt)( 3pt,0pt)
		\psbezier[border=1pt](20pt,11pt)(12pt,12pt)(11pt,15pt)(17pt,16pt)
		\psbezier( 3pt, 0pt)(11pt, 2pt)(11pt, 2pt)(20pt, 0pt)
		\psbezier( 3pt,11pt)( 9pt,12pt)( 8pt,15pt)( 0pt,16pt)
		\psline(17pt,16pt)(17pt,14pt)
		\psline( 0pt,16pt)( 0pt,5pt)
		\psline(20pt,11pt)(20pt,0pt)
		\psbezier(6.67pt,14pt)(7pt,5pt)(13pt,5pt)(13.33pt,14pt)
	\end{centerpict}%
}

\makeatletter
\def\cobSaddle#1#2{\begingroup
		\rput(0,#1){%
			\psbezier(0.0, 0.2)(0.4, 0.1)(0.4, 0.1)(0.2, 0)
			\psbezier(0.8, 0.2)(0.6, 0.1)(0.6, 0.1)(1.0, 0)
		}
		\psbezier(0.0, 0.2)(0.45, 0.1)(0.45, 0.1)(0.8, 0.2)
		\psclip{\psframe[linestyle=none,fillstyle=solid,fillcolor=white](0.2,0)(1,0.3)}
			\psdash\psbezier(0.0, 0.2)(0.45, 0.1)(0.45, 0.1)(0.8, 0.2)
		\endpsclip
		\psbezier(0.2, 0.0)(0.55, 0.1)(0.55, 0.1)(1.0, 0.0)
		\COBloady\COBdim@a{#1}%
		\COBdim@b=\COBdim@a \advance\COBdim@b 0.2\psyunit
		\COBdim@c=\COBdim@a \advance\COBdim@c 0.05136\psyunit
		\psline(0.0, 0.2)(0.0, \the\COBdim@b)
		\psline(0.2, 0.0)(0.2, \the\COBdim@a)
		\psdash\psline(0.8, 0.2)(0.8, \the\COBdim@c)
		\psline(0.8, \the\COBdim@c)(0.8, \the\COBdim@b)
		\psline(1.0, 0.0)(1.0, \the\COBdim@a)
		\COBdim@b=\COBdim@a	\advance\COBdim@b 0.1\psyunit
		\COBdim@c=\COBdim@a \advance\COBdim@c 0.13\psyunit
		\psline(0.33,\the\COBdim@b)(0.33,#2)
		\psline(0.67,\the\COBdim@c)(0.67,#2)
		\rput(0,#2){%
			\psbezier(0.33, 0)(0.33,-0.6)(0.67,-0.6)(0.67, 0)
			\psbezier[linewidth=0.3pt,border=0.5pt,arrowsize=3pt,arrowlength=1.2]%
						{<-}(0.55,-0.6)(0.53,-0.2)(0.5,-0.2)(0.45,-0.4)
		}
\endgroup}
\makeatother

\def\clap#1{\hbox to 0pt{\hss#1\hss}}
\def\mathllap{\mathpalette\mathllapinternal}
\def\mathrlap{\mathpalette\mathrlapinternal}
\def\mathclap{\mathpalette\mathclapinternal}
\def\mathllapinternal#1#2{\llap{$\mathsurround=0pt#1{#2}$}}
\def\mathrlapinternal#1#2{\rlap{$\mathsurround=0pt#1{#2}$}}
\def\mathclapinternal#1#2{\clap{$\mathsurround=0pt#1{#2}$}}

\newtheorem{theorem}{Theorem}[section]
\newtheorem{lemma}[theorem]{Lemma}
\newtheorem{conjecture}[theorem]{Conjecture}
\newtheorem{proposition}[theorem]{Proposition}
\newtheorem{corollary}[theorem]{Corollary}

\theoremstyle{definition}

\newtheorem{definition}[theorem]{Definition}
\newtheorem{example}[theorem]{Example}
\newtheorem{remark}[theorem]{Remark}

\definecolor{internalLink}{rgb}{0.5,0,0}
\definecolor{citeLink}{rgb}{0,0.5,0}
\definecolor{urlLink}{rgb}{0,0,0.5}

\hypersetup{linkcolor=internalLink}
\hypersetup{citecolor=citeLink}
\hypersetup{urlcolor=urlLink}

\DeclareMathOperator{\id}{id}		
\def\blank{\raisebox{0.3ex}{\underline{\ \ }}}		

\newcommand*{\Ob}{\mathrm{Ob}}
\newcommand*{\Mor}{\mathrm{Mor}}

\newcommand*{\OFun}{\mathrm{Fun}^{\mathrm{or}}}

\newcommand*{\FFun}{\mathrm{Fun}^{\mathrm{fr}}}

\newcommand*{\bS}{\mathbb{S}^1}			
\newcommand*{\Disk}{\mathbb{D}^2}		
\newcommand*{\Z}{\mathbb{Z}}				
\newcommand*{\R}{\mathbb{R}}				

\newcommand*{\Aut}{\mathrm{Aut}}

\newcommand*{\Hess}{\mathit{Hess}}

\newcommand*{\cone}{C}
\newcommand*{\Kom}{\textrm{\normalfont\textbf{Kom}}}

\newcommand*{\cat}[1]{\textrm{\normalfont\textbf{#1}}}		
\newcommand*{\catAdd}[1]{\mathrm{Mat}(#1)}
\newcommand*\Mod[1]{\cat{Mod}_{#1}}

\def\newcobcategory#1#2{%
	\expandafter\def\csname #1\endcsname{#2}
	\expandafter\def\csname #1L\endcsname{#2_{/\ell}}
	\expandafter\def\csname #1D\endcsname{#2_{\bullet}}
	\expandafter\def\csname k#1\endcsname{\scalars#2}
	\expandafter\def\csname k#1L\endcsname{\scalars#2_{/\ell}}
	\expandafter\def\csname k#1D\endcsname{\scalars#2_{\bullet}}
}

\newcobcategory{Cob}{\cat{Cob}}
\newcobcategory{ChCob}{\cat{ChCob}}
\newcobcategory{EmbChCob}{\cat{ChCob}^{e}}
\newcobcategory{kChCob}{\scalars\cat{ChCob}}

\newcommand*{\F}{\mathcal{F}}
\newcommand*{\FA}{\mathcal{F}_{\!A}}

\newcommand*{\Fev}{\mathcal{F}_{\!ev}}
\newcommand*{\Fodd}{\mathcal{F}_{\!odd}}
\newcommand*{\Fcov}{\mathcal{F}_{\!cov}}
\newcommand*{\FD}{\mathcal{F}_{\!\bullet}}

\newcommand*\SingLocus[1]{S(#1)}

\newcommand{\quotient}[2]{\raisebox{0.3\baselineskip}{$#1$}\!\Big/\!\raisebox{-0.3\baselineskip}{$#2$}}

\newcommand*{\vv}[2]{v_{#1}\otimes v_{#2}}

\def\vvvv#1,{v_{#1}\vvvvv}
\def\vvvvv#1,{%
	\ifx\endvvv#1\relax\else
		\otimes\ifx*#1\else v_{#1}\fi
	\expandafter\vvvvv\fi}

\newbox\fntbox

\makeatletter
\def\defineMathSymbol#1(#2)(#3)[#4]#5{%
	\expandafter\edef\csname fntMath#1\endcsname{\noexpand\mathchoice
			{\expandafter\noexpand\csname fntMath#1@draw\endcsname\noexpand\textfont\noexpand\displaystyle}%
			{\expandafter\noexpand\csname fntMath#1@draw\endcsname\noexpand\textfont\noexpand\textstyle}%
			{\expandafter\noexpand\csname fntMath#1@draw\endcsname\noexpand\scriptfont\noexpand\scriptstyle}%
			{\expandafter\noexpand\csname fntMath#1@draw\endcsname\noexpand\scriptscriptfont\noexpand\scriptscriptstyle}%
	}%
	\expandafter\def\csname fntMath#1@draw\endcsname##1##2{\begingroup
		\fntinkwidth=\fontdimen8##13\relax
		\setbox\fntbox=\hbox{$\m@th##2\mathrm{M}$}%
		\fntsize = \wd\fntbox
		\psset{unit=\fntsize,linewidth=\fntinkwidth}%
		\begin{pspicture}[shift=#4](#2)(#3)#5\end{pspicture}%
	\endgroup}
	\edef\next{\noexpand\DeclareMathOperator\expandafter\noexpand\csname#1\endcsname{\expandafter\noexpand\csname fntMath#1\endcsname}}\next
}
\makeatother

\defineMathSymbol ldsum(-0.2,0)(0.8,0.8)[0pt]{%
		\psset{arrowsize=0.4,arrowlength=0.8,arrowinset=0.5}%
		\psline{c-c}(0,0.7)(0,0.04)(0.6,0.04)(0.6,0.7)
		\psline{->}(0.0,0.55)(0.0,0.25)
		\psline{->}(0.6,0.25)(0.6,0.55)
}
\defineMathSymbol rdsum(-0.2,0)(0.8,0.8)[0pt]{%
		\psset{arrowsize=0.4,arrowlength=0.8,arrowinset=0.5}%
		\psline{c-c}(0,0.7)(0,0.04)(0.6,0.04)(0.6,0.7)
		\psline{->}(0.6,0.75)(0.6,0.25)
		\psline{->}(0.0,0.25)(0.0,0.55)
}

\defineMathSymbol connsum(0,-0.2)(0.7,0.75)[-0.2]{%
		\psline{c-c}(0.15,-0.15)(0.33,0.7)
		\psline{c-c}(0.37,-0.15)(0.55,0.7)
		\psline{c-c}(0.0,0.17)(0.7,0.17)
		\psline{c-c}(0.0,0.39)(0.7,0.39)
}

\defineMathSymbol lcsum(0,-0.2)(0.7,0.75)[-0.2]{%
		\pspolygon[fillcolor=black,fillstyle=solid](0.05,-0.15)(0.097,0.07)(-0.021,0.07)(-0.068,-0.15)
		\pspolygon[fillcolor=black,fillstyle=solid](0.721,0.49)(0.605,0.49)(0.65,0.7)(0.766,0.7)
		\psline{c-c}(0.15,-0.15)(0.33,0.7)
		\psline{c-c}(0.37,-0.15)(0.55,0.7)
		\psline{c-c}(0.0,0.17)(0.7,0.17)
		\psline{c-c}(0.0,0.39)(0.7,0.39)
}

\defineMathSymbol rcsum(0,-0.2)(0.7,0.75)[-0.2]{%
		\pspolygon[fillcolor=black,fillstyle=solid](0.23,0.71)(0.184,0.49)(0.021,0.49)(0.067,0.71)
		\pspolygon[fillcolor=black,fillstyle=solid](0.47,-0.16)(0.518,0.07)(0.679,0.07)(0.631,-0.16)
		\psline{c-c}(0.15,-0.15)(0.33,0.71)
		\psline{c-c}(0.37,-0.16)(0.55,0.7)
		\psline{c-c}(0.0,0.17)(0.7,0.17)
		\psline{c-c}(0.0,0.39)(0.7,0.39)
}

\makeatletter
\def\KaufBracket{\@ifstar\KaufBracketScaled\KaufBracketSimple}
\def\KhBracket{\@ifstar\KhBracketScaled\KhBracketSimple}
\def\KhCube{\@ifstar\KhCubeScaled\KhCubeSimple}
\def\KhCubeSigned{\@ifstar\KhCubeSignedScaled\KhCubeSignedSimple}
\def\KhCubeGraded{\@ifstar\KhCubeGradedScaled\KhCubeGradedSimple}
\def\KhCubeGradedSigned{\@ifstar\KhCubeGradedSignedScaled\KhCubeGradedSignedSimple}
\makeatother

\newcommand*{\KhCubeScaled}[1]{\mathcal{I}\left(#1\right)}
\newcommand*{\KhCubeSimple}[1]{\mathcal{I}(#1)}
\newcommand*{\KhCubeSignedScaled}[2]{\mathcal{I}^{#2}\left(#1\right)}
\newcommand*{\KhCubeSignedSimple}[2]{\mathcal{I}^{#2}(#1)}
\newcommand*{\KhCubeGradedScaled}[1]{\mathcal{I}_{\mathrm{gr}}\left(#1\right)}
\newcommand*{\KhCubeGradedSimple}[1]{\mathcal{I}_{\mathrm{gr}}(#1)}
\newcommand*{\KhCubeGradedSignedScaled}[2]{\mathcal{I}_{\mathrm{gr}}^{#2}\left(#1\right)}
\newcommand*{\KhCubeGradedSignedSimple}[2]{\mathcal{I}_{\mathrm{gr}}^{#2}(#1)}
\newcommand*{\KhBracketScaled}[1]{\left\llbracket#1\right\rrbracket}
\newcommand*{\KhBracketSimple}[1]{\llbracket#1\rrbracket}
\newcommand*{\KaufBracketScaled}[1]{\left\langle#1\right\rangle}
\newcommand*{\KaufBracketSimple}[1]{\langle#1\rangle}

\newcommand*{\KhCom}{K\!h}

\newcommand*{\Kh}{\mathcal{H}}
\newcommand*{\EKh}{\Kh_{ev}}
\newcommand*{\OKh}{\Kh_{odd}}

\newcommand*{\KhCov}{\Kh_{cov}}

\newcommand*{\DKh}{\Kh_{\bullet}}

\def\chdeg{\mathrm{chdeg}\mskip\thinmuskip}

\newcommand*{\CPO}{\mathcal{CPO}}

\def\permMM{X}
\def\permSS{Y}
\def\permMS{Z}
\def\permSM{Z^{-1}}
\def\permTpos{1}
\def\permTneg{\permMM\permSS}
\def\permT{\alpha}

\let\scalars\Bbbk
\def\invScalars{\scalars^*}

\def\Zev{\Z_{ev}}
\def\Zodd{\Z_{odd}}
\def\Zpi{\mathbb{Z}_{\pi}}

\makeatletter
\def\scalarsLong{\@ifstar
	{\mathbb{Z}[\permMM,\permSS,\permMS^{\pm1}]/(\permMM^2{=}\permSS^2{=}1)}%
	{\mathbb{Z}[\permMM,\permSS,\permMS^{\pm1}]/(\permMM^2=\permSS^2=1)}%
}
\def\ZevLong{\@ifstar{\Zpi/(\pi{-}1)}{\Zpi/(\pi-1)}}
\def\ZoddLong{\@ifstar{\Zpi/(\pi{+}1)}{\Zpi/(\pi+1)}}
\def\ZpiLong{\@ifstar{\mathbb{Z}[\pi]/(\pi^2{-}1)}{\mathbb{Z}[\pi]/(\pi^2-1)}}

\makeatother

\def\quot#1{`#1'}

\allowdisplaybreaks

\begin{document}

\title{A 2-category of chronological cobordisms and odd Khovanov homology}

\author[Krzysztof K. Putyra]{Krzysztof K. Putyra}
\address{Department of Mathematics, Columbia University\\ New York, NY 10027}
\email{putyra@math.columbia.edu}

\begin{abstract}
	We create a~framework for odd Khovanov homology in the~spirit of Bar-Natan's construction for the~ordinary Khovanov homology. Namely, we express the~cube of resolutions of a~link diagram as a~diagram in a~certain $2$-category of chronological cobordisms and show that it is $2$-commutative: the~composition of $2$-morphisms along any $3$-dimensional subcube is trivial. This allows us to create a~chain complex, whose homotopy type modulo certain relations is a~link invariant. Both the~original and the~odd Khovanov homology can be recovered from this construction by applying certain strict $2$-functors. We describe other possible choices of functors, including the~one that covers both homology theories and another generalizing dotted cobordisms to the~odd setting. Our construction works as well for tangles and is conjectured to be functorial up to sign with respect to tangle cobordisms.
\end{abstract}

\keywords{%
	categorification,
	chronology,
	cobordism,
	covering homology,
	dotted cobordisms,
	Frobenius algebra,
	functoriality,
	Khovanov homology,
	knot,
	link,
	odd homology,
	planar algebra,
	tangle%
}

\maketitle

\section{Introduction}\label{sec:intro}
The~Khovanov homology \cite{KhCatJones} opened to knot theorists a~new and interesting world of powerful invariants, of which knot polynomials are only shadows. For instance, the~Euler characteristic of the~Khovanov homology is the~Jones polynomial of a~link. It did not take much time to prove usefulness of these invariants. For instance, the~Lee deformation of the~Khovanov's differential \cite{Lee} leads to a~spectral sequence, from which J.~Rasmussen extracted a~lower bound for the~knot genus, giving a~combinatorial proof of Milnor Conjecture \cite{RasmGenus}. Moreover, the~Khovanov homology detects the~unknot \cite{KhDetectsUnknot} and unlinks \cite{KhDetectsUnlinks}, although the~question whether the~Jones polynomial is an~unknot detector is still open. This raised a~question, whether there were other link homology theories categorifying the~Jones polynomial. In particular, D.~Bar-Natan \cite{DrorCobs} described a~very general construction that produces link homology for rank two Frobenius algebras satisfying some additional relations. Then M.~Khovanov classified all theories that arise from Frobenius systems \cite{KhFrobExt}, proving that the~Bar-Natan's theory of dotted cobordisms is universal.

When it seemed that categorifications of the~Jones polynomial were well understood, P.~Ozsv\'ath, J.~Rasmussen and Z.~Szab\'o published a~paper with a~distinct construction \cite{ORS} based on a~\emph{projective} TQFT. Their invariant also categorifies the~Jones polynomial, but the~algebra used in the~construction is not cocommutative and even not coassociative. They call it \emph{odd Khovanov homology}, because of similarity with the~original construction, which we now refer to as \emph{even}. Both theories agree modulo $2$, but they are not equivalent over $\Z$. In particular, results of A.~Shumakovitch \cite{ShumComp} provide examples of pairs of knots that can be distinguished by one theory but not by the~other. Moreover, it was proved by J.~Bloom that the~odd Khovanov homology is mutation invariant \cite{Bloom}, generalizing the~similar result by S.~Wehrli for even Khovanov homology with $\Z_2$ coefficients \cite{Wehrli}. On the~other hand, the~even Khovanov homology detects mutant links, but the~problem is still open for knots.

Both theories are obtained from the~cube of resolutions of a~link diagram. Namely, given a~link diagram $D$ with $n$ crossings we create $2^n$ pictures, by resolving each crossing horizontally (type $0$ resolution) or vertically (type $1$ resolution):
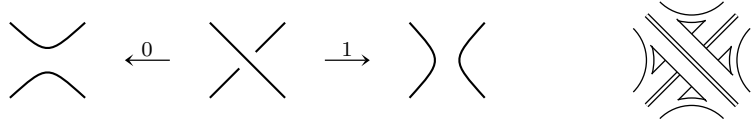
\begin{equation*}
	\psset{unit=0.5}
	\begin{centerpict}(-1,-1)(1,1)
		\psbezier(-1,-1)(0,-0.1)(0,-0.1)(1,-1)
		\psbezier(-1, 1)(0, 0.1)(0, 0.1)(1, 1)
	\end{centerpict}
		\quad\to/<-/^0\quad
	\begin{centerpict}(-1,-1)(1,1)
		\psline(-1,-1)(1,1)
		\psline[border=5\pslinewidth](-1,1)(1,-1)
	\end{centerpict}
		\quad\to^1\quad
	\begin{centerpict}(-1,-1)(1,1)
		\psbezier(-1,-1)(-0.1,0)(-0.1,0)(-1,1)
		\psbezier( 1,-1)( 0.1,0)( 0.1,0)( 1,1)
	\end{centerpict}
		\hskip 2cm
	\begin{centerpict}(-1.5,-1.5)(1.5,1.5)
		\psset{linewidth=0.5pt}
		\psline[doubleline=true,doublesep=1pt](-1.2,-1.2)(1.2, 1.2)
		\psline(0.35,1.05)(-1.05,-0.35)
		\psline(-0.35,-1.05)(1.05,0.35)
		\psline[linecolor=white,linewidth=1](-1,1)(1,-1)
		\psline[doubleline=true,doublesep=1pt](-1.2, 1.2)(1.2,-1.2)
		\psline(-0.4,1.1)(1.1,-0.4)
		\psline(0.4,-1.1)(-1.1,0.4)
		\psarc(-2.4,0){1.197}{-45}{45}
		\psarc( 2.4,0){1.197}{135}{225}
		\psarc(0,-2.4){1.197}{ 45}{135}
		\psarc(0, 2.4){1.197}{225}{315}
		\psarc{c-c}(-2.4,0){1.395}{-15}{ 15}
		\psarc{c-c}( 2.4,0){1.395}{165}{195}
		\psarc{c-c}(0,-2.4){1.395}{ 75}{105}
		\psarc{c-c}(0, 2.4){1.395}{255}{285}
	\end{centerpict}
\end{equation*}
The~picture of crossing highways is placed to the~right to help to remember the~naming convention: a~resolution of a~crossing can be seen as leaving one highway by turning right (assuming the~traffic is on the~right side). In type $0$ we leave the~lower highway, while in type $1$ the~upper one. We place all such pictures in vertices of an $n$-dimensional cube and decorate its edges with certain cobordisms. This cube commutes and by applying a~TQFT functor we obtain a~commuting cube of abelian groups and homomorphisms, which can be collapsed to a~chain complex (after changing signs of some maps). On the~other hand, a~projective TQFT from \cite{ORS} produces a~cube that commutes only up to signs, which has to be fixed before collapsing. %
It is a~kind of mystery, why this is possible.

\wrapfigure[r]<12>{%
	\psset{unit=1cm}%
	\begin{pspicture}(-2.4,-1ex)(2.4,2)
		\diagnode cov( 0,1.55)[\Kh(L)]
		\diagnode  ev(-1.7,0)[\EKh(L)]
		\diagnode odd( 1.7,0)[\OKh(L)]
		\diagarrow|b{npos=0.6,labelsep=3pt}|[cov`ev;\permMM,\permSS,\permMS\,\mapsto<0.7em>\mathrlap1]
		\diagarrow|a{npos=0.65,labelsep=3pt}|[cov`odd;\substack{%
					\ \mathllap{\permMM,\permMS\,}\mapsto<0.7em> 1\phantom{-}\\
					\ \mathllap{\permSS\,}\mapsto<0.7em>-1}]
	\end{pspicture}}%
The~last step is exactly why the~odd theory does not fit into Bar-Natan's framework. The~latter starts with a~cube of resolutions and cobordisms, and invariance is proved at this level, before applying a~TQFT functor. The~author extended this framework \cite{Putyra} using cobordisms with an~additional structure, a~\emph {chronology}, which is a~framed Morse function $\tau\colon W\to I$ that separates critical points \cite{IgusaFMF}. Isotopies of these functions equip the~category of chronological cobordisms with a~structure of a~$2$-category and we can express the~projective functor from \cite{ORS} as a~strict $2$-functor. By translating Bar-Natan's construction into this new framework, we were able to show invariance of the~complex built from chronological cobordisms. Applying different $2$-functors recovers both the~even and odd Khovanov homology. In~particular, it follows from contractibility of certain loops in the~space of framed functions that in the~odd theory we can always distribute signs over edges of the~cube to make it commute. In addition to that, we have found several theories with parameters, especially the~\emph{covering homology} $H^{cov}(L)$. It is a~sequence of graded modules over the~ring of truncated polynomials $\Z[\permMM,\permSS,\permMS^{\pm1}]/(\permMM^2=\permSS^2=1)$, from which we can obtain both even and odd Khovanov homology as illustrated to the~right. The~specializations take place at the~level of chains. This construction was first described in \cite{Putyra}. Another example is given by chronological cobordisms with dots that generalizes the~universal Bar-Natan's theory to the~odd setting. By an~analogy to the~even case it is proved to be universal, see Theorem~\ref{thm:universality}. A~motivation was to find an~odd analog of Lee's deformation, the~goal that has not been reached.

\subsection*{A~connection with categorified quantum groups}
The~existence of covering homology theory fits nicely with recent discoveries regarding odd categorifications of quantum groups. It is known that the~even Khovanov homology can be recovered from categorical representations of categorified $U_q(\mathfrak{sl}_2)$ \cite{Webster}. A~recent discovery of odd nilHecke algebras \cites{EllisLaudaKhov,KangKashTsu}, which categorifies the~negative half of $U_q(\mathfrak{sl}_2)$, suggests the~existence of the~odd Khovanov homology may also possess a~representation-theoretical explanation.\linebreak

\noindent\wrapfigure[r]<3>{\begin{diagps}(-2.4,-0.2)(2.6,2)
		\node Uqp(0,1.5)[U_{q,\pi}]
		\node sl2(-1.7,0)[U_q(\mathfrak{sl}_2)]
		\node osp12( 1.7,0)[U_q(\mathfrak{osp}_{1|2})]
		\arrow|b|[Uqp`sl2;\pi=1]
		\arrow|a|[Uqp`osp12;\pi=-1]
	\end{diagps}}
The~odd nilHecke algebras appeared to be connected with the~Lie superalgebra $U_q(\mathfrak{osp}_{1|2})$. Both $U_q(\mathfrak{sl}_2)$ and $U_q(\mathfrak{ops}_{1|2})$ are covered by a~Kac-Moody algebra $U_{q,\pi}$ introduced by S.~Clark, D.~Hill and W.~Wang \cites{ClarkHillWang,HillWang}, where $\pi$ is a~formal parameter with $\pi^2=1$. The~relationship is illustrated in the~diagram to the~right. Recently, A.~Lauda and A.~Ellis categorified the~covering algebra $U_{q,\pi}$ using graded supermonoidal categories, in which the~relation $(f\otimes\id)\circ(\id\otimes g) = (\id\otimes g)\circ(f\otimes\id)$ holds up to a~sign in a~coherent way \cite{oddSl2}. It is expected that this categorification leads to homologies covering both odd and even homology theories and the~author believes that the~covering Khovanov homology described in this paper is one of them.

\subsection*{Outline}
We start the~paper with a~picture visualizing the~construction of the~Khovanov complex for the~trefoil knot, see Fig.~\ref{diag:khov-cube}. We hope it will serve as a~motivation for the~next two sections, in which we define chronological cobordisms and analyze changes of chronologies. Section~\ref{sec:chron} describes the~$2$-category of chronological cobordisms of any dimension and explains a~symmetric monoidal structure induced by a~disjoint union. The~section ends with a~detailed description of the~two-dimensional case. A~refined version of chronological cobordisms embedded in $\Disk\times I$ is described in Section~\ref{sec:cob-linear}, together with a~solution for \emph{chronological relations}: permuting critical points corresponds to scaling a~cobordism by an~invertible scalar.

Details of the~construction of the~generalized Khovanov complex for a~tangle diagram are given in Sections~\ref{sec:complex} and \ref{sec:tangles}. The~former deals with link diagrams only, whereas the~latter describes how to extend the~construction to tangles in the~spirit of Bar-Natan, using planar algebras. Unfortunately, the~functors forming a~planar algebra of chronological cobordisms are not strict, so that we cannot combine complexes for tangles in the~naive way. This issue is partially resolved in Section~\ref{sec:invariance}, where we prove invariance of the~generalized complex under Reidemeister moves. Section \ref{sec:properties} contains several straightforward properties of the~complex.

The~next few sections are devoted to computation of the~homology of the~complex. We recover both odd and even Khovanov homology from our construction in Section~\ref{sec:homology}. The~covering homology is defined in Section~\ref{sec:chron-Frob}, in which we define a~chronological version of a~Frobenius system. Similarly to the~ordinary case, a~chronological Frobenius system induces a~TQFT $2$-functor from the~category of chronological cobordisms to a~$2$-category of graded symmetric bimodules. They are analyzed in the~next section. In particular, we describe dotted chronological cobordisms and their algebra, proving it is universal among all Frobenius systems fitting in our framework.

Section~\ref{sec:odds-ends} contains several remarks and constructions related to this paper, but not fully explored. We discuss, following \cite{DrorCobs}, functoriality up to \quot{sign} of our construction, where a~\quot{sign} is understood as an~invertible scalar in degree $0$. Then we analyze a~choice in defining chronological relations: there is one type of changes for which the~associated coefficients are defined only up to a~scalar $\permMM\permSS$, although the~whole construction is independent of this choice. Finally, we analyze a~possible connection of our construction to the~one based on $\mathfrak{sl}_2$ foams \cite{Caprau}. We suppose there is a~parallel theory of chronological foams, closely connected to our construction.

The~construction of chronological cobordisms utilizes the~theory of framed functions, which is an~interesting enrichment of Morse theory. It is described in Appendix~\ref{sec:framed-function} following \cite{IgusaFMF}. In particular, we describe all singularities of these functions up to codimension two.

The~paper uses also several $2$-categorical constructions, including semi-strict monoidal structure and braiding. These are explained in Appendix~\ref{sec:2-cats}.

\subsection*{Acknowledgement}
This paper would have never been written without help of many people. The~problem of creating a~framework for odd Khovanov homology was suggested by Dror Bar-Natan while the~author was at the~University of Toronto. The~dotted algebra was understood with the~help of Anna Beliakova, when the~author visited her in Z\"urich. Several ideas used to clarify the~original construction came out after discussions with Aaron Lauda, Robert Lipshitz, Mikhail Khovanov, J\'ozef Przytycki and Alexander Shumakovitch. The~author is also thankful to Alexander Ellis, Maria Hempel, Vasilly Manturov, Cotton Seed and Joshua Batson for interesting discussions and remarks. Further, the~author would like to thank the~referees for their insightful remarks on an~earlier version of this paper.

\section{The~picture}\label{sec:cube-picture}
We begin with describing elements of the~big diagram in Fig.~\ref{diag:khov-cube}. In the~next few sections we shall create a~framework for this picture.

\begin{figure}[b]
	\input{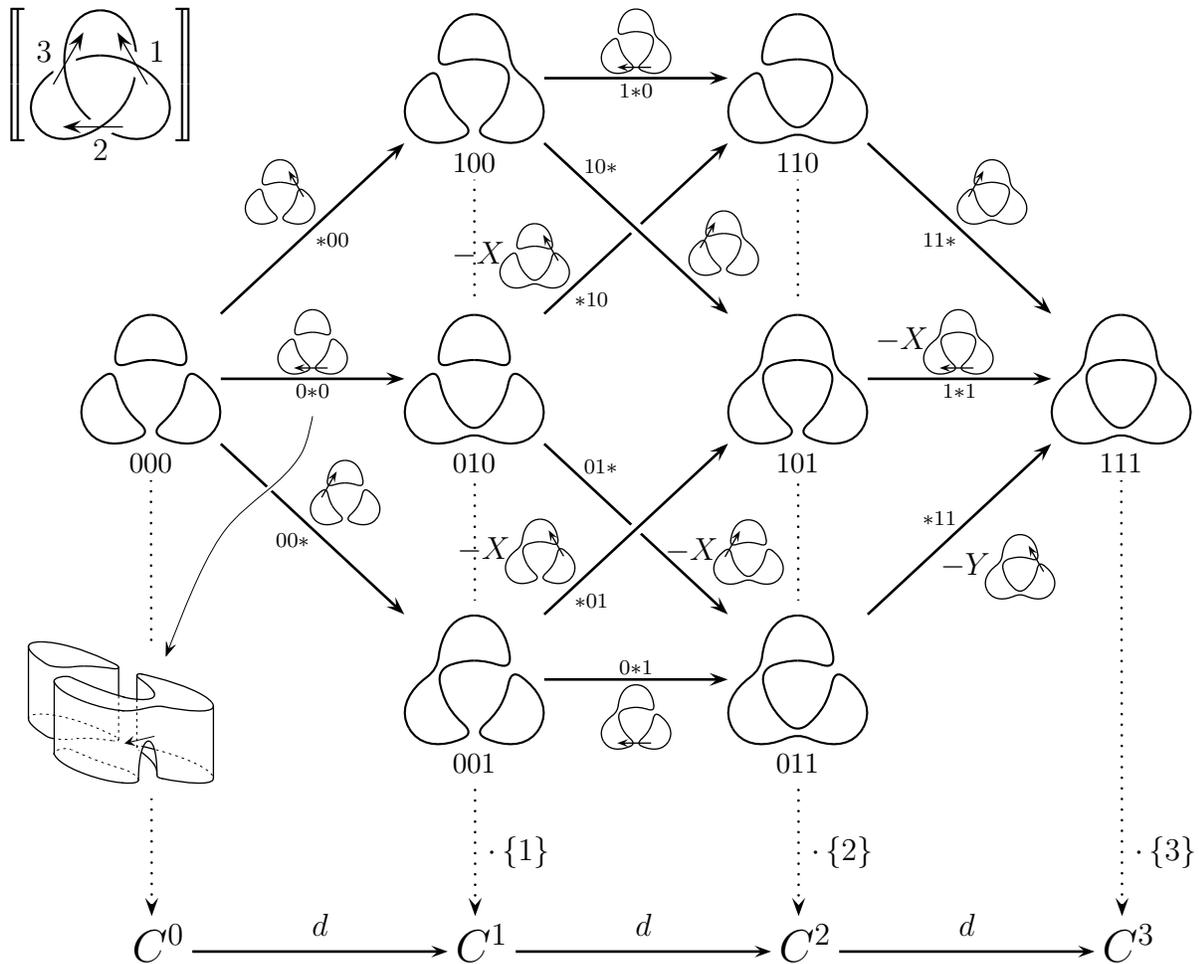}%
	\caption{The Khovanov bracket for the~trefoil.}\label{diag:khov-cube}
\end{figure}

\subsection*{Knot}
In the~left top corner we can see a~diagram $D$ of the~left-handed trefoil with enumerated crossings. Each crossing is equipped with an~arrow oriented in such a~way that it connects the two arcs in the~type~$0$ resolution at this crossing (there are two choices of such an~arrow). These arrows do not appear in the~construction of the~even Khovanov complex \cites{KhCatJones,DrorCobs}, but it is crucial for the~odd Khovanov complex \cite{ORS}.

\subsection*{Vertices in the~cube}
Most of the picture is occupied by resolutions of the~diagram $D$, placed in vertices of a~unit three-dimensional cube. A~vertex $\xi$ of the~cube is encoded by a~sequence $(\xi_1,\xi_2,\xi_3)$ with each $\xi_i=0$ or $1$, and it is decorated with a~diagram $D_\xi$ obtained from $D$ by replacing $i$-th crossing with the~resolution of type $\xi_i$. The~cube is drawn slant, to have all diagrams grouped in columns with respect to the~weight of the~vertex $\|\xi\|:=\xi_1+\xi_2+\xi_3$.

\subsection*{Edges in the~cube}
Edges are encoded by sequences $\zeta=(\zeta_1,\zeta_2,\zeta_3)$ with exactly one $\zeta_i$ being a~star $*$. The~star indicates direction of the~edge: replacing it with $0$ or $1$ results in the~source or the~target vertex respectively. Choose an~edge $\zeta\colon\xi\to\xi'$ and let $U$ be a~small neighborhood of the~$i$-th crossing, where $\zeta_i=*$. It is decorated with a~unique cobordism $D_\zeta\subset\mathbb{R}^2\times I$ that has only one critical point: a~cylinder $(D - U)\times I$ with a~saddle \drawSaddle{} inserted over $U$. We equip this cobordism with a~height function $h\colon D_\zeta\to I$. The~small arrow over the~crossing determines the~\emph{framing}, hence the~\emph{orientation} of the~saddle (see Appendix~\ref{sec:framed-function}). For simplicity we represent the~cobordism by its input together with an~arrow, which determines both the~place and the~orientation of the~saddle. This is the~same arrow that decorates the~$i$-th crossing in the~diagram of the~knot. A~3D picture of the~cobordism decorating the~edge $0{*}0$ is given in the~left-bottom corner.

\subsection*{An~underlying diagram with holes}
\wrapfigure[r]<0>{%
	\psset{unit=1cm}%
	\begin{pspicture}(-1,-0.8)(1,1.1)\trefoilarcdiag{1}{2}{3}\end{pspicture}%
}
The~two paragraphs above can be unified by a~single construction, which also explains how to create the~cube for any link diagram $D$. Take the~diagram $D$ and remove a~small neighborhood of \emph{each} crossing, obtaining a~new diagram $D_\bullet$.\footnote{
	This is an~example of a~planar arc diagram, see Section~\ref{sec:tangles}.
}
For instance, the~trefoil diagram produces a~diagram with three holes. Copy the~numbers associated to crossings to the~holes---this gives an~ordering of them. The~picture $D_\xi$ at a~vertex $\xi$ is obtained from $D_\bullet$ by filling the~holes with resolutions, type $\xi_i$ at the~$i$-th hole. To obtain the~cobordism $D_\zeta$ associated to an~edge $\zeta$, where $\zeta_i=*$, copy the~arrow from $i$-th crossing to the~$i$-th hole. For a~3D picture, take a~product $D_\bullet\times I$ and insert into the~$i$-th hole\footnote{
	Notice that holes in $D_\bullet\times I$ are three-dimensional.
}
either a~pair of vertical rectangles, when $\zeta_i\neq *$, or a~saddle for $\zeta_i=*$ with a~framing induced by the~small arrow over the~crossing, see Fig.~\ref{fig:cob-resol}.

\begin{figure}[h]%
	\rnode[c]{crossing}{\begin{pspicture}(-0.1,-0.1)(1.1,1.1)
		\psline(0,0)(1,1)
		\psline[border=5\pslinewidth](0,1)(1,0)
		\psline[linewidth=0.5pt,border=0.5pt,arrowsize=3pt,arrowlength=1.2]{->}(0.5,0.8)(0.5,0.2)
	\end{pspicture}}
	
	\vskip 1.5cm
	\psset{yunit=1.5cm,xunit=1.2cm,linewidth=0.5pt,nodesep=3pt,dash=1pt 1.5pt}%
	\rnode[t]{r0}{\begin{pspicture}(0,-0.2)(1,1)
		\psbezier(0.0, 0.2)(0.45, 0.1)(0.45, 0.1)(0.8, 0.2)
		\psbezier(0.0, 1.0)(0.45, 0.9)(0.45, 0.9)(0.8, 1.0)
		\psline(0.8, 0.2)(0.8, 1.0)
		\psline(0.0, 0.2)(0.0, 1.0)
		\psclip{\psframe[linestyle=none,fillstyle=solid,fillcolor=white](0.2,0)(1,0.843)}
			\psdash\psbezier(0.0, 0.2)(0.45, 0.1)(0.45, 0.1)(0.8, 0.2)
			\psdash\psline(0.8, 0.2)(0.8, 1.0)
		\endpsclip
		\psbezier(0.2, 0.8)(0.55, 0.9)(0.55, 0.9)(1.0, 0.8)
		\psbezier(0.2, 0.0)(0.55, 0.1)(0.55, 0.1)(1.0, 0.0)
		\psline(0.2, 0.0)(0.2, 0.8)
		\psline(1.0, 0.0)(1.0, 0.8)
	\end{pspicture}}
	\hskip 1.5cm
	\rnode[t]{r*}{\begin{pspicture}(0,-0.2)(1,1)
		\psbezier(0.0, 1.0)(0.4, 0.9)(0.4, 0.9)(0.2, 0.8)
		\psbezier(0.8, 1.0)(0.6, 0.9)(0.6, 0.9)(1.0, 0.8)
		\psbezier(0.0, 0.2)(0.45, 0.1)(0.45, 0.1)(0.8, 0.2)
		\psclip{\psframe[linestyle=none,fillstyle=solid,fillcolor=white](0.2,0)(1,0.3)}
			\psdash\psbezier(0.0, 0.2)(0.45, 0.1)(0.45, 0.1)(0.8, 0.2)
		\endpsclip
		\psbezier(0.2, 0.0)(0.55, 0.1)(0.55, 0.1)(1.0, 0.0)
		\psline(0.0, 0.2)(0.0, 1.0)
		\psline(0.2, 0.0)(0.2, 0.8)
		\psdash\psline(0.8, 0.2)(0.8, 0.85136)
		\psline(0.8, 0.85136)(0.8, 1.0)
		\psline(1.0, 0.0)(1.0, 0.8)
		\psbezier(0.33, 0.9)(0.33, 0.3)(0.67, 0.3)(0.67, 0.93)
		\psbezier[linewidth=0.3pt,border=0.5pt,arrowsize=3pt,arrowlength=1.2]%
					{<-}(0.55, 0.3)(0.53, 0.7)(0.5, 0.7)(0.45, 0.5)
	\end{pspicture}}
	\hskip 1.5cm
	\rnode[t]{r1}{\begin{pspicture}(0,-0.2)(1,1)
		\psbezier(0.0, 0.2)(0.4, 0.1)(0.4, 0.1)(0.2, 0.0)
		\psbezier(0.0, 1.0)(0.4, 0.9)(0.4, 0.9)(0.2, 0.8)
		\psclip{\psframe[linestyle=none,fillstyle=solid,fillcolor=white](0.2,0.1)(0.4,0.2)}
			\psdash\psbezier(0.0, 0.2)(0.4, 0.1)(0.4, 0.1)(0.2, 0.0)
		\endpsclip
		\psline(0.0, 0.2)(0.0, 1.0)
		\psline(0.2, 0.0)(0.2, 0.8)
		\psline(0.33, 0.9)(0.33, 0.1)
		\psbezier(0.8, 0.2)(0.6, 0.1)(0.6, 0.1)(1.0, 0.0)
		\psbezier(0.8, 1.0)(0.6, 0.9)(0.6, 0.9)(1.0, 0.8)
		\psclip{\psframe[linestyle=none,fillstyle=solid,fillcolor=white](0.6,0.11)(1,0.3)}
			\psdash\psbezier(0.8, 0.2)(0.6, 0.1)(0.6, 0.1)(1.0, 0.0)
		\endpsclip
		\psline(0.67, 0.91)(0.67, 0.11)
		\psdash\psline(0.8, 0.2)(0.8, 0.85136)
		\psline(0.8, 0.85136)(0.8, 1.0)
		\psline(1.0, 0.0)(1.0, 0.8)
	\end{pspicture}}%
	\psset{arrowsize=5pt,arrowlength=1}
	\ncline{->}{crossing}{r0}\nbput {$\scriptstyle \zeta_i=0$}%
	\ncline{->}{crossing}{r*}\ncput*{$\scriptstyle \zeta_i=*$}%
	\ncline{->}{crossing}{r1}\naput {$\scriptstyle \zeta_i=1$}%
	\caption{3D resolutions of a~crossing decorated by an~arrow.}%
	\label{fig:cob-resol}%
\end{figure}

\subsection*{Faces}
Consider a~two-dimensional face $S$ of the~cube of resolutions of a~diagram $D$ with $n$ crossings. It is encoded by a~sequence $\nu=(\nu_1,\dots,\nu_n)$ having $0$ or $1$ at all except two positions, $i_0$ and $i_1$, where we put stars. Denote the~vertices of $S$ by $S_{ab}$, where $a,b$ are the~replacements for the~two stars. We label $S$ with the~resolution of $D$ decorating the~vertex $S_{00}$ with two arrows placed over smoothings of the~$i_0$-th and $i_1$-th crossing. This is a~surgery diagram of a~cobordism with two saddles, and there are two height functions depending on which of the~two saddle points is below the~other. Instead of picking any of them, we interpret the~diagram as a~linear homotopy between the~two height functions, see Fig.~\ref{fig:surgery-description-of-a-change} for an~example.

\begin{figure}[t]
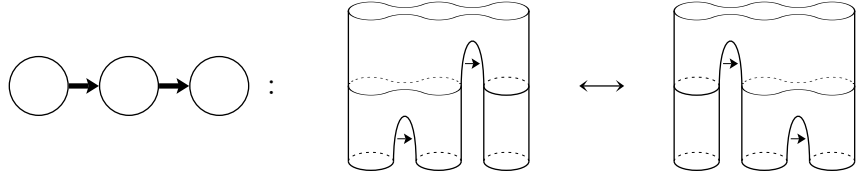

	\centering
	\psset{linewidth=0.5pt,arrowlength=0.8,arrowinset=0.2,arrowsize=4pt,dimen=middle,dash=1pt 1.5pt}%
	\begin{centerpict}(-1.7,-0.5)(1.7,0.5)
		\diagConnMM{}{->}{}{->}
	\end{centerpict}
	$:\qquad$
	\begin{centerpict}(0,0.3)(2.4,2.7)
		\rput(0,0.5){
			\psellipticarc(0.3,0)(0.3,0.12){-180}{0}
			\psellipticarc(1.2,0)(0.3,0.12){-180}{0}
			\psellipticarc(2.1,0)(0.3,0.12){-180}{0}
			\psdash\psellipticarc(0.3,0)(0.3,0.12){0}{180}
			\psdash\psellipticarc(1.2,0)(0.3,0.12){0}{180}
			\psdash\psellipticarc(2.1,0)(0.3,0.12){0}{180}
		}
		\rput(0,1.5){
			\psellipticarc(2.1,0)(0.3,0.12){-180}{0}
			\psdash\psellipticarc(2.1,0)(0.3,0.12){0}{180}
			\pscustom{%
				\scale{1 0.4}%
				\psarc(0.3,0){0.3}{180}{315}
				\psarcn(0.75,-0.45){0.33639}{135}{45}
				\psarc(1.2,0){0.3}{225}{360}
			}
			\psdash\pscustom{%
				\scale{1 0.4}%
				\psarc(1.2,0){0.3}{0}{135}
				\psarcn(0.75, 0.45){0.33639}{315}{225}
				\psarc(0.3,0){0.3}{ 45}{180}
		}}
		\rput(0,2.5){\pscustom{%
			\scale{1 0.4}%
			\psarc(0.3,0){0.3}{ 45}{315}
			\psarcn(0.75,-0.45){0.33639}{135}{45}
			\psarc(1.2,0){0.3}{225}{315}
			\psarcn(1.65,-0.45){0.33639}{135}{45}
			\psarc(2.1,0){0.3}{225}{135}
			\psarcn(1.65, 0.45){0.33639}{315}{225}
			\psarc(1.2,0){0.3}{ 45}{135}
			\psarcn(0.75, 0.45){0.33639}{315}{225}
		}}
		\psline(0,0.5)(0,2.5)
		\psbezier(0.6,0.5)(0.6,1.3)(0.9,1.3)(0.9,0.5)
		\psline(1.5,0.5)(1.5,1.5)
		\psbezier(1.5,1.5)(1.5,2.3)(1.8,2.3)(1.8,1.5)
		\psline(1.8,0.5)(1.8,1.5)
		\psline(2.4,0.5)(2.4,2.5)
		\psline[linewidth=0.3pt]{->}(0.65,0.8)(0.85,0.8)
		\psline[linewidth=0.3pt]{->}(1.55,1.8)(1.75,1.8)
	\end{centerpict}
	$\quad\to/<->/\quad$
	\begin{centerpict}(0,0.3)(2.4,2.7)
		\rput(0,0.5){
			\psellipticarc(0.3,0)(0.3,0.12){-180}{0}
			\psellipticarc(1.2,0)(0.3,0.12){-180}{0}
			\psellipticarc(2.1,0)(0.3,0.12){-180}{0}
			\psdash\psellipticarc(0.3,0)(0.3,0.12){0}{180}
			\psdash\psellipticarc(1.2,0)(0.3,0.12){0}{180}
			\psdash\psellipticarc(2.1,0)(0.3,0.12){0}{180}
		}
		\rput(0,1.5){
			\psellipticarc(0.3,0)(0.3,0.12){-180}{0}
			\psdash\psellipticarc(0.3,0)(0.3,0.12){0}{180}
			\pscustom{%
				\scale{1 0.4}%
				\psarc(1.2,0){0.3}{180}{315}
				\psarcn(1.65,-0.45){0.33639}{135}{45}
				\psarc(2.1,0){0.3}{225}{360}
			}
			\psdash\pscustom{%
				\scale{1 0.4}%
				\psarc(2.1,0){0.3}{0}{135}
				\psarcn(1.65, 0.45){0.33639}{315}{225}
				\psarc(1.2,0){0.3}{ 45}{180}
		}}
		\rput(0,2.5){\pscustom{%
			\scale{1 0.4}%
			\psarc(0.3,0){0.3}{ 45}{315}
			\psarcn(0.75,-0.45){0.33639}{135}{45}
			\psarc(1.2,0){0.3}{225}{315}
			\psarcn(1.65,-0.45){0.33639}{135}{45}
			\psarc(2.1,0){0.3}{225}{135}
			\psarcn(1.65, 0.45){0.33639}{315}{225}
			\psarc(1.2,0){0.3}{ 45}{135}
			\psarcn(0.75, 0.45){0.33639}{315}{225}
		}}
		\psline(0,0.5)(0,2.5)
		\psline(0.6,0.5)(0.6,1.5)
		\psbezier(0.6,1.5)(0.6,2.3)(0.9,2.3)(0.9,1.5)
		\psline(0.9,0.5)(0.9,1.5)
		\psbezier(1.5,0.5)(1.5,1.3)(1.8,1.3)(1.8,0.5)
		\psline(2.4,0.5)(2.4,2.5)
		\psline[linewidth=0.3pt]{->}(1.55,0.8)(1.75,0.8)
		\psline[linewidth=0.3pt]{->}(0.65,1.8)(0.85,1.8)
	\end{centerpict}
	\caption{A~two-arrow surgery diagram encodes a~permutation of two saddle points.}
	\label{fig:surgery-description-of-a-change}
\end{figure}

\subsection*{Commutativity cochain}
Take the~two-arrow description of a~face and disregard all circles that are not touched by any of the~two arrows. What remains is one of the~pictures listed in Tab.~\ref{tab:cube-faces}. We gathered all such configurations into groups labeled with some monomials from a~commutative ring $\scalars:=\Z[\permMM,\permSS,\permMS^{\pm1}]/(\permMM^2{=}\permSS^2{=}1)$ (they are explained in Section~\ref{sec:cob-linear}). They define a~$2$-cochain $\psi\in C^2(I^n; \invScalars)$, where $\invScalars$ is the~group of invertible elements in $\scalars$. Here one must be careful with the~two configurations placed in Tab.~\ref{tab:cube-faces} below the~letter $\permMS$---the~value of $\psi$ is either $\permMS$ or $\permSM$, depending on the~orientation of the~face:
\begin{equation}\label{eq:SM-group-explained}
	\psset{nrot=:U,labelsep=3pt}
	\psi\left(
	\begin{diagps}(-0.1,-0.0)(2.1,2.2)
		\node tl(0,1)[\bullet]	\node tr(1,2)[\bullet]
		\node bl(1,0)[\bullet]	\node br(2,1)[\bullet]
		\arrow[nodesep=1pt]|a{labelsep=3pt}|{->}[tl`tr;\textrm{merge}]
		\arrow[nodesep=1pt]|a{labelsep=3pt}|{->}[tr`br;\textrm{split}]
		\arrow[nodesep=1pt]|b{labelsep=3pt}|{->}[tl`bl;\textrm{split}]
		\arrow[nodesep=1pt]|b{labelsep=3pt}|{->}[bl`br;\textrm{merge}]
		\psset{linewidth=0.5pt,arrowsize,arrowsize=4pt 1.2,arrowlength=0.7,arrowinset=0.5}
		\psarcn(1,1.1){0.2}{135}{-135}\psline{<-}(0.85858,0.95858)(0.85929,0.95828)
	\end{diagps}\right) = \permMS,
	\qquad\mathrm{but}\qquad
	\psi\left(
	\begin{diagps}(-0.1,-0.0)(2.1,2.2)
		\node tl(0,1)[\bullet]	\node tr(1,2)[\bullet]
		\node bl(1,0)[\bullet]	\node br(2,1)[\bullet]
		\arrow[nodesep=1pt]|a{labelsep=3pt}|{->}[tl`tr;\textrm{split}]
		\arrow[nodesep=1pt]|a{labelsep=3pt}|{->}[tr`br;\textrm{merge}]
		\arrow[nodesep=1pt]|b{labelsep=3pt}|{->}[tl`bl;\textrm{merge}]
		\arrow[nodesep=1pt]|b{labelsep=3pt}|{->}[bl`br;\textrm{split}]
		\psset{linewidth=0.5pt,arrowsize,arrowsize=4pt 1.2,arrowlength=0.7,arrowinset=0.5}
		\psarcn(1,1.1){0.2}{135}{-135}\psline{<-}(0.85858,0.95858)(0.85929,0.95828)
	\end{diagps}\right) = \permSM.
\end{equation}
We call $\psi$ the~\emph{commutativity $2$-cochain}.

\begin{table}
	\begin{minipage}{5.5cm}
		\begin{center}
			{\Large $\permMM$}
			\par\pictDisMM{}{-}{}{-}
			\par\pictConnMM{}{-}{}{-}
			\par\pictConnX{}{->}{}{<-}
		\end{center}
	\end{minipage}
	\begin{minipage}{4.5cm}
		\begin{center}
			{\Large $\permSS$}
			\par\pictDisSS{}{-}{}{-}
			\par\pictConnSS{}{-}{}{-}
			\par\pictConnX{}{->}{}{->}
		\end{center}
	\end{minipage}
	\vskip\baselineskip
	\begin{minipage}[t]{5cm}
		\begin{center}
			{\Large $\permMS$}
			\par\pictDisMS{1}{-}{2}{-}
			\par\pictConnMS{1}{-}{2}{-}
		\end{center}
	\end{minipage}
	\begin{minipage}[t]{2.5cm}
		\begin{center}
			{\Large $\permTpos$}
			\vskip 0.75\baselineskip
			\pictConnT{}{->}{}{->}
		\end{center}
	\end{minipage}
	\begin{minipage}[t]{2.5cm}
		\begin{center}
			{\Large $\permTneg$}
			\vskip 0.75\baselineskip
			\pictConnT{}{->}{}{<-}
		\end{center}
	\end{minipage}
	\vskip 0.5\baselineskip
	\caption{%
		Diagrams for faces that can appear in a~cube of resolutions, grouped by values of the~commutativity cochain $\psi$. All~coefficients live in the~commutative ring $\scalars:=\Z[\permMM,\permSS,\permMS^{\pm1}]/(\permMM^2{=}\permSS^2{=}1)$. Thin lines are the~input circles and thick arrows visualize saddle points. Orientations of the~arrows are omitted if $\psi$ does not depend on them. The~small numbers $1$ and $2$ in the~two configurations placed under the~letter $\permMS$ indicate an~initial order of critical points, see~\eqref{eq:SM-group-explained}. For the~opposite order take $\permSM$.}\label{tab:cube-faces}
\end{table}

\subsection*{Coefficients on edges}
Edges in Fig.~\ref{diag:khov-cube} are labeled with elements of $\scalars$, describing a~$1$-cochain $\epsilon\in C^1(I^3; \invScalars)$ (take $1$ if no coefficient is present). The~product of these elements around each face $S$ is equal to $-\psi(S)$, i.e. $\psi=-d\epsilon$. Such a~cochain $\epsilon$ is called a~\emph{sign assignment}, following \cite{ORS}. It exists for any link diagram and, in some sense, it is unique (see Section~\ref{sec:complex}).

\subsection*{Complex}
The~bottom line in Fig.~\ref{diag:khov-cube} shows a~sequence of objects and maps between them. This is the~\emph{Khovanov bracket} of the~trefoil: think of the~objects $C^i$ as columns of the~diagrams above and the~maps $d^i$ as bundles of arrows between the~columns. We give more meaning to this in Section~\ref{sec:complex}, showing that $(C,d)$ is a~chain complex.

\subsection*{A word about tangles}
In the~same manner we can create a~cube of resolutions for a~tangle diagram, using cobordisms with corners. However, it has to be explained what we mean by a~$2$-cochain $\psi$ in this case, as faces are more complicated. This is done in Section~\ref{sec:tangles}.

\section{Cobordisms and chronologies}\label{sec:chron}
We start creating the~framework for Fig.~\ref{diag:khov-cube} by describing a~$2$-category\footnote{
	A~brief introduction to the~theory of $2$-categories is included in Appendix~\ref{sec:2-cats}.}
of chronological cobordisms.

An~$(n+1)$-manifold $W$ is a~\emph{cobordism} between two oriented $n$-manifolds $\Sigma_0$ and $\Sigma_1$ if its boundary is diffeomorphic to $\Sigma_0 \sqcup -\Sigma_1$ (the~minus sign stands for the~opposite orientation of $\Sigma_1$). We will often write $W_{in}$ and $W_{out}$ for the~components of $\partial W$ identified with $\Sigma_0$ and $-\Sigma_1$ respectively, and call them the~\emph{input} and the~\emph{output} of $W$.

Given cobordisms $W$ from $\Sigma_0$ to $\Sigma_1$ and $W'$ from $\Sigma_1$ to $\Sigma_2$ one can glue them together along the~orientation reversing diffeomorphism $W_{out}\approx\Sigma_1\approx W'_{in}$ to obtain a~cobordism $W'W$. Unfortunately, this operation is defined only up to a~diffeomorphism, the~issue we can address by considering cobordisms with \emph{collars}. Namely, think of an~$n$-manifold $\Sigma$ as an~open cylinders $\Sigma\times(-\varepsilon,\varepsilon)$ for a~fixed small $\varepsilon>0$, and a~cobordism from $\Sigma_0$ to $\Sigma_1$ as a~manifold $W$ with a~pair of embeddings $\Sigma_0\times[0,\varepsilon)\to W\to/<-/\Sigma_1\times(-\varepsilon,0]$. If $W'$ is another cobordism from $\Sigma_1$ to $\Sigma_2$, then the~gluing $W'W := W'\cup(\Sigma_1\times(-\varepsilon,\varepsilon))\cup W$ has a~well-defined smooth structure.

\begin{definition}\label{def:chronology}
	Let $W$ be a~cobordism and $\tau\colon W\to I$ an~oriented Morse function separating critical points. The~pair $(W,\tau)$ is called a~\emph{chronological cobordism} if $\tau^{-1}([0,\epsilon))$ and \mbox{$\tau^{-1}((1-\epsilon,1])$} are the~collars of $W_{in}$ and $W_{out}$ respectively, on which $\tau$ is the~projection on the~second factor. A~homotopy of $\tau$ in the~space of oriented Igusa functions is called a~\emph{change of a~chronology}.
\end{definition}

\noindent
We now explain some notions from the~definition, referring for details to Appendix~\ref{sec:framed-function}. An~Igusa function $f\colon W\to I$ is allowed to have two types of critical points:
\begin{itemize}
	\item $A_1$ or \emph{Morse singularities}, characterized by the~property that the~Hessian $\Hess_p(f)$ of $f$ is nondegenerate, and
	\item $A_2$ or \emph{birth-death singularities}, for which the~Hessian has a~one dimensional kernel $N(p)$, on which the~third derivative of $f$ does not vanish.
\end{itemize}
Choose a~Riemannian metric on $W$. For a~critical point $p$ we denote by $E^{\pm}(p)$ the~positive or negative eigenspace of the~Hessian $\Hess_p(f)\colon T_pW\to T_pW$. A~choice of orientations for negative eigenspaces over all critical points is called an~\emph{orientation} of $f$. We denote the~space of oriented Igusa functions on $W$ by $\OFun(W)$.

A~generic function $f\colon W\to I$ has only Morse singularities, but we need birth-death singularities for generic homotopies. Higher singularities are unnecessary for higher homotopies if we equip these functions with \emph{framing}, i.e.\ a~choice of a~basis for each $E^-(p)$, see Theorem~\ref{thm:framed-contr}. The~space of oriented functions can be seen as a~quotient  of this space, as explained at the~end of Appendix~\ref{sec:framed-function}. This space may no longer be contractible, but it is simply connected.

We are interested only in the~order of critical points of the~function $\tau$, so that we identify chronologies that differ by a~change preserving the~order.

\begin{definition}
	Chronological cobordisms $(W,\tau)$ and $(W,\tau')$ are \emph{equivalent} if there exists a~path $\gamma$ in $\OFun(W)$ from $\tau$ to $\tau'$ such that each $\gamma_t\colon W\to I$ is a~Morse function that separates critical points.\footnote{
		We are allowed to deform not only the~function $\tau$, but also the~chosen Riemannian structure on $W$. As shown in \cite{IgusaFMF} all Riemannian structures can be related by such deformations.
	}
	In such case we write $(W,\tau)\sim(W,\tau')$ or $\tau\sim\tau'$.
\end{definition}

Given cobordisms $(W,\tau)$ from $\Sigma_0$ to $\Sigma_1$ and $(W,\tau')$ from $\Sigma_1$ to $\Sigma_2$ we define a~chronology $\tau''$ on the~gluing $W'W$ by concatenation:
\begin{equation}\label{eq:chron-gluing}
	\tau''(p) := \begin{cases*}
		\frac{1}{2}\tau(p),				&	for $p\in W$,\\
		\frac{1}{2}(\tau'(p)+1),	&	for $p\in W'$.
	\end{cases*}
\end{equation}
The~assumed behavior of a~chronology on collars of a~cobordism guarantees $\tau''$ is smooth. Hence, we have an~associative and unital operation on equivalence classes of cobordisms, where units are given by cylinders $\Sigma\times I$ with the~simplest chronology---the~projection on $I$.

Recall that given two paths $\gamma,\gamma'\colon I\to X$ in a~topological space $X$ such that $\gamma(1)=\gamma'(0)$ we define their concatenation $\gamma'\star\gamma$ by the~formula
\begin{equation}
	(\gamma'\star\gamma)(t) := \begin{cases}
		\gamma(2t),    &   0\leqslant t\leqslant 1/2,\\
		\gamma'(2t-1), & 1/2\leqslant t\leqslant 1.
	\end{cases}
\end{equation}

\begin{definition}
\wrapfigure<1>{\begin{diagps}(-1.7em,-0.5ex)(9.7em,12.2ex)
		\square<8em,9.5ex>[%
			(W,H_0)`(W,H_1)`(W,H'_0)`(W,H'_1);%
			H`\gamma`\gamma'`H']%
		\arrow[nodesep=1.2em]|b{nrot=:D}|{=>}[v1`v2;\textit{htpy}]
	\end{diagps}}
	Let $H,H'\colon W\times I\to I$ be changes of chronologies such that $(W,H_0)$ and $(W,H'_0)$ are equivalent chronological cobordisms, i.e.\ there is a~path $\gamma$ in $\OFun(W)$ between $H_0$ and $H'_0$. We say $H$ and $H'$ are \emph{equivalent} if there is a~path $\gamma'$ from $H_1$ to $H'_1$ such that the~paths $\gamma'_t\star H_t$ and $H'_t\star\gamma_t$ are homotopic in $\OFun(W)$, see the~square to the~right. In such case we write $H\sim H'$.
\end{definition}

\begin{remark}\label{rmk:change-equiv-hpty-exists}
	The~connectivity of $\OFun(W)$ implies the~homotopy in the~definition above always exists. Hence, any two changes $H, H'\colon W\times I\to I$, for which $H_0\sim H'_0$ and $H_1\sim H'_1$, are equivalent.
\end{remark}

We can \emph{juxtapose} changes occurring on different regions of a~cobordism. Formally, if $H$ and $H'$ are changes of chronologies on $W$ and $W'$ respectively, and cobordisms $W$ and $W'$ can be glued together, there is a~change of a~chronology on $W'W$ induced by the~map
\begin{equation}\label{eq:chron-juxtapose}
	(H'\circ H)(p,t) := \begin{cases}
			H(p,t),	&	p\in W,\\
			H'(p,t),&	p\in W',
	\end{cases}
\end{equation}
which may need to be smoothed. This operation is clearly associative and unital.

\emph{Concatenation} of changes of chronologies is a~bit cumbersome: we cannot combine homotopies $H,H'\colon W\times I\to I$ if $(W,H_1)$ and $(W,H'_0)$ are only equivalent, as $H_1$ may not agree with $H'_0$. Instead, we define
\begin{equation}\label{eq:chron-concatenate}
	(H'\star H)(p,t) := \begin{cases}
		H(p,3t),							&   0 \leqslant t\leqslant 1/3,\\
		\gamma(p,2t-1),				& 1/3 \leqslant t\leqslant 2/3,\\
		H'(p,3t-2),	&	2/3 \leqslant t\leqslant 1,
	\end{cases}
\end{equation}
where $\gamma$ is a~path in $\OFun(W)$ from $H_1$ to $H'_0$. Hence, $H'\star H$ is given as the~following sequence of homotopies:
\begin{equation}
	(W,H_0)\dblto^{H}(W,H_1)\dblto^{\gamma}(W,H'_0)\dblto^{H'}(W,H'_1).
\end{equation}
This operation is well-defined up to equivalence due to Remark~\ref{rmk:change-equiv-hpty-exists} (in particular, it does not depend on the~path $\gamma$), and it is clearly associative and unital up to equivalence.

\begin{lemma}
	Choose pairs of equivalent changes of chronologies $H\sim\widetilde H$ and $H'\sim\widetilde H'$ on a~cobordism $W$ such that $H'$ and $H$ can be concatenated. Then we can concatenate $\widetilde H'$ with $\widetilde H$, and $H'\star H\sim \widetilde H'\star\widetilde H$.
\end{lemma}
\begin{proof}
	The~asserted equivalences guarantee $(W,H_i)\sim(W,\widetilde H_i)$ and $(W,H'_i)\sim(W,\widetilde H'_i)$ for $i=0,1$, and since $H'$ can be concatenated with $H$, $(W,H_1)$ and $(W,H'_0)$ are equivalent as well. Hence, we have a~sequence of equivalences $(W,\widetilde H_1)\sim(W,H_1)\sim(W,H'_0)\sim(W',\widetilde H'_0)$, which shows $\widetilde H'$ and $\widetilde H$ can be composed together. The~thesis follows, since the~rectangle below
	\begin{equation}\begin{diagps}(0em,-0.5ex)(21em,12ex)
		\node t0( 0em,10ex)[(W,H_0)]		\node b0( 0em,0ex)[(W,\widetilde H_0)]
		\node t1( 7em,10ex)[(W,H_1)]		\node b1( 7em,0ex)[(W,\widetilde H_1)]
		\node t2(14em,10ex)[(W,H'_0)]		\node b2(14em,0ex)[(W,\widetilde H'_0)]
		\node t3(21em,10ex)[(W,H'_1)]		\node b3(21em,0ex)[(W,\widetilde H'_1)]
		\arrow{=>}[t0`t1;H_t]						\arrow|a{npos=0.45}|{=>}[t1`t2;\omega_t]
		\arrow{=>}[b0`b1;\widetilde H_t]\arrow|a{npos=0.45}|{=>}[b1`b2;\widetilde\omega_t]
		\arrow{=>}[t2`t3;H'_t]					\arrow{=>}[b2`b3;\widetilde H'_t]
		\arrow{=>}[t0`b0;\gamma_t]			\arrow{=>}[t3`b3;\gamma_t']
	\end{diagps}\end{equation}
	commutes up to homotopy due to Remark~\ref{rmk:change-equiv-hpty-exists}, where $\gamma$, $\gamma'$, $\omega$, and $\widetilde\omega'$ are paths of Morse functions given by corresponding equivalences of chronological cobordisms.
\end{proof}

All the~above almost shows that chronological cobordisms form a~2-category---it remains to check the~interchange law \eqref{eq:interchange-law} holds; this follows immediately as concatenation and juxtaposition commute with each other. We state this as the~following proposition.

\begin{proposition}
	There is a~strict 2-category of chronological cobordisms $n\cat{ChCob}$ with oriented manifolds of dimension $(n-1)$ as objects, equivalence classes of chronological cobordisms as morphisms, and homotopy classes of changes of chronologies as 2-morphisms. The composition of morphisms is induced by gluing, and the~two compositions of 2-morphisms are given as juxtaposition (the~horizontal one) and concatenation (the~vertical one).
\end{proposition}

\begin{remark}
	For a~chronological cobordism $W$ the~set of critical points $crit(W)$ is linearly ordered by the~chronology: we write $x<y$ if $\tau(x)<\tau(y)$. This order is invariant under equivalence of cobordisms, but it is affected by changes of chronologies.
\end{remark}

One of the~important operations on cobordisms is the~\emph{disjoint union}. For chronological cobordisms it has to be defined carefully: with the~naive definition one might get two critical points at the~same level, which is prohibited. Instead, we have to shift critical points of the~left or the~right cobordism below critical points of the~other one, obtaining a~\quot{left-then-right} and a~\quot{right-then-left} disjoint unions denoted respectively by $\ldsum$ and $\rdsum$ (see Fig.~\ref{fig:left-right-disjoint-union}). Formally, we equip $(W,\tau)\rdsum(W',\tau')$ with the~chronology
\begin{equation}\label{eq:chron-disjoint-sum}
	\tau_r(p) := \begin{cases}
				\beta_{1/2}^1(\tau(p)), & p\in W,\\
				\beta_0^{1/2}(\tau'(p)),& p\in W',
		\end{cases}
\end{equation}

\noindent\unskip
\wrapfigure[r]<2>{%
	\psset{linewidth=0.5pt,arrowsize=4pt,xunit=3cm,yunit=1.5cm}
	\begin{pspicture}(-1em,-2ex)(1.25,1.5)
		\psline{<->}(0,1.3)(0,0)(1.2,0)
		\uput[l](0,1.3){$\scriptstyle y$}
		\uput[d](1.2,0){$\scriptstyle x$}
		\psset{linewidth=1pt}
		\psbezier(0,0)(0.2,0.05)(0.2,0.05)(0.4,0.5)
		\psbezier(0.4,0.5)(0.6,0.95)(0.6,0.95)(1,1)
		\rput[B](1,1.15){$\scriptstyle y=\beta_a^b(x)$}
		\psset{linewidth=0.5pt,linestyle=dashed,dash=2pt 3pt}
		\psline(0,1)(1,1)\uput[l](0,1){$\scriptstyle 1$}
		\psline(0.2,0)(0.2,1)\uput[d](0.2,0){$\scriptstyle a\vphantom{b}$}
		\psline(0.6,0)(0.6,1)\uput[d](0.6,0){$\scriptstyle b$}
		\psline(1,0)(1,1)\uput[d](1,0){$\scriptstyle 1$}
	\end{pspicture}%
}
where $\beta_a^b\colon I\to I$ is a~perturbed \quot{bump function}: an~increasing function which is very close to $0$ on the~interval $[0,a]$ and very close to $1$ on $[b,1]$. The~chronology $\tau_\ell$ on $(W,\tau)\ldsum(W',\tau')$ is defined in a~similar way. Finally, the~formula \eqref{eq:chron-disjoint-sum} can be naturally extended to changes of chronologies---replace $p$ with a~pair $(p,t)$.

\wrapfigure[r]<0>{\hskip 1em\parbox[c][3.5\baselineskip]{3.75cm}{\quad}}
This is the~first place where we can see that chronological cobordisms indeed require a~richer structure than just a~category: the disjoint unions defined above are functorial only up to a~change of a~chronology $\sigma^{\sqcup}_{W,W'}\colon (W\ldsum W',\tau_r)\dblto (W\rdsum W',\tau_\ell)$ that pulls $W$ below $W'$. This can be done by a~linear interpolation: $\sigma_{W,W'}^{\sqcup}(p,t) := (1-t)\tau_\ell(p) + t\tau_r(p)$.

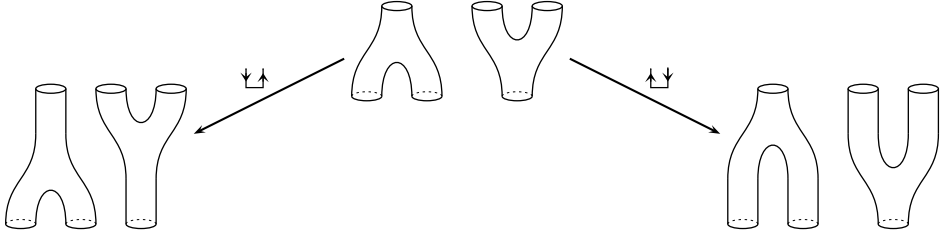
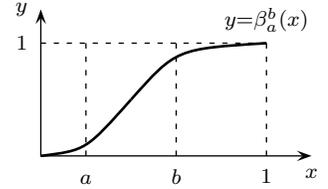
\begin{figure}[t]
	\centering
		\begin{pspicture}(-6.5,0)(6.5,4)
			\COBmerge(-1.4,2)
			\COBsplit( 0.6,2)
			\cobordism[2](-6,0.3)(M,1)(sI)
			\cobordism[1](-4.4,0.3)(sI)(S,1)
			\cobordism[2](3.6,0.3)(sI)(M,1)
			\cobordism[1](5.6,0.3)(S,1)(sI)
			\psline{->}(-1.5,2.5)(-3.5,1.5)\rput[br](-2.5,2.1){$\ldsum$}
			\psline{->}( 1.5,2.5)( 3.5,1.5)\rput[bl]( 2.5,2.1){$\rdsum$}
		\end{pspicture}
	\caption{A~disjoint union in a~chronological setting requires a~shift.}
	\label{fig:left-right-disjoint-union}
\end{figure}

\begin{theorem}\label{thm:ChCob-gray}
	The~2-category $n\cat{ChCob}$ is Gray monoidal. The~monoidal product is induced by the~`right-then-left' disjoint union $\rdsum$ and the~unit is given by the~empty manifold $\emptyset$.
\end{theorem}
\begin{proof}
	We have to check conditions from Definition~\ref{def:Gray-monoidal}. First, $\rdsum$ is cubical. Indeed, the~conditions from Definition~\ref{def:cubical} are trivially satisfied, as $\sigma^{\sqcup}_{W,W'}\colon W\ldsum W'\dblto W\rdsum W'$ does nothing if either $W$ or $W'$ has no critical points. Next, it is coherent with 2-morphism, i.e.\ the~square \eqref{eq:cubical-2-morphisms} commutes. Indeed, given changes of chronologies $\alpha$ on $W$ and $\beta'$ on $W'$ we construct a~homotopy
	\begin{equation}
		h_s := \sigma^{\sqcup}_{W,W'}|_{[0,s]}\star \left((1-s)\alpha\ldsum\beta' + s\alpha\rdsum\beta'\right)
												                 \star \sigma^{\sqcup}_{W,W'}|_{[s,1]}
	\end{equation}
	where $\sigma^{\sqcup}_{W,W'}|_{[a,b]}$ is a~restriction of $\sigma^{\sqcup}_{W,W'}$ to $t\in[a,b]$. The~homotopy $h_s$ first shifts $W$ and $W'$ a~bit towards their final position, then it applies the~changes $\alpha$ and $\beta'$ on appropriate levels, and after that it shifts $W$ and $W'$ further to their final positions. Finally, commutativity of \eqref{eq:cubical-monoidal} follows easily: the~two changes $\sigma^{\sqcup}\star(\sigma^{\sqcup}\circ\id)$ and $\sigma^{\sqcup}\star(\id\circ\sigma^{\sqcup})$ are homotopic by a~linear interpolation.
	
	The~unitarity condition is clear, which leaves only associativity to check. This follows from the~way $\rdsum$ is defined: the~two chronologies on $W\rdsum (W'\rdsum W'')$ and $(W\rdsum W')\rdsum W''$ are homotopic by a~reparametrization of the~target interval $I$.
\end{proof}

\wrapfigure[r]{%
	\def\COBxsize{0.4}\def\COBysize{1.0}%
	\begin{pspicture}(-0.3,-0.5ex)(3.2,3.3)
		\psset{linewidth=0.5pt,arrowsize=4pt,arrowlength=1,dimen=mid}
		\cobordism[2](0,1.1)(I)
		\psellipse(0.2,0.1)(0.2,0.06)\psellipse(1.0,0.1)(0.2,0.06)
		\psellipse(0.2,3.1)(0.2,0.06)\psellipse(1.0,3.1)(0.2,0.06)
		\psline{->}(0.2,0.3)(0.2,0.9)\psline{->}(0.25,2.95)(1.0,2.25)
		\psline{->}(1.0,0.3)(1.0,0.9)\psline{->}(0.95,2.95)(0.2,2.25)
		\diagnode 01(2.5,0.0)[\Sigma_0\sqcup\Sigma_1]
		\diagnode 10(2.5,3.0)[\Sigma_1\sqcup\Sigma_0]
		\diagnode  W(2.5,1.5)[W]
		\diagarrow{->}[01`W;]
		\diagarrow{->}[10`W;]
	\end{pspicture}}
The~ordinary category of cobordisms is not only monoidal, but it possesses a~symmetry induced by a~family of permutation diffeomorphisms $c\colon\Sigma_1\sqcup\Sigma_0\to^\approx\Sigma_0\sqcup\Sigma_1$. Namely, take the~cylinder $(\Sigma_0\sqcup\Sigma_1)\times I$ with the~standard inclusion as its input and the~diffeomorphism $c$ as its output (see the~picture to the~side). In case of chronological cobordisms, these permutation cylinders form natural transformations between unary functors $C\rdsum(\blank)$ and $(\blank)\rdsum C$, where $C$ stands for any cylinder. This suggests the~permutation cylinders equip $n\cat{ChCob}$ with a~strict symmetry, see Definition~\ref{def:strict-symmetry}. Indeed, commutativity of the~triangle \eqref{diag:strict-symm-triangle} follows easily from this construction.

\begin{corollary}
	The~Gray monoidal category $(n\cat{ChCob},\rdsum,\emptyset)$ has a~strict symmetry induced by permutation diffeomorphisms $c\colon\Sigma_1\sqcup\Sigma_0\to^\approx\Sigma_0\sqcup\Sigma_1$.
\end{corollary}

There is another operation on chronological cobordisms similar to the~disjoint unions---the~connected sum. Given chronological cobordisms $W$ and $W'$ remove vertical cylinders from them (a~cylinder $C$ in $(W,\tau)$ is \emph{vertical} if the~restriction $\tau|_C$ is a~regular function) and construct $W\connsum W'$ by identifying the~cobordisms along the~newly created boundary. Likewise for the~disjoint unions, there are two connected sums of chronological cobordisms $W$ and $W'$, the~`left-then-right' $W\lcsum W'$ and the~`right-then-left' one $W\rcsum W'$, related by a~change of a~chronology $\sigma^{\connsum}_{W,W'}\colon W\lcsum W'\dblto W'\rcsum W$ that permutes the~critical points.

Let $n\cat{ChCob}_\circ$ be the~category of nonempty manifolds with two distinguished points, and cobordisms between them, decorated with two vertical lines connecting the~basepoints of the~boundary manifolds. Then the~connected sums are well-defined (choose small neighborhoods of the~distinguished lines), and we have the~following analog of Theorem~\ref{thm:ChCob-gray}.

\begin{corollary}
	The~2-category $n\cat{ChCob}_\circ$ is Gray monoidal. The~product is induced by the~`right-then-left' connected sum $\rcsum$ and the~unit is given by the~$(n-1)$-dimensional ball.
\end{corollary}

We end this section with a~combinatorial description of $2\cat{ChCob}$.

\begin{proposition}\label{prop:2ChCob-presentation}
	$2\cat{ChCob}$ is a~symmetric Gray monoidal category with objects freely generated by a~circle\/ $\bS$ and morphisms freely generated by the~following five cobordisms:
	\begin{equation}\label{diag:mor-gens}\begin{centerpict}(10,2.5)
		\COBmergeFrLeft(0,1.1)\COBsplitFrBack(2.6,1.1)
		\COBbirth(5.2,1.1)\COBposDeath(7.0,1.1)\COBnegDeath(9.2,1.1)
		\rput[B](0.6,2ex){\textnormal{a~merge}}
		\rput[B](2.8,2ex){\textnormal{a~split}}
		\rput[B](5.0,2ex){\textnormal{a~birth}}
		\rput[B](7.2,2ex){\begin{minipage}[c]{10ex}
			\centering\baselineskip=0.7\baselineskip\textnormal{a~positive death}
		\end{minipage}}
		\rput[B](9.4,2ex){\begin{minipage}[c]{10ex}
			\centering\baselineskip=0.7\baselineskip\textnormal{a~negative death}
		\end{minipage}}
	\end{centerpict}\end{equation}
	with a~twist {\psset{unit=0.6cm}\textcobordism[2](P)} acting as a~strict symmetry.
\end{proposition}

One should read the~pictures above from bottom to top: the~bottom circles form the~input of a~cobordism, the~top ones form the~output and the~height function determines a~chronology. Orientations of critical points are visualized by arrows.

\begin{proof}
	Every $1$-dimensional manifold is a~family of circles, so that objects of $2\cat{ChCob}$ are freely generated under the~disjoint union by a~single circle $\bS$. Since all orientation preserving diffeomorphisms of $\bS$ are isotopic to the~identity, chronological cobordisms with no critical points are generated by a~permutation of two circles, the~symmetry of the~monoidal structure. Morse theory provides a~description of cobordisms with a~single critical point. Since the~order of critical points is fixed, it remains to analyze orientations of the~critical points.
	
	We need only one merge and one split---an~orientation of the~saddle point can be reversed by attaching a~twist. The~tangent space to a~point of index 0 (a~birth) is stable, so that there is only one choice for orientation (the~empty frame), but it is unstable at points of index 2 (deaths). Hence, a~choice of an~orientation of a~death is equivalent to an~orientation of the~tangent space, which can be either coherent with the~orientation of the~cobordism or not.
\end{proof}

We shall use Cerf theory (see Section~\ref{sec:framed-function}) to describe 2-morphisms in terms of generators and relations. Most of them are easy to draw directly, but for some it will be useful to use other presentations: \emph{movies} and \emph{surgery diagrams}.

A~\emph{movie presentation} of a~chronological cobordism is a~sequence of its regular levels, dense enough to capture all topological changes: such a~sequence contains at least one regular level between any two critical ones. Two consecutive diagrams in the~sequence differ in one of the~following ways:
\begin{itemize}
	\item they are isotopic, so that there is no critical level in between,
	\item one diagram is obtained form the~other by a~saddle move $\fntHorRes\to\fntVertRes$; this corresponds to a~merge or a~split,
	\item a~one circle component appear (for a~birth) or disappear (for a~death).
\end{itemize}
We can add additional information to encode orientations of the~critical points: an~oriented chord for a~saddle move, or $a/c$ for a~death oriented anti- or clockwise. We provide below one example.
\begin{equation}
	\psset{linewidth=0.5pt,arrowsize=5pt,arrowlength=1}
	\textcobordism[1](S-B)(srD+,2)
	\qquad=\qquad
	\begin{movie}[h](1.5,1.5){4}
		\movieclip{\pscircle(0,0){0.3}}
		\movieclip{\pscustom{%
				\moveto(-0.25,0.3)%
				\psarc(-0.25,0){0.3}{90}{270}%
				\curveto(0.1,-0.3)(0,-0.15)(0.2,-0.15)\curveto(0.3,-0.15)(0.2,-0.2)(0.4,-0.2)
				\psarc(0.4,0){0.2}{-90}{90}
				\curveto(0.2,0.2)(0.3,0.15)(0.2,0.15)\curveto(0,0.15)(0.1,0.3)(-0.25,0.3)
			}%
			\psline[linewidth=3\pslinewidth]{->}(0.2,-0.15)(0.2,0.15)
		}
		\movieclip{\pscircle(-0.25,0){0.3}\pscircle(0.4,0){0.2}\rput(0.4,0){$\scriptstyle a$}}
		\movieclip{\pscircle(-0.25,0){0.3}}
	\end{movie}
\end{equation}

Movie presentations are a~good way to visualize cobordisms. However, if a~cobordism $(W,\tau)$ has only saddle points, a~more compact description is given by its \emph{surgery diagram}: a~collection of circles with enumerated oriented chords between them. The~circles illustrate the~input of the~cobordism $W$, whereas the~chords represent 1-handles in the~handle decomposition of $W$ with respect to the~chronology $\tau$. Performing surgeries along the~chords in the~specified order results in a~movie presentation of $W$. However, we can get more: a~diagram with two chords encodes two chronological cobordisms, depending on the~order of the~chords, and a~change that permutes the~two points, see~Fig.~\ref{fig:surgery-description-of-a-change} on page \pageref{fig:surgery-description-of-a-change}.

\begin{proposition}\label{prop:changes-generators}
	Changes of chronologies are generated under composition and disjoint union by the~following:
\begin{enumerate}
	\item \emph{creation} and \emph{annihilation} changes
	\begin{equation}\label{change:creation-annihilation}
		\psset{unit=7mm}
		\def\arraystretch{0.5}
			\textcobordism[1](sI)(I)
				\,\begin{array}{c}\to/=>/ \\ \to/<=/\end{array}\,
			\textcobordism*[1](slB)(M-L)
			\hskip 1.5cm
			\textcobordism[1](I)(sI)
				\,\begin{array}{c}\to/=>/ \\ \to/<=/\end{array}\,
			\textcobordism*[1](S-B)(srD+,2)
			\hskip 1.5cm
			\textcobordism[1](I)(sI)
				\,\begin{array}{c}\to/=>/ \\ \to/<=/\end{array}\,
			\textcobordism*[1](S-B)(slD-,1)
	\end{equation}
	in which the~orientations of deaths are determined by the~monotonicity condition for $d^3\tau$ at an~$A_2$ singularity (take the~arrow at the~saddle and rotate it towards the~vertical cylinder),

	\item the~\emph{disjoint sum permutations}
	\begin{equation}
		\psset{unit=7mm}
		\begin{centerpict}(-0.1,-0.1)(3.5,2.5)
			\COBcylinder(0.1,0)(0.1,2.4)\COBcylinder(1.0,0)(1.0,2.4)
			\COBcylinder(2.0,0)(2.0,2.4)\COBcylinder(2.9,0)(2.9,2.4)
			\rput[c](0.77,0.1){$\scriptstyle\cdots$}\rput[c](0.77,2.25){$\scriptstyle\cdots$}
			\rput[c](2.67,0.1){$\scriptstyle\cdots$}\rput[c](2.67,2.25){$\scriptstyle\cdots$}
			\psframe[framearc=0.5,fillstyle=solid](-0.05,0.3)(1.55,1.2)
			\psframe[framearc=0.5,fillstyle=solid]( 1.85,1.2)(3.45,2.1)
			\rput[c](0.8,0.75){$W'$}
			\rput[c](2.7,1.65){$W\phantom'$}
		\end{centerpict}
		\,\to/=>/^{\sigma^{\sqcup}_{W,W'}}\,
		\begin{centerpict}(-0.1,-0.1)(3.5,2.5)
			\COBcylinder(0.1,0)(0.1,2.4)\COBcylinder(1.0,0)(1.0,2.4)
			\COBcylinder(2.0,0)(2.0,2.4)\COBcylinder(2.9,0)(2.9,2.4)
			\rput[c](0.77,0.1){$\scriptstyle\cdots$}\rput[c](0.77,2.25){$\scriptstyle\cdots$}
			\rput[c](2.67,0.1){$\scriptstyle\cdots$}\rput[c](2.67,2.25){$\scriptstyle\cdots$}
			\psframe[framearc=0.5,fillstyle=solid](-0.05,1.2)(1.55,2.1)
			\psframe[framearc=0.5,fillstyle=solid]( 1.85,0.3)(3.45,1.2)
			\rput[c](0.8,1.65){$W'$}
			\rput[c](2.7,0.75){$W\phantom'$}
		\end{centerpict}
	\end{equation}	

	\item the~\emph{connected sum permutations}
	\begin{equation}
		\psset{unit=7mm}
		\begin{centerpict}(-0.1,-0.1)(2.9,2.5)
			\COBcylinder(0.1,0)(0.1,2.4)\COBcylinder(1.2,0)(1.2,2.4)\COBcylinder(2.3,0)(2.3,2.4)
			\rput[c](0.87,0.1){$\scriptstyle\cdots$}\rput[c](0.87,2.25){$\scriptstyle\cdots$}
			\rput[c](1.97,0.1){$\scriptstyle\cdots$}\rput[c](1.97,2.25){$\scriptstyle\cdots$}
			\psframe[framearc=0.5,fillstyle=solid](-0.05,0.3)(1.40,1.2)
			\psframe[framearc=0.5,fillstyle=solid]( 1.40,1.2)(2.85,2.1)
			\rput[c](0.725,0.75){$W'$}
			\rput[c](2.125,1.65){$W\phantom'$}
		\end{centerpict}
		\,\to/=>/^{\sigma^{\connsum}_{W,W'}}\,
		\begin{centerpict}(-0.1,-0.1)(2.9,2.5)
			\COBcylinder(0.1,0)(0.1,2.4)\COBcylinder(1.2,0)(1.2,2.4)\COBcylinder(2.3,0)(2.3,2.4)
			\rput[c](0.87,0.1){$\scriptstyle\cdots$}\rput[c](0.87,2.25){$\scriptstyle\cdots$}
			\rput[c](1.97,0.1){$\scriptstyle\cdots$}\rput[c](1.97,2.25){$\scriptstyle\cdots$}
			\psframe[framearc=0.5,fillstyle=solid](-0.05,1.2)(1.40,2.1)
			\psframe[framearc=0.5,fillstyle=solid]( 1.40,0.3)(2.85,1.2)
			\rput[c](0.725,1.65){$W'$}
			\rput[c](2.125,0.75){$W\phantom'$}
		\end{centerpict}
	\end{equation}

	\item the~\emph{exceptional permutation} changes, represented by the~following movies
	\begin{equation}\label{change:connected-permutations}
		\psset{unit=5mm,linewidth=0.5pt}
		\begin{movie}[v](2.5,2){3}
			\movieclip{
				\pscustom{\psarc ( 0.4,0){0.6}{-90}{90}\psarc( 0.4, 0.4){0.2}{90}{270}
				          \psarcn( 0.4,0){0.2}{90}{-90}\psarc( 0.4,-0.4){0.2}{90}{270}}
				\pscustom{\psarcn(-0.4,0){0.6}{-90}{90}\psarcn(-0.4, 0.4){0.2}{90}{270}
				          \psarc (-0.4,0){0.2}{90}{-90}\psarcn(-0.4,-0.4){0.2}{90}{270}}
				\psline[linewidth=4\pslinewidth](-0.2,0.4)(0.2,0.4)
			}
			\movieclip{
				\pscustom{\psarc ( 0.4, 0.0){0.6}{-90}{90}\psarc (-0.4, 0.0){0.6}{90}{270}
				          \psarc (-0.4,-0.4){0.2}{-90}{90}\psarcn(-0.4, 0.0){0.2}{270}{90}
									\psarcn( 0.4, 0.0){0.2}{90}{-90}\psarc ( 0.4,-0.4){0.2}{90}{270}}
				\psline[linewidth=4\pslinewidth](-0.2,-0.4)(0.2,-0.4)
			}
			\movieclip{
				\pscustom{\psarc(0.4,0){0.6}{-90}{90}\psarc(-0.4,0){0.6}{90}{270}\lineto(0.4,-0.6)}
				\pscustom{\psarc(0.4,0){0.2}{-90}{90}\psarc(-0.4,0){0.2}{90}{270}\lineto(0.4,-0.2)}
			}
		\end{movie}
			\,\begin{array}{c}\to/=>/ \\ \to/<=/\end{array}\,
		\begin{movie}[v](2.5,2)3
			\movieclip{
				\pscustom{\psarc ( 0.4,0){0.6}{-90}{90}\psarc( 0.4, 0.4){0.2}{90}{270}
				          \psarcn( 0.4,0){0.2}{90}{-90}\psarc( 0.4,-0.4){0.2}{90}{270}}
				\pscustom{\psarcn(-0.4,0){0.6}{-90}{90}\psarcn(-0.4, 0.4){0.2}{90}{270}
				          \psarc (-0.4,0){0.2}{90}{-90}\psarcn(-0.4,-0.4){0.2}{90}{270}}
				\psline[linewidth=4\pslinewidth](-0.2,-0.4)(0.2,-0.4)
			}
			\movieclip{
				\pscustom{\psarc (-0.4,0.0){0.6}{90}{270}\psarc (0.4,0.0){0.6}{-90}{90}
				          \psarc ( 0.4,0.4){0.2}{90}{270}\psarcn(0.4,0.0){0.2}{90}{270}
									\psarcn(-0.4,0.0){0.2}{270}{90}\psarc(-0.4,0.4){0.2}{-90}{90}}
				\psline[linewidth=4\pslinewidth](-0.2,0.4)(0.2,0.4)
			}
			\movieclip{
				\pscustom{\psarc(0.4,0){0.6}{-90}{90}\psarc(-0.4,0){0.6}{90}{270}\lineto(0.4,-0.6)}
				\pscustom{\psarc(0.4,0){0.2}{-90}{90}\psarc(-0.4,0){0.2}{90}{270}\lineto(0.4,-0.2)}
			}
		\end{movie}
		\hskip 1cm\textrm{and}\hskip 1cm
		\begin{movie}[v](2.5,2)3
			\movieclip{
				\pscustom{\psarc (-0.4,0.0){0.6}{90}{270}\psarc (0.4,0.0){0.6}{-90}{90}
				          \psarc ( 0.4,0.4){0.2}{90}{270}\psarcn(0.4,0.0){0.2}{90}{270}
									\psarcn(-0.4,0.0){0.2}{270}{90}\psarc(-0.4,0.4){0.2}{-90}{90}}
				\psline[linewidth=4\pslinewidth](-0.2,0.4)(0.2,0.4)
			}
			\movieclip{
				\pscustom{\psarc(0.4,0){0.6}{-90}{90}\psarc(-0.4,0){0.6}{90}{270}\lineto(0.4,-0.6)}
				\pscustom{\psarc(0.4,0){0.2}{-90}{90}\psarc(-0.4,0){0.2}{90}{270}\lineto(0.4,-0.2)}
				\psline[linewidth=4\pslinewidth](0,-0.2)(0,-0.6)
			}
			\movieclip{
				\pscustom{\psarc ( 0.4, 0.0){0.6}{-90}{90}\psarc (-0.4, 0.0){0.6}{90}{270}
				          \psarc (-0.4,-0.4){0.2}{-90}{90}\psarcn(-0.4, 0.0){0.2}{270}{90}
									\psarcn( 0.4, 0.0){0.2}{90}{-90}\psarc ( 0.4,-0.4){0.2}{90}{270}}
			}
		\end{movie}
			\,\begin{array}{c}\to/=>/ \\ \to/<=/\end{array}\,
		\begin{movie}[v](2.5,2)3
			\movieclip{
				\pscustom{\psarc (-0.4,0.0){0.6}{90}{270}\psarc (0.4,0.0){0.6}{-90}{90}
				          \psarc ( 0.4,0.4){0.2}{90}{270}\psarcn(0.4,0.0){0.2}{90}{270}
									\psarcn(-0.4,0.0){0.2}{270}{90}\psarc(-0.4,0.4){0.2}{-90}{90}}
				\psline[linewidth=4\pslinewidth](0,-0.2)(0,-0.6)
			}
			\movieclip{
				\pscustom{\psarc ( 0.4,0){0.6}{-90}{90}\psarc( 0.4, 0.4){0.2}{90}{270}
				          \psarcn( 0.4,0){0.2}{90}{-90}\psarc( 0.4,-0.4){0.2}{90}{270}}
				\pscustom{\psarcn(-0.4,0){0.6}{-90}{90}\psarcn(-0.4, 0.4){0.2}{90}{270}
				          \psarc (-0.4,0){0.2}{90}{-90}\psarcn(-0.4,-0.4){0.2}{90}{270}}
				\psline[linewidth=4\pslinewidth](-0.2,0.4)(0.2,0.4)
			}
			\movieclip{
				\pscustom{\psarc ( 0.4, 0.0){0.6}{-90}{90}\psarc (-0.4, 0.0){0.6}{90}{270}
				          \psarc (-0.4,-0.4){0.2}{-90}{90}\psarcn(-0.4, 0.0){0.2}{270}{90}
									\psarcn( 0.4, 0.0){0.2}{90}{-90}\psarc ( 0.4,-0.4){0.2}{90}{270}}
			}
		\end{movie}
	\end{equation}
	to which we refer respectively as a~${\times}$-change and a~$\Diamond$-change, because of the~shapes of cobordisms involved.
\end{enumerate}
\end{proposition}
\begin{proof}
	According to Cerf theory there are two types of changes:
	\begin{itemize}
		\item those generated by $A_2$-singularies, i.e.\ creation and annihilation changes, and
		\item those induced by homotopies $H_t$, such that $H_{t_0}$ has two critical points at one level for some $t_0$.
	\end{itemize}
	In the~latter case, we refer to the~critical level of $H_{t_0}$ as the~\emph{singular section} of $H_t$. Its components carrying the~critical points is a~four-valent graph $\Gamma_{\!H}$. Consider the~connectivity of the~graph:\footnote{
		A~graph $\Gamma$ is $n$-connected if at least $n$ edges must be removed to split it into two components.
	} the~homotopy $H$ represents a~disjoint union permutation if $\Gamma_{\!H}$ has two components, a~connected sum permutation if $\Gamma_{\!H}$ is 2-connected, or one of the~exceptional changes if $\Gamma_{\!H}$ is 4-connected.
\end{proof}

\begin{remark}
	\wrapfigure{\begin{pspicture}(-1.2,-1)(1.2,1)
		\rput(0, 0.6){\diagConnX{}{->}{}{->}}%
		\rput(0,-0.6){\diagConnX{}{->}{}{<-}}%
	\end{pspicture}}
	The~surgery diagram of an~${\times}$-change consists of two circles connected by two arrows. However, the~arrows can either point to the~same or to different circles, and the~two cases lead to surgery diagrams encoding non-equivalent changes. Hence, there are two versions of the~${\times}$-change.

	\parshape=0
	On the~other hand, reversing one chord in a~surgery diagram of a~$\Diamond$-change results in the~inverse change. Indeed, the~only topological information we have is the~order of chords induced by the~arc connecting their heads (there is a~natural orientation of the~circle in the~surgery diagram induced from the~orientation of the~underlying cobordism). This order may or may not coincide with the~order of critical points, induced by the~initial chronology, and the~two cases lead to mutually inverse permutation changes.
\end{remark}

We shall now proceed to a~description of \emph{relations} between the~generators of the~set of 2-morphisms. These are given by homotopies of paths in the~space of Igusa functions listed in Section~\ref{sec:framed-function}. As before, not all of them can be easily drawn, especially the~homotopies relating the~two ways of switching levels of three critical points. We shall encode them with three-chord surgery diagrams---such a~diagram represent six cobordisms, depending on the~order of critical points, call them $a$, $b$, $c$, and six permutation changes between these cobordisms forming a~hexagonal diagram
\begin{equation}\label{eq:hexagon-relation}
	\begin{diagps}(-8em,-6ex)(8em,7.5ex)
		\node[href=-0.5]bac(-3em, 6ex)[W\raisebox{0.2ex}{$\scriptstyle(b<a<c)$}]	\node[href= 0.5]bca( 3em, 6ex)[W\raisebox{0.2ex}{$\scriptstyle(b<c<a)$}]
		\node[href= 0.5]abc(-8em, 0ex)[W\raisebox{0.2ex}{$\scriptstyle(a<b<c)$}]	\node[href=-0.5]cba( 8em, 0ex)[W\raisebox{0.2ex}{$\scriptstyle(c<b<a)$}]
		\node[href=-0.5]acb(-3em,-6ex)[W\raisebox{0.2ex}{$\scriptstyle(a<c<b)$}]	\node[href= 0.5]cab( 3em,-6ex)[W\raisebox{0.2ex}{$\scriptstyle(c<a<b)$}]
		\arrow{=>}[abc`bac;]	\arrow{=>}[bac`bca;]	\arrow{=>}[bca`cba;]
		\arrow{=>}[abc`acb;]	\arrow{=>}[acb`cab;]	\arrow{=>}[cab`cba;]
	\end{diagps}
\end{equation}
The~notation $W\raisebox{0.2ex}{$\scriptstyle(x<y<z)$}$ is used for the~cobordism with the~point $x$ at the~lowest critical level, $y$ in the~middle, and $z$ at the~highest one. The~relation imposed by the~homotopy makes this hexagon commute.

\begin{proposition}\label{prop:changes-relations}
	The~following is the~complete set of relations among the~generating changes of chronologies listed in Proposition~\ref{prop:changes-generators}:
	\begin{enumerate}
	\item the~squares below commute for any cobordisms $W$, $W'$, $W''$, and a~2-morphism $\alpha$:
	\begin{equation}\label{eq:rel-dsum-csum-vs-change}
		\begin{diagps}(0em,-1ex)(21em,12.5ex)
			\square<7em,10ex>{=>`=>`=>`=>}[%
				W\ldsum W'`W\ldsum W''`W\rdsum W'`W\rdsum W'';
				\id\sqcup\alpha`\sigma^{\sqcup}_{W,W'}`\sigma^{\sqcup}_{W,W''}`\id\sqcup\alpha
			]
			\square(14em,0ex)<7em,10ex>{=>`=>`=>`=>}[%
				W\lcsum W'`W\lcsum W''`W\rcsum W'`W\rcsum W'';
				\id\connsum\alpha`\sigma^{\connsum}_{W,W'}`\sigma^{\connsum}_{W,W''}`\id\connsum\alpha
			]
		\end{diagps}
	\end{equation}
	
	\item hexagons encoded by the~following surgery diagrams commute:
	\begin{gather}
		\label{eq:rel-exceptional-planar}
		\psset{linewidth=0.5pt,fillstyle=solid,fillcolor=white}
		\begin{centerpict}(-1,-0.6)(1,0.9)
			\pspolygon[linewidth=1.5pt](-0.519,-0.3)(0.519,-0.3)(0,0.6)
			\pscircle(-0.519,-0.3){0.3}
			\pscircle( 0.519,-0.3){0.3}
			\pscircle(0,0.6){0.3}
		\end{centerpict}
		\hskip 1cm
		\begin{centerpict}(-1,-0.3)(1,0.3)
			\psellipse[linewidth=1.5pt](0,0)(0.8,0.2)
			\pscircle(0.5,0){0.3}
			\psclip{\pscircle(-0.5,0){0.3}}
				\psline[linewidth=1.5pt](-0.8,0)(-0.2,0)
			\endpsclip
		\end{centerpict}
		\hskip 1cm
		\begin{centerpict}(-1,-0.5)(1,0.5)
			\psellipse[linewidth=1.5pt](0,0.2)(0.7,0.3)
			\psclip{\psellipse(0,0)(0.5,0.3)}
				\psline[linewidth=1.5pt](-0.2,-0.3)(-0.2,0.3)
				\psline[linewidth=1.5pt]( 0.2,-0.3)( 0.2,0.3)
			\endpsclip
		\end{centerpict}
		\hskip 1cm
		\begin{centerpict}(-1,-0.6)(1,0.6)
			\psline[linewidth=1.5pt](-0.5,-0.3)(0.5,-0.3)
			\psline[linewidth=1.5pt](-0.5, 0  )(0.5, 0  )
			\psline[linewidth=1.5pt](-0.5, 0.3)(0.5, 0.3)
			\psellipse(-0.5,0)(0.3,0.5)
			\psellipse( 0.5,0)(0.3,0.5)
		\end{centerpict}
		\\[2ex]
		\label{eq:rel-exceptional-nonplanar}
		\psset{linewidth=0.5pt,fillstyle=solid,fillcolor=white}
		\begin{centerpict}(-1,-0.8)(1,1)
			\SpecialCoor
			\psline[linewidth=1.5pt](0,0)(1;90)
			\psline[linewidth=1.5pt](0,0)(1;210)
			\psline[linewidth=1.5pt](0,0)(1;330)
			\psccurve[fillstyle=none,curvature=.75 .1 0](1;90)(1;210)(1;330)
			\psccurve(0.6;90)(0.15;330)(0.6;210)(0.15;90)(0.6;330)(0.15;210)
		\end{centerpict}
		\hskip 1cm
		\begin{centerpict}(-1,-0.8)(1,1)
			\SpecialCoor
			\pspolygon[linewidth=1.5pt](0.7;105)(0.7;75)(0.7;345)(0.7;315)(0.7;225)(0.7;195)
			\psccurve[curvature=.75 .1 0]%
				(0.8;150)(0.9;105)(0.2;210)(0.9;315)%
				(0.8;270)(0.9;225)(0.2;330)(0.9; 75)%
				(0.8; 30)(0.9;345)(0.2; 90)(0.9;195)
		\end{centerpict}
	\end{gather}
	where the~crossings in the~last two diagrams are the~artifacts of projecting the~diagrams to the~plane (singular levels of the~corresponding homotopies are not planar). 
	\end{enumerate}
\end{proposition}
\begin{proof}
	We shall analyze the~three groups of homotopies from Section~\ref{sec:framed-function} on page~\pageref{rel:far-away}.
	
	\smallskip\noindent
	\emph{Group I: two changes occur simultaneously at different levels} \eqref{rel:far-away}. This is the~exchange law for 2-morphisms, so that this group does not introduce new relations.
	
	\smallskip\noindent
	\emph{Group II: nontransverse changes} \eqref{rel:inverses}. These imply a~change followed by its inverse is equivalent to the~trivial one. Again, no interesting relations.
	
	\smallskip\noindent
	\emph{Group III: several critical points at the~same level} \eqref{rel:permutations}. This group introduces interesting relations between generating 2-morphisms. A~homotopy $H_{s,t}$ from this group admits a~\emph{singular level}: the~critical level of some $H_{s_0,t_0}$ containing all the~critical points (either three Morse singularities, or one Morse and one birth-death point). Denote by $\Gamma_{\!H}$ the~components of the~singular level carrying the~critical points; it is a~graph with two types of vertices: 4-valent ones for Morse singularies, and 2-valent to birth-death singularities.
	
	If the~graph $\Gamma_{\!H}$ is disconnected, it must have a~component with a~single 4-valent vertex. In such a~case the~relation imposed by $H$ is commutativity of the~left square in \eqref{eq:rel-dsum-csum-vs-change}, where the~cobordism $W$ contains the~component of $\Gamma_{\!H}$ with a~single 4-vertex, and $\alpha$ is a~change encoded by the~other components (a~creation or annihilation if the~component contains one 2-valent vertex, or a~permutation otherwise).
	
	If $\Gamma_{\!H}$ is 2-connected, break its two edges to obtain two components. The~reverse operation is the~connected sum---this shows a~homotopy with such a~graph impose commutativity of the~right square in \eqref{eq:rel-dsum-csum-vs-change}.
	
	\wrapfigure{%
		\begin{pspicture}(-1.5,-0.6)(1.5,0.6)
			\psdot(-1,0)\psdot(0,0)\psdot(1,0)
			\pscurve(-1,0)(-0.7, 0.2)(-0.3, 0.2)(0,0)\pscurve(1,0)(0.7, 0.2)(0.3, 0.2)(0,0)
			\pscurve(-1,0)(-0.7,-0.2)(-0.3,-0.2)(0,0)\pscurve(1,0)(0.7,-0.2)(0.3,-0.2)(0,0)
			\pscurve(-1,0)(-1.3, 0.2)(-1.3, 0.3)(0, 0.6)(1.3, 0.3)(1.3, 0.2)(1,0)
			\pscurve(-1,0)(-1.3,-0.2)(-1.3,-0.3)(0,-0.6)(1.3,-0.3)(1.3,-0.2)(1,0)
		\end{pspicture}}
	Finally, $\Gamma_{\!H}$ can be 4-connected, which requires three 4-valent vertices. There is only one such graph, shown to the~right. Take a~look on a~regular level of $H_{s_0,t_0}$ just below the~singular one---it is a~collection of circles obtained from $\Gamma_{\!H}$ by replacing a~neighborhood of each vertex with two arcs (not necessarily in a~planar way). Join the~arcs with a~chord to obtain a~three-chord surgery diagram for $H$. All such diagrams are listed in lines \eqref{eq:rel-exceptional-planar} and \eqref{eq:rel-exceptional-nonplanar}.
\end{proof}

\begin{example}
	The~following sequence of homotopies
	\begin{equation}\label{eq:rel-create-perm}
	\psset{unit=7mm}
		\begin{centerpict}(-0.2,0)(1.6,2.4)
				\cobordism[1](0,0)(I)(I)		\cobordism[1](1.2,0)(I)(I)
				\psframe[fillstyle=solid,fillcolor=white,framearc=0.5](-0.2,0.7)(1.4,1.7)
				\rput(0.8,0.2){$\scriptstyle\cdots$}\rput(0.8,2.2){$\scriptstyle\cdots$}
				\rput(0.65,1.2){$W$}
		\end{centerpict}
		\dblto
		\begin{centerpict}(-0.2,0)(2.4,2.4)
				\cobordism[1](0,0)(I)(I)		\cobordism[1](1.2,0)(sI)(srB,2)(rM-R)
				\psframe[fillstyle=solid,fillcolor=white,framearc=0.5](-0.2,0.3)(1.4,1.1)
				\rput(0.8,0.1){$\scriptstyle\cdots$}\rput(0.8,1.8){$\scriptstyle\cdots$}
				\rput(0.65,0.7){$W$}
			\end{centerpict}
		\dblto
		\begin{centerpict}(-0.2,0)(2.4,2.4)
				\cobordism[1](0,0)(I)(I)		\cobordism[1](1.2,0)(srB,2)(rM-R)(sI)
				\psframe[fillstyle=solid,fillcolor=white,framearc=0.5](-0.2,1.3)(1.4,2.1)
				\rput(0.8,0.7){$\scriptstyle\cdots$}\rput(0.8,2.2){$\scriptstyle\cdots$}
				\rput(0.65,1.7){$W$}
		\end{centerpict}
		\dblto
		\begin{centerpict}(-0.2,0)(1.6,2.4)
				\cobordism[1](0,0)(I)(I)			\cobordism[1](1.2,0)(I)(I)
				\psframe[fillstyle=solid,fillcolor=white,framearc=0.5](-0.2,0.7)(1.4,1.7)
				\rput(0.8,0.2){$\scriptstyle\cdots$}\rput(0.8,2.2){$\scriptstyle\cdots$}
				\rput(0.65,1.2){$W$}
		\end{centerpict}
	\end{equation}
	represents a~trivial change for any cobordism $W$, not necessarily connected. Indeed, go around the~right square in \eqref{eq:rel-dsum-csum-vs-change}, where $\alpha$ is a~creation change. Likewise a~similar change involving a~split and a~death is trivial.
\end{example}

\section{Emdedded cobordisms and linearization}\label{sec:cob-linear}
In the~view of the~construction of odd Khovanov homology it is unfortunate to have only one $\Diamond$-change up to inverse, see Tab.~\ref{tab:cube-faces}. One solution to this issue is to use cobordisms embedded in $\Disk\times I$, in which case we can easily define chronological cobordisms with corners---they are necessary to construct the~generalized Khovanov bracket for tangles. These cobordisms have a~natural Riemannian structure induced from the~ambient space.

\begin{definition}
	Given a~natural number $k$ define the~2-category $\EmbChCob(k)$ as follows.
	\begin{enumerate}
	
		\item Objects are families of disjoint circles and $k$ intervals properly embedded in a~two-dimensional disk $\Disk$.
	
		\item A~morphism is a~properly embedded surface $W\subset\Disk\times I$, such that the~restriction $pr|_W$ of the~projection $pr\colon\Disk\times I\to I$ to $W$ is a~separative Morse function. We call it a~\emph{chronology} on $W$ and, as before, we orient critical points of $pr|_W$. Moreover, we assume that $\partial W$ consists of three parts: the~input $W\cap(\Disk\times\{0\})$ of $W$, the~output $W\cap(\Disk\times\{1\})$, and $2k$ vertical lines $W\cap(\partial\Disk\times I)$.
		
		\item Finally, a~$2$-morphism is an~\emph{admissible} diffeotopy $\varphi\colon(\Disk\times I)\times I\to\Disk\times I$, i.e.\ the~one that fixes boundary points and at every moment $t\in I$ the~restriction $pr|_{\varphi_t(W)}$ is an~Igusa function.
		
	\end{enumerate}
	We call $\EmbChCob(k)$ the~$2$-category of \emph{embedded chronological cobordisms}.
\end{definition}

\begin{remark}
	We shall refer to orientations of deaths as \emph{clockwise} or \emph{anticlockwise} by comparing them with the~standard orientation of $\Disk\times\{t\}\subset\Disk\times I$.
\end{remark}

We shall identify cobordisms related by diffeotopies $\varphi_t$ for which $pr|_{\varphi_t(W)}$ is separative Morse at every moment $t\in I$. In~particular, this holds for the~following deformations:
\begin{itemize}
	\item \emph{level-preserving} diffeotopies: $pr\circ\varphi_t = pr$ for every $t\in I$,
	\item \emph{vertical} diffeotopies: $\varphi_t(p,z) = (p, h_t(z))$ for some diffeotopy $h_t$ of the~interval $I$.
\end{itemize}
Another important family consists of locally vertical diffeotopies---they are vertical only over a~collection of disks, while constant beyond them.

\begin{definition}
	Choose a~family of disjoint vertical cylinders $C_1,\dots,C_r$ in $\Disk\times I$ and an~embedded chronological cobordism $W$ that is vertical in the~annular neighborhood of each $\partial C_i$. A~diffeotopy $\varphi_t$ is \emph{locally vertical} if it is vertical on all $C_i$'s, but fixes all points outside them except small annular neighborhoods of $\partial C_i$'s, in which we interpolate the~two behaviors.
\end{definition}

The~requirement that $W$ intersects each $\partial C_i$ in vertical lines implies that $\varphi_t$ cannot create critical points. Hence, each interpolation $(1-s)\varphi_1+s\id$ induces a~chronology on $W$, so that locally vertical diffeotopies can be \quot{straightened up} (compare this with Theorem~\ref{thm:framed-contr}).

\begin{proposition}\label{prop:loc-vert-unique}
	Let $\varphi_t$ and $\varphi'_t$ be diffeotopies locally vertical with respect to the~same family of cylinders. If $\varphi_1=\varphi'_1$, then they are homotopic in the~space of admissible diffeotopies. In~particular, a~locally vertical diffeotopy $\varphi_t$ satisfying $\varphi_1=\id$ is trivial.
\end{proposition}
\begin{proof}
	Take a~linear homotopy $h_{t,s} := s\varphi_t + (1-s)\varphi'_t$. Because both $\varphi_t$ and $\varphi'_t$ are locally vertical, each $h_{t,s}$ is a~diffeomorphism of $\Disk\times I$ such that $pr|_{h_{t,s}(W)}$ is a~Morse function.
\end{proof}

The~proposition above makes it possible to define disjoint unions in $\EmbChCob(0)$ (more general operations on all categories $\EmbChCob(k)$ are defined in Section~\ref{sec:tangles}). Given embedded cobordisms $W$ and $W'$ with no corners we define the~`left-then-right' and `right-then-left' \emph{disjoint unions} $W\ldsum W'$ and $W\rdsum W'$ by placing the~cobordisms next to each other and pushing the~critical points of $W$ below or above those of $W'$ respectively. The~disjoint union permutation $\sigma^{\sqcup}_{W,W'}\colon W\ldsum W'\dblto W\rdsum W'$ is realized as a~locally vertical diffeotopy, so that it equips $\rdsum$ with a~structure of a~cubical functor.

\begin{corollary}
	$\EmbChCob(0)$ is a~Gray monoidal category, with a~monoidal structure given by the~\quot{right-then-left} disjoint union $\rdsum$
\end{corollary}

\begin{remark}
	This monoidal structure is strictly braided (see Definition~\ref{def:strict-symmetry}) with a~braiding induced by twists {\psset{unit=0.6cm}\textcobordism[2](rP) and \textcobordism[2](lP)}. We shall not use this fact in our paper.
\end{remark}

The~connected sum $W\rcsum W'$ of embedded cobordisms with no corners is formed from $W\rdsum W'$ by performing a~surgery along a~vertical curtain in $\Disk\times I$ with one edge on $W$ and the~other on $W'$. Again, there is some choice involved, and to make it a~well defined operation one has to decorate objects and morphisms of $\EmbChCob(0)$ with additional data, such as embedded arcs originating at the~circles and ending at the~boundary of $\Disk$.

\wrapfigure[r]<1>{\begin{pspicture}(-0.8,0)(0.8,2.6)\rput(0,0.4){\diagConnT{}{->}{}{->}}\rput(0,1.8){\diagConnT{}{->}{}{<-}}\end{pspicture}}
The~2-category $\EmbChCob(0)$ is a~finer version of abstract cobordisms. For instance, there are two kinds of merges, depending on whether the~input circles are nested or not, and likewise for splits. We shall usually ignore this additional structure except one case: we split $\Diamond$-changes into two groups using the~intersection number of the~two arrows in their surgery description (the~two-arrow diagrams). In other words, rotate the~diagram to make the~inner arrow points upwards, and check the~direction of the~outer one---it points either to the~left or to the~right (as shown in the~diagrams to the~right), and the~two changes encoded by the~diagrams are not equivalent.

\subsection*{Linearization of cobordisms}
The~$2$-category $\EmbChCob(0)$ is not good for homological constructions and we shall \quot{linearize} it. More precisely, choose a~commutative ring $R$ with a~function $\iota\colon 2\Mor(\EmbChCob(0))\to R$ that is multiplicative with respect to both compositions of $2$-morphisms, and define a~category $R\EmbChCob_\iota(0)$ as follows:
\begin{enumerate}
	\item the~set of objects is not changed and it consists of families of circles in the~plane $\Disk$, and
	\item morphisms are finite linear combinations of chronological cobordisms $r_1W_1 + \ldots + r_kW_k$, with $r_i\in R$, modulo \emph{chronological relations\/ }$W'=\iota(\varphi)W$, one per every $2$-morphism $\varphi\colon W\dblto W'$.
\end{enumerate}
We extend the~composition of cobordisms to formal sums in a~linear way. The~function $\iota$ can be considered as a~part of a~2-functor $\EmbChCob(0)\to R\EmbChCob_\iota(0)$, where 2-morphisms in the~target category are scalings by elements of the~ring $R$. We want this functor to be \quot{faithful enough} to support the~construction of odd Khovanov homology. We start with a~few observations.

\begin{lemma}
	For any function $\iota$ as above there is another one, $\hat\iota$, which assigns $1$ to creations and annihilations, such that the~linearizations $R\EmbChCob_\iota(0)$ and $R\EmbChCob_{\hat\iota}(0)$ are isomorphic.
\end{lemma}
\begin{proof}
	Each of the~three creations \eqref{change:creation-annihilation} involve different generators. Hence, we can force the~coefficients associated to them to be $1$ by scaling births and deaths accordingly.
\end{proof}

\begin{lemma}\label{lem:disjoint-vs-connected-permutation}
	We have $\iota(\sigma^{\sqcup}_{W,W'}) = \iota(\sigma^{\connsum}_{W,W'})$ whenever each of\/ $W$ and $W'$ is a~merge or a~split.
\end{lemma}
\begin{proof}
	It follows from the~right square in \eqref{eq:rel-dsum-csum-vs-change} for the~cobordism $W$ and the~connected sum permutation $\alpha:=\sigma^{\connsum}\colon M\lcsum W'\dblto M\rcsum W'$, where $M$ is a~merge. Indeed, commutativity of the~square implies
	\begin{equation}
		\iota(\sigma^{\connsum}_{W,M})\iota(\sigma^{\sqcup}_{W,W'})\iota(\sigma^{\connsum}_{M,W'})
		=
		\iota(\sigma^{\connsum}_{M,W'})\iota(\sigma^{\connsum}_{W,W'})\iota(\sigma^{\connsum}_{W,M})
	\end{equation}
	so that the~middle terms must be equal.
\end{proof}

As a~result, we have to specify $\iota$ only for disconnected union permutations and exceptional changes. Instead of finding the~most general formula, and keeping in mind we want to regard embedded cobordisms as close to the~abstract ones as possible, we shall define $\iota(\sigma^{\sqcup}_{W,W'})$ using the~following \emph{bidegree} $\chdeg W\in \Z\times\Z$, which counts critical points of the~cobordism $W$ as follows:
\begin{equation}\label{eq:2chcob-degree}
	\chdeg W := (\#\text{births}-\#\text{merges}, \#\text{deaths}-\#\text{splits}).
\end{equation}
The~following result shows a~connection between this bidegree with other topological properties of a~cobordism.

\begin{lemma}\label{lem:chdeg-vs-bdry}
	Given a~chronological cobordism $W$ of degree $\chdeg W = (a,b)$ with $n$ inputs and $m$ outputs, $a+b=\chi(W)$ and $a+n = b+m$.
\end{lemma}
\begin{proof}
	Straightforward, by checking for generating cobordisms \eqref{diag:mor-gens}.
\end{proof}

It follows the~bidegree is preserved by changes of chronologies, so that $R\EmbChCob(0)$ is a~graded category (morphisms between two objects form an~$R$-module graded by $\Z\times\Z$, and the~degree function is additive with respect to composition). The~following is determined by the~requirement that $\iota(\sigma^{\sqcup}_{W,W'})$ depends only on the~degrees of $W$ and $W'$.

\begin{table}
	\def\arraystretch{0.7}%
	\centering
	\psset{linewidth=0.5pt,fillstyle=solid,fillcolor=white}
	\begin{tabular}{cccc}%
		\begin{pspicture}(-1.25,-0.2)(1.25,1)
			\pspolygon[linewidth=1.5pt](-0.519,-0.3)(0.519,-0.3)(0,0.6)
			\pscircle(-0.519,-0.3){0.3}
			\pscircle( 0.519,-0.3){0.3}
			\pscircle(0,0.6){0.3}
		\end{pspicture}&%
		\begin{pspicture}(-1.25,-0.5)(1.25,0.6)
			\psellipse[linewidth=1.5pt](0,0)(0.8,0.2)
			\pscircle(0.5,0){0.3}
			\psclip{\pscircle(-0.5,0){0.3}}
				\psline[linewidth=1.5pt](-0.8,0)(-0.2,0)
			\endpsclip
		\end{pspicture}&%
		\begin{pspicture}(-1.25,-0.5)(1.25,0.6)
			\psellipse[linewidth=1.5pt](0,0.2)(0.7,0.3)
			\psclip{\psellipse(0,0)(0.5,0.3)}
				\psline[linewidth=1.5pt](-0.2,-0.3)(-0.2,0.3)
				\psline[linewidth=1.5pt]( 0.2,-0.3)( 0.2,0.3)
			\endpsclip
		\end{pspicture}&%
		\begin{pspicture}(-1.25,-0.4)(1.25,0.6)
			\psline[linewidth=1.5pt](-0.5,-0.3)(0.5,-0.3)
			\psline[linewidth=1.5pt](-0.5, 0  )(0.5, 0  )
			\psline[linewidth=1.5pt](-0.5, 0.3)(0.5, 0.3)
			\psellipse(-0.5,0)(0.3,0.5)
			\psellipse( 0.5,0)(0.3,0.5)
		\end{pspicture}\\
		$\begin{array}{c}\\
									   \scriptstyle(03|00|300)\\
										 \scriptstyle(21|00|300)\end{array}$ &
		$\begin{array}{c}\scriptstyle(10|11|100)\\
										 \scriptstyle(01|20|100)\\
										 \scriptstyle(01|02|100)\end{array}$ &
		$\begin{array}{c}\scriptstyle(10|20|010)\\
										 \scriptstyle(10|02|010)\\
										 \scriptstyle(01|11|010)\end{array}$ &
		$\begin{array}{c}\\
		                 \scriptstyle(30|00|030)\\
										 \scriptstyle(11|00|030)\end{array}$
	\end{tabular}
	\vskip 0.5\baselineskip
	\caption[Surgery diagrams of homotopies relating permutation changes]{Surgery diagrams of homotopies relating permutation changes. The~numbers below each diagram count how many times various permutations occur: $\times$-changes with parallel or opposite arrows (the~first group), $\Diamond$-changes with outer arrows oriented to the~left or to the~right (the~second group) and the~other changes grouped by the~value of $\iota$ (respectively $\permMM$, $\permSS$ and $\permMS$). Different sequences correspond to different orientations of chords.}
	\label{tab:chcob3-hexagon-rels}
\end{table}

\begin{proposition}
	Choose invertible elements $\permMM,\permSS,\permMS\in R$ such that $\permMM^2=\permSS^2=1$ and define $\iota$ on generating changes of chronologies by the~following rules:
	\begin{enumerate}
		\item creations and annihilations are sent to $1$,
		\item the~coefficient associated to a~disjoint union and connected sum permutation involving cobordisms of degrees $(a,b)$ and $(c,d)$ is given by $\lambda(a,b,c,d)=\permMM^{ac}\permSS^{bd}\permMS^{ad-bc}$,
		\item a~$\times$-change is sent to $\permSS$ if the~arrows point to the~same circle and to $\permMM$ otherwise, and
		\item a~$\Diamond$-change with a~diagram in which the~inner arrow is oriented upwards is sent to $1$ or $\permMM\permSS$ depending on whether the~outer arrow is oriented to the~left or to the~right respectively.
	\end{enumerate}
	Then $\iota\colon2\Mor(\EmbChCob(0))\to R$ is a~well-defined multiplicative function.
\end{proposition}
\begin{proof}
	First, coherence of $\iota$ with the~interchange law for 2-morphisms \eqref{eq:interchange-law} follows from commutativity of $R$. Next, $\iota(\alpha)\iota(\alpha^{-1})=1$ for every elementary change $\alpha$: this is trivial for creations and annihilations, and follows easily for disjoint union and connected sum permutations from the~way $\lambda$ is defined. If $\alpha$ is an~exceptional permutations, $\iota(\alpha^{-1})=\iota(\alpha)$ is a~square root of 1.
	
	The~commutativity of squares \eqref{eq:rel-dsum-csum-vs-change}, in particular the~triviality of \eqref{eq:rel-create-perm}, is preserved due to the~way $\lambda$ is defined---it is a~group homomorphism in each variable. Finally, it remains to check the~relations given by the~four planar diagrams in \eqref{eq:rel-exceptional-planar}. For that see Tab.~\ref{tab:chcob3-hexagon-rels}: the~numbers below each diagram indicate how many times a~particular elementary change occurs when we go around the~hexagon \eqref{eq:hexagon-relation}. The~product of values of $\iota$ is equal to $1$ in each case.
\end{proof}

\wrapfigure[r]{%
	\begin{tabular}{c|cccc}
		\begin{centerpict}(0,0)(1,1)
			\rput[bl](-0.05,0.15){$W$}
			\rput[tr](1.05,0.85){$W'$}
			\psline[linewidth=0.3pt](0.2,0.8)(0.8,0.2)
		\end{centerpict}
		& birth & merge & split & death\\
		\hline
		birth & $\permMM$ & $\permMM$ & $\permSM$ & $\permMS$\\
		merge & $\permMM$ & $\permMM$ & $\permMS$ & $\permSM$\\
		split & $\permMS$ & $\permSM$ & $\permSS$ & $\permSS$\\
		death & $\permSM$ & $\permMS$ & $\permSS$ & $\permSS$%
	\end{tabular}%
}
We shall often use the~values of $\iota$ for disjoint union and connected sum permutations; they are gathered in the~table to the~right. For instance, we have
\begingroup\psset{unit=7mm}
\begin{align*}
	\permMM\textcobordism[2](M-L)(sI)\textcobordism[0](I)(sB)
	\ =
	\ \textcobordism[2](sI)(M-L)\hskip\COBxsize\psxunit\textcobordism[0](sB)(I)\,,
	\hskip 1cm
	\permMS\textcobordism[2](M-L)(S-B)
	\ =
	\textcobordism[2](S-B,2)(M-L)\,.
	\hskip 0.4\textwidth
\end{align*}
\endgroup

\begin{corollary}\label{cor:orientation-rev}
	The~following rules for reversing orientations hold:
	\begin{equation}\label{rel:reverse-orientation}
		\psset{unit=8mm}%
		\textcobordism[2](M-R) = \permMM\textcobordism[2](M-L),\hskip 1cm
		\textcobordism[1](S-F) = \permSS\textcobordism[1](S-B),\hskip 1cm
		\psset{unit=1cm}%
		\textcobordism[1](sD-) = \permSS\textcobordism[1](sD+).
	\end{equation}
\end{corollary}
\begin{proof}
	The~last rule follows from the~following change
	\begin{equation}\label{eq:reverse-death}
	\textcobordism*[1](I)(sI)(sD+)
		\,\dblto\,
	\textcobordism*[1](S-B)(slD-)(sD+)
		\,\dblto\,
	\textcobordism*[1](S-B)(srD+,2)(sD-)
		\,\dblto\,
	\textcobordism*[1](I)(sI)(sD-)
	\end{equation}
	and the~first one from
	\begin{equation}\label{eq:reverse-merge}
	\psset{unit=8mm}%
	\textcobordism*[2](sI)(M-L)(sI)
		\,\dblto\,
	\textcobordism*{\COBshortCylinderVert(0,0)\COBshortLeftBirth[1](1.6,0)}
								 {\COBmergeFrRight[3](0,0.6)}
								 {\COBmergeFrLeft(0.4,1.8)}
		\,\dblto\,
	\textcobordism*{\COBshortCylinderVert(0,0)\COBshortLeftBirth[1](1.6,0)}
								 {\COBmergeFrLeft[3,2](0,0.6)}
								 {\COBmergeFrRight(0.4,1.8)}
		\,\dblto\,
	\textcobordism*[2](sI)(M-R)(sI)
	\end{equation}
	Reversing an~orientation of a~split is done in a~similar way.
\end{proof}

\begin{remark}
	We shall usually omit the~subscript, writing $R\EmbChCob(0)$ for the~linearized category. If the~choice of $\iota$ is important, we shall write $R\EmbChCob_{abc}(0)$ for the~quotient by chronological relations with parameters $\permMM$, $\permSS$, and $\permMS$ set to $a$, $b$, and $c$ accordingly.
\end{remark}

\begin{remark}
 A~choice of parameters $\permMM$, $\permSS, \permMS\in R$ as above is equivalent to specifying a~ring homomorphism $\scalars\to<1em> R$, where $\scalars:=\scalarsLong*$. Hence, there is a~base change isomorphism $R\EmbChCob(0)\cong R\otimes\kEmbChCob(0)$ implying $\kEmbChCob(0)$ is the~universal linearization of\/ $\EmbChCob(0)$ with respect to the~function $\iota$ defined as in the proposition above.
\end{remark}

From now on we shall take $\scalars$ as the~ring of coefficients. Choose a~change of a~chronology $\varphi\colon W\dblto W'$ that is not a~$\Diamond$-change. Despite $\varphi$ being a~diffeotopy of the~ambient space, the~value $\iota(\varphi)$ depends only the~restriction of $\varphi$ to the~cobordism $W$, which is a~change of a~chronology in the~abstract sense. Even more, given a~diffeomorphic cobordism $\widetilde W\approx W$ and a~corresponding change $\widetilde\varphi$ on $\widetilde W$, $\iota(\widetilde\varphi) = \iota(\varphi)$.

$\Diamond$-changes do not introduce essential relations in $\scalars\EmbChCob(0)$---they force a~merge followed by a~split to be annihilated by $(1-\permMM\permSS)$, a~relation that is a~consequence of the~others, see Corollary~\ref{cor:orientation-rev}. Hence, we can safely forget the~ambient space and identify diffeomorphic cobordisms obtaining another category, which we shall denote by $\kChCob(0)$. Formally, morphisms of $\kChCob(0)$ are $\scalars$-linear combinations of diffeomorphism classes of chronological cobordisms modulo the~relations induced by $\iota$: we set $W'=\iota(\varphi)W$ for any embedding of $W$ and $W'$ into $\Disk\times I$ and a~diffeotopy $\varphi\colon W\dblto W'$.

\begin{remark}
	One should not confuse $\kChCob(0)$ with a~linearization of $2\ChCob$---in the~latter one must have $\permMM=\permSS$ not only because there is only one type, up to inverse, of a~$\Diamond$-change, but this equality is also imposed by the~additional relations coming from the~non-planar diagrams \eqref{eq:rel-exceptional-nonplanar}. This is a~reason why it is so difficult to extend the~definition of odd Khovanov homology to virtual links, even if we restrict to those on orientable surfaces: the~non-planar diagrams \eqref{eq:rel-exceptional-nonplanar} encode the~cube of resolutions for the~virtual Borromean rings, which are realized on a~torus.
\end{remark}

Because we identify in $\kChCob(0)$ diffeomorphic cobordisms, there exists a~cobordism $W$ such that $W=k W$ for some $k\in\scalars$. Indeed, it is enough to find a~nontrivial change of a~chronology between diffeomorphic cobordisms, such as a~permutation of two spheres:
\begin{equation}
	\psset{linewidth=0.5pt,dash=2pt 3pt}
	\def\mysphere(#1){\rput(#1){%
		\pscircle(0,0){0.3}
		\psellipticarc[linestyle=dashed,linewidth=0.5\pslinewidth](0,0)(0.3,0.1){0}{180}
		\psellipticarc(0,0)(0.3,0.1){180}{360}
	}}%
	\begin{centerpict}(-0.5,-0.7)(0.5,0.7)\mysphere(0,0.4)\mysphere(0,-0.4)\end{centerpict}\to
	\begin{centerpict}(-0.7,-0.7)(0.7,0.7)\mysphere(-0.2,0.4)\mysphere(0.2,-0.4)\end{centerpict}\to
	\begin{centerpict}(-0.9,-0.7)(0.9,0.7)\mysphere(-0.4,0.2)\mysphere(0.4,-0.2)\end{centerpict}\to
	\begin{centerpict}(-0.9,-0.7)(0.9,0.7)\mysphere(-0.4,-0.2)\mysphere(0.4,0.2)\end{centerpict}\to
	\begin{centerpict}(-0.7,-0.7)(0.7,0.7)\mysphere(-0.2,-0.4)\mysphere(0.2,0.4)\end{centerpict}\to
	\begin{centerpict}(-0.5,-0.7)(0.5,0.7)\mysphere(0,0.4)\mysphere(0,-0.4)\end{centerpict}
\end{equation}
Another example is a~twice punctured torus---reverse orientations of both saddle points and then rotate the~cobordism. The~following result states that nothing more can happen.

\begin{theorem}\label{thm:aut(W)}
	Choose an~embedded chronological cobordism\/ $W$ in $\kChCob(0)$, where $\scalars=\scalarsLong$, and write $\Aut(W) := \{k\in\scalars\ |\ kW=W\}$. Then
	\begin{equation}
		\Aut(W)=\begin{cases*}
						\{1\},                & if\/ $W$ has genus $0$ and at most one closed component,\\
						\{1, \permMM\permSS\},& otherwise.
					\end{cases*}
	\end{equation}
\end{theorem}

\noindent
A~proof of this theorem is postponed to Section~\ref{sec:chron-Frob}. Notice that elements of $\Aut(W)$ are invertible, since they are products of values of $\iota$.

\section{Khovanov complex}\label{sec:complex}
Now we go back to the~construction of the~generalized Khovanov complex. For this section fix a~link diagram $D$ with $n$ crossings, among which there are $n_+$ positive and $n_-$ negative ones. We need to make a~few choices: enumerate the~crossings, and choose for each of them an~arrow connecting the~two arcs in its horizontal resolution, i.e.\ \fntPosCrArrowUp\ or \fntPosCrArrowDown.

Fig.~\ref{diag:khov-cube} visualizes the~construction for the~trefoil knot. We can first see it as a~diagram $\KhCube{D}$ in the~$2$-category $\EmbChCob(0)$: vertices are $1$-manifolds (resolutions of the~diagram $D$), edges are chronological cobordisms between these manifolds and faces are decorated with changes of chronologies. It should be obvious how to create such a~diagram for a~link diagram $D$ (see the~discussion about diagrams with holes in Section~\ref{sec:cube-picture}). This diagram commutes in the~$2$-categorical sense: a~composition of $2$-morphisms along any $3$-dimensional subcube is trivial:
\begin{equation}\label{eq:faces-commute}
	\psset{xunit=7ex,yunit=7ex}
	\begin{diagps}(-0.2,-1.1)(3.2,1.2)
		\node 000(0, 0)[\scriptstyle000]	\node 100(1, 1)[\scriptstyle100]
		\node 010(1, 0)[\scriptstyle010]	\node 001(1,-1)[\scriptstyle001]
		\node 110(2, 1)[\scriptstyle110]	\node 011(2,-1)[\scriptstyle011]
		\node 111(3, 0)[\scriptstyle111]
		\arrow{->}[000`100;]	\arrow{->}[000`010;]	\arrow{->}[000`001;]
		\arrow{->}[100`110;]	\arrow{->}[010`110;]	\arrow{->}[010`011;]
		\arrow{->}[001`011;]	\arrow{->}[110`111;]	\arrow{->}[011`111;]
		\arrow[nodesep=5ex]{=>}[110`011;]
		\arrow[nodesep=2ex,offsetA=-0.5ex,offsetB=0.5ex]{=>}[100`010;]
		\arrow[nodesep=2ex,offsetA=0.5ex,offsetB=-0.5ex]{=>}[010`001;]
	\end{diagps}
	\quad = \quad
	\begin{diagps}(-0.2,-1.1)(3.2,1.2)
		\node 000(0, 0)[\scriptstyle000]	\node 100(1, 1)[\scriptstyle100]
		\node 001(1,-1)[\scriptstyle001]	\node 110(2, 1)[\scriptstyle110]
		\node 101(2, 0)[\scriptstyle101]	\node 011(2,-1)[\scriptstyle011]
		\node 111(3, 0)[\scriptstyle111]
		\arrow{->}[000`100;]	\arrow{->}[000`001;]	\arrow{->}[100`110;]
		\arrow{->}[100`101;]	\arrow{->}[001`101;]	\arrow{->}[001`011;]
		\arrow{->}[110`111;]	\arrow{->}[101`111;]	\arrow{->}[011`111;]
		\arrow[nodesep=5ex]{=>}[100`001;]
		\arrow[nodesep=2ex,offsetA=0.5ex,offsetB=-0.5ex]{=>}[110`101;]
		\arrow[nodesep=2ex,offsetA=-0.5ex,offsetB=0.5ex]{=>}[101`011;]
	\end{diagps}
\end{equation}
This follows from Proposition~\ref{prop:loc-vert-unique}, as the~two changes are locally vertical with respect to small cylinders over the~crossings of $D$.

Apply the~function $\iota\colon 2\Mor(\EmbChCob(0))\to\scalars$ from Section~\ref{sec:cob-linear} to faces of the~cube $\KhCube{D}$, where $\scalars=\scalarsLong$; the~faces are now decorated with elements from $\invScalars$, the~group of invertible elements in $\scalars$, according to Tab.~\ref{tab:cube-faces}. They define a~$2$-cochain $\psi\in C^2(I^n; \invScalars)$. A~$1$-cochain $\epsilon\in C^1(I^n; \invScalars)$ is called a~\emph{sign assignment} if $d\epsilon = -\psi$. This means the~corrected cube $\KhCubeSigned{D}{\epsilon}$ anticommutes, where $\KhCubeSigned{D}{\epsilon}$ has the~same vertices as $\KhCube{D}$, but for an~edge $\zeta$ one has $\KhCubeSigned{D}{\epsilon}(\zeta) = \epsilon(\zeta)\cdot\KhCube{D}(\zeta)$. Existence of such a~cochain follows easily.

\begin{proposition}\label{prop:cochain-is-cocycle}
	The~cochain $\psi$ is a~cocycle for any link diagram $D$. Hence, $-\psi = d\epsilon$ for some sign assignment $\epsilon$.
\end{proposition}
\begin{proof}
	The~$2$-commutativity of faces \eqref{eq:faces-commute} of any $3$-dimensional subcube in $\KhCube{D}$ implies that $d(-\psi)=d\psi=1$. The~existence of $\epsilon$ follows from the~contractibility of $I^n$.
\end{proof}

Motivated by \cite{DrorCobs} we construct the~generalized Khovanov bracket in the~\emph{additive closure} $\catAdd{\kChCob(0)}$ of the~category $\kChCob(0)$.

\begin{definition}
	The~\emph{additive closure} $\catAdd{\cat{C}}$ of an~$R$-linear category $\cat{C}$, where $R$ is a~commutative ring, is defined as follows:
	\begin{itemize}
		
		\item objects are formal direct sums $\displaystyle{\bigoplus_{i=1}^n C_i}$ of objects from $\cat{C}$,
		
		\item a~morphism $\displaystyle{F\colon\bigoplus_{i=1}^n A_i\to\bigoplus_{j=1}^m B_j}$ is a~matrix $\left(F_{ij}\colon A_j\to B_i\right)$ of morphisms from $\cat{C}$,
		
		\item the~composition of morphisms $F\circ G\vphantom{\displaystyle{\bigoplus^n}}$ mimics the~formula for a~product of matrices
		\begin{equation}
			(F\circ G)_{ij} := \sum_k F_{ik}\circ G_{kj}.
		\end{equation}

	\end{itemize}
	This category is $R$-linear with a~natural action of $R$ and addition defined as addition of matrices: $(F+G)_{ij} := F_{ij} + G_{ij}$.
\end{definition}

We can represent objects of $\catAdd{\cat{C}}$ by finite sequences (vectors) of objects in $\cat{C}$ and morphisms between such sequences by bundles\footnote{
	In the~colloquial sense, not the~mathematical one.
}
(matrices) of morphisms in $\cat{C}$, see Fig.~\ref{fig:additive-closure}. It means each column in Fig.~\ref{diag:khov-cube} forms a~single object $C^i$, as indicated by the~dotted arrows going downwards, and all edges between two columns form a~single morphism $d\colon C^i\to C^{i+1}$. Because every square in $\KhCubeSigned{D}{\epsilon}$ anticommutes, $d^2=0$.

\begin{figure}%
	\psset{unit=6ex,linewidth=0.5pt,arrowsize=3.5pt,dash=3pt 2pt}
	\begin{pspicture}(8,3)
		\psbezier(-0.3,0.8)(-0.45,1.2)(-0.45,1.9)(-0.3,2.3)
		\rput(0,2){$A_1$}
		\rput(0,1){$A_2$}
		\psbezier(0.3,0.8)(0.45,1.2)(0.45,1.9)(0.3,2.3)
		\psellipticarc(2.05,1.5)(0.5,1.2){90}{130}
		\psellipticarc(2.05,1.5)(0.5,1.2){-130}{-90}
		\psline[border=3pt]{->}(0.5,2.1)(3.4,2.6)
		\psline{->}(0.5,2.0)(3.4,1.6)
		\psline{->}(0.5,1.9)(3.4,0.6)
		\psline[linestyle=dashed,border=3pt]{->}(0.5,0.9)(3.4,0.4)
		\psline[linestyle=dashed]{->}(0.5,1.0)(3.4,1.4)
		\psline[linestyle=dashed]{->}(0.5,1.1)(3.4,2.4)
		\psellipticarc[border=3pt](2.05,1.5)(0.5,1.2){-90}{90}
		\rput(2,0){$G$}
		\psbezier(3.7,0.3)(3.55,1.0)(3.55,2.1)(3.7,2.8)
		\rput(4,2.5){$B_1$}
		\rput(4,1.5){$B_2$}
		\rput(4,0.5){$B_3$}
		\psbezier(4.3,0.3)(4.45,1.0)(4.45,2.1)(4.3,2.8)
		\psellipticarc(6.05,1.5)(0.5,1.2){90}{120}
		\psellipticarc(6.05,1.5)(0.5,1.2){-120}{-90}
		\psline[border=3pt]{->}(4.6,0.4)(7.5,0.85)
		\psline{->}(4.6,1.4)(7.5,1.00)
		\psline{->}(4.6,2.4)(7.5,1.15)
		\psline[linestyle=dashed,border=3pt]{->}(4.6,2.6)(7.5,2.15)
		\psline[linestyle=dashed]{->}(4.6,1.6)(7.5,2.00)
		\psline[linestyle=dashed]{->}(4.6,0.6)(7.5,1.85)
		\psellipticarc[border=3pt](6.05,1.5)(0.5,1.2){-90}{90}
		\rput(6,0){$F$}
		\psbezier(7.7,0.8)(7.55,1.2)(7.55,1.9)(7.7,2.3)
		\rput(8,2){$C_1$}
		\rput(8,1){$C_2$}
		\psbezier(8.3,0.8)(8.45,1.2)(8.45,1.9)(8.3,2.3)
	\end{pspicture}
	\caption[The~composition of morphisms in the~additive closure of a~category]{The~composition of morphisms in the~additive closure of a~category. The~component $(F\circ G)_{21}$ is indicated by solid lines.}\label{fig:additive-closure}
\end{figure}
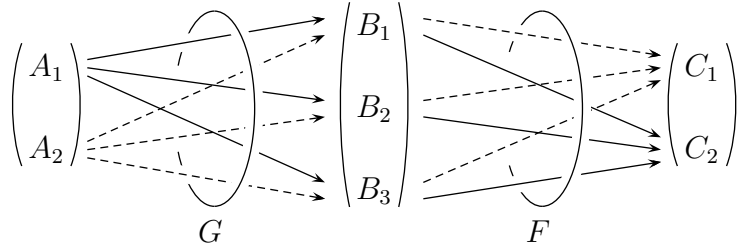

There is one more ingredient to Fig.~\ref{diag:khov-cube}: the~numbers in curly brackets along the~dotted arrows. As usual, this is a~notation for degree shifts.

\begin{definition}\label{def:cat-grad}
	Choose an~abelian group $G$. We say an~$R$-linear category $\cat{C}$ is $G$-\emph{graded}, if
	\begin{enumerate}
		\item for any objects $A,B$ the~set $\Mor(A,B)$ is a~$G$-graded $R$-module such that $\id_A$ is homogeneous of degree $0$ for any object~$A$,
		\item the~degree function is additive with respect to composition: $\deg(f\circ g) = \deg f + \deg g$, for homogeneous $f$ and $g$, and
		\item there is a~\emph{degree shift} functor $\Ob(\cat{C})\times G\ni(A,m)\longmapsto A\{m\}\in\Ob(\cat{C})$ preserving morphisms, i.e.\ $\Mor(A\{m\},B\{n\}) = \Mor(A,B)$, but degrees are changed: if a~morphism $f\in\Mor(A,B)$ has degree $d$, then $\deg f = d+n-m$ when regarded as an~element of $\Mor(A\{m\},B\{n\})$.
	\end{enumerate}
\end{definition}

We have already defined a~$\Z{\times}\Z$-valued degree function for chronological cobordisms \eqref{eq:2chcob-degree}. Here, we shall collapse it to an~integral grading, by summing up both numbers, so that $\deg W=\chi(W)$ is the~Euler characteristic of a~cobordism $W$. Degree shifts are introduced artificially: add formal objects $\Sigma\{m\}$ for every $1$-manifold $\Sigma$ and $m\in\Z$, and extend the~degree map as in the~definition above:
\begin{equation}
	\deg W := \chi(W) + n - m
		\qquad\textrm{for}\qquad
	W\colon\Sigma_0\{m\}\to\Sigma_1\{n\}.
\end{equation}
All cobordisms in the~cube of resolutions have degree $-1$. Hence, taking $C^i\{i\}$ at $i$-th place results in a~complex with a~degree $0$ differential.

This is the~last piece of the~construction. Below we summarize everything, giving a~full definition of the~bracket.

\begin{definition}
	Given a~link diagram $D$ with $n$ crossings construct its cube of resolutions $\KhCube{D}$ and choose a~sign assignment $\epsilon$. The~\emph{generalized Khovanov bracket} of $D$ is a~chain complex $\KhBracket{D}_\epsilon$ with:
	\begin{align}
		\KhBracket{D}_\epsilon^i & := \bigoplus_{\mathclap{\|\xi\|=i}}D_\xi\{i\}, \\
		d_\epsilon|_{D_\xi} & := \sum_{\mathclap{\zeta\colon\xi\to\xi'}} \epsilon(\zeta)\cdot D_\zeta,
	\end{align}
	where $\|\xi\| := \xi_1+\ldots+\xi_n$ is the~\emph{weight} of a~vertex $\xi$.
\end{definition}

\begin{corollary}
	The~sequence $(C,d)$ at the~bottom line of Fig.~\ref{diag:khov-cube} is a~chain
	complex.
\end{corollary}
\begin{proof}
	This follows from anticommutativity of the~corrected cube
	$\KhCubeSigned{D}{\epsilon}$.
\end{proof}

There are a~few choices involved in the~construction of $\KhBracket{D}_\epsilon$: an~order of crossings, arrows at the~crossings, and the~sign assignment $\epsilon$. We shall now show that different choices lead to isomorphic complexes. First, different orientation of the~arrows over crossings can be compensated by an~edge assignment.

\begin{lemma}\label{lem:kh-cube-indep-of-arr}
	Let $D_1, D_2$ be diagrams of a~link $L$ with~$n$ crossings, which differ only in directions of arrows over crossings. Then for any sign assignment~$\epsilon_1$ for $\KhCube{D_1}$ there exists a~sign assignment $\epsilon_2$ for $\KhCube{D_2}$ such that $\KhCubeSigned{D_1}{\epsilon_1} = \KhCubeSigned{D_2}{\epsilon_2}$.
\end{lemma}
\begin{proof}
	Without loss of generality we may assume $D_1$ and $D_2$ differ only in the~direction of the~arrow at the~$i$-th crossing. Reversing the~arrow changes orientation of critical points of	cobordisms at edges $\zeta$ with $\zeta_i=*$. Let $\psi_i$ be the~commutativity cocycle of the~cube $\KhCube{D_i}$. Given a~sign assignment $\epsilon_1$ for $\KhCube{D_1}$ we define
	\begin{equation}
		\epsilon_2(\zeta) := 
			\begin{cases*}
				\phantom{\permMM}\epsilon_1(\zeta),	& if $\zeta_i\neq *$,\\
				\permMM\epsilon_1(\zeta),	& if $\zeta_i=*$ and $D_\zeta$ is a~merge,\\
				\permSS\epsilon_1(\zeta),	& if $\zeta_i=*$ and $D_\zeta$ is a~split.
			\end{cases*}
	\end{equation}
	A~direct computation shows $d\epsilon_2 = -\psi_2$, and $\epsilon_2$ is the~desired sign assignment for $\KhCube{D_2}$.
\end{proof}

\noindent
A~sign assignment for a~given cube is unique up to an~isomorphism, where an~isomorphism of cubes $\eta\colon\mathcal{I}\to\mathcal{I}'$ is a~collection of invertible morphisms $\eta_\xi\colon\mathcal{I}_\xi\to\mathcal{I}'_\xi$ such that the~square commutes
\begin{equation}\label{eq:cube-morphism}
	\begin{diagps}(0,-1ex)(6em,12ex)
		\square<6em,10ex>[%
			\mathcal{I}_\xi`\mathcal{I}'_\xi`\mathcal{I}_{\xi'}`\mathcal{I}'_{\xi'};%
			\eta_\xi`\mathcal{I}_\zeta`\mathcal{I}'_\zeta`\eta_{\xi'}]
	\end{diagps}
\end{equation}
for every edge $\zeta\colon\xi\to\xi'$.

\begin{lemma}\label{lem:kh-cube-indep-of-sign-assign}
	Let $\epsilon$ and $\epsilon'$ be two sign assignments for $\KhCube{D}$. Then the~cubes $\KhCubeSigned{D}{\epsilon}$ and $\KhCubeSigned{D}{\epsilon'}$ are isomorphic.
\end{lemma}
\begin{proof}
	The~equality $d\epsilon = -\psi = d\epsilon'$ and contractibility of $I^n$ implies that $\epsilon' = d\eta\cdot\epsilon$ for some $0$-cochain $\eta\in C^0(I^n; \invScalars)$. The~family of morphisms $f_\xi := \eta(\xi)\cdot\id$ form then a~desired isomorphism $f\colon\KhCubeSigned{D}{\epsilon}\to\KhCubeSigned{D}{\epsilon'}$.
\end{proof}

An~isomorphism of cubes induces an~isomorphism of complexes, resulting in the~following statement.

\begin{proposition}
	The~isomorphism class of the~Khovanov bracket $\KhBracket{D}_\epsilon$ depends only on the~link diagram $D$.
\end{proposition}
\begin{proof}
	Changing the~order of crossings results in a~different parametrization of the~cube $\KhCube{D}$, but the~chain objects $\KhBracket{D}^i_\epsilon$ are preserved and likewise for the~differential. Independence of the~other choices follows from Lemmas~\ref{lem:kh-cube-indep-of-arr} and \ref{lem:kh-cube-indep-of-sign-assign}, as an~isomorphism of anticommutative cubes descends to an~isomorphism of complexes.
\end{proof}

The~generalized bracket, even up to chain homotopies, is not a~link invariant, but it is not very far from it. To construct an~invariant we have to take an~oriented diagram and shift degrees (both the~internal grading and the~homological one) according to the~number of positive and negative crossings.

\begin{definition}
	Let $D$ be an~oriented link diagram with $n_+$ positive and $n_-$ negative crossings. The~\emph{generalized Khovanov complex} $\KhCom(D)$ of $D$ is obtained from the~bracket $\KhBracket{D}$ by the~degree shifts $\KhCom(D) := \KhBracket{D}[-n_-]\{n_+-2n_-\}$, i.e.\ $\KhCom^i(D) = \KhBracket{D}^{i-n_-}\{n_+-2n_-\}$.
\end{definition}

We shall show invariance of the~complex in Section~\ref{sec:invariance}, but before that we describe an~extension of the~construction to tangles.

\section{Tangles and planar algebras}\label{sec:tangles}
\def\parpic(#1)[#2]#3{}%
Tangles have a~rich algebraic structure called a~planar algebra \cite{JonesPA}: they can be combined together to produce larger tangles, by connecting some of their endpoints. We follow here the~exposition from \cite{DrorCobs}.

\begin{definition}\label{def:diag-plan}
	\wrapfigure[r]{\begin{pspicture}(0,0)(3,2.9)%
		\newgray{diskcolor}{0.875}
		\psset{linewidth=0.3pt,dotsize=3pt,fillstyle=solid,fillcolor=white}%
		\pscircle[fillcolor=diskcolor](1.5,1.5){1.5}\psdot(2.56,0.44)
		\pscircle(0.7,1.7){0.33}\rput(0.7,1.7){$\scriptstyle 1$}\psdot(0.56,1.40)
		\pscircle(2.0,1.5){0.43}\rput(2.0,1.5){$\scriptstyle 2$}\psdot(2.42,1.43)
		\pscircle(1.4,2.4){0.23}\rput(1.4,2.4){$\scriptstyle 3$}\psdot(1.36,2.62)
		\pscircle(1.2,0.6){0.23}\rput(1.2,0.6){$\scriptstyle 4$}\psdot(0.99,0.65)
		\psset{linewidth=0.8pt,fillstyle=none}%
		\pscircle[dimen=mid,fillstyle=none](2.25,2.35){0.15}\psline{->}(2.15,2.458)(2.12,2.444)
		\psbezier    (0.20096,0.75000)(0.56924,0.96262)(0.23000,1.17475)(0.23000,1.60000)	
		\psbezier{<-}(0.43934,2.56066)(0.78696,2.21304)(0.23000,2.09160)(0.23000,1.60000)
		\psbezier{->}(0.70000,2.03000)(0.70000,2.45296)(1.04482,2.47613)(0.86607,2.85946)	
		\psbezier{->}(2.97721,1.23953)(2.52561,1.31916)(2.19411,0.39913)(2.07467,1.07653)	
		\psbezier{->}(0.98697,0.09046)(1.05448,0.27595)(0.65892,0.28761)(1.00081,0.48500)	
		\psbezier    (2.21500,1.87239)(2.34778,2.10237)(1.90000,2.08756)(1.90000,2.30000)	
		\psbezier{->}(1.90000,2.30000)(1.90000,2.48487)(1.83934,2.65366)(1.59919,2.51500)
		\psbezier{->}(1.12134,0.81613)(0.99623,1.15984)(1.46094,1.26067)(0.98579,1.53500)	
		\psbezier{->}(1.36263,0.43737)(1.70763,0.09237)(1.70763,1.10763)(1.36263,0.76263)	
		\pscurve{->}(1.4,2.17)(1.35,1.5)(1.7,0.95)(2.05,0.8)%
								(2.1,0.6)(1.9,0.4)(1.79,0.25)(1.76047,0.02279)
	\end{pspicture}}
	A~\emph{planar arc diagram} $D$ with~$d$ inputs is a~disk $\Disk$ missing $d$ smaller disks $\Disk_i$, together with a~proper embedding of disjoint circles and closed intervals. Each boundary component carries a~basepoint and meets an~even number of intervals. We say that $D$ is \emph{oriented} if the~embedded circles and intervals are oriented. Both oriented and non-oriented planar arc diagrams are considered up to planar isotopies.
\end{definition}

We can compose planar arc diagrams by placing one of them in a~hole of another. This operation is associative: when composing more than two diagrams, the~final result does not depend in which order they are composed.

\begin{definition}\label{def:alg-plan}
	A~\emph{planar algebra} $\mathcal{P}$ is a~collection of sets $\mathcal{P}(k)$ together with operators
	\begin{equation}
		D\colon\mathcal{P}(k_1)\times\cdots\times\mathcal{P}(k_s)\to\mathcal{P}(k),
	\end{equation}
	one for each planar arc diagram $D$, whose composition is associative and radial diagrams (i.e.\ those with a~single input and radially emdedded intervals) correspond to identity maps. An~\emph{oriented planar algebra} is defined similarly, using oriented planar arc diagrams.
\end{definition}

\begin{example}
	Given a~planar arc diagram $D$ we can insert into its holes some tangle diagrams, creating another tangle diagram. This results in a~map
	\begin{equation}\label{eq:planar-operator}
		D\colon\mathcal{T}^0(k_1)\times\cdots\times\mathcal{T}^0(k_s)\to\mathcal{T}^0(k),
	\end{equation}
	where $\mathcal{T}^0(k)$ is the set of all tangle diagrams with $2k$ endpoints embedded in a~disk $\Disk$ with a~basepoint on its boundary (the~basepoints make this operation well-defined). Because Reidemeister moves are local, we can replace $\mathcal{T}^0(k)$ with sets of tangles $\mathcal{T}(k)$ and the~operation induced by $D$ is still well-defined. In a~similar way oriented diagrams allow us to combine oriented tangles. Here, we group tangles (or tangle diagrams) into sets $\mathcal{T}_+(\!\vec{\,s})$ (respectively $\mathcal{T}^0_+(\!\vec{\,s})$), labeled with finite sequences $\!\vec{\,s}$ of \raisebox{0.2ex}{$\scriptstyle+$}'s and \raisebox{0.2ex}{$\scriptstyle-$}'s encoding orientation of the~endpoints.
\end{example}

\begin{definition}\label{def:alg-plan-mor}
	A~\emph{morphism} of planar algebras $\Phi\colon\mathcal{P}_1\longrightarrow\mathcal{P}_2$ is a~collection of morphisms $\Phi_k\colon\mathcal{P}_1(k)\longrightarrow\mathcal{P}_2(k)$ commuting with planar operators, i.e.
	\begin{equation}
		D\circ(\Phi_{k_1},\dots,\Phi_{k_s}) = \Phi_k\circ D.
	\end{equation}
	for every operator $D$. In a~similar way one defines morphisms of oriented planar algebras.
\end{definition}

\begin{example}
	There is a~natural morphism from the~planar algebra of tangle diagrams to the~planar algebra of tangles that maps a~tangle diagram into the~tangle it represents.
\end{example}

We shall now construct the~Khovanov complex for a~tangle diagram $T$ with $2k$ endpoints. As mentioned in Section~\ref{sec:cube-picture}, we can construct the~cube of resolutions $\KhCube{T}$ using cobordisms with corners $\EmbChCob(k)$. These $2$-categories also form a~planar algebra, with a~cubical functor
\begin{equation}\label{eq:cubical-D-functor}
	D\colon\EmbChCob(k_1)\times\dots\EmbChCob(k_d)\longrightarrow\EmbChCob(k)
\end{equation}
associated to every planar arc diagram $D$ as follows.
\begin{itemize}
	\item The~object $D(\Sigma_1,\dots,\Sigma_d)$ is defined in the~same way as a~composition of tangle diagrams: simply insert the~pictures $\Sigma_i$
				into holes of $D$.
	
	\item\wrapfigure[r]{\begin{pspicture}(-1.5,0)(1.5,2)
		\psset{linewidth=0.5pt,dimen=outer,xunit=1.2,dash=3pt 2pt}
		\psellipse(0,0.2)(1.3,0.2)
		\psset{border=1pt,fillstyle=solid,fillcolor=white}
		\psframe(-0.95,0.2)(-0.45,1.8)
		\psframe(-0.25,0.2)( 0.25,1.8)
		\psframe( 0.45,0.2)( 0.95,1.8)
		\psellipse[fillstyle=none](0,1.8)(1.3,0.2)
		\psset{border=0pt}
		\psellipticarc(-0.7,0.2)(0.25,0.05){-180}{0}
		\psdash\psellipticarc(-0.7,0.2)(0.25,0.05){0}{180}
		\psellipse(-0.7,1.8)(0.25,0.05)
		\psellipticarc( 0.0,0.2)(0.25,0.05){-180}{0}
		\psdash\psellipticarc( 0.0,0.2)(0.25,0.05){0}{180}
		\psellipse( 0.0,1.8)(0.25,0.05)
		\psellipticarc( 0.7,0.2)(0.25,0.05){-180}{0}
		\psdash\psellipticarc( 0.7,0.2)(0.25,0.05){0}{180}
		\psellipse( 0.7,1.8)(0.25,0.05)
		\psset{fillstyle=none}
		\psline(-1.3,0.2)(-1.3,1.8)\psline(1.3,0.2)(1.3,1.8)
		\psset{linewidth=0.3pt}
		\psellipticarc(-0.7,1.2)(0.25,0.05){-180}{0}
		\psellipticarc( 0.0,0.7)(0.25,0.05){-180}{0}
		\psellipticarc( 0.0,1.2)(0.25,0.05){-180}{0}
		\psellipticarc( 0.7,0.7)(0.25,0.05){-180}{0}
		\psset{linestyle=dashed}
		\psellipticarc(-0.7,1.2)(0.25,0.05){0}{180}
		\psellipticarc( 0.0,0.7)(0.25,0.05){0}{180}
		\psellipticarc( 0.0,1.2)(0.25,0.05){0}{180}
		\psellipticarc( 0.7,0.7)(0.25,0.05){0}{180}
		\rput(-0.7,1.4){$\scriptstyle W_1$}
		\rput( 0.0,0.9){$\scriptstyle W_2$}
		\rput( 0.7,0.4){$\scriptstyle W_3$}
	\end{pspicture}}
				For cobordisms take a~\emph{curtain diagram} $D\times I$ (see Fig.~\ref{fig:curtains} for an~example) and fill its holes with the~cobordisms. Here, one has to do the~same trick as with the~disjoint sum---to shift all critical points, placing the~critical points of the~first cobordism at the~top, below the~critical points of the~second one and so on.
	
	\item\wrapfigure[r]{}
				Finally, for changes of chronologies $\alpha_i\colon W_i\Rightarrow W'_i$ there is an~induced change
				\begin{equation}
					D(\alpha_1,\dots,\alpha_d)\colon D(W_1,\dots,W_d)\Longrightarrow D(W'_1,\dots,W'_d)
				\end{equation}
				defined as a~composition $\bar\alpha_1\circ\ldots\circ\bar\alpha_d$, where $\bar\alpha_i=D(\id_{W_1},\ldots,\alpha_i,\ldots,\id_{W'_d})$ is given by $\alpha_i$ in the~$i$-th hole (reparametrized accordingly) and fixed beyond it. Simply speaking, all changes $\alpha_i$ are applied at the~same time, but on different regions.\footnote{
					A~change defined in this way might not be generic.
				}
\end{itemize}
It follows directly from Proposition~\ref{prop:loc-vert-unique} that the~functors defined above are cubical. They extend naturally to cubes in $\EmbChCob(k)$.

\begin{figure}
 	\psset{xunit=2cm,yunit=0.7cm}%
	\begin{pspicture}(0,0)(3,5)%
		\newgray{tubecolor}{0.875}%
		\psset{linewidth=0.3pt,dotsize=3pt,fillstyle=none,dimen=middle}%
		\psset{hatchcolor=tubecolor,hatchwidth=0.3pt,hatchsep=0.5pt,hatchangle=0}%
		\psellipse(1.5,1.5)(1.5,1.5)
		\psframe[linestyle=none,fillstyle=solid,fillcolor=white](0.20096,1.5)(2.97721,3.1)
		\psellipse(1.5,3.5)(1.5,1.5)
		\psline(0.00,1.5)(0.00,3.5)
		\psline(3.00,1.5)(3.00,3.5)
		\psclip{\psframe[linestyle=none](0.2,1.0)(0.86,2.0)}%
			\psellipticarc(0.7,1.7)(0.33,0.33){-180}{0}
		\endpsclip
		\psellipse[fillstyle=vlines](0.7,3.7)(0.33,0.33)\rput(0.8,3.7){$\scriptstyle 1$}
		\psline(0.37,1.7)(0.37,3.7)
		\psline(1.03,3.49)(1.03,3.7)
		\psellipse[fillstyle=vlines](2.0,3.5)(0.43,0.43)\rput(2.00,3.5){$\scriptstyle 2$}
		\psline(1.57,3.05)(1.57,3.5)
		\psline(2.43,2.95)(2.43,3.5)
		\psellipse[fillstyle=vlines](1.4,4.4)(0.23,0.23)\rput(1.40,4.4){$\scriptstyle 3$}
		\psline(1.17,3.4)(1.17,4.4)
		\psline(1.63,3.72)(1.63,4.4)
		\psclip{\psframe[linestyle=none](1,0)(1.365,0.6)}
			\psellipticarc(1.2,0.6)(0.23,0.23){-180}{0}
		\endpsclip
		\psellipse[fillstyle=vlines](1.2,2.6)(0.23,0.23)\rput(1.25,2.6){$\scriptstyle 4$}
		\psline(1.43,2.4)(1.43,2.6)
		\psline(0.97,2.5)(0.97,2.6)
		\psset{linewidth=0.8pt}
		\psclip{\psframe[linestyle=none](0.2,0.6)(0.4,1.0)}%
			\psbezier(0.20096,0.75000)(0.56924,0.96262)(0.23000,1.17475)(0.23000,1.60000)
		\endpsclip
		\psline(0.20096,0.75000)(0.20096,2.75000)
		\psline(0.37226,0.98396)(0.37226,2.98396)
		\psline(0.23000,2.78000)(0.23000,3.60000)
		\psline(0.55353,4.00000)(0.55353,4.34298)
		\psline(0.43934,4.10000)(0.43934,4.56066)
		\psbezier(0.20096,2.75000)(0.56924,2.96262)(0.23000,3.17475)(0.23000,3.60000)
		\psbezier(0.43934,4.56066)(0.78696,4.21304)(0.23000,4.09160)(0.23000,3.60000)
		\psline(0.70000,3.37000)(0.70000,4.03000)
		\psline(0.91747,3.93000)(0.91747,4.66190)
		\psline(0.86607,4.50000)(0.86607,4.85946)
		\psbezier(0.70000,4.03000)(0.70000,4.45296)(1.04482,4.47613)(0.86607,4.85946)
		\psclip{\psframe[linestyle=none](2.1,0.5)(3,1.5)}
			\psbezier(2.97721,1.23953)(2.52561,1.31916)(2.19411,0.39913)(2.07467,1.07653)
		\endpsclip
		\psline(2.97721,1.23953)(2.97721,3.23953)
		\psline(2.07467,2.75000)(2.07467,3.07653)
		\psbezier(2.97721,3.23953)(2.52561,3.31916)(2.19411,2.39913)(2.07467,3.07653)
		\psclip{\psframe[linestyle=none](1.62917,0)(1.76047,2)}%
			\pscurve(1.4,2.17)(1.35,1.5)(1.7,0.95)(2.05,0.8)%
		\endpsclip
		\psecurve(2.05,2.8)(2.1,0.6)(1.9,0.4)(1.79,0.25)(1.76047,0.02279)(1.67,0)
		\psline(1.76047,0.02279)(1.76047,2.02279)
		\psline(2.1,0.6)(2.1,2.6)
		\psline(1.34,2.77)(1.34,3.65)
		\psline(1.4,3.3)(1.4,4.17)
		\pscurve(1.4,4.17)(1.35,3.5)(1.7,2.95)(2.05,2.8)%
						(2.1,2.6)(1.9,2.4)(1.79,2.25)(1.76047,2.02279)
		\psclip{\psframe[linestyle=none](0.7,0.0)(1.2,0.355)}%
			\psbezier(0.98697,0.09046)(1.05448,0.27595)(0.65892,0.28761)(1.00081,0.48500)
		\endpsclip
		\psline(0.98697,0.09046)(0.98697,2.09046)
		\psline(0.85969,0.34424)(0.85969,2.34424)
		\psline(1.00081,0.15000)(1.00081,2.48500)
		\psbezier(0.98697,2.09046)(1.05448,2.27595)(0.65892,2.28761)(1.00081,2.48500)
		\psline(1.12134,2.38387)(1.12134,2.81613)
		\psline(1.09908,2.80000)(1.09908,2.93615)
		\psline(1.20492,2.83000)(1.20492,3.28932)
		\psline(0.98579,2.67000)(0.98579,3.53500)
		\psbezier(1.12134,2.81613)(0.99623,3.15984)(1.46094,3.26067)(0.98579,3.53500)
		\psclip{\psframe[linestyle=none](1.3,0.0)(1.8,0.5722)}
			\psbezier(1.36263,0.43737)(1.64248,0.15752)(1.76714,0.86446)(1.40845,0.69720)
		\endpsclip
		\psline(1.36263,0.43737)(1.36263,2.43737)
		\psline(1.62917,0.57211)(1.62917,2.57211)
		\psline(1.40845,2.50280)(1.40845,2.69720)
		\psbezier(1.36263,2.43737)(1.64248,2.15752)(1.76714,2.86446)(1.40845,2.69720)
		\psline(2.21500,3.12761)(2.21500,3.87239)
		\psline(2.23961,3.85000)(2.23961,3.95025)
		\psline(1.90000,3.90000)(1.90000,4.30000)
		\psline(1.59919,4.28500)(1.59919,4.51500)
		\psbezier(2.21500,3.87239)(2.34778,4.10237)(1.90000,4.08756)(1.90000,4.30000)
		\psbezier(1.90000,4.30000)(1.90000,4.48487)(1.83934,4.65366)(1.59919,4.51500)
		\psline(2.1,4.10)(2.1,4.35)
		\psline(2.4,3.65)(2.4,4.35)
		\psellipse[dimen=mid,fillstyle=none](2.25,4.35)(0.15,0.15)
	\end{pspicture}
	\caption{%
		The~planar operator on the~set of~cobordisms with corners associated to the~planar diagram from Definition~\ref{def:diag-plan} consists of a~cylinder with hollow tubes and curtains, i.e.\ properly embedded vertical rectangles.
	}\label{fig:curtains}%
\end{figure}
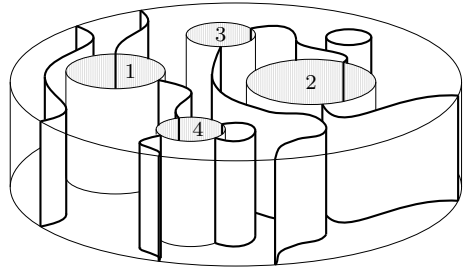

\begin{corollary}\label{cor:plan-alg-cubes}
	The~function $T\mapsto \KhCube{T}$ is a~morphism of planar algebras, i.e.
	\begin{equation}
		\KhCube{T} = D(\KhCube{T_1},\dots,\KhCube{T_d})
	\end{equation}
	for a~planar arc diagram $D$ and tangle diagrams $T_1,\dots,T_d$ with an~appropriate number of endpoints.
\end{corollary}

\begin{remark}
	We can extend the~integral degree to cobordisms with corners by setting $\deg W := \chi(W) - k$ for $W\in\EmbChCob(k)$, and this degree is preserved by planar algebra operators. However, cobordisms with corners do not admit a~natural $\Z\times\Z$-grading defined in Section~\ref{sec:cob-linear}.
\end{remark}

The~main problem when defining the~Khovanov complex for tangles is to understand the~commutativity cocycle $\psi$. For instance, a~single saddle is a~part of both a~merge and a~split:
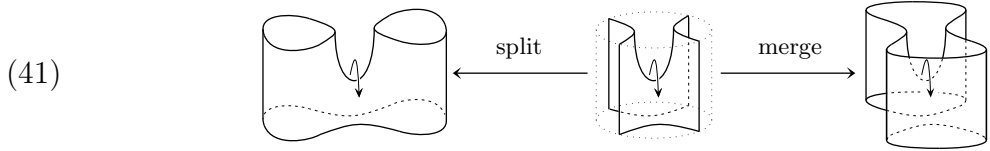
\begin{equation}
	\psset{yunit=1.5cm,xunit=1.2cm,linewidth=0.5pt,dash=1pt 1.5pt}
	\Rnode{split}{\begin{centerpict}(-0.5,-0.2)(1.5,1.2)
		\psbezier(0.0,1.0)(-0.2,1.05)(-0.6,1.05)(-0.5,0.85)
		\psbezier(0.11581, 0.77194)(0.01581,0.72194)(-0.4,0.65)(-0.5,0.85)
		\psbezier(1.0,0.8)(1.2,0.75)(1.6,0.75)(1.5,0.95)
		\psbezier(0.88419,1.02806)(0.98419,1.07806)(1.4,1.15)(1.5,0.95)
		\psbezier(0.88419, 1.02806)(0.68419, 0.92806)(0.6, 0.9)(1.0, 0.8)		
		\psbezier(0.0, 1.0)(0.4, 0.9)(0.31581, 0.87194)(0.11581, 0.77194)		
		\psclip{\psframe[linestyle=none](-1,0.12)(0.12,-1)}
			\psbezier(0.0,0.2)(-0.2,0.25)(-0.6,0.25)(-0.5,0.05)
		\endpsclip
		\psclip{\psframe[linestyle=none](-1,0.10808)(1,1)}
			\psdash\psbezier(0.0,0.2)(-0.2,0.25)(-0.6,0.25)(-0.5,0.05)
		\endpsclip
		\psbezier(0.11581,-0.02806)(0.01581,-0.07806)(-0.4,-0.15)(-0.5,0.05)
		\psclip{\psframe[linestyle=none](0.5,-1)(2,0.09)}
			\psbezier(1.0,0.0)(1.2,-0.05)(1.6,-0.05)(1.5,0.15)
		\endpsclip
		\psdash\psbezier(0.88419,0.22806)(0.98419,0.27806)(1.4,0.35)(1.5,0.15)
		\psdash\psbezier(0.0, 0.2)(0.45, 0.1)(0.53419, 0.12806)(0.88419, 0.22806)	
		\psbezier(0.11581,-0.02806)(0.46581, 0.07194)(0.55, 0.1)(1.0, 0.0)	
		\psbezier(0.28306, 0.88342)(0.28306, 0.3)(0.71694, 0.3)(0.71694, 0.91685)	
		\psbezier[linewidth=0.3pt,border=0.5pt,arrowsize=3pt,arrowlength=1.2]%
						 {<-}(0.55, 0.3)(0.53, 0.7)(0.5, 0.7)(0.45, 0.5)
		\psline(-0.51579,0.10808)(-0.51579,0.90808)
		\psline( 1.51579,0.08192)( 1.51579,0.88192)
	\end{centerpict}}
	\hskip 2cm
	\Rnode{saddle}{\begin{centerpict}(-0.15,-0.2)(1.15,1.2)
		\psbezier(0.0, 1.0)(0.4, 0.9)(0.31581, 0.87194)(0.11581, 0.77194)		
		\psbezier(0.88419, 1.02806)(0.68419, 0.92806)(0.6, 0.9)(1.0, 0.8)		
		\psbezier(0.28306, 0.88342)(0.28306, 0.3)(0.71694, 0.3)(0.71694, 0.91685)	
		\psbezier[linewidth=0.3pt,border=0.5pt,arrowsize=3pt,arrowlength=1.2]%
						 {<-}(0.55, 0.3)(0.53, 0.7)(0.5, 0.7)(0.45, 0.5)
		\psline(0.11581,-0.02806)(0.11581, 0.77194)		
		\psline(1.0, 0.0)(1.0, 0.8)										
		\psbezier(0.11581,-0.02806)(0.46581, 0.07194)(0.55, 0.1)(1.0, 0.0)			
		\psclip{\psframe[linestyle=none](0.11581,0)(0,1)}
			\psbezier(0.0, 0.2)(0.45, 0.1)(0.53419, 0.12806)(0.88419, 0.22806)	
		\endpsclip
		\psclip{\psframe[linestyle=none](0.11581,0)(1,1)}
			\psdash\psbezier(0.0, 0.2)(0.45, 0.1)(0.53419, 0.12806)(0.88419, 0.22806)	
		\endpsclip
		\psclip{\psbezier[linestyle=none](0.88419, 1.02806)(0.68419, 0.92806)(0.6, 0.9)(1.0, 0.8)}
			\psline(0.88419, 0.22806)(0.88419, 1.02806)		
		\endpsclip
		\psclip{\pscustom[linestyle=none]{%
			\moveto(0.88419, 1.02806)
			\curveto(0.68419, 0.92806)(0.6, 0.9)(1.0, 0.8)
			\lineto(1.0,0)\lineto(0.5,0)\lineto(0.5,1.02806)\lineto(0.88419,1.02806)
		}}
			\psdash\psline(0.88419, 0.22806)(0.88419, 1.02806)		
		\endpsclip
		\psline(0.0, 0.2)(0.0, 1.0)										
		\begingroup
			\psset{linestyle=dotted,linewidth=0.3pt,dotsep=2pt,dimen=middle}
			\psellipse(0.5,0.1)(0.64031,0.16008)
			\psellipse(0.5,0.9)(0.64031,0.16008)
			\psline(-0.14031,0.1)(-0.14031,0.9)
			\psline( 1.14031,0.1)( 1.14031,0.9)
		\endgroup
	\end{centerpict}}
	\hskip 2cm
	\Rnode{merge}{\begin{centerpict}(-0.2,-0.2)(1.2,1.2)
		\psbezier(0.11581,0.77194)(-0.08519,0.67194)(0.3,0.65)(0.65,0.65)		
		\psbezier(0.65,0.65)(1.0,0.65)(1.4,0.7)(1.0,0.8)										
		\psbezier(0.0,1.0)(-0.4, 1.1)(0.0,1.15)(0.35,1.15)									
		\psbezier(0.35,1.15)(0.7,1.15)(1.08419, 1.12806)(0.88419,1.02806)		
		\psbezier(0.0, 1.0)(0.4, 0.9)(0.31581, 0.87194)(0.11581, 0.77194)		
		\psbezier(0.88419, 1.02806)(0.68419, 0.92806)(0.6, 0.9)(1.0, 0.8)		
		\psclip{\psframe[linestyle=none](0.06046,0)(1,1)}
			\psdash\psbezier(0.0,0.2)(-0.4, 0.3)(0.0,0.35)(0.35,0.35)									
			\psdash\psbezier(0.35,0.35)(0.7,0.35)(1.08419, 0.32806)(0.88419,0.22806)	
			\psdash\psbezier(0.0,0.2)(0.45, 0.1)(0.53419, 0.12806)(0.88419,0.22806)		
		\endpsclip
		\psclip{\psframe[linestyle=none](0.06046,0)(-1,1)}
			\psbezier(0.0,0.2)(-0.4, 0.3)(0.0,0.35)(0.35,0.35)									
			\psbezier(0.0,0.2)(0.45, 0.1)(0.53419, 0.12806)(0.88419,0.22806)		
		\endpsclip
		\psline( 0.06046,-0.07715)( 0.06046, 0.7285)		
		\psline( 1.16697,-0.07583)( 1.16697, 0.72417)		
		\psclip{\pspolygon[linestyle=none]( 0.06046,-0.07715)( 1.16697,-0.07583)(1.16697,1)(0.06046,1)}
			\psdash\psbezier(0.11581,-0.02806)(-0.08519,-0.12806)(0.3,-0.15)(0.65,-0.15)	
			\psdash\psbezier(0.65,-0.15)(1.0,-0.15)(1.4,-0.1)(1.0,0.0)										
			\psdash\psbezier(0.11581,-0.02806)(0.46581, 0.07194)(0.55, 0.1)(1.0, 0.0)			
		\endpsclip
		\psclip{\psframe[linestyle=none](0,-0.05)(1.2,-1)}
			\psbezier(0.11581,-0.02806)(-0.08519,-0.12806)(0.3,-0.15)(0.65,-0.15)	
			\psbezier(0.65,-0.15)(1.0,-0.15)(1.4,-0.1)(1.0,0.0)										
		\endpsclip
		\psdash\psline( 0.93916, 0.27702)( 0.93916, 0.81)		
		\psline( 0.93916, 0.81)( 0.93916, 1.07702)			
		\psline( 0.06046,-0.07715)( 0.06046, 0.7285)		
		\psline( 1.16697,-0.07583)( 1.16697, 0.72417)		
		\psline(-0.16697, 0.27583)(-0.16697, 1.07583)		
		\psclip{\pscustom[linestyle=none]{%
			\moveto(0,1)
			\curveto(0.4, 0.9)(0.31581, 0.87194)(0.11581, 0.77194)
			\curveto(-0.08519,0.67194)(0.3,0.65)(0.65,0.65)
			\curveto(1.0,0.65)(1.4,0.7)(1.0,0.8)
			\curveto(0.6,0.9)(0.68419, 0.92806)(0.88419, 1.02806)
		}}
			\psbezier(0.28306, 0.88342)(0.28306, 0.3)(0.71694, 0.3)(0.71694, 0.91685)	
		\endpsclip
		\psclip{\pscustom[linestyle=none]{%
			\moveto(0,0)
			\lineto(0,0.77194)\lineto(0.11581, 0.77194)
			\curveto(-0.08519,0.67194)(0.3,0.65)(0.65,0.65)
			\curveto(1.0,0.65)(1.4,0.7)(1.0,0.8)
			\lineto(1.2,0.8)\lineto(1.2,0)
		}}
			\psdash\psbezier(0.28306, 0.88342)(0.28306, 0.3)(0.71694, 0.3)(0.71694, 0.91685)	
			\psdot[linecolor=white,dotsize=6pt](0.55,0.35)
			\psbezier[linewidth=0.3pt,arrowsize=3pt,arrowlength=1.2]%
							{<-}(0.55, 0.3)(0.53, 0.7)(0.5, 0.7)(0.45, 0.5)
		\endpsclip
	\end{centerpict}}
	\diagline{<-}{split}{saddle}\naput[npos=0.55]{\textrm{\scriptsize split}}
	\diagline{->}{saddle}{merge}\naput{\textrm{\scriptsize merge}}
\end{equation}
and any of the~diagrams in Tab.~\ref{tab:cube-faces} is a~closure of two saddles. Therefore, a~coefficient associated to a~change of a~chronology cannot be a~single element of the~ring $\scalars$, but it must be a~gadget that returns such an~element after all corners and vertical boundaries are connected in pairs.

\begin{definition}
	A~\emph{closure planar diagram} is a~planar arc diagram with one input (hence, it is an~annulus) and embedded intervals only with endpoints on the~input boundary.
\end{definition}

\noindent
Denote by $\CPO(k)$ the~set of all closure planar diagrams. If $T$ is a~tangle diagram with $2k$ endpoints and $D\in\CPO(k)$ is a closure operator, then $D(T)$ is a~link. A~diagram $D\in\CPO(k)$ induces a~strict $2$-functor\footnote{
	A~cubical functor taking one argument is automatically strict.
}
\begin{equation}
	\EmbChCob(k) \to \EmbChCob(0) \to \kChCob(0),
\end{equation}
which suggests the~commutativity cochain $\psi$ takes values in $\scalars(k) := \{f\colon\CPO(k)\to\scalars\}$, the~ring of all functions from the~set of closure planar operators to $\scalars$. To compute $\psi(S)(D)$ identify the~picture $D(S)$ in Tab.~\ref{tab:cube-faces} on page~\pageref{tab:cube-faces} (an~example is given in Tab.~\ref{tab:psi-for-corners}). It follows immediately that $\psi$ is a~cocycle and that $\KhBracket{T}_\epsilon$, up to an~isomorphism, does not depend on a~sign assignment $\epsilon$.

\begin{table}%
	\begin{tabular}{l|cc}
		\hline
		Closure diagram
		&
		\psset{unit=1.5cm}
		\begin{pspicture}[shift=-0.6](-1,-0.7)(1,0.7)
			\psset{linewidth=0.5pt,linestyle=dotted,dotsep=2pt,dotsize=3pt,dimen=middle}
			\pscircle(0,0){0.6}
			\pscircle(0,0){0.2}
			\psdot(0,-0.195)
			\psset{linestyle=solid}
			\psarc(-0.2828,0){0.2}{  45}{315}
			\psarc( 0.2828,0){0.2}{-135}{135}
		\end{pspicture}
		&
		\psset{unit=1.5cm}
		\begin{pspicture}[shift=-0.6](-1,-0.7)(1,0.7)
			\psset{linewidth=0.5pt,linestyle=dotted,dotsep=2pt,dotsize=3pt,dimen=middle}
			\pscircle(0,0){0.6}
			\pscircle(0,0){0.2}
			\psdot(0,-0.195)
			\psset{linestyle=solid}
			\psarc(0,-0.2828){0.2}{135}{ 45}
			\psarc(0, 0.2828){0.2}{-45}{225}
		\end{pspicture}
		\\
		\hline
		\raisebox{-0.7cm}{\rule{0pt}{1.6cm}}%
		$\psi\left(\begin{pspicture}[shift=-0.6](-0.6,-0.7)(0.6,0.7)
			\psset{linewidth=0.5pt,linestyle=dotted,dotsep=2pt,dotsize=3pt,dimen=middle}
			\pscircle(0,0){0.6}
			\psdot(0,-0.595)
			\psset{linestyle=solid}
			\pscircle(0,0){0.2}
			\psarc(-0.765,0){0.3}{-45}{45}
			\psarc( 0.765,0){0.3}{135}{225}
			\psset{linewidth=2pt,arrowsize=5pt,arrowlength=0.8}
			\psline{->}(-0.48,0)(-0.2,0)
			\psline{->}( 0.48,0)( 0.2,0)
		\end{pspicture}\right)$
		&
		\psset{unit=0.8}%
		\begin{pspicture}(-1.9,0.2)(1.9,1.3)
			\rput[B](0,-0.4){$\permMM$}
			\rput[b](0,0){\pictConnMM{}{->}{}{<-}}
		\end{pspicture}
		&
		\psset{unit=0.8}%
		\begin{pspicture}(-1.9,0.2)(1.9,1.3)
			\rput[B](0,-0.4){$\permSS$}
			\rput[b](0,0){\pictConnX{}{->}{}{->}}
		\end{pspicture}
		\\
		\raisebox{-0.7cm}{\rule{0pt}{1.6cm}}%
		$\psi\left(\begin{pspicture}[shift=-0.6](-0.6,-0.7)(0.6,0.7)
			\psset{linewidth=0.5pt,linestyle=dotted,dotsep=2pt,dotsize=3pt,dimen=middle}
			\pscircle(0,0){0.6}
			\psdot(0,-0.595)
			\psset{linestyle=solid}
			\psarc( 0.765,0){0.3}{135}{225}
			\psbezier(-0.15, 0.212)(-0.26, 0.106)(-0.44, 0.106)(-0.55, 0.212)
			\psbezier(-0.15,-0.212)(-0.26,-0.106)(-0.44,-0.106)(-0.55,-0.212)
			\psbezier(-0.15, 0.212)(-0.10, 0.265)(0.2, 0.3)(0.2,0)
			\psbezier(-0.15,-0.212)(-0.10,-0.265)(0.2,-0.3)(0.2,0)
			\psset{linewidth=2pt,arrowsize=5pt,arrowlength=0.8}
			\psline{->}(-0.35,-0.13)(-0.35,0.13)
			\psline{->}( 0.48,0)( 0.2,0)
		\end{pspicture}\right)$
		&
		\psset{unit=0.8}%
		\begin{pspicture}(-1.9,0.2)(1.9,1.3)
			\rput[B](0,-0.4){$\permSM$}
			\rput[b](0,0){\pictConnSM{}{->}{}{<-}}
		\end{pspicture}
		&
		\psset{unit=0.8}%
		\begin{pspicture}(-1.9,0.2)(1.9,1.3)
			\rput[B](0,-0.4){$\permTneg$}
			\rput[b](0,0){\pictConnT{}{->}{}{<-}}
		\end{pspicture}
		\\
		\hline
	\end{tabular}
	
	\hskip 0.5\baselineskip
	\caption{Values of $\psi$ for some diagrams with four endpoints.}\label{tab:psi-for-corners}
\end{table}

\begin{remark}
	For simplicity, the~linearization of $\EmbChCob(k)$ with coefficients in $\scalars(k)$ will be written as $\kChCob(k)$.
\end{remark}

\begin{remark}
	It can be shown that categories $\kChCob(k)$ form a~planar algebra. However, there is no analogue of Corollary~\ref{cor:plan-alg-cubes} for signed cubes: planar operators are only cubical functors and as such they do not preserve anticommutativity. In~particular, we cannot use planar operators to combine generalized brackets together as it was done in \cite{DrorCobs}.
\end{remark}

\section{Proof of invariance}\label{sec:invariance}
The~Khovanov complex $\KhCom(T)$ is not a~tangle invariant. For example, it depends on the~number of crossings in a~chosen diagram. This dependence disappears after passing to the~homotopy category of complexes and imposing modified versions of Bar-Natan's \textit{S}, \textit{T}, and \textit{4Tu} relations \cite{DrorCobs} explained below.
\begin{enumerate}
	\item[(\textit{S}\/)]\wrapfigure[r]{$\pictRelS = 0$}
				The~\textit{S} relation replaces with $0$ all cobordisms that have a~sphere as a~connected component. The~number and orientations of critical points do not matter.
	\item[(\textit{T}\/)]\wrapfigure[r](0pt,0.5ex){$\pictRelT = \permMS(\permMM+\permSS)$}
				The~\textit{T} relation allows us to remove a~standard torus component at a~cost of multiplying the~cobordism with $\permMS(\permMM+\permSS)$. Here, the~standard torus is defined as a~torus with four critical points and an~arrow at the~merge pointing to the~circle originating on the~left hand side of the~split. The~death is oriented clockwise.
	\item[(\textit{4Tu}\/)]\wrapfigure[r](0pt,0.5ex){\begingroup\psset{unit=0.8cm}
		$\permMS\pictRelTuL + \permMS\pictRelTuR = \permMM\pictRelTuB + \permSS\pictRelTuT$%
	\endgroup}
				The~four tube relation \textit{4Tu} involves four cobordisms from two circles to two circles. Each of them consists of a~tube and two caps, but the~position of the~tube is different in each picture: for the~first two cobordisms the~tube is a~vertical cylinder over one of the~two circles, while in the~remaining two cases it connects either the~input or the~output circles. Notice the~choice of framing for saddle points and heights of caps (the~left caps are smaller than the~right ones). Again, all deaths are oriented clockwise.
\end{enumerate}
The~relations above, especially \textit{T} and \textit{4Tu}, are local. This means all other critical points can appear only above or below the~pictures shown.\footnote{
	This is exactly how the~right disjoint union of chronological cobordisms behaves.
}
All relations are homogeneous---the~degree of the~standard torus is zero, whereas each cobordism involved in \textit{4Tu} has degree $(-1,-1)$---so that they are coherent with changes of chronologies. Let $\kChCobL(k)$ be the~quotient of $\kChCob(k)$ by these relations.

\begin{theorem}\label{thm:invariance}
	Let $T$ be a~tangle with $2k$ endpoints. The~homotopy type of the~generalized Khovanov complex $\KhCom(T)$, regarded as an~object of $\Kom(\kChCobL(k))$, is an~invariant of $T$, i.e.\ complexes for two tangle diagrams related by any of the~Reidemeister moves are chain homotopy equivalent.
\end{theorem}

We shall first prove this theorem locally, for the~tangles defining Reidemeister moves. Using the~planar algebra of chronological cobordisms we shall then extend the~homotopy equivalences to complexes for bigger tangles. Proofs will be done on pictures of cobordisms and for simplicity we omit some details, keeping in mind the~following conventions:
\begin{enumerate}
	\item basepoints should be chosen in the~same place for all pictures involved
				in each proof,
	\item all deaths are oriented clockwise, and
	\item arrows orienting saddles are directed either to the~right or to the~front.
\end{enumerate}
In particular, we can cancel at no cost a~merge or a~split with a~birth or a~death respectively on its right-hand side, while a~left-hand cancellation costs a~multiplication by $\permMM$ or $\permSS$.

\begin{definition}
	We say that a~chain complex $D$ is a~\emph{strong deformation retract} of a~chain complex $C$ if there are chain maps $f\colon D\to C$ and $g\colon C\to D$ such that $gf=\id$ and $fg - \id = dh+hd$ for a~homotopy $h$ such that $hf = 0$.\footnote{
		We will often omit the~composition sign $\circ$.
	}
	The~chain map $f$ is called an~\emph{inclusion into a~deformation retract}.
\end{definition}

\begin{lemma}\label{lem:kh-inv-R1}
	The~bracket $\KhBracket*{\fntRIh*}\{1\}$ is a~strong deformation retract of $\KhBracket*{\fntRIx*}$. Hence, $\KhCom(\fntRIh*)$ and $\KhCom(\fntRIx*)$ are homotopy equivalent for any orientation of the~tangle.
\end{lemma}
\vskip -0.75\baselineskip
\begin{proof}
	\wrapfigure[r]<1>{%
		\psset{unit=0.75cm}%
		\begin{pspicture}(-5.1,-0.3)(5.6,5.7)
			\rput[r](-3,5){$\KhBracket*{\vcenter{\hbox{\Large\tangleRIh*}}} :$}
			\rput[r](-3,1){$\KhBracket*{\vcenter{\hbox{\Large\tangleRIx*}}} :$}
			\rput[c]( 0,5){\rnode{tl}{\tangleRIh*}}
			\rput[c]( 5,5){\rnode{tr}{$0$}}
			\rput[c]( 0,1){\rnode{bl}{\tangleRIv*}}
			\rput[c]( 5,1){\rnode{br}{\tangleRIh*}}
			\diagline{->}{tl}{tr}\naput[labelsep=1pt]{$\scriptstyle 0$}
			\diagline{<->}{tr}{br}\naput[labelsep=1pt]{$\scriptstyle 0$}
			\diagline[offset=-2pt]{->}{bl}{br}
						\nbput{$\scriptstyle d\,=\,\epsilon\fntCobRId$}
			\diagline[offset= 2pt]{<-}{bl}{br}
						\naput{$\scriptstyle h\,=\,-\epsilon^{-1}\fntCobRIh$}
			\diagline[offset=-2pt]{->}{tl}{bl}
						\nbput{$\scriptstyle\permSS\left(\!\!\permMM\fntCobRIfa-\permMS\fntCobRIfb\!\!\right)\,=\,f^0$}
			\diagline[offset= 2pt]{<-}{tl}{bl}
						\naput[npos=0.3]{$\scriptstyle g^0\,=\,\permMM\permSM\fntCobRIg$}
		\end{pspicture}%
	}
	The~second statement follows from the first one, because no matter how the~tangle is oriented, its crossing is positive. Consider now the~diagram to the~right. Rows together with morphisms pointing to the~right represent the~Khovanov brackets $\KhBracket*{\fntRIh*}$ (the top row) and $\KhBracket*{\fntRIx*}$ (the~bottom row), whereas the~morphisms pointing to the~left are chain homotopies in these complexes (zero at the top and a~curtain with a~birth at the~bottom). The~coefficient $\epsilon$ comes from a~sign assignment---although we could take $\epsilon=1$, this more general situation turns out to be useful when extending the~proof to the~global case.
	
	\parshape=0
	Vertical arrows define chain maps $f\colon\KhBracket*{\fntRIh*}\to\KhBracket*{\fntRIx*}$ and $g\colon\KhBracket*{\fntRIx*}\to\KhBracket*{\fntRIh*}$, which is obvious for $g$, but requires the~following short computation for $f$:
	\begin{equation}
		\epsilon^{-1}df^0
					= \permMM\permSS\insertCob{R1-df-0} - \permSS\permMS\insertCob{R1-df-1}
					= \permSS\permMS\insertCob{R1-df-3} - \permSS\permMS\insertCob{R1-df-1}
					= 0.
	\end{equation}
	When the~degree shifts are applied, both $f$ and $g$ have degree $0$. They are chain homotopy equivalences, as the~relation \textit{T} implies $gf = \id$:
	\begin{equation}
		g^0f^0	=	\permSS\permSM\insertCob{R1-GF-0} - \permMM\permSS\insertCob{R1-GF-1}
						= \left(\permSS(\permMM+\permSS)-\permMM\permSS\right)\insertCob{R1-GF-2} = \id,
	\end{equation}
	whereas \textit{4Tu} makes $f^0g^0 - \id = hd$:
	\begin{align*}
		0  &= \phantom{\permSS}\permMS\insertCob{R1-FG-0}\phantom{\permMM}
				+ \permMS\insertCob{R1-FG-1}\phantom{\permMS}
				- \permMM\insertCob{R1-FG-2}
				- \permSS\insertCob{R1-FG-3}\\
			 &= \permSS\permMS\insertCob{R1-FG-4}
				+ \permMM\permMS\insertCob{R1-FG-5}
				- \permMM\permMS\insertCob{R1-FG-6}\phantom{\permSS}
				- \insertCob{R1-FG-7}
			 = -\permMM\permMS(f^0g^0 - \id - hd).
	\end{align*}
	After expanding $f^0g^0$ we can see that the~last cobordism should appear with the~coefficient $-\permMM\permSS$. The~equality holds, because the~cobordism has a~handle, hence, it is annihilated by $(1-\permMM\permSS)$. Together with $dh = -\id = f^1g^1-\id$ (remove the~birth), this shows that the~maps $f$ and $g$ are mutually inverse homotopy equivalences. To finish the~proof, notice that $hf = 0$ trivially.
\end{proof}

\begin{remark}\label{rmk:R1-homotopy}
	Suppose $\fntRIx*$ is a~part of a~bigger tangle diagram $T$ and consider the~corrected cube of resolutions $\KhCubeSigned{T}{\epsilon}$. If we replace edges $d_\zeta$ corresponding to the~crossing in $\fntRIx*$ with homotopies $h_\zeta$ defined as in Lemma~\ref{lem:kh-inv-R1} (this reverses directions of the~edges), the~new cube still anticommutes. Indeed, as $d_\zeta$ is always a~merge and $h_\zeta$ is a~birth, checking anticommutativity reduces to comparing the~following squares.
	\begin{equation}
	\begin{diagps}(-0.2,-0.2)(2.2,2.2)
		\rput(0,1){\Rnode{l}{$\bullet$}}
		\rput(1,2){\Rnode{t}{$\bullet$}}
		\rput(1,0){\Rnode{b}{$\bullet$}}
		\rput(2,1){\Rnode{r}{$\bullet$}}
		\diagline{->}{l}{t}	\naput[nrot=:U,labelsep=3pt]{\scriptsize saddle}
		\diagline{->}{t}{r}	\naput[nrot=:U,labelsep=3pt]{\scriptsize merge}
		\diagline{->}{l}{b}	\nbput[nrot=:U,labelsep=4pt]{\scriptsize merge}
		\diagline{->}{b}{r}	\nbput[nrot=:U,labelsep=4pt]{\scriptsize saddle}
		\psset{linewidth=0.5pt,arrowsize,arrowsize=4pt 1.2,arrowlength=0.7,arrowinset=0.5}
		\psarcn[linewidth=0.5pt](1,1){0.2}{135}{-135}\psline{<-}(0.85858,0.85858)(0.85929,0.85828)
	\end{diagps}
	\hskip 1cm\to/<->/<3em>\hskip 1cm
	\begin{diagps}(-0.2,-0.2)(2.2,2.2)
		\rput(0,1){\Rnode{l}{$\bullet$}}
		\rput(1,2){\Rnode{t}{$\bullet$}}
		\rput(1,0){\Rnode{b}{$\bullet$}}
		\rput(2,1){\Rnode{r}{$\bullet$}}
		\diagline{->}{l}{t}	\naput[nrot=:U,labelsep=3pt]{\scriptsize saddle}
		\diagline{<-}{t}{r}	\naput[nrot=:U,labelsep=3pt]{\scriptsize birth}
		\diagline{<-}{l}{b}	\nbput[nrot=:U,labelsep=4pt]{\scriptsize birth}
		\diagline{->}{b}{r}	\nbput[nrot=:U,labelsep=4pt]{\scriptsize saddle}
		\psset{linewidth=0.5pt,arrowsize,arrowsize=4pt 1.2,arrowlength=0.7,arrowinset=0.5}
		\psarcn[linewidth=0.5pt](1,1){0.2}{135}{-135}\psline{<-}(0.85858,0.85858)(0.85929,0.85828)
	\end{diagps}
	\end{equation}
	Whatever the~saddle is, the~relation between the~top and the~bottom cobordism in the~left square is exactly the~same as the~relation between the~left and the~right cobordism in the~right square. Because we corrected $h_\zeta$ with the~inverse of the~coefficient for $d_\zeta$, coefficients along the~circular arrows are the~same and anticommutativity of one of the~squares implies anticommutativity of the~other.
\end{remark}

\begin{lemma}\label{lem:kh-inv-R2}
	The~bracket $\KhBracket*{\fntRIIhhs*}\{1\}$ is a~strong deformation retract of $\KhBracket*{\fntRIIxx*}[1]$. Hence, $\KhCom(\fntRIIhhs*)$ and $\KhCom(\fntRIIxx*)$ are homotopy equivalent for any orientation of the~tangles involved.
\end{lemma}
\begin{proof}
	As before, the~second claim follows from the~first one, as the~two crossings in \fntRIIxx* have different signs for any orientation of the~tangle. The~first sentence follows from the~diagram in Fig.~\ref{fig:khov-inv-R2}. As before, $\epsilon$'s come from some sign assignment (so that the~lower square anticommutes). The~nontrivial components of $f$ and $g$ are chosen as compositions $f^0 := h_{{*}1}d_{1{*}}$ and $g^0 := d_{{*}0}h_{0{*}}$. Again, both $f$ and $g$ have degree $0$ after the~degree shifts are applied.

	The~morphisms $f$ and $g$ are chain maps: the~equalities $df = 0$ and $dg = 0$ are either trivial or they follow easily from the~chronological relations. The~relation \textit{S} makes both $gf=\id$ and $hf=0$ and it remains to show that $h$ is a~chain homotopy between $fg$ and the~identity morphism. The~only nontrivial case is in the~middle, were we have to check the~matrix equality
	\begin{equation}
		\left(\begin{matrix}	g^0f^0 & f^0\\ g^0 & I\end{matrix}\right)
			-
		\left(\begin{matrix} \id & 0 \\ 0 & \id\end{matrix}\right)
			= 
		\left(\begin{matrix}
				h_{{*}1}d_{{*}1}+d_{0{*}}h_{0{*}} & h_{{*}1}d_{1{*}}\\
				d_{{*}0}h_{0{*}} & 0
		\end{matrix}\right).
	\end{equation}
	It follows from definitions of $f^0$ and $g^0$ and the~\textit{4Tu} relation:
	\begin{align*}
		0\quad &= \phantom{\permMM}\permMS\insertCob{R2-FG-0a}\phantom{\gamma\varphi}
				+ \phantom{\permMM}\permMS\insertCob{R2-FG-0b}\phantom{\permMS}
				- \permMM\insertCob{R2-FG-0c}\phantom{\permMM}
				- \phantom{\permMS}\permSS\insertCob{R2-FG-0d}\\
		   &= \permMM\permMS\insertCob{R2-FG-1a}\phantom{\gamma}
			  + \phantom{\varphi\permMM}\mathllap{\permMM\permSS}\permMS\insertCob{R2-FG-1b}
				- \permMM\permMS\insertCob{R2-FG-1c}
				- \permMM\permSS\permMS\insertCob{R2-FG-1d}\\
			 &= \permMM\permMS\insertCob{R2-FG-1a}
			  - \gamma\varphi\permMM\permMS\insertCob{R2-FG-2}
				- \permMM\permMS\insertCob{R2-FG-1c}
				- \permMM\permSS\permMS\insertCob{R2-FG-1d}\\
			 &= \permMM\permMS(-f^0g^0+\id+h_{{*}1}d_{{*}1}+d_{0{*}}h_{0{*}})\rule{0pt}{0.7cm}.
	\end{align*}
	The~coefficient $\permMM$ at the~first term appears, because the~birth is canceled with a~merge from the~left side. The~same happens in the~last two terms, but in the~third one we also have to reverse an~orientation of the~lower merge. Finally, to modify the~second term, we first used chronological relations and then anticommutativity of the~lower square in Fig.~\ref{fig:khov-inv-R2} (erase the~caps to see compositions of differentials).
\end{proof}

\begin{remark}\label{rmk:R2-homotopy}
	As before, if we take a~cube for a~bigger tangle diagram the~homotopies $h_{0{*}}$ and $h_{{*}1}$ anticommute with edges corresponding to other crossings than the~two involved in the~second Reidemeister move. This can be shown similarly as in Remark~\ref{rmk:R1-homotopy}: the~two homotopies are paired with edges that are always a~merge or a~split, however we close the~diagram.
\end{remark}

\begin{figure}
	\psset{unit=0.75}
	\begin{diagps}(0,0)(14.5,10.2)
		\rput(0,9.6){$\KhBracket*{\vcenter{\hbox{\tangleRIIhhs*}}}\mathrlap{:}$}
		\rput(0,3.0){$\KhBracket*{\vcenter{\hbox{\tangleRIIxx*}}}\mathrlap{\![\,1]:}$}
		\psline{<->}( 0,7.25)(0,5.25)\uput[ur]( 0,7.25){$\scriptstyle g$}\uput[dr](0,5.25){$\scriptstyle f$}
		\psline{<->}(-1,6.25)(1,6.25)\uput[ul](-1,6.25){$\scriptstyle h$}\uput[ur](1,6.25){$\scriptstyle d$}
		\rput( 4.2,9.6){\rnode{t-1}{$0$}}
		\rput( 9.6,9.6){%
			\rnode{t 0}{$\KhBracket*{\vcenter{\hbox{\tangleRIIhhs*}}}$}%
		}%
		\rput(15.0,9.6){\rnode{t 1}{$0$}}
		\rput( 4.2,3.0){%
			\rlap{\raisebox{-3ex}[0pt][0pt]{$\scriptstyle\mskip\thickmuskip 00$}}%
			\rnode{b00}{$\KhBracket*{\vcenter{\hbox{\tangleRIIvh*}}}$}%
		}
		\rput( 8.4,1.2){%
			\rlap{\raisebox{-3ex}[0pt][0pt]{$\scriptstyle\mskip\thickmuskip 01$}}%
			\rnode{b01}{$\KhBracket*{\vcenter{\hbox{\tangleRIIvv*}}}$}%
		}
		\rput(10.8,4.8){%
			\rlap{\raisebox{-3ex}[0pt][0pt]{$\scriptstyle\mskip\thickmuskip 10$}}%
			\rnode{b10}{$\KhBracket*{\vcenter{\hbox{\tangleRIIhh*}}}$}%
		}
		\rput(15.0,3.0){%
			\rlap{\raisebox{-3ex}[0pt][0pt]{$\scriptstyle\mskip\thickmuskip 11$}}%
			\rnode{b11}{$\KhBracket*{\vcenter{\hbox{\tangleRIIhv*}}}$}%
		}
		\diagline{->}{t-1}{t 0}\naput{$\scriptstyle 0$}
		\diagline{->}{t 0}{t 1}\naput{$\scriptstyle 0$}
		\diagline{->}{b00}{b10}\naput[labelsep=-1pt,npos=0.3]{$\scriptstyle \epsilon_{{*}0}\fntCob{R2-d*0}$}
		\diagline{->}{b10}{b11}\naput[labelsep=-7pt,npos=0.4]{$\scriptstyle \epsilon_{1{*}}\fntCob{R2-d1*}$}
		\diagline[offset= 2pt]{->}{b00}{b01}
				\naput[labelsep=-4pt]{$\scriptstyle \epsilon_{0{*}}\fntCob{R2-d0*}$}
		\diagline[offset= 2pt]{->}{b01}{b11}
				\naput[labelsep=1pt,npos=0.4]{$\scriptstyle \epsilon_{{*}1}\fntCob{R2-d*1}$}
		\diagline[offset=-2pt]{<-}{b00}{b01}
				\nbput[labelsep= 0pt]{$\scriptstyle-\epsilon_{0{*}}^{-1}\permSS\fntCob{R2-h0*}$}
		\diagline[offset=-2pt]{<-}{b01}{b11}
				\nbput[labelsep=-4pt]{$\scriptstyle-\epsilon_{{*}1}^{-1}\fntCob{R2-h*1}$}
		\diagline{<->}{b00}{t-1}\naput{$\scriptstyle 0$}
		\diagline{<->}{b11}{t 1}\nbput{$\scriptstyle 0$}
		\diagarc[arcangle= 12]{<->}{t 0}{b10}
				\naput[npos=0.3]{$\scriptstyle\id$}
		\diagarc[arcangle= 12,offset= 2pt,border=3\pslinewidth]{->}{b01}{t 0}
				\naput[labelsep=-2pt,npos=0.75]{$\scriptstyle\gamma\fntCob{R2-G}$}
		\diagarc[arcangle= 12,offset=-2pt,border=3\pslinewidth]{<-}{b01}{t 0}
				\nbput[labelsep=-1pt,npos=0.65]{$\scriptstyle\varphi\!\fntCob{R2-F}$}
		\rput[c](12.8,7.2){$\setlength\arraycolsep{0pt}\begin{array}{rl}
				\scriptstyle\gamma  &\scriptstyle= -\epsilon_{{*}0}^{\vphantom{-1}}\epsilon_{0{*}}^{-1}\permSS\\
				\scriptstyle\varphi &\scriptstyle= -\epsilon_{1{*}}^{\vphantom{-1}}\epsilon_{{*}1}^{-1}
			\end{array}$}
	\end{diagps}
	\caption{Invariance under the~$R_2$ move.}\label{fig:khov-inv-R2}
\end{figure}

The~case of the~third move is the~simplest one, although it deals with the~largest complex. This is because it can be derived from the~invariance under the~second move, as it is done in the~case of the~Kauffman bracket. For this, we need one property of mapping cone complexes.

\begin{definition}
	The~\emph{mapping cone} of a~chain map $\psi\colon C\to D$ is the~chain complex $\cone(\psi)$ defined as
	\begin{equation}
		\cone(\psi)^i := C^{i+1}\oplus D^i,\qquad
		d=\left(\begin{matrix}-d_C & 0\\ \psi & d_D\end{matrix}\right)
	\end{equation}
\end{definition}

\begin{lemma}\label{lem:cone-homot}
	The~homotopy type of a~mapping cone is preserved under compositions with inclusions into strong deformation retracts. More precisely, given a~pair of strong deformation retracts
	\begin{equation}
		C_a\twoways<2em>^{g_a}_{f_a}D_a
			\qquad\textit{and}\qquad
		C_b\twoways<2em>^{g_b}_{f_b}D_b
	\end{equation}
	and a~chain map $\psi\colon C_a\to C_b$, the~mapping cones $\cone(\psi f_a)$ and $\cone(f_b\psi)$ are strong deformation retracts of $\cone(\psi)$.
\end{lemma}
\begin{proof}
	Let $h$ be the~homotopy associated to the~retract $D$. Then there is a~diagram with commuting squares
	\begin{equation}\begin{diagps}(-2.5,0.8)(8,3.4)
		\rput[B](-2.5,3){$\cone(\psi f_a):$}
		\rput[B](-2.5,1){$\cone(\psi):$}
		\node t0(0,3)[\cdots]	\node[href=-0.2] t1(2.5,3)[D_a^r\oplus C_b^{r-1}]	\node[href=0.25] t2(5.5,3)[D_a^{r+1}\oplus C_b^r]	\node t3(8,3)[\cdots]
		\node b0(0,1)[\cdots]	\node[href=-0.2] b1(2.5,1)[C_a^r\oplus C_b^{r-1}]	\node[href=0.25] b2(5.5,1)[C_a^{r+1}\oplus C_b^r]	\node b3(8,1)[\cdots]
		\arrow{->}[t0`t1;]	\arrow{->}[t1`t2;d]	\arrow{->}[t2`t3;]
		\arrow{->}[b0`b1;]	\arrow{->}[b2`b3;]
		\arrow[offset=2pt]|a{labelsep=3pt}|{->}[b1`b2;d]
		\arrow[offset=-2pt]|b{labelsep=3pt}|{<-}[b1`b2;h]
		\arrow[offset= 2pt]|a{labelsep=3pt}|{->}[t1`b1;\tilde f_a^{r}]
		\arrow[offset=-2pt]|b{labelsep=3pt}|{<-}[t1`b1;\tilde g_a^{r}]
		\arrow[offset= 2pt]|a{labelsep=3pt}|{->}[t2`b2;\tilde f_a^{r+1}]
		\arrow[offset=-2pt]|b{labelsep=3pt}|{<-}[t2`b2;\tilde g_a^{r+1}]
	\end{diagps}\end{equation}
	with morphisms $\tilde f_a$, $\tilde g_a$, and $\tilde h$ given by matrices
	\begin{equation}
		\def\arraystretch{0.75}%
		\tilde f_a^r = \left(\begin{matrix} f_a & 0 \\ 0 & \id\end{matrix}\right),
			\hskip 1cm
		\tilde g_a^r = \left(\begin{matrix} g_a & 0\\ -\psi h & \id\end{matrix}\right),
			\hskip 1cm
		\tilde h^r = \left(\begin{matrix}-h & 0\\ 0 & 0\end{matrix}\right).
	\end{equation}
	A~quick computation shows $\tilde h\tilde f_a=0$, $\tilde g_a\tilde f_a = \id$, and $\tilde f_a\tilde g_a-\id = d\tilde{h}+\tilde{h}d$, which proves $\cone(\psi f_a)$ is a~strong deformation retract of $\cone(\psi)$. The~other case is shown in a~similar way.
\end{proof}

\begin{lemma}\label{lem:kh-inv-R3}
	The~complexes $\KhBracket*{\fntRIIIax*}$ and~$\KhBracket*{\fntRIIIbx*}$ are homotopy equivalent and so are $\KhCom(\fntRIIIax*)$ and $\KhCom(\fntRIIIbx*)$ for any orientation of tangles.
\end{lemma}
\begin{proof}
	\wrapfigure[r]{\psset{unit=0.6cm}%
	\begin{diagps}(-1.5,0)(6.7,6)
			\rput(0.0,4.75){\rnode{t00}{$\fntRIIIahhh*$}}
			\rput(2.3,4.00){\rnode{t01}{$\fntRIIIahhv*$}}
			\rput(3.7,5.50){\rnode{t10}{$\fntRIIIahvh*$}}
			\rput(6.0,4.75){\rnode{t11}{$\fntRIIIahvv*$}}
			\rput(0.0,1.25){\rnode{b00}{$\fntRIIIahhh*$}}
			\rput(2.3,0.50){\rnode{b01}{$\fntRIIIahhv*$}}
			\rput(3.7,2.00){\rnode{b10}{$\fntRIIIahvh*$}}
			\rput(6.0,1.25){\rnode{b11}{$\fntRIIIahvv*$}}
			\psset{border=3\pslinewidth}%
			\diagline{->}{t00}{b00}\diagline{->}{t10}{b10}
			\diagline[nodesep=0pt]{->}{t00}{t01}\diagline[nodesep=0pt]{->}{t00}{t10}
			\diagline[nodesep=0pt]{->}{t01}{t11}\diagline[nodesep=0pt]{->}{t10}{t11}
			\diagline[nodesep=0pt]{->}{b00}{b01}\diagline[nodesep=0pt]{->}{b00}{b10}
			\diagline[nodesep=0pt]{->}{b01}{b11}\diagline[nodesep=0pt]{->}{b10}{b11}
			\diagline{->}{t01}{b01}\diagline{->}{t11}{b11}
			\rput[l](-1.5,3){$\Psi\big\downarrow$}
	\end{diagps}}
	Again, the~second part follows from the~first one, because \fntRIIIax* and \fntRIIIbx* have same crossings, regardless of orientation. The~complex $\KhBracket{\fntRIIIax}$ is a~mapping cone of the~chain map $\Psi := \KhBracket{\fntRIIIac*}\colon\KhBracket{\fntRIIIah*}\to\KhBracket{\fntRIIIav*}$ visualized by the~four vertical morphisms to the~right. Consider the~chain map $f$ from the~proof of Lemma~\ref{lem:kh-inv-R2}. It is an~inclusion into a~strong deformation retract and Lemma~\ref{lem:cone-homot} implies $\KhBracket{\fntRIIIax*}$ is homotopy equivalent to the~mapping cone of $\Psi_L := \Psi\circ f$ given in Fig~\ref{fig:psi-cone}. For the~same argument $\KhBracket{\fntRIIIbx*}$ is homotopy equivalent to $\Psi_R$. Since tangle diagrams \fntRIIIav* and \fntRIIIbv* are isotopic, the~mapping cone complexes $\cone(\Psi_L)$ and $\cone(\Psi_R)$ are isomorphic.
\end{proof}

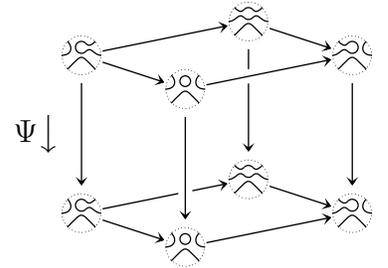
\begin{figure}
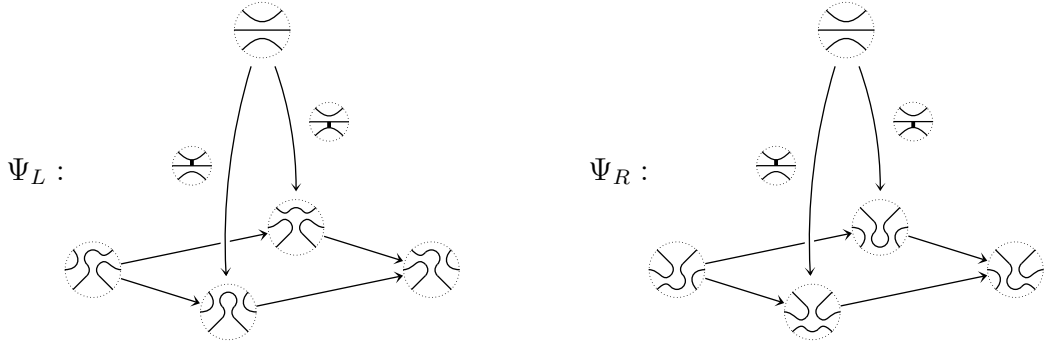

	\psset{unit=0.75cm}
	\begin{diagps}(-1,0)(6.5,6)
			\rput[r](-0.5,3){$\Psi_L:$}
			\rput(3.0,5.50){\rnode{top}{\tangleIIIVert*}}
			\rput(0.0,1.25){\rnode{b00}{\tangleRIIIavhh*}}
			\rput(2.4,0.50){\rnode{b01}{\tangleRIIIavhv*}}
			\rput(3.6,2.00){\rnode{b10}{\tangleRIIIavvh*}}
			\rput(6.0,1.25){\rnode{b11}{\tangleRIIIavvv*}}
			\diagline[nodesep=0pt]{->}{b00}{b01}\diagline[nodesep=0pt]{->}{b00}{b10}
			\diagline[nodesep=0pt]{->}{b01}{b11}\diagline[nodesep=0pt]{->}{b10}{b11}
			\diagarc[arcangle=-10,border=3\pslinewidth]{->}{top}{b01}\nbput{\fntIIIVertT*}
			\diagarc[arcangle= 10]{->}{top}{b10}\naput{\fntIIIVertB*}
	\end{diagps}
	\hskip 2cm
	\begin{diagps}(-1,0)(6.5,6)
			\rput[r](-0.5,3){$\Psi_R:$}
			\rput(3.0,5.50){\rnode{top}{\tangleIIIVert*}}
			\rput{180}(0.0,1.25){\rnode{b00}{\tangleRIIIavhh*}}
			\rput{180}(2.4,0.50){\rnode{b01}{\tangleRIIIavvh*}}
			\rput{180}(3.6,2.00){\rnode{b10}{\tangleRIIIavhv*}}
			\rput{180}(6.0,1.25){\rnode{b11}{\tangleRIIIavvv*}}
			\diagline[nodesep=0pt]{->}{b00}{b01}\diagline[nodesep=0pt]{->}{b00}{b10}
			\diagline[nodesep=0pt]{->}{b01}{b11}\diagline[nodesep=0pt]{->}{b10}{b11}
			\diagarc[arcangle=-10,border=3\pslinewidth]{->}{top}{b01}\nbput{\fntIIIVertT*}
			\diagarc[arcangle= 10]{->}{top}{b10}\naput{\fntIIIVertB*}
	\end{diagps}
	\caption{Morphisms describing complexes for the~two tangles defining the~$R_3$ move.}
	\label{fig:psi-cone}
\end{figure}

\begin{proof}[Proof of Theorem~\ref{thm:invariance}]
It remains to show that the~above local proofs extend to diagrams of bigger tangles. Each case follows the~same pattern. Assume there is a~chain map $\psi\colon\KhCom(T_1)\to\KhCom(T_2)$ defined for whichever sign assignments were chosen to construct the~complexes. Choose a~tangle $T$ and a~planar arc diagram $D$ with two inputs, and construct a~corrected cube $\KhCubeSigned{D(T,T_1)}{\epsilon_1}$ using some sign assignment $\epsilon_1$. We can collapse it partially to obtain a~cube of complexes as in Fig.~\ref{fig:cube-of-complexes}. Namely, a~resolution $T_\xi$ of the~tangle $T$ picks a~subcube $\KhCubeSigned{D(T_\xi,T_1)}{\epsilon_1|_{\xi}}$, which collapses to the~complex $\KhCom(D(T_\xi,T_1))$. Put these complexes in vertices of an~$n$-dimensional cube, where $n$ is the~number of crossings of $T$. Since the~original cube $\KhCubeSigned{D(T,T_1)}{\epsilon_1}$ anticommutes, the~edge morphisms corresponding to changing resolutions of $T$ induce \quot{anti-chain} maps
\begin{equation}
	d_\zeta\colon\KhCom(D(T_\xi,T_1))\to\KhCom(D(T_{\xi'},T_1)),
\end{equation}
i.e. morphisms that anticommute with differentials.

We can do the~same with the~tangle $T_2$, obtaining a~cube of complexes $\KhCom(D(T_\xi,T_2))$. Because planar operators with one input are strict $2$-functors, $\KhCom(D(T_\xi,T_1)) = D(T_\xi,\KhCom(T_1))$ and there are chain maps
\begin{equation}\label{eq:D(psi)-chain-maps}
	D(T_\xi,\psi)\colon \KhCom(D(T_\xi,T_1))\to\KhCom(D(T_\xi,T_2)),
\end{equation}
one for each resolution $T_\xi$. Hence, we have two cubes of complexes and a~morphism between them. Collapsing these cubes (while taking care about homological grading of complexes in vertices\footnote{
	This can be achieved for instance by shifting a~homological degree of a~complex $\KhCom(D(T_\xi,T_i))$ by $-\|\xi\|$ and then taking a~direct sum of complexes over all vertices.
}) results in the~complexes $\KhCom(D(T,T_i))$. If the~chain maps $D(T_\xi,\psi)$ commute with the~edge morphisms $d_\zeta$, they induce a~chain map $\Psi\colon\KhCom(D(T,T_1))\to\KhCom(D(T,T_2))$. In particular, if all $\psi$ are homotopy equivalences, so is $\Psi$.

\begin{figure}
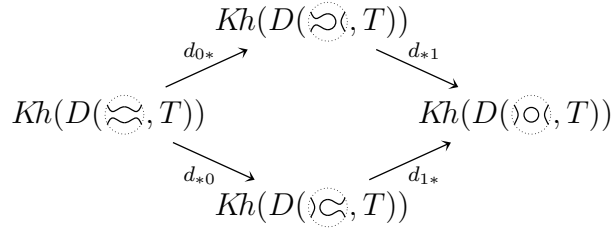

	\begin{diagps}(0,0)(8,3)
		\rput[l](0,1.5){\Rnode{00}{$\KhCom(D(\fntRIIhh*,T))$}}
		\rput[t](4,3.0){\Rnode{01}{$\KhCom(D(\fntRIIhv*,T))$}}
		\rput[b](4,0.0){\Rnode{10}{$\KhCom(D(\fntRIIvh*,T))$}}
		\rput[r](8,1.5){\Rnode{11}{$\KhCom(D(\fntRIIvv*,T))$}}
		\diagline{->}{00}{01}\diagaput{$\scriptstyle d_{0*}$}
		\diagline{->}{00}{10}\diagbput{$\scriptstyle d_{*0}$}
		\diagline{->}{01}{11}\diagaput{$\scriptstyle d_{*1}$}
		\diagline{->}{10}{11}\diagbput{$\scriptstyle d_{1*}$}
	\end{diagps}
	\caption{%
		The~cube of complexes for a~tangle $T$ induced by a~planar diagram $D$ and a~tangle \protect\fntFullTwist* with 2 crossings. Each arrow is a~degree $0$ morphism that anticommutes with differentials.
	}
	\label{fig:cube-of-complexes}
\end{figure}

There is nothing to do for the~second Reidemester move. If a~tangle diagram $T$ can be reduced to $T'$ by this move, consider $\KhCube{T'}$ as a~subcube of $\KhCube{T}$. Remark~\ref{rmk:R2-homotopy} implies that both homotopies $h_{0{*}}$ and $h_{{*}1}$ from Lemma~\ref{lem:kh-inv-R2} anticommute with edge morphisms from $\KhCube{T'}$, so that the~morphisms $f^0 := h_{{*}1}d_{1{*}}$ and $g^0 := d_{{*}0}h_{0{*}}$ commute with them.
	
Invariance under the~third Reidemeister move follows from the~same argument as the~one used to prove Lemma~\ref{lem:kh-inv-R3}: the~chain map $f$ from the~previous paragraph is again an~inclusion into a~strong deformation retract.
	
The~first Reidemeister move is the~most challenging one. As before, choose a~diagram $T$ that can be reduced to $T'$ by this move, and consider $\KhCube{T'}$ as a~subcube of $\KhCube{T}$. The~morphisms $f$ and $g$ does not commute with the~edge morphisms of $\KhCube{T'}$: for an~edge $\zeta\colon\xi\to\xi'$ decorated with a~morphism $d_\zeta$ we have
\begin{equation}
	d_\zeta g_\xi = \lambda(\chdeg d_\zeta, \chdeg(\textnormal{a birth})) g_{\xi'} d_\zeta
\end{equation}
and similarly for $f$. To fix it, we define a~$0$-cochain $\eta\in C^0(I^{n-1}; \invScalars)$ in the~following way. Pick any oriented path in $\KhCube{T'}$ from the~origin $(0,\ldots,0)$ to a~vertex $\xi$. It represents some chronological cobordism $W$, whose degree $\chdeg W$ depends only on $\xi$, but not on the~path. Define $\eta(\xi) := \lambda(\chdeg W,\chdeg(\textnormal{a birth}))$. Then
\begin{equation}
	\eta(\xi')g_{\xi'}d_\zeta =
		\eta(\xi)\lambda(\chdeg d_\zeta, \chdeg(\textnormal{a birth}))g_{\xi'}d_\zeta =
		\eta(\xi)d_\zeta g_\xi.
\end{equation}
Hence, $\eta g$ commutes with edge morphisms. In a~similar way we show that $\eta^{-1}f$ induces a~chain map.
\end{proof}

\section{Basic properties}\label{sec:properties}
Directly from its definition the~generalized Khovanov bracket satisfies the~following properties, similar to the~rules of the~Kauffman bracket:
\begin{enumerate}[label={(KB\arabic{*})},ref={KB\arabic{*}},leftmargin=!,labelwidth=1.5cm,labelsep=5mm]
	\item $\KhBracket\emptyset = \emptyset$,
	\item $\KhBracket{T\sqcup T'} = \KhBracket{T}\sqcup T'$, if $T'$ has no crossings\footnote{
		Think of $(\blank)\sqcup T'$ as a~functor on embedded cobordisms, which we extend naturally to complexes. It maps an~object $\Sigma$ to $\Sigma\sqcup T'$ and a~cobordisms $W$ to $W\sqcup(T'\times I)$.
	}, and
	\item\label{KhBracket:seq}
				$\KhBracket\fntNWSECr = \cone\left(\KhBracket\fntHorResCob\colon
					\KhBracket\fntHorRes\to\KhBracket\fntVertRes\{1\}\right)[1]$.
\end{enumerate}
In the~last property, the~symbols \fntNWSECr, \fntHorRes{} and \fntVertRes{} represent three tangle diagrams that are identical except the~indicated region and the~morphism $\KhBracket\fntHorResCob$ is induced by edge maps in the~cube $\KhCube\fntNWSECr$ at which the~resolution of the~distinguished crossing is changed.

The~property \eqref{KhBracket:seq} implies a~long exact sequence of generalized Khovanov complexes that mimics the~Jones skein relation. Say that a~sequence $\dots\to A_i\to A_{i+1}\to A_{i+2}\to\dots$ in $\catAdd{\kChCob}$ is \emph{exact} if its image under any additive functor $\F\colon\catAdd{\kChCob}\to\cat{A}$ is exact, where $\cat{A}$ is any abelian category.

\begin{proposition}
	There is an~exact sequence of complexes
	\begin{equation}\label{eq:Kh(T)-exact-seq}
		0 \longrightarrow\KhCom(\fntNoCr)[2]\{1\}
			\longrightarrow\KhCom(\fntNegCr)[2]\{2\}
			\longrightarrow\KhCom(\fntPosCr)\{-2\}
			\longrightarrow\KhCom(\fntNoCr)\{-1\}
			\longrightarrow 0.
	\end{equation}
\end{proposition}
\begin{proof}
	The~property \eqref{KhBracket:seq} for diagrams \fntNESWCr\ and \fntNWSECr\ implies the~following sequences are exact:
	\begin{gather}
		0	\longrightarrow\KhBracket\fntVertRes[1]\{1\}
			\longrightarrow\KhBracket\fntNWSECr
			\longrightarrow\KhBracket\fntHorRes
			\longrightarrow 0,\\
		0	\longrightarrow\KhBracket\fntHorRes[1]\{1\}
			\longrightarrow\KhBracket\fntNESWCr
			\longrightarrow\KhBracket\fntVertRes
			\longrightarrow 0.
	\end{gather}
	Gluing them together results in an~exact sequence
	\begin{gather}
		0	\longrightarrow\KhBracket\fntVertRes[2]\{1\}
			\longrightarrow\KhBracket\fntNWSECr[1]
			\longrightarrow\KhBracket\fntNESWCr \{-1\}
			\longrightarrow\KhBracket\fntVertRes\{-1\}
			\longrightarrow 0,
	\end{gather}
	which is the~same as \eqref{eq:Kh(T)-exact-seq} up to grading shifts.
\end{proof}

\noindent
Next, the~generalized Khovanov complex $\KhCom(T)$ depends on the~orientation of $T$ in a~well understood way.

\begin{proposition}
	Choose an~oriented tangle $T$. Denote by $-T$ the~same tangle with reversed orientation of
	all its components and by $T'$ the~tangle when only the~orientation of a~single component
	$T_0$ is reversed. Then
	\begin{align}
		\label{eq:Kh(-T)}\KhCom(-T)^r &\cong \KhCom(T)^r,\\
		\label{eq:Kh(T-T0)}\KhCom(T')^r &\cong \KhCom(T)^{r-2\ell}\{-6\ell\},
	\end{align}
	where $\ell=\mathrm{lk}(T-T_0,T_0)$ is the~linking number of $T_0$ with the~remaining components of $T$.
\end{proposition}
\begin{proof}
	For \eqref{eq:Kh(-T)} it is enough to see that after reversing orientation of all components
	the~signs of crossings are the~same. If we reverse the~orientation only of one component $T_0$,
	then the~crossings of $T_0$ with other components change signs, so that $n_+(T') = n_+(T)-2\ell$
	and $n_-(T') = n_-(T)+2\ell$.
\end{proof}

\wrapfigure[r]<2>{%
				\psset{unit=3mm,labelsep=2pt}
				\begin{pspicture}(-6,-4.5)(6,4.5)
					\rput[c](-5,0){\Rnode{hor}{\begin{pspicture}[shift=-0.5](-1,-1)(1,1)
							\psbezier(-1,-1)(0,-0.1)(0,-0.1)(1,-1)
							\psbezier(-1, 1)(0, 0.1)(0, 0.1)(1, 1)
					\end{pspicture}}}
					\rput[c](5,0){\Rnode{vert}{\begin{pspicture}[shift=-0.5](-1,-1)(1,1)
							\psbezier(-1,-1)(-0.1,0)(-0.1,0)(-1,1)
							\psbezier( 1,-1)( 0.1,0)( 0.1,0)( 1,1)
					\end{pspicture}}}
					\rput[c](0,3){\Rnode{NWSE}{\begin{pspicture}[shift=-1](-1,-1)(1,1)
							\psline(-1,-1)(1,1)
							\psline[border=5\pslinewidth](-1,1)(1,-1)
					\end{pspicture}}}
					\rput[c](0,-3){\Rnode{NESW}{\begin{pspicture}(-1,-1)(1,1)
							\psline(-1,1)(1,-1)
							\psline[border=5\pslinewidth](-1,-1)(1,1)
					\end{pspicture}}}
					\diagline{->}{NWSE}{hor} \nbput{$\scriptstyle0$}
					\diagline{->}{NWSE}{vert}\naput{$\scriptstyle1$}
					\diagline{->}{NESW}{hor} \naput{$\scriptstyle1$}
					\diagline{->}{NESW}{vert}\nbput{$\scriptstyle0$}
					\diagline[linestyle=dotted]{<->}{NWSE}{NESW}\naput{$\scriptstyle!$}
				\end{pspicture}}
Given a~tangle diagram $T$ we form its mirror image $T^!$ by replacing every crossing with the~other one---think about placing a~mirror below the~diagram. It follows the~cube of resolutions of $T^!$ is a~reflection of $\KhCube{T}$: we start in the~terminal state of $T$, which is the~initial state of $T^!$, and proceed backwards (see the~picture to the~right). Formally the~symmetry comes from a~duality functor $(\blank)^*\colon\kChCob(k)\to\kChCob(k)$ induced by the~vertical flip of $\Disk\times I$. We must be careful with defining orientations of critical points in $W^*$: if $p$ is a~critical point of $W$, its orientation determines an~orientation of the~stable part of $T_pW^*$. We choose for the~unstable part the~complementary orientation with respect to the~outward orientation of the~cobordism $W$. The~only exception is a~death, as there is only one orientation of births: if $W$ is a~negatively oriented death, we first rewrite it as a~positively oriented death scaled by $\permSS$, and then we make the~flip.

The~convention for orientation of critical points can be described also diagrammatically in the~following way. Color each region in the~complement of $W$ black or white, so that the~unbounded region is white and regions with same colors do not meet. Then for a~saddle point $p$ rotate the~framing arrow in $W^*$ clockwise, if the~region below $p\in W$ is white, and anticlockwise otherwise:
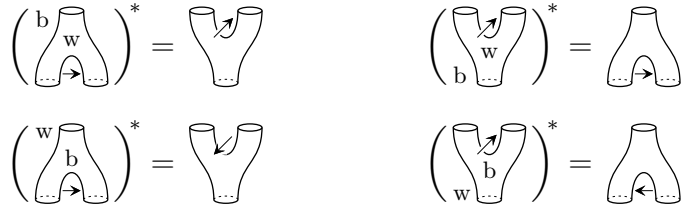
\begin{gather*}
	\psset{unit=8mm}
	\left(\begin{centerpict}(-0.1,-0.1)(1.3,1.3)
		\rput[tl](0,1.2){\textnormal{\scriptsize b}}%
		\rput(0.6,0.7){\textnormal{\scriptsize w}}%
		\cobordism[2](0,0)(M-R)
	\end{centerpict}\right)^{\!\!*} = \textcobordism[1](S-B)
		\hskip 2cm
	\left(\begin{centerpict}(-0.1,-0.1)(1.3,1.3)
		\rput[bl](0,0.0){\textnormal{\scriptsize b}}%
		\rput(0.6,0.5){\textnormal{\scriptsize w}}%
		\cobordism[1](0.4,0)(S-B)
	\end{centerpict}\right)^{\!\!*} = \textcobordism[2](M-R)
		\\[1ex]
	\psset{unit=8mm}
	\left(\begin{centerpict}(-0.1,-0.1)(1.3,1.3)
		\rput[tl](0,1.2){\textnormal{\scriptsize w}}%
		\rput(0.6,0.7){\textnormal{\scriptsize b}}%
		\cobordism[2](0,0)(M-R)
	\end{centerpict}\right)^{\!\!*} = \textcobordism[1](S-F)
		\hskip 2cm
	\left(\begin{centerpict}(-0.1,-0.1)(1.3,1.3)
		\rput[bl](0,0.0){\textnormal{\scriptsize w}}%
		\rput(0.6,0.5){\textnormal{\scriptsize b}}%
		\cobordism[1](0.4,0)(S-B)
	\end{centerpict}\right)^{\!\!*} = \textcobordism[2](M-L)
\end{gather*}
Since we want the~duality functor to be coherent with annihilations and creations, there is no choice left for births and deaths:
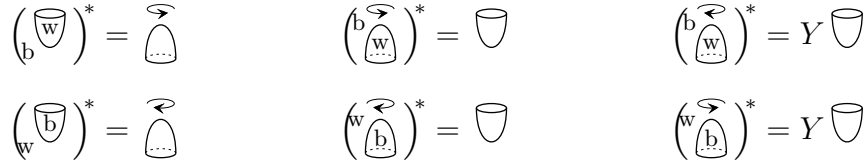
\begin{gather*}
		\left(\begin{centerpict}(-0.1,-0.1)(0.5,0.7)
			\rput[br](0,0){\textnormal{\scriptsize b}}%
			\rput(0.2,0.4){\textnormal{\scriptsize w}}%
			\cobordism[0](0.4,0)(sB)
		\end{centerpict}\right)^{\!\!*} = \textcobordism[1](sD+)
			\hskip 2cm
		\left(\begin{centerpict}(-0.1,-0.1)(0.5,0.7)
			\rput[tr](0,0.6){\textnormal{\scriptsize b}}%
			\rput(0.2,0.2){\textnormal{\scriptsize w}}%
			\cobordism[1](0,0)(sD+)
		\end{centerpict}\right)^{\!\!*} = \textcobordism[0](sB)
			\hskip 2cm
		\left(\begin{centerpict}(-0.1,-0.1)(0.5,0.7)
			\rput[tr](0,0.6){\textnormal{\scriptsize b}}%
			\rput(0.2,0.2){\textnormal{\scriptsize w}}%
			\cobordism[1](0,0)(sD-)
		\end{centerpict}\right)^{\!\!*} = \permSS\textcobordism[0](sB)
			\\[1ex]
		\left(\begin{centerpict}(-0.1,-0.1)(0.5,0.7)
			\rput[br](0,0){\textnormal{\scriptsize w}}%
			\rput(0.2,0.4){\textnormal{\scriptsize b}}%
			\cobordism[0](0.4,0)(sB)
		\end{centerpict}\right)^{\!\!*} = \textcobordism[1](sD-)
			\hskip 2cm
		\left(\begin{centerpict}(-0.1,-0.1)(0.5,0.7)
			\rput[tr](0,0.5){\textnormal{\scriptsize w}}%
			\rput(0.2,0.2){\textnormal{\scriptsize b}}%
			\cobordism[1](0,0)(sD-)
		\end{centerpict}\right)^{\!\!*} = \textcobordism[0](sB)
			\hskip 2cm
		\left(\begin{centerpict}(-0.1,-0.1)(0.5,0.7)
			\rput[tr](0,0.5){\textnormal{\scriptsize w}}%
			\rput(0.2,0.2){\textnormal{\scriptsize b}}%
			\cobordism[1](0,0)(sD+)
		\end{centerpict}\right)^{\!\!*} = \permSS\textcobordism[0](sB)
\end{gather*}

Flipping a~cobordism permutes its degree components, $\chdeg W^* = (b,a)$ if $\chdeg W = (a,b)$, but it also intertwines the two disjoint unions, $(W\rdsum W')^* = W^*\ldsum W'^*$. Hence, in the~linearized case, the~roles of $\permMM$ and $\permSS$ are exchanged, but the~role of $\permMS$ is preserved. Therefore, the~flipping operation is a~functor $(\blank)^*\colon\kChCob_{\permMM\permSS\permMS}(k)\to\kChCob_{\permSS\permMM\permMS}(k)$ between two different categories. It is coherent with all chronological relations, as well as with relations \textit{S}, \textit{T}, and \textit{4Tu}. We extend it to categories of complexes, by reflecting the~homological grading, i.e.\ we set $(C^*)^i := (C^{-i})^*$.

\begin{proposition}\label{prp:kh-dual}
	The~generalized Khovanov complexes of a~tangle $T$ and its mirror image $T^!$ are dual to each other: $\KhCom_{\permMM\permSS\permMS}(T^!) \cong \KhCom_{\permSS\permMM\permSM}(T)^*$, where $\KhCom_{abc}$ stands for a~Khovanov complex constructed in the~category $\kChCobL(k)$ with chronology change coefficients $\permMM$, $\permSS$, and $\permMS$ set to $a$, $b$, and $c$ respectively.
\end{proposition}
\begin{proof}
	Choose a~diagram of $T$ with $n$ enumerated crossings and arrows over them. To obtain a~diagram for $T^!$ replace first each crossing \fntNWSECr\ with the~opposite one \fntNESWCr, and rotate the~arrows over crossings using the~same convention as for $(\blank)^*$: color regions black and white and rotate an~arrow anticlockwise, when it is placed over white regions, and clockwise otherwise. With this choice of diagrams $\KhCube{T^!}=\KhCube{T}^*$, which follows directly from the~construction of the~cube of resolutions. Moreover, a~sign assignment $\epsilon\in C^1(I^n;\invScalars)$ for $\KhCube{T}$ is automatically a~sign assignment for $\KhCube{T}^*$. Therefore, $(\KhBracket{T}_\epsilon)^*[n] = \KhBracket{T^!}_\epsilon$ and the~proposition follows.
\end{proof}

\section{Homology}\label{sec:homology}
Although the~complex $\KhCom(T)$ is an~invariant of the~tangle $T$, it is a~difficult problem to determine whether two complexes in $\kChCobL$ are homotopy equivalent. One can obtain a~partial answer, by applying a~functor $\F\colon\kChCobL\to\cat{A}$ to some abelian category $\cat{A}$. Such a~functor extends naturally to complexes $\F\colon\Kom(\kChCobL)\to\Kom(\cat{A})$ and the~homology $H(\F\KhCom(T))$ is an~invariant of the~tangle $T$.

For simplicity, we will consider only functors $\F\colon\kChCobL(0)\to\Mod\scalars$, producing invariants of links. If we restrict to $\Z{\times}\Z$-graded $\scalars$-modules and $\F$ preserves degrees of morphisms, then homology groups $H^i(\F\KhCom(T))$ are $\Z$-graded (recall, that in $\KhCom$ we collapse the~$\Z{\times}\Z$-grading into the~$\Z$-grading, by replacing $(a,b)$ with $a+b$).

\subsection*{Even Khovanov homology}

Denote by $\Zev$ the~ring of integers with the~trivial action of $\scalars$, i.e.\ all $\permMM$, $\permSS$, and $\permMS$ act as $1$. Then $\Zev\otimes\kChCob$ is the~category of ordinary cobordisms, so that all invariants decribed in \cite{DrorCobs} can be computed from $\KhCom(L)$. In particular, we can take a~functor $\Fev$ that sends a~family of $s$ circles in $\Disk$ into an~$s$-folded tensor product\footnote{
	Strictly speaking, one should think of $A^{\otimes s}$ as an~orderless tensor product, which makes sense in any symmetric monoidal category. Otherwise, $\Fev$ is defined on objects only up to an~isomorphism, since it requires ordering of circles. However, there is a~canonical isomorphism induced by the~symmetric structure, and coherence result for symmetric monoidal categories implies it is unique. The~same issue arises in other examples of functors described in this paper.}
$A^{\otimes s}$ of a~rank $2$ module $A=\Zev v_+\oplus\Zev v_-$, graded with $\deg v_+ = (1,0)$ and $\deg v_- = (0,-1)$. For cobordisms we define $\Fev$ as below
\begin{align}
	\Fev\left(\textcobordism[2](M)\right)&\colon A\otimes A\to A,\phantom\Zev\quad
			\begin{cases}
					v_+\otimes v_+ \mapsto v_+, &\qquad v_-\otimes v_+ \mapsto v_-,\\
					v_+\otimes v_- \mapsto v_-, &\qquad v_-\otimes v_- \mapsto 0,
				\end{cases} \\
	\Fev\left(\textcobordism[1](S)\right)&\colon A\to A\otimes A,\phantom\Zev\quad
			\begin{cases}
				v_+\mapsto v_-\otimes v_+ + v_+\otimes v_-,\\
				v_-\mapsto v_-\otimes v_-,
			\end{cases} \\
	\Fev\left(\textcobordism[0](sB)\right)&\colon\Zev\to A,\phantom{A\otimes A}\quad
			\begin{cases}
				1\mapsto v_+,
			\end{cases}\\
	\label{eq:Fev-counit}
	\Fev\left(\textcobordism[1](sD)\right)&\colon A\to\Zev,\phantom{A\otimes A}\quad
			\begin{cases}
				v_+\mapsto 0,\\
				v_-\mapsto 1.
			\end{cases}
\end{align}	
The~above turns $A$ into a~Frobenius algebra, so that $\Fev$ is well-defined. Compatibility with the~three relations \textit{S}, \textit{T}, and \textit{4Tu} is easy to check \cite{DrorCobs}. The~resulting homology $\EKh(L):=H(\Fev\KhCom(L))$ is the~categorification of the~Jones polynomial from \cite{KhCatJones}.

\subsection*{Odd Khovanov homology}

Assume now that $\permMM$ and $\permMS$ act on integers as $1$, but $\permSS$ as $-1$, and denote this $\scalars$-algebra by $\Zodd$. This choice provides a~framework for the~odd Khovanov homology \cite{ORS}. The~functor $\Fodd\colon\kChCobL(0)\to\Mod\scalars$ associates to a~family of $s$ circles in $\Disk$ the~exterior algebra $\Lambda_s:=\bigwedge^*\Zodd\langle a_1,\dots,a_s\rangle$ with one generator $a_i$ for each circle. A~merge of circles labeled $a_i$ and $a_j$ is realized by the~canonical projection $\Lambda_s\to/->>/ \quotient{\Lambda_s}{(a_i-a_j)}\cong \Lambda_{s-1}$ that identifies appropriate generators. Dually, splitting a~circle into two, labeled $a_i$ and $a_j$, is given as
\begin{equation}
	\Lambda_{s-1}\cong\quotient{\Lambda_s}{(a_i-a_j)} \ni [w]\longmapsto (a_i-a_j)\wedge w \in \Lambda_s,
\end{equation}
assuming the~$i$-th circle in the~target configuration is to the~left of the~framing arrow and the~$j$-th one is to the~right. A~birth is an~inclusion of algebras and an~anticlockwise death of an~$i$-th circle is the~Kronecker delta~function $a_j\mapsto \delta_{i,j}$ wedged with identity, i.e. it strips off $a_i$ from the~element $w$ from the~left hand side, if it is present, or sends $w$ to $0$ otherwise.

One can directly check that $\Fodd$ defined in this way is a~functor. It is shown in \cite{ORS} that $\OKh(L) := H(\Fodd\KhCom(L))$ is an~invariant of a~link $L$. The~group $\Lambda_s$ is graded with an~element $a_{i_1}\wedge\cdots\wedge a_{i_r}$ in degree $s-2r$, which makes $\Fodd$ a~degree-preserving functor. Both a~sphere and a~torus evaluate to zero ($a_i-a_j$ becomes $0$ after merging $i$-th and $j$-th circles). and \textit{4Tu}
follows from the~table below.
\begin{center}
	\psset{unit=8mm}%
	\begin{tabular}{c|cccc}
		\hline \rule[-6mm]{0pt}{14mm} &
		\ \pictRelTuL\  & \ \pictRelTuR\  & \ \pictRelTuB\  & $-$\pictRelTuT\phantom{$-$} \\
		\hline
	  $1$ & $0$ & $0$ & $0$ & $0$ \\
		$a_1$ & $0$ & $1$ & $1$ & $0$ \\
		$a_2$ & $1$ & $0$ & $1$ & $0$ \\
		$a_1\wedge a_2$ & $-a_1$ & $a_2$ & $0$ & $a_2-a_1$ \\
		\hline
	\end{tabular}
\end{center}
Therefore, invariance of $\Fodd\KhCom(L)$ also follows from Theorem~\ref{thm:invariance}.

\section{Chronological Frobenius algebras}\label{sec:chron-Frob}
We shall now construct a~natural target for a~chronological TQFT functor. Choose a~commutative ring $R$ and a~category of symmetric $R$-bimodules graded by an~abelian group $G$.

\begin{definition}
	Choose a~function $\lambda\colon G\times G\to U(R)$ that is a~group homomorphism in each variable, where $U(R)$ is the~group of invertible elements in $R$. Define the~\emph{graded tensor product} for $G$-graded modules in the~ordinary way, but for homogeneous homomorphisms $f$ and $g$ we define the~product $f\otimes g$ by the~formula
	\begin{equation}\label{eq:tensor-cubical}
		(f\otimes g)(m\otimes n) := \lambda(\deg g, \deg m) f(m)\otimes g(n).
	\end{equation}
	There is a~\emph{braiding} $\sigma_{M,N}\colon M\otimes N\to N\otimes M$, $m\otimes n\mapsto\lambda(\deg m,\deg n)n\otimes m$, which is a~\emph{symmetry} if $\lambda(a,b)\lambda(b,a)=1$ for all $a,b\in G$.
\end{definition}

The~graded tensor product generalizes the~Koszul product ($G=\Z_2$ and $\lambda(a,b) = (-1)^{ab}$), and the~anyonic braiding ($G=\Z$ and $\lambda(a,b) = \zeta^{ab}$ for some root of unity $\zeta$).

\begin{lemma}
	The~following hold
	\begin{align}
	\label{eq:graded-product}
			(f'\otimes g')\circ(f\otimes g) &= \lambda(\deg g',\deg f)(f'\circ f)\otimes (g'\circ g), \text{ and}\\
	\label{eq:graded-brading}
			\sigma_{M',N'}\circ(f\otimes g) &= \lambda(\deg f,\deg g)(g\otimes f)\circ\sigma_{M,N}
	\end{align}
	for any homogeneous homomorphisms $M\to^f M'\to^{f'} M''$ and $N\to^g N'\to^{g'} N''$.
\end{lemma}
\begin{proof}
	Straightforward.
\end{proof}

\begin{example}
	The~category $\Mod\scalars$ of $\Z{\times}\Z$-graded modules over a~commutative ring $\scalars=\scalarsLong*$ is a~graded tensor category in the~above sense with $\lambda$ defined as $\lambda(a,b,c,d) = \permMM^{ac}\permSS^{bd}\permMS^{ad-bc}$.
\end{example}

There is a~nice graphical interpretation of the~formulas \eqref{eq:graded-product} and \eqref{eq:tensor-cubical}. We represent a~homomorphism $f\colon M\to N$ by a~box labeled $f$ with two legs: one at the~bottom, labeled with $M$, and one at the~top, labeled with $N$. Composition of morphism is given by placing the~boxes one over the~other and a~tensor product of homomorphisms by placing them side by side, the~left on the~higher level than the~right one. Then we have the~following relation for homogeneous morphisms $f$ and $g$:
\begin{equation}\label{diag:graded-product}
	\psset{unit=0.5}
	\begin{centerpict}(0,1)(4,4.5)
		\psset{linewidth=1pt}
		\psline(1,1)(1,4.5)
		\psline(3,1)(3,4.5)
		\psset{linewidth=0.5pt}
		\rput(1,2.0){\psframe[framearc=0.5,fillstyle=solid](-0.7,-0.5)(0.7,0.5)\rput[B](0,-2pt){$\scriptstyle f$}}
		\rput(3,3.5){\psframe[framearc=0.5,fillstyle=solid](-0.7,-0.5)(0.7,0.5)\rput[B](0,-2pt){$\scriptstyle g$}}
	\end{centerpict}
	 = \lambda(\deg g,\deg f)
	\begin{centerpict}(0,1)(4,4.5)
		\psset{linewidth=1pt}
		\psline(1,1)(1,4.5)
		\psline(3,1)(3,4.5)
		\psset{linewidth=0.5pt}
		\rput(1,3.5){\psframe[framearc=0.5,fillstyle=solid](-0.7,-0.5)(0.7,0.5)\rput[B](0,-2pt){$\scriptstyle f$}}
		\rput(3,2.0){\psframe[framearc=0.5,fillstyle=solid](-0.7,-0.5)(0.7,0.5)\rput[B](0,-2pt){$\scriptstyle g$}}
	\end{centerpict}.
\end{equation}
For example, \eqref{eq:tensor-cubical} is coherent with the~following simple calculation, where we represent an~element $m\in M$ of a~module $M$ by a~box with no input (think of it as a~homomorphisms $R\to M$ taking $1$ to $m$):
\begin{equation}
	\psset{unit=0.5}
	\begin{centerpict}(0,0)(4,5)
		\psset{linewidth=1pt}
		\psline(1,2)(1,5)
		\psline(3,1)(3,5)
		\psset{linewidth=0.5pt}
		\rput(1,4){\psframe[framearc=0.5,fillstyle=solid](-0.7,-0.5)(0.7,0.5)\rput[B](0,-2pt){$\scriptstyle f$}}
		\rput(3,3){\psframe[framearc=0.5,fillstyle=solid](-0.7,-0.5)(0.7,0.5)\rput[B](0,-2pt){$\scriptstyle g$}}
		\rput(1,2){\psframe[framearc=0.5,fillstyle=solid](-0.7,-0.5)(0.7,0.5)\rput[B](0,-2pt){$\scriptstyle m$}}
		\rput(3,1){\psframe[framearc=0.5,fillstyle=solid](-0.7,-0.5)(0.7,0.5)\rput[B](0,-2pt){$\scriptstyle n$}}
	\end{centerpict}
	\quad\to^{\cdot\lambda}\quad
	\begin{centerpict}(0,0)(4,5)
		\psset{linewidth=1pt}
		\psline(1,3)(1,5)
		\psline(3,1)(3,5)
		\psset{linewidth=0.5pt}
		\rput(1,4){\psframe[framearc=0.5,fillstyle=solid](-0.7,-0.5)(0.7,0.5)\rput[B](0,-2pt){$\scriptstyle f$}}
		\rput(1,3){\psframe[framearc=0.5,fillstyle=solid](-0.7,-0.5)(0.7,0.5)\rput[B](0,-2pt){$\scriptstyle m$}}
		\rput(3,2){\psframe[framearc=0.5,fillstyle=solid](-0.7,-0.5)(0.7,0.5)\rput[B](0,-2pt){$\scriptstyle g$}}
		\rput(3,1){\psframe[framearc=0.5,fillstyle=solid](-0.7,-0.5)(0.7,0.5)\rput[B](0,-2pt){$\scriptstyle n$}}
	\end{centerpict}
	\quad=\quad
	\begin{centerpict}(0,0)(4,5)
		\psset{linewidth=1pt}
		\psline(1,3.5)(1,5)
		\psline(3,1.1)(3,5)
		\psset{linewidth=0.5pt}
		\rput(1,3.5){\psframe[framearc=0.5,fillstyle=solid](-1,-0.5)(1,0.5)\rput[B](0,-2pt){$\scriptstyle f(m)$}}
		\rput(3,1.5){\psframe[framearc=0.5,fillstyle=solid](-1,-0.5)(1,0.5)\rput[B](0,-2pt){$\scriptstyle g(n)$}}
	\end{centerpict}
\end{equation}
where the~arrow $\to^{\cdot\lambda}$ indicates that the~picture to the~right must be scaled by $\lambda(\deg g,\deg m)$.

If $\lambda(a,b)\lambda(b,a)=1$ for all $a,b\in G$, we represent the~symmetry $\sigma_{M,N}\colon M\otimes N\to N\otimes M$ by a~crossing:
\begin{equation}
	\psset{unit=0.5}
	\begin{centerpict}(1,0.5)(3,5)
		\psset{linewidth=1pt}
		\psline(1,2.5)(1,3)(3,5)
		\psline(3,1)(3,3)(1,5)
		\psset{linewidth=0.5pt}
		\rput(1,2.5){\psframe[framearc=0.5,fillstyle=solid](-0.7,-0.5)(0.7,0.5)\rput[B](0,-2pt){$\scriptstyle m$}}
		\rput(3,1.0){\psframe[framearc=0.5,fillstyle=solid](-0.7,-0.5)(0.7,0.5)\rput[B](0,-2pt){$\scriptstyle n$}}
	\end{centerpict}
	\qquad=\qquad
	\begin{centerpict}(1,0.5)(3,5)
		\psset{linewidth=1pt}
		\psline(1,1)(1,5)
		\psline(3,2.5)(3,5)
		\psset{linewidth=0.5pt}
		\rput(3,2.5){\psframe[framearc=0.5,fillstyle=solid](-0.7,-0.5)(0.7,0.5)\rput[B](0,-2pt){$\scriptstyle m$}}
		\rput(1,1.0){\psframe[framearc=0.5,fillstyle=solid](-0.7,-0.5)(0.7,0.5)\rput[B](0,-2pt){$\scriptstyle n$}}
	\end{centerpict}
	\qquad=\lambda(\deg m,\deg n)
	\begin{centerpict}(0.3,0.5)(3.7,5)
		\psset{linewidth=1pt}
		\psline(1,2.5)(1,5)
		\psline(3,1)(3,5)
		\psset{linewidth=0.5pt}
		\rput(1,2.5){\psframe[framearc=0.5,fillstyle=solid](-0.7,-0.5)(0.7,0.5)\rput[B](0,-2pt){$\scriptstyle n$}}
		\rput(3,1.0){\psframe[framearc=0.5,fillstyle=solid](-0.7,-0.5)(0.7,0.5)\rput[B](0,-2pt){$\scriptstyle m$}}
	\end{centerpict}
\end{equation}
This does not work in the~braided case. Indeed, one can first change the~heights of boxes labeled $m$ and $n$, which results in $\sigma(m\otimes n) = \lambda(\deg n,\deg m)^{-1}n\otimes m$. Comparing the~two values we conclude it must be $\lambda(\deg n,\deg m)\lambda(\deg m,\deg n)=1$. One solution to this issue is to add horizontal lines originating at all boxes and pointing leftwards, in which case the~relation \eqref{diag:graded-product} appears in two versions
\begin{align}
	\psset{unit=0.5}\begin{centerpict}(-0.5,1)(4,4.5)
		\psset{linewidth=1pt}
		\psline(1,1)(1,4.5)
		\psline(3,1)(3,4.5)
		\psline[linecolor=blue,border=3\pslinewidth](0.3,2.0)(-0.5,2.0)
		\psline[linecolor=blue,border=3\pslinewidth](2.3,3.5)(-0.5,3.5)
		\psset{linewidth=0.5pt}
		\rput(1,2.0){\psframe[framearc=0.5,fillstyle=solid](-0.7,-0.5)(0.7,0.5)\rput[B](0,-2pt){$\scriptstyle f$}}
		\rput(3,3.5){\psframe[framearc=0.5,fillstyle=solid](-0.7,-0.5)(0.7,0.5)\rput[B](0,-2pt){$\scriptstyle g$}}
	\end{centerpict}
	&=
	\psset{unit=0.5}\begin{centerpict}(-1.2,1)(4,4.5)
		\psset{linewidth=1pt}
		\psline(1,1)(1,4.5)
		\psline(3,1)(3,4.5)
		\pscustom[linecolor=blue]{%
			\moveto(0.3,3.5)\curveto(-0.5,3.5)(-0.5,2.0)(-1.2,2.0)}
		\pscustom[linecolor=blue]{%
			\moveto(2.3,2.0)\lineto(0.3,2.0)\curveto(-0.5,2.0)(-0.5,3.5)(-1.2,3.5)
			\stroke[linecolor=white,linewidth=7\pslinewidth]}
		\psset{linewidth=0.5pt}
		\rput(1,3.5){\psframe[framearc=0.5,fillstyle=solid](-0.7,-0.5)(0.7,0.5)\rput[B](0,-2pt){$\scriptstyle f$}}
		\rput(3,2.0){\psframe[framearc=0.5,fillstyle=solid](-0.7,-0.5)(0.7,0.5)\rput[B](0,-2pt){$\scriptstyle g$}}
	\end{centerpict},\quad\text{and}
	\\[1ex]
	\psset{unit=0.5}\begin{centerpict}(-0.5,1)(4,4.5)
		\psset{linewidth=1pt}
		\psline[linecolor=blue](0.3,2.0)(-0.5,2.0)
		\psline[linecolor=blue](2.3,3.5)(-0.5,3.5)
		\psline[border=3\pslinewidth](1,1)(1,4.5)
		\psline[border=3\pslinewidth](3,1)(3,4.5)
		\psset{linewidth=0.5pt}
		\rput(1,2.0){\psframe[framearc=0.5,fillstyle=solid](-0.7,-0.5)(0.7,0.5)\rput[B](0,-2pt){$\scriptstyle f$}}
		\rput(3,3.5){\psframe[framearc=0.5,fillstyle=solid](-0.7,-0.5)(0.7,0.5)\rput[B](0,-2pt){$\scriptstyle g$}}
	\end{centerpict}
	 &=
	\psset{unit=0.5}\begin{centerpict}(-1.2,1)(4,4.5)
		\psset{linewidth=1pt}
		\pscustom[linecolor=blue]{%
			\moveto(2.3,2.0)\lineto(0.3,2.0)\curveto(-0.5,2.0)(-0.5,3.5)(-1.2,3.5)}
		\pscustom[linecolor=blue]{%
			\moveto(0.3,3.5)\curveto(-0.5,3.5)(-0.5,2.0)(-1.2,2.0)
			\stroke[linecolor=white,linewidth=7\pslinewidth]}
		\psline[border=3\pslinewidth](1,1)(1,4.5)
		\psline[border=3\pslinewidth](3,1)(3,4.5)
		\psset{linewidth=0.5pt}
		\rput(1,3.5){\psframe[framearc=0.5,fillstyle=solid](-0.7,-0.5)(0.7,0.5)\rput[B](0,-2pt){$\scriptstyle f$}}
		\rput(3,2.0){\psframe[framearc=0.5,fillstyle=solid](-0.7,-0.5)(0.7,0.5)\rput[B](0,-2pt){$\scriptstyle g$}}
	\end{centerpict}.
\end{align}
Decorating the~horizontal lines with degrees of the~boxes, we add untwisting relations
\begin{equation}
	\psset{linecolor=blue,linewidth=1pt,xunit=1em,yunit=1ex}%
	\begin{centerpict}(-1,-2.1)(1.5,2.1)
		\psbezier(-1,-2)(0,-2)(0, 2)(1, 2)
		\psbezier[border=3\pslinewidth](-1, 2)(0, 2)(0,-2)(1,-2)
		\rput[l](1.1, 2){$\scriptstyle b$}
		\rput[l](1.1,-2){$\scriptstyle a$}
	\end{centerpict}
	= \lambda(a,b)
	\begin{centerpict}(-1,-2.1)(1.5,2.1)
		\psline(-1,-2)(1,-2)\psline(-1,2)(1,2)
		\rput[l](1.1, 2){$\scriptstyle b$}
		\rput[l](1.1,-2){$\scriptstyle a$}
	\end{centerpict}
	\hskip 1cm\text{and}\hskip 1cm
	\begin{centerpict}(-1,-2)(1.5,2)
		\psbezier(-1, 2)(0, 2)(0,-2)(1,-2)
		\psbezier[border=3\pslinewidth](-1,-2)(0,-2)(0, 2)(1, 2)
		\rput[l](1.1, 2){$\scriptstyle b$}
		\rput[l](1.1,-2){$\scriptstyle a$}
	\end{centerpict}
	= \lambda(b,a)^{-1}
	\begin{centerpict}(-1,-2)(1.5,2)
		\psline(-1,-2)(1,-2)\psline(-1,2)(1,2)
		\rput[l](1.1, 2){$\scriptstyle b$}
		\rput[l](1.1,-2){$\scriptstyle a$}
	\end{centerpict}
\end{equation}
This can be done only at the~left edge of the~diagram. The~product $f\otimes g$ is then represented by the~diagram in which the~line originating from $g$ passes over the~input for $f$, and we can represent $\sigma$ by the~positive crossing \fntNESWCr. However, the~composition of boxes becomes more complicated---one cannot simply join two boxes, unless their horizontal lines pass all other lines in the~same way. We shall not go deeper into the~braided case, as all graded tensor products considered in this paper are symmetric.

\begin{definition}\label{def:graded-properties}
	Choose a~ring $S$ that is a~$G$-graded $\scalars$-algebra, and consider the~category of $G$-graded modules over $S$. We say that
	\begin{enumerate}

		\item the~ring $S$ is \emph{commutative} if $rs = \lambda(\deg r,\deg s)sr$ for homogeneous elements $r,s\in S$,

		\item a~$G$-graded bimodule $M$ over $S$ is \emph{symmetric} if $sm = \lambda(\deg s, \deg m)ms$ for homogeneous elements $s\in S$, $m\in M$, and

		\item a~homogeneous function $f\colon M\to N$ between $G$-graded bimodules over $S$ is \emph{right linear} if $f(ms) = f(m)s$, but \emph{left linear} if $f(sm) = \lambda(\deg f,\deg s) sf(m)$ for a~homogeneous element $s\in S$.
	
	\end{enumerate}
\end{definition}

If we think of linearity as a~commutativity of a~map $f$ with the~action of $S$, then the~last definition follows easily from the~graphical calculus (notice that the~actions of $S$ are degree $0$ maps):
\begin{equation}
	\psset{unit=3ex}
	\begin{centerpict}(1,0.5)(3,6)
		\psset{linewidth=1pt}
		\psline(1,2)(1,3)(2,4)(3,3)(3,1)
		\psline(2,4)(2,6)
		\psset{linewidth=0.5pt}
		\rput(1,2){\psframe[framearc=0.5,fillstyle=solid](-0.7,-0.5)(0.7,0.5)\rput[B](0,-2pt){$\scriptstyle s$}}
		\rput(3,1){\psframe[framearc=0.5,fillstyle=solid](-0.7,-0.5)(0.7,0.5)\rput[B](0,-2pt){$\scriptstyle m$}}
		\rput(2,5){\psframe[framearc=0.5,fillstyle=solid](-0.7,-0.5)(0.7,0.5)\rput[B](0,-2pt){$\scriptstyle f$}}
	\end{centerpict}
	\qquad=\qquad
	\begin{centerpict}(1,0.5)(3,6)
		\psset{linewidth=1pt}
		\psline(1,2)(1,4)(2,5)(3,4)(3,1)
		\psline(2,5)(2,6)
		\psset{linewidth=0.5pt}
		\rput(1,2){\psframe[framearc=0.5,fillstyle=solid](-0.7,-0.5)(0.7,0.5)\rput[B](0,-2pt){$\scriptstyle s$}}
		\rput(3,1){\psframe[framearc=0.5,fillstyle=solid](-0.7,-0.5)(0.7,0.5)\rput[B](0,-2pt){$\scriptstyle m$}}
		\rput(3,3){\psframe[framearc=0.5,fillstyle=solid](-0.7,-0.5)(0.7,0.5)\rput[B](0,-2pt){$\scriptstyle f$}}
	\end{centerpict}
	\qquad\to^{\cdot\lambda}\qquad
	\begin{centerpict}(1,0.5)(3,6)
		\psset{linewidth=1pt}
		\psline(1,3.5)(1,4)(2,5)(3,4)(3,1)
		\psline(2,5)(2,6)
		\psset{linewidth=0.5pt}
		\rput(1,3.5){\psframe[framearc=0.5,fillstyle=solid](-0.7,-0.5)(0.7,0.5)\rput[B](0,-2pt){$\scriptstyle s$}}
		\rput(3,1.0){\psframe[framearc=0.5,fillstyle=solid](-0.7,-0.5)(0.7,0.5)\rput[B](0,-2pt){$\scriptstyle m$}}
		\rput(3,2.5){\psframe[framearc=0.5,fillstyle=solid](-0.7,-0.5)(0.7,0.5)\rput[B](0,-2pt){$\scriptstyle f$}}
	\end{centerpict}
	\qquad=\qquad
	\begin{centerpict}(1,0.5)(3,6)
		\psset{linewidth=1pt}
		\psline(1,3)(1,4)(2,5)(3,4)(3,2)
		\psline(2,5)(2,6)
		\psset{linewidth=0.5pt}
		\rput(1,3){\psframe[framearc=0.5,fillstyle=solid](-0.7,-0.5)(0.7,0.5)\rput[B](0,-2pt){$\scriptstyle s$}}
		\rput(3,2){\psframe[framearc=0.5,fillstyle=solid](-1.0,-0.5)(1.0,0.5)\rput[B](0,-2pt){$\scriptstyle f(m)$}}
	\end{centerpict}
\end{equation}
With these conventions we can define a~tensor product of $G$-graded bimodules $M\otimes_S N$ in the~usual way, with actions of $S$ given as $s(m\otimes n) := (sm)\otimes n$ and $(m\otimes n)s := m\otimes(ns)$. If both $M$ and $N$ are symmetric in the~graded sense, so is $M\otimes_S N$.

By an~analogy to ordinary cobordisms, a~chronological TQFT $\F\colon\kChCob(0)\to\Mod\scalars$ is determined by the~pair $(\F(\emptyset),\F(\fntCircle))$, a~variant of a~Frobenius system over $\scalars$.

\begin{definition}\label{def:chron-Frob}
	Choose an~abelian group $G$ and a~commutative ring $R$. A~\emph{chronological Frobenius system} in the~category $\Mod{R}$ with a~symmetric $G$-graded tensor product of type $\lambda$ is a~pair $(S,A)$ of two $R$-modules such that $S$ is a~graded ring and $A$ a~symmetric $S$-bimodule, together with four homogeneous operations, a~unit $\eta\colon S\to A$, a~counit $\epsilon\colon A\to S$, a~multiplication $\mu\colon A\otimes_S A\to A$, and a~comultiplication $\Delta\colon A\to A\otimes_S A$, subject to the~following conditions:
	\begin{gather}
		\label{eq:chron-frob-first}
		\mu\circ(\mu\otimes\id) = \lambda(\deg\mu,\deg\mu) \mu\circ(\id\otimes\mu), \\
		\label{eq:chron-frob-assoc-delta}
		(\Delta\otimes\id)\circ\Delta = \lambda(\deg\Delta,\deg\Delta) (\id\otimes\Delta)\circ\Delta, \\
		\label{eq:chron-frob-units}
		\mu\circ(\eta\otimes\id) = \id, \hskip 2cm
		(\epsilon\otimes\id)\circ\Delta = \id, \\
		\label{eq:chron-frob-comm}
		\mu\circ\sigma = \lambda(\deg\mu,\deg\mu)\mu, \hskip 2cm
		\sigma\circ\Delta = \lambda(\deg\Delta,\deg\Delta)\Delta, \\
		\label{eq:chron-frob-last}
		(\mu\otimes\id)\circ(\id\otimes\Delta)
			= \lambda(\deg\mu,\deg\Delta) \Delta\circ\mu
			= (\id\otimes\mu)\circ(\Delta\otimes\id).
	\end{gather}
	We call $A$ a~\emph{chronological Frobenius algebra} over $S$.
\end{definition}

The~conditions for a~chronological Frobenius algebra reflect the~chronological relations: \eqref{eq:chron-frob-first}, \eqref{eq:chron-frob-assoc-delta} and \eqref{eq:chron-frob-last} are like the~connected sum permutations changes, \eqref{eq:chron-frob-units} mimics the~creation and the~annihilation changes, whereas \eqref{eq:chron-frob-comm} is the~orientation reversion. Therefore, this is not a~surprise that they give chronological TQFT functors.

\begin{proposition}\label{prop:TQFT-from-Frob}
	Choose a~chronological Frobenius system $(S,A)$ in the~category of $G$-graded modules $\Mod{R}$ of type $\lambda$. Then there is a~group homomorphism $\psi\colon\Z\times\Z\to G$, a~$\scalars$-algebra structure on $R$, and a~$\scalars$-linear functor $\FA\colon\kChCob\to\Mod{R}$ that sends a~family of $s$ circles to the~tensor product $A^{\otimes s}$ and
	\psset{unit=7mm}%
	\begin{align}
		\FA\left(\textcobordism[2](M-L)\right) &= \Big(\mu\colon A\otimes A\to A\Big), &
		\FA\left(\textcobordism[0](sB) \right) &= \Big(\eta\colon S\to A\Big), \\
		\FA\left(\textcobordism[1](S-B)\right) &= \Big(\Delta\colon A\to A\otimes A\Big), &
		\FA\left(\textcobordism[1](sD-)\right) &= \Big(\epsilon\colon A\to S\Big).
	\end{align}
	This functor is graded in the~sense that $\deg\F(W) = \psi(\deg W)$ for a~cobordism $W$.
\end{proposition}
\begin{proof}
	The~condition $\deg\F(W)=\psi(\deg W)$ requires $\psi(1,0)=\deg\eta$ and $\psi(0,1)=\deg\epsilon$, while the~ring homomorphism $\scalars\to R$ is determined by $\lambda$ as below:
	\begin{equation*}
		\permMM\mapsto\lambda(\deg\mu,\deg\mu),
			\hskip 1cm
		\permSS\mapsto\lambda(\deg\Delta,\deg\Delta),
			\hskip 1cm
		\permMS\mapsto\lambda(\deg\mu,\deg\Delta).
	\end{equation*}
	It remains to check that $\FA$ preserves the~chronological relations. Most cases follow from \eqref{eq:tensor-cubical} and conditions \eqref{eq:chron-frob-first}--\eqref{eq:chron-frob-last}, with the~exception of ${\times}$- and $\Diamond$-changes. The~former follows from \eqref{eq:chron-frob-comm}, as an~${\times}$-change adds a~twist on one side of the~cobordism. In the~latter both cobordisms are equivalent, so it is enough to show that $1-\permMM\permSS$ annihilates $\mu\circ\Delta$. This follows from \eqref{eq:chron-frob-comm}:
	$\mu\circ\Delta
			= \permMM\permSS(\mu\circ\sigma)\circ (\sigma\circ\Delta)
			= \permMM\permSS\mu\circ\sigma^2\circ\Delta
			= \permMM\permSS\mu\circ\Delta$.
\end{proof}

\begin{example}[Covering Khovanov homology]
Let $S:=\scalars$ and take $A:=\scalars v_+\oplus\scalars v_-$. As before, we grade $A$ by setting $\deg v_+ = (1,0)$ and $\deg v_- = (0,-1)$, and we equip it with the~following operations
\setlength\arraycolsep{1pt}%
\begin{align}
	\mu\colon A\otimes A\to A,&\qquad\left\{\begin{array}{rlrl}
		\vv++ &\mapsto v_+, \quad\ & \vv-+ &\mapsto \permMM\permMS v_-,\\
		\vv+- &\mapsto v_-,        & \vv-- &\mapsto 0,
	\end{array}\right. \\
	\Delta\colon A\to A\otimes A,&\qquad\left\{\begin{array}{l}
		v_+\mapsto \vv-+ + \permSS\permMS\vv+-,\\
		v_-\mapsto \vv--,
	\end{array}\right. \\
	\eta\colon\scalars\to A,&\qquad \left\{\begin{array}{l}
		1\mapsto v_+,
	\end{array}\right.\\
	\epsilon\colon A\to\scalars,&\qquad \left\{\begin{array}{l}
		v_+\mapsto 0,\\
		v_-\mapsto 1.
	\end{array}\right.
\end{align}
One can directly check that conditions \eqref{eq:chron-frob-first}--\eqref{eq:chron-frob-last} hold. The~induced functor $\Fcov$ clearly satisfies the~sphere relation and a~direct calculation shows that a~standard torus evaluates to $\permMS(\permMM+\permSS)$. Finally, the~\textit{4Tu} relation follows from the~table below.
\begin{center}
	\psset{unit=8mm}%
	\begin{tabular}{c|cccc}
		\hline \rule[-6mm]{0pt}{14mm}
		& $\permMS\pictRelTuL\phantom{\permMS}$ & $\permMS\pictRelTuR\phantom{\permMS}$
		& $\permMM\pictRelTuB\phantom{\permMM}$ & $\permSS\pictRelTuT\phantom{\permSS}$ \\
		\hline
	  $\vv++$ & $0$            & $0$            & $0$            & $0$ \\
		$\vv+-$ & $\permMM\vv++$ & $0$            & $\permMM\vv++$ & $0$ \\
		$\vv-+$ & $0$            & $\permMS\vv++$ & $\permMS\vv++$ & $0$ \\
		$\vv--$ & $\permSS\vv-+$ & $\permMS\vv+-$ & $0$            & $\permSS v_-\otimes v_+
																															  + \permMS v_+\otimes v_-$ \\
		\hline
	\end{tabular}
\end{center}
Therefore, this algebra defines a~functor $\F\colon\kChCobL\to\Mod\scalars$. We call the~invariant $\KhCov(L):=H(\Fcov\KhCom(L))$ the~\emph{covering Khovanov homology} of the~link $L$.
\end{example}

Recall we defined two $\scalars$-module structure on $\Z$, depending on the~actions of the~generators $\permMM,\permSS,\permMS\in\scalars$: all three act as $1$ in $\Zev$, but $\permSS$ acts as $-1$ in $\Zodd$. The~following proposition explains the~name \emph{covering homology}.

\begin{proposition}
	For any link $L$ there are isomorphisms
	\begin{equation}
		\EKh(L) \cong \KhCov(L; \Zev)
			\qquad\textrm{and}\qquad
		\OKh(L)\cong \KhCov(L; \Zodd),
	\end{equation}
	where $\KhCov(L; M) := H(\Fcov\KhCom(L)\otimes M)$ for any $\scalars$-module $M$.
\end{proposition}
\begin{proof}
	The~first isomorphism follows directly from the~construction: replacing $\permMM$, $\permSS$ and $\permMS$ with $1$'s in the~definition of the~algebra $A$ results in the~Khovanov algebra. For the~second one it is enough to show that functors $\Fcov(\blank)\otimes\Zodd$ and $\Fodd$ are equivalent. This follows from applying an~isomorphism $i\colon A^{\otimes s}\otimes\Zodd\to\Lambda_s$ that sends any $v_+$ into $1$ and $v_-$ at the $i$-th position to $a_i$. Comparing the~two definitions, one can easily see that $\Fodd(M) = i\circ\left(\Fcov(M)\otimes\Zodd\right)\circ i^{-1}$ for any generating cobordism $M$.
\end{proof}

The~above proposition is an~example of a~more general operation called a~\emph{base change}: given a~chronological Frobenius system $(S,A)$ in $\Mod{R}$ and a~symmetric $R$-module $S'$, which is also a~ring, together with a~degree zero homomorphism of $R$-algebras $S\to S'$, the~pair $(S',A')$ with $A':=A\otimes_S S'$ is another chronological Frobenius system, called a~\emph{base change} of $(S,A)$. Clearly, $H(\F_{A'})\cong H(\F_A; S')$.

\begin{example}\label{ex:tautological-circle}
	One of the~consequences of the~\textit{4Tu} relation is the~following equality
	\begin{equation}\label{eq:bad-neck-cutting}
	\psset{linewidth=0.5pt,dash=1pt 1.5pt}
	\permMS(\permMM+\permSS)
	\begin{centerpict}(0,0)(0.7,2)
		\COBcylinder(0.1,0.1)(0.1,1.9)
	\end{centerpict}
	=
	\begin{centerpict}(0,0)(1.2,2)
		\COBdeath(0.4,0.1)
		\psellipse(0.6,1.9)(0.2,0.075)
		\psbezier(0.4,1.9)(0.4,1.5)(0.2,1.5)(0.2,1.1)
		\psbezier(0.8,1.9)(0.8,1.5)(1.0,1.5)(1.0,1.1)
		\psellipticarc(0.6,1.1)(0.4,0.3){180}{360}
		\psellipticarc(0.6,1.2)(0.25,0.18){190}{350}
		\psellipticarc(0.6,1.0)(0.2,0.2){30}{150}
	\end{centerpict}
	+
	\begin{centerpict}(0,0)(1.2,2)
		\COBbirth(0.8,0.7)
		\psdash\psellipticarc(0.6,0.1)(0.2,0.075){0}{180}
		\psellipticarc(0.6,0.1)(0.2,0.075){-180}{0}
		\psbezier(0.4,0.1)(0.4,0.5)(0.2,0.5)(0.2,0.9)
		\psbezier(0.8,0.1)(0.8,0.5)(1.0,0.5)(1.0,0.9)
		\psellipticarc(0.6,0.9)(0.4,0.3){0}{180}
		\psellipticarc(0.6,0.9)(0.25,0.18){190}{350}
		\psellipticarc(0.6,0.7)(0.2,0.2){30}{150}
	\end{centerpict},
	\end{equation}
	called a~neck-cutting relation. Again, we omitted the~orienting arrows, but the~convention is to orient all death clockwise, merges with arrow pointing leftwards, and splits with arrows pointing to the~back. If we impose the~relation $\permMM+\permSS=0$, we can use \eqref{eq:bad-neck-cutting} to move handles freely between components of a~cobordism (up to multiplication by $\permMM\permMS^a$). A~similar theory over the~two-element field $\mathbb{F}_2$ was analyzed in \cite{DrorCobs}, suggesting we have found its lift to $\Z$ in the~odd setting. Namely, we have an~algebra $A_H:=\Mor(\fntCircle,\fntCircle)$ over the~ring $R_H:=\Z[H,\permMM,\permMS^{\pm1}]/(2H,\permMM^2-1)$, where $H$ has degree $(-1,-1)$ and represents a~handle. Unfortunately, $H$ is a~torsion element, as it is annihilated by $1-\permMM\permSS = 1+\permMM^2 = 2$.	One can check that $A_H$ is a~free module generated by $v_+$ and $v_-$ of degrees $(1,0)$ and $(0,-1)$ respectively, with multiplication and comultiplication given by the~formulas
	\begin{align}
		\mu\colon A_H\otimes A_H\to A_H,&\qquad\left\{\begin{array}{rlrl}
			\vv++ &\mapsto v_+, \quad\ & \vv-+ &\mapsto \permMM\permMS v_-,\\
			\vv+- &\mapsto v_-,        & \vv-- &\mapsto Hv_-,
		\end{array}\right. \\
		\Delta\colon A_H\to A_H\otimes A_H,&\qquad\left\{\begin{array}{l}
			v_+\mapsto \vv-+ + \permMM\permMS\vv+- - H\permMM\permSM\vv++,\\
			v_-\mapsto \vv--.
		\end{array}\right.
	\end{align}
	The~generator $v_+$ is represented by a~death followed by a~birth and $v_-$ by a~vertical cylinder. In tensor products, each $v_+$ is represented by a~birth and all other circles are boundaries of a~single component built from splits only (or a~single death, if there is no $v_-$). See \cite{DrorCobs} for details.
\end{example}

We shall end this section with a~proof of the~nondegeneracy result for chronological cobordisms. For that we define a~universal rank 2 Frobenus system, with scalars in a~$\Z{\times}\Z$-graded commutative ring
\begin{equation}
	R_U := \quotient{\scalars[a,c,e,f,t,h]}{\bigg(\begin{array}{l}
					(\permMM\permSS-1)h, (\permMM\permSS-1)t, af+ce, \\
					ae+ceh+\permSS\permMS cft-1
		\end{array}\bigg)}
\end{equation}
where $\deg a=\deg e=(0,0)$, $\deg c=\deg f=(1,1)$, $\deg h=(-1,-1)$ and $\deg t=(-2,-2)$. The~element $\permMM\permSS-1$ annihilates not only polynomials in $h$ and $t$, but also $c^2$ and $f^2$ due to the~graded commutativity, see Definition~\ref{def:graded-properties}. Consider a~rank two chronological Frobenius agebra $A_U$ over $R_U$ with the~following operations:
\begin{align}&\left\{\begin{array}{rlrl}
		\mu(\vv++) &= v_+, & \quad\mu(\vv-+) &= \permMM\permMS v_-,\\[0.5ex]
		\mu(\vv+-) &= v_-, &      \mu(\vv--) &= hv_- + tv_+,
	\end{array}\right. \\[1ex]
	&\left\{\begin{array}{l}
		\Delta(v_+) = (ft-\permSS\permSM eh)\vv++ + e(\vv-+ + \permSS\permMS\vv+-) + \permMS^2 f\vv--,\\[0.5ex]
		\Delta(v_-) = \permMS^{-2}et\vv++ + ft(\permSS\permSM\vv-+ + \vv+-) + (e+fh)\vv--,
	\end{array}\right. \\[1ex]
	&\left\{\begin{array}{l}
		\eta(1) = v_+,
	\end{array}\right.\\[1ex]
	&\left\{\begin{array}{l}
		\epsilon(v_+) = c,\\[0.5ex]
		\epsilon(v_-) = a.
	\end{array}\right.
\end{align}
It is a~graded version of the~system $(R_4,A_4)$ in \cite{KhFrobExt} and it has the~same universality property. The~following proposition is proven in the~same way as Proposition~4 in \cite{KhFrobExt}.

\begin{proposition}
	Let $(R',A')$ be a~homogeneous chronological Frobenius system in $\Mod\scalars$ of rank two. Then there is a~unique graded ring homomorphism $R_U\to R'$ such that $A'\cong A\otimes_{R_U} R'$.
\end{proposition}

We are now ready to prove the~nondegeneracy result for $\kChCob(0)$.

\begin{proof}[Proof of Theorem~\ref{thm:aut(W)}]
	Given a~chronological cobordism $W$ we want to compute the~group $\Aut(W):=\{k\in\scalars\ |\ kW=W\}$; its elements are products of values of $\iota$, hence, they are invertible.
	
	We shall first show that $\Aut(W)$ is a~subgroup of $\{1,\permMM\permSS\}$. For that take a~graded ring $R_1=R_U/(\permMM-\permSS,a,e,h)=\Z[\permMM,\permMS^{\pm1},c,f,t]/(\permMM^2=\permMM\permMS cft=1)$, and consider a~chronological Frobenius system $(R_1,A_1)$ with $A_1=A_U\otimes R_1$. It has the~following operations:
	\begin{align}&\left\{\begin{array}{rlrl}
			\mu(\vv++) &= v_+, & \quad\mu(\vv-+) &= \permMM\permMS v_-,\\[0.5ex]
			\mu(\vv+-) &= v_-, &      \mu(\vv--) &= tv_+,
		\end{array}\right.
		&&\left\{\begin{array}{l}
			\eta(1) = v_+,
		\end{array}\right.
		\\[1ex]
		&\left\{\begin{array}{l}
			\Delta(v_+) = ft\vv++ + \permMS^2 f\vv--,\\[0.5ex]
			\Delta(v_-) = ft\vv+- + \permMM\permSM ft\vv-+,
		\end{array}\right.
		&&\left\{\begin{array}{l}
			\epsilon(v_+) = c,\\[0.5ex]
			\epsilon(v_-) = 0.
		\end{array}\right.
	\end{align}
	In particular, $\mu(\Delta(v_+)) = (1+\permMS^2)ftv_+$. Since $c$, $f$, and $t$ are invertible and polynomials in $Z$ are not zero divisors, it follows $\F_1(W)$ is not a~zero divisor for any closed surface $W$. This implies $\Aut(W)$ is a~subgroup of $\{1, \permMM\permSS\}$. If $\partial W\neq\emptyset$, create a~closed surface $\widehat W$ by capping its boundary components with births and deaths. Then $\Aut(W)\subset\Aut(\widehat W)$, as every 2-morphism $\varphi\colon W\dblto W$ in $\EmbChCob(0)$ extends to $\widehat W$ in a~way that preserves the~value of $\iota$ (juxtapose $\varphi$ with the~identity 2-morphisms on the~caps).
	
	Now assume $W$ is a~surface of genus $0$ with at most one closed component. Choose the~graded ring $R_2:=R_U/(c^2,a-1,e-1,h)\cong\scalars[c,t]/(c^2, (\permMM\permSS-1)t)$ and consider a~chronological Frobenius system $(R_2,A_2)$ with $A_2=A_U\otimes R_2$. In particular, the~unit and counit are given by formulas
	\begin{equation}
			\eta(1) = v_+,
		\hskip 1cm
			\epsilon(v_+) = c,
		\hskip 1cm
			\epsilon(v_-) = 1,
	\end{equation}
	and a~sphere is evaluated to $c$. Create $\widehat W$ by capping some inputs and outputs of $W$ so that, up to a~change of a~chronology, $\widehat W$ is a~disjoint union of caps and at most one spherical component. The~homomorphism $\F_2(\widehat W)\colon A^{\otimes k}\to A^{\otimes\ell}$ takes $(v_-)^{\otimes k}$ to $(v_+)^{\otimes\ell}$ or $c(v_+)^{\otimes\ell}$, perhaps multiplied by a~monomial in $\permMM$, $\permSS$ and $\permMS$. Since none of $r\in\scalars$ annihilates $c$, $(1-r)W=0$ implies $r=1$, which shows the~group $\Aut(\widehat W)$ is trivial.
\end{proof}

\section{Dotted cobordisms}\label{sec:dots}
A~very generic example of a~chronological Frobenius algebra is given by the~tautological functor $\Mor(\Sigma,\blank)$, where $\Sigma$ is any object of $\kChCob(0)$.

\begin{proposition}
	Given an~object $\Sigma\in\kChCob(0)$, the~group of morphisms $\Mor(\Sigma,\emptyset)$ is a~ring with multiplication induced by the~`right-then-left' disjoint sum and $\Mor(\Sigma, \fntCircle)$ is a~chronological Frobenius algebra over $\Mor(\Sigma,\emptyset)$.
\end{proposition}

The~case $\Sigma=\fntCircle$ was analyzed in Example~\ref{ex:tautological-circle} under the~assumption $\permMM+\permSS=0$, in which case $\Mor(\fntCircle,\fntCircle)$ was a~free rank $2$ module over $\Mor(\fntCircle,\emptyset)\cong\Z[H,\permMM,\permMS^{\pm1}]/(2H,X^2-1)$. However, the~rank of $\Mor(\Sigma,\fntCircle)$ over $\Mor(\Sigma,\emptyset)$ is in general infinite, but the~neck-cutting relation \eqref{eq:bad-neck-cutting} suggests a~way how to reduce it to the~finite case.

\begin{definition}
	The~category $\kChCobD(k)$ consists of chronological cobordisms (with $2k$ vertical boundary lines) and dots on regular levels. A~single dot has a~degree $(-1,-1)$ and two dots cannot lie on the~same level. In addition to chronological relations, we allow dots to move past other dots and critical points at the~cost specified by $\lambda$, and we impose the~following three local relations:
	\begin{align*}
		(\textit{S\/})\quad&\pictRelS = 0, \hskip 1.5cm
		(\textit{D\/})\quad\pictRelD = 1,\\[2ex]
		(\textit{N\/})\quad&\pictRelNeckI = \pictRelNeckB + \pictRelNeckT - \pictRelNeckM.
	\end{align*}
	where all deaths are oriented clockwise.
\end{definition}

Dots are a~part of the~chronological structure and one can think of them as \quot{infinitesimal} handles, which are \quot{frozen}, so that a~dot is not annihilated by $1-\permMM\permSS$. But a~cobordism with two dots on one component is, because permuting two dots costs $\permMM\permSS$. All relations are homogeneous, thence coherent with changes of chronologies. Even more: the~neck cutting relation \textit{N} together with the~cubical structure of the~disjoint sum determines all coefficients for changes of chronologies, except the $\Diamond$-change. For example,
\begingroup\psset{unit=0.8cm}
\begin{align*}
		\textcobordism*[3](cI,2)(M-L)(M-L)\phantom{\permMM}
			&= \phantom{\permMM}\textcobordism*[3](D*B,2)(M-L)(M-L)\phantom{\permMM}
			 + \phantom{\permMM}\permMS^2\textcobordism*[3](DB,2)(M-L)(M*-L)\phantom{\permMM}
			 - \phantom{\permMM}\textcobordism*[3](D**B,2)(M-L)(M-L)\phantom{\permMM}
		\\[1ex]&
			 = \permMM\textcobordism*{\COBshortCylinderVert(0,0)
																\COBshortCylinderVert(1.6,0)
																\COBdeathDotNoBirth(0.8,0)}
												 {\COBcylinderRight(0,0.6)\COBcylinderLeft(1.6,0.6)}
												 {\COBmergeFrLeft(0.4,1.8)}\phantom{\permMM}
			 + \permMM\permMS^2\textcobordism*{\COBshortCylinderVert(0,0)
																				 \COBshortCylinderVert(1.6,0)
																				 \COBdeathNoBirth(0.8,0)}
												 {\COBcylinderRight(0,0.6)\COBcylinderLeft(1.6,0.6)}
												 {\COBmergeDotFrLeft(0.4,1.8)}\phantom{\permMM}
			 - \permMM\textcobordism*{\COBshortCylinderVert(0,0)
																\COBshortCylinderVert(1.6,0)
																\COBdeathDotsNoBirth(0.8,0)}
												 {\COBcylinderRight(0,0.6)\COBcylinderLeft(1.6,0.6)}
												 {\COBmergeFrLeft(0.4,1.8)}\phantom{\permMM}
		\\[1ex]&
			 = \permMM\textcobordism*[3](D*B,2)(M-L,2)(M-L)\phantom{\permMM}
			 + \permMM\permMS^2\textcobordism*[3](DB,2)(M-L,2)(M*-L)\phantom{\permMM}
			 - \permMM\textcobordism*[3](D**B,2)(M-L,2)(M-L)\phantom{\permMM}
		=\ \permMM\textcobordism*[3](cI,2)(M-L,2)(M-L)
\end{align*}\endgroup
where we moved dots in the~middle pictures from the~birth to the~top by the~cost of $\permMS^2$. Dotted cobordisms satisfy also the~other relations from $\kChCobL(k)$. Hence, we can think of $\kChCobD(k)$ as an~abelian extension of $\kChCobL(k)$.

\begin{lemma}
	Relations \textit{T} and \textit{4Tu} follow from \textit{S}, \textit{D} and \textit{N}. Therefore, there are natural functors $k\ChCobL(k)\to\kChCobD(k)$.
\end{lemma}
\begin{proof}
	For the~\textit{T} relation take a~standard torus and cut its handle. In the~resulting expression, one term has a~sphere as its component and the~other two can be reduced to dotted spheres by changing chronologies:
	\begin{equation}
		\pictRelTfromNi = \pictRelTfromNb + \pictRelTfromNt - \pictRelTfromNm
		= (\permMM\permMS+\permSS\permMS)\pictRelTfromNd
	\end{equation}
	The~\textit{4Tu} relation is proved in a~similar way, by cutting the~unique tube in each term. Again, by changing chronologies we can reduce each term to four caps, with left caps smaller than the~right ones, possibly with a~two-dotted sphere in the~middle:
	\begin{align}\label{eq:4TuL-resolved}
			\permMS\pictRelTuL &= \permMM\pictRelTuNbl
													+ \permSS\pictRelTuNtl
													- \permMM\permSS\permMS\pictRelTuNmid,\\[2ex]
	\label{eq:4TuR-resolved}
			\permMS\pictRelTuR &= \permMS\pictRelTuNbr
													+ \permMS\pictRelTuNtr\phantom{\permMM}
													- \phantom{\permSS}\permMS\pictRelTuNmid,\\[2ex]
	\label{eq:4TuB-resolved}
			\permMM\pictRelTuB &= \permMM\pictRelTuNbl
													+ \permMS\pictRelTuNbr
													- \permMM\permSS\permMS\pictRelTuNmid,\\[2ex]
	\label{eq:4TuT-resolved}
			\permSS\pictRelTuT &= \permSS\pictRelTuNtl
													+ \permMS\pictRelTuNtr
													- \permMM\permSS\permMS\pictRelTuNmid.  
	\end{align}
	Because a~two-dotted sphere is annihilated by $(\permTneg-\permTpos)$, the~sum of right hand sides of \eqref{eq:4TuL-resolved} and \eqref{eq:4TuR-resolved} is equal to the~sum of right hand sides of \eqref{eq:4TuB-resolved} and \eqref{eq:4TuT-resolved}.
\end{proof}

The~additive closure $\catAdd{\kChCobD(0)}$ is equivalent to a~category of finitely generated free graded symmetric bimodules over a~certain ring. This follows from the~proposition below.

\begin{proposition}[Delooping]
	The~following two morphisms
	\begin{equation}\label{eq:delooping}
		\begin{diagps}(-3,-1.2)(3,1.4)
			\node l(-3,0)[\vcenter{\hbox{\psset{unit=3ex}\bigCircle}}]
			\node r( 3,0)[\vcenter{\hbox{\psset{unit=3ex}\bigCircle}}]
			\node[vref=0.3] t(0, 1)[\emptyset\{-1\}]
			\node[vref=0.0] b(0,-1)[\emptyset\{+1\}]
			\rput[B](0,0){$\oplus$}
			\psset{dash=1pt 1.5pt,linewidth=0.5pt}
			\arrow|a|{->}[l`t;\pictDeloopTL]
			\arrow|b|{->}[l`b;-\,\pictDeloopBLb\,+\,\pictDeloopBLa]
			\arrow|a|{->}[t`r;\pictDeloopTR]
			\arrow|b|{->}[b`r;\pictDeloopBR]
		\end{diagps}
	\end{equation}
	form a~pair of inverse isomorphisms in the~additive closure $\catAdd{\kChCobD}$.
\end{proposition}
\begin{proof}
	Call the~left map $f$ and the~right one $g$. The~equality $g\circ f=\id$ is exactly the~neck-cutting relation \textit{N}, whereas the~other composition is the~identity $2\times 2$ matrix---this follows directly from relations \textit{D} and \textit{S}.
\end{proof}

\begin{corollary}
	The~tautological functor $\Mor(\emptyset,\blank)\colon\kChCobD(0)\to\Mod{R'}$ is full and faithful, where $R' := \Mor(\emptyset,\emptyset)$. Hence, we can identify $\kChCobD(0)$ with the~category of finitely generated free graded symmetric $\Mor(\emptyset,\emptyset)$-bimodules.
\end{corollary}

We shall now compute a~presentation of the~ring $\Mor(\emptyset,\emptyset)$.

\begin{proposition}
	There is an~isomorphism of graded commutative rings
	\begin{equation}
		\Mor(\emptyset,\emptyset)\cong R_\bullet :=
			\quotient{\scalars[h,t]}{\big((\permTneg{-}\permTpos)t, (\permTneg{-}\permTpos)h\big)},
	\end{equation}
	where $\deg h=(-1,-1)$ and $\deg t=(-2,-2)$, such that
	\begin{equation}
		\begin{pspicture}[shift=-0.2](-0.3,-0.3)(0.3,0.3)
			\pscircle(0,0){0.3}\psdot(0,0.1)\psdot(0,-0.1)
		\end{pspicture}\mapsto h
			\hskip 1cm\textnormal{and}\hskip 1cm
		\begin{pspicture}[shift=-0.3](-0.4,-0.4)(0.4,0.4)
			\pscircle(0,0){0.4}
			\psdot(0,0)\psdot(0,0.2)\psdot(0,-0.2)
		\end{pspicture}\mapsto \permMM\permMS t+h^2.
	\end{equation}
\end{proposition}
\begin{proof}
	It is enough to show that the~above defines a~homomorphism---it is clearly invertible if it exists. We begin with constructing a~graded monoidal functor $\FD\colon\kChCobD\to\Mod{R_\bullet}$. For that take a~free rank two symmetric bimodule $A_\bullet=R_\bullet v_+\oplus R_\bullet v_-$ with $\deg v_+=(1,0)$ and $\deg v_-=(0,-1)$ as usual. This module is a~chronological Frobenius algebra with operations
	\setlength\arraycolsep{1pt}%
	\begin{align}
		\mu\colon A_\bullet\otimes A_\bullet\to A_\bullet,&\qquad\left\{\begin{array}{rlrl}
			\vv++ &\mapsto v_+,		&\qquad \vv-+ &\mapsto \permMM\permMS v_-,\\
			\vv+- &\mapsto v_-,   &       \vv-- &\mapsto tv_+ + hv_-,
		\end{array}\right. \\
		\Delta\colon A_\bullet\to A_\bullet\otimes A_\bullet,&\qquad\left\{\begin{array}{l}
			v_+\mapsto \vv-+ + \permSS\permMS\vv+- - \permSS\permSM h\vv++,\\
			v_-\mapsto \vv-- + \permMS^{-2}t\vv++,
		\end{array}\right. \\
		\eta\colon R_\bullet\to A_\bullet,&\qquad \left\{\begin{array}{l}
			1\mapsto v_+,
		\end{array}\right.\\
		\epsilon\colon A_\bullet\to R_\bullet,&\qquad \left\{\begin{array}{l}
			v_+\mapsto 0,\\
			v_-\mapsto 1
		\end{array}\right.
	\end{align}
	These tell us how to define $\F_\bullet$ on all generators except one, a~cylinder decorated with a~dot. Associate to it the~following homomorphism:
	\begin{equation}
		\theta\colon A_\bullet\to A_\bullet,\qquad\left\{\begin{array}{l}
			v_+ \mapsto v_-,\\
			v_- \mapsto \permMM\permSM(tv_+ + hv_-) = v_+t\permMM\permMS + v_-h.
		\end{array}\right.
	\end{equation}
	Clearly, $\epsilon\circ\eta = 0$ and $\epsilon\circ\theta\circ\eta = 1$, so that $\F_\bullet$ preserves relations \textit{S} and \textit{T}. It remains to show that $\F_\bullet$ is also coherent with the~neck-cutting relation \textit{N}. This follows from computing the~terms on the~right hand side of \textit{N}:
	\begin{align}
		\textcobordism[1](D*B)\colon A_\bullet\to A_\bullet, & \qquad\left\{\begin{array}{l}
					v_+ \mapsto v_+,\\
					v_- \mapsto v_+\cdot h,
				\end{array}\right.\\
		\textcobordism[1](DB*)\colon A_\bullet\to A_\bullet, & \qquad\left\{\begin{array}{l}
					v_+ \mapsto 0,\\
					v_- \mapsto v_-,
				\end{array}\right.\\
		\textcobordism[1](D**B)\colon A_\bullet\to A_\bullet, & \qquad\left\{\begin{array}{l}
					v_+ \mapsto 0,\\
					v_- \mapsto v_+\cdot h.
				\end{array}\right.
	\end{align}
	Summing the~first two and subtracting the~last homomorphism results in the~identity on $A_\bullet$. The~functor $\F_\bullet$ induces a~homomorphism $\varphi\colon\Mor(\emptyset, \emptyset)\to R_\bullet$ by associating an~element from the~ring to any closed surface with dots. In particular, we compute
	\begin{equation}
		\varphi\left(\begin{pspicture}[shift=-0.2](-0.3,-0.3)(0.3,0.3)
			\pscircle(0,0){0.3}\psdot(0,0.1)\psdot(0,-0.1)
		\end{pspicture}\right) = h
			\hskip 1cm\textnormal{and}\hskip 1cm
		\varphi\left(\begin{pspicture}[shift=-0.3](-0.4,-0.4)(0.4,0.4)
			\pscircle(0,0){0.4}
			\psdot(0,0)\psdot(0,0.2)\psdot(0,-0.2)
		\end{pspicture}\right) = \permMM\permMS t+h^2,
	\end{equation}
	which is the~desired homomorphism.
\end{proof}

\begin{remark}
	Similarly to the~even case, dotted cobordisms lead us to a~deformation of odd theory, although both $t$ and $h$ are torsion elements: $2t=2h=0$ if $\permMM\permSS=-1$. In particular, we cannot set $t=1$ to obtain Lee deformation, unless we work with $\Z_2$ coefficients.
\end{remark}

The~homology theory defined by the~algebra $A_\bullet$ is universal: it carries the~most information among all chronological Frobenius algebras producing link homology. The~proof follows the~argument from \cite{KhFrobExt} and it is based on the~following observation.

Given a~chronological Frobenius algebra $A$ and an~invertible element $y\in A$ of degree $(1,0)$, we can twist its coalgebra structure by $y$ as follows:
\begin{equation}
	\epsilon'(a) := \epsilon(ya),
		\hskip 1cm
	\Delta'(a) := \Delta(y^{-1}a).
\end{equation}
If $\Delta$ and $\epsilon$ are homogeneous, so are their twisted version $\Delta'$ and $\epsilon'$. The~degrees are not changed. Because $\deg y=-\deg\mu$, there is an~equality $\Delta(y^{-1}a) = y^{-1}\Delta(a)$:
\begin{equation}
	\psset{unit=5mm}
	\begin{centerpict}(0,1)(4,6)
		\psset{linewidth=1pt}
		\psline(1,2)(1,3.5)\psline(3,1)(3,3.5)\psline(2,3.5)(2,5)\psline(1,5)(1,6)\psline(3,5)(3,6)
		\psset{linewidth=0.5pt}
		\rput(2,5){\psframe[framearc=0.5,fillstyle=solid](-1.7,-0.5)(1.7,0.5)\rput[B](0,-2pt){$\scriptstyle \Delta$}}
		\rput(2,3.5){\psframe[framearc=0.5,fillstyle=solid](-1.7,-0.5)(1.7,0.5)\rput[B](0,-2pt){$\scriptstyle \mu$}}
		\rput(1,2){\psframe[framearc=0.5,fillstyle=solid](-0.7,-0.5)(0.7,0.5)\rput[B](0,-2pt){$\scriptstyle y^{-1}$}}
	\end{centerpict}
	= Z^{-1}
	\begin{centerpict}(0,1)(5,6)
		\psset{linewidth=1pt}
		\psline(1,2)(1,5)\psline(3,1)(3,3.5)\psline(1.65,5)(1.65,6)\psline(2.33,3.5)(2.33,5)\psline(3.67,3.5)(3.67,6)
		\psset{linewidth=0.5pt}
		\rput(3,3.5){\psframe[framearc=0.5,fillstyle=solid](-1.35,-0.5)(1.35,0.5)\rput[B](0,-2pt){$\scriptstyle \Delta$}}
		\rput(1.65,5){\psframe[framearc=0.5,fillstyle=solid](-1.35,-0.5)(1.35,0.5)\rput[B](0,-2pt){$\scriptstyle \mu$}}
		\rput(1,2){\psframe[framearc=0.5,fillstyle=solid](-0.7,-0.5)(0.7,0.5)\rput[B](0,-2pt){$\scriptstyle y^{-1}$}}
	\end{centerpict}
	=
	\begin{centerpict}(-0.15,1)(4,6)
		\psset{linewidth=1pt}
		\psline(1,3.5)(1,5)\psline(1.65,5)(1.65,6)\psline(2.33,1.5)(2.33,5)\psline(3.67,1.5)(3.67,6)\psline(3,1)(3,2)
		\psset{linewidth=0.5pt}
		\rput(3,2){\psframe[framearc=0.5,fillstyle=solid](-1.35,-0.5)(1.35,0.5)\rput[B](0,-2pt){$\scriptstyle \Delta$}}
		\rput(1.65,5){\psframe[framearc=0.5,fillstyle=solid](-1.35,-0.5)(1.35,0.5)\rput[B](0,-2pt){$\scriptstyle \mu$}}
		\rput(1,3.5){\psframe[framearc=0.5,fillstyle=solid](-0.7,-0.5)(0.7,0.5)\rput[B](0,-2pt){$\scriptstyle y^{-1}$}}
	\end{centerpict}
\end{equation}

\begin{lemma}[cf. \cite{KhFrobExt}]
	Assume that $\F$ and $\F'$ are two functors induced by an~algebra $A$ and its twisted version $A'$. Then the~complexes $\F\KhCom(L)$ and $\F'\KhCom(L')$ are isomorphic.
\end{lemma}
\begin{proof}
	Consider cubes $\F\KhCubeSigned{L}{\epsilon}$ and $\F'\KhCubeSigned{L}{\epsilon}$, both corrected by a~sign assignment $\epsilon$. They have the~same $R$-modules in vertices and the~only difference is in edges labeled with comultiplications. The~isomorphism is constructed inductively, starting with the~identity homomorphism on the~initial vertex $(0,\dots,0)$ and applying the~following rule at every face:
	\begin{equation}
		\begin{diagps}(0,-1ex)(22em,12ex)
			\square<7em,10ex>[%
				\F\mathcal{I}_{\xi}`\F'\mathcal{I}_{\xi'}`\F\mathcal{I}_{\xi'}`\F'\mathcal{I}_{\xi};
				f`\mu`\mu`f]
			\square(15em,0ex)<7em,10ex>[%
				\F\mathcal{I}_{\xi}`\F'\mathcal{I}_{\xi'}`\F\mathcal{I}_{\xi'}`\F'\mathcal{I}_{\xi};
				f`\Delta`\Delta`y^{-1}\cdot f]
		\end{diagps}
	\end{equation}
	where in the~case of a~split we multiply by $y^{-1}$ the~element from the~copy of $A$ corresponding to the~circle that appears to the~left of the~split.
\end{proof}

\begin{theorem}\label{thm:universality}
	Any homogeneous rank two chronological Frobenius system $(R',A')$ in $\Mod{R}$ is obtained from $(R_\bullet, A_\bullet)$ by a~base change and a~twist. In particular, $\DKh(L):=H(\FD\KhCom(L))$ is the~most general link homology theory in our framework.
\end{theorem}
\begin{proof}
	Recall the~Frobenius system $(R_U,A_U)$ is universal with respect to the~base change operation. An~element $y=ev_+ + \permSS\permMS fv_- \in A_U$ is invertible and of degree $(1,0)$, with an~inverse $y^{-1} = (a+ch)v_+ - \permSS\permMS cv_-$. The~dotted algebra $A_\bullet$ arises as the~twisting of $(R_U,A_U)$ by this element.
\end{proof}

\section{Odds and ends}\label{sec:odds-ends}
\subsection*{Tangle cobordisms}

Let $\cat{Cob}^4(k)$ be the~category of tangles with $2k$ endpoints and tangle cobordisms between them, i.e.\ surfaces $W\subset\mathbb{D}^3\times I$ with its boundary decomposing into the~input and the~output tangles $T_i\subset\mathbb{D}^3\times\{i\}$, $i=0,1$, and vertical lines on $\partial\mathbb{D}^3\times I$. In particular, cobordisms between empty links are $2$-knots.

There is a~presentation of $\cat{Cob}^4(k)$ due to Carter and Saito \cite{CarterSaito} using \emph{movies}: sequences of sections of $W$ cutting the~cobordism into simple pieces, each with at most one singularity. There are nine singularities, corresponding to nine generators: the~three Reidemeister moves (each represents two generators), a~saddle move, a~birth, and a~death (see Fig.~\ref{fig:cob4-generators}). They are subject to a~number of relations, called \emph{movie moves}, that represent isotopic cobordisms, see \cite{CarterSaito}.

The~even Khovanov homology $\EKh(L)$ was proven to be functorial up to sign \cites{JacobFun,KhFunctorial,DrorCobs}, and corrected later to a~functor \cites{ClarkMorWalker,Blanchet}. This means there is a~chain map $\KhCom(W)\colon\KhCom(L_0)\to\KhCom(L_1)$ for any surface $W\subset\mathbb{R}^3\times I$ with $L_0$ and $L_1$ as its boundary, $L_i\subset\R^3\times\{i\}$.

It is not obvious how functoriality should be understood for odd homology. For instance, consider a~cobordism $W\colon \fntCircle \fntCircle \Longrightarrow \emptyset$ from the~two-component unlink to an~empty diagram given by two deaths. Depending on how we decompose $W$ into simple pieces (i.e.\ which link component vanishes first), we obtain two chain maps that differ by $\permSS$. One can try to show $\KhCom$ is a~weak $2$-functor, where movie moves are $2$-morphisms in $\cat{Cob}^4(k)$. However, this approach requires understanding of higher singularities of embedded cobordisms.

Functoriality up to \quot{sign} of the~generalized Khovanov complex $\KhCom(\blank)$, where by a~\quot{sign} we mean any degree $0$ invertible element of $\scalars$, is more promising. One can try to modify the~proof of the~even case presented in \cite{DrorCobs}, showing that for most tangles the~
automorphism groups of $\KhCom(T)$ are multiplies of the~identity map. We can define the~chain maps for generators as in the~table below.
\begin{center}
	\begin{tabular}{p{0.25\textwidth}p{0.6\textwidth}}
		\hline
		\multicolumn{1}{c}{\textbf{Movie}} & \multicolumn{1}{c}{\textbf{Chain map on $\KhCom(D)$}}\\
		\hline
		Reidemeister moves	& Homotopy equivalences from Theorem~\ref{thm:invariance} \\[1ex]
		Saddle move					& The~chain map $\KhBracket{\fntHorResCob}\colon\KhBracket{\fntHorRes}\to\KhBracket{\fntVertRes}\{1\}$
													obtained from the~cube or resolutions of the~tangle $\fntNWSECr$.	\\[1ex]
		Birth/death move		& The~chain maps induced by births and clockwise deaths. \\
		\hline
	\end{tabular}
\end{center}
The~last chain map requires some explanation. Consider a~morphism $b\colon\KhCubeSigned{T}{\epsilon}\to\KhCubeSigned{T\sqcup\fntCircle}{\epsilon}$ of anticommutative cubes with each component $b_\xi$ given by a~birth. They do not commute with edge morphisms of the~cubes, but we can fix it by the~same argument we used in the~proof of Theorem~\ref{thm:invariance}: scale $b_\xi$ by $\lambda((1,0),\chdeg W)$, where $W\subset\Disk\times I$ is a~cobordism given by any path from the~initial vertex $(0,\dots,0)$ to $\xi$. In a~similar way we define the~chain map for a~death.

Unfortunately, the~proof of functoriality in \cite{DrorCobs} does not translate immediately to our setting---the~problem is with Lemma~8.8, which states that a~tangle $T$ is $\KhCom$-simple (i.e.\ the~only automorphisms of $\KhCom(T)$ are $\pm\id$) if $TX$ is such ($TX$ is the~tangle obtained from $T$ by adding one extra crossing along its boundary). Functoriality of planar operations is used in the~original proof, the~property that does not hold in our setting. However, we believe this can be fixed with some generalization of the~argument used in the~proof of Theorem~\ref{thm:invariance}.

\begin{figure}
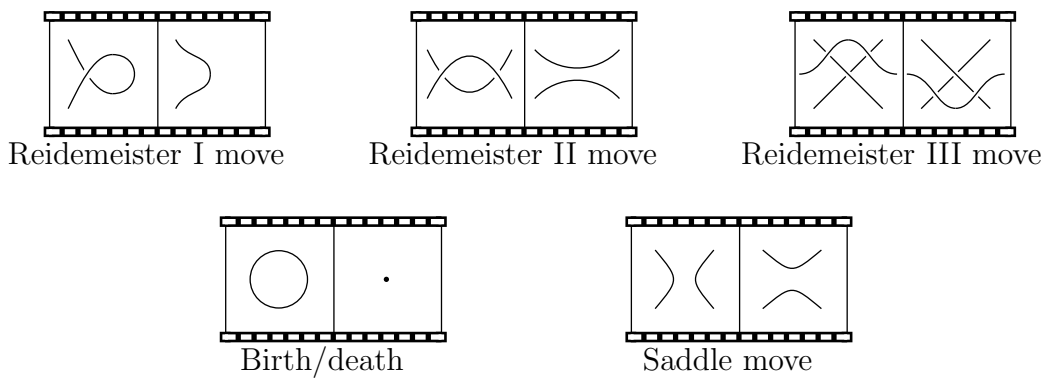

	\psset{unit=1.3}%

	\vskip\baselineskip
	\begin{tabular}{ccccc}
		\begin{movie}(1.1,1.1){2}
			\movieclip{\tangleRIx}
			\movieclip{\tangleRIline}
		\end{movie} & \quad\quad &
		\begin{movie}(1.1,1.1){2}
			\movieclip{\tangleRIIxx}
			\movieclip{\tangleRIIhhs}
		\end{movie} & \quad\quad &
		\begin{movie}(1.1,1.1){2}
			\movieclip{\tangleRIIIax}
			\movieclip{\tangleRIIIbx}
		\end{movie}	\\
		Reidemeister I move &&
		Reidemeister II move &&
		Reidemeister III move
	\end{tabular}
	
	\vskip\baselineskip
	\begin{tabular}{ccc}
		\begin{movie}(1.1,1.1){2}
			\movieclip{\pscircle[linewidth=0.5pt](0,0){0.3}}
			\movieclip{\psdot[dotsize=2pt](0,0)}
		\end{movie} &\quad\hskip 1cm\quad&
		\begin{movie}(1.1,1.1){2}
			\movieclip{%
				\psset{linewidth=0.5pt}%
				\psbezier(-0.3,-0.3)(-0.05,0)(-0.05,0)(-0.3,0.3)
				\psbezier( 0.3,-0.3)( 0.05,0)( 0.05,0)( 0.3,0.3)
			}%
			\movieclip{%
				\psset{linewidth=0.5pt}%
				\psbezier(-0.3,-0.3)(0,-0.05)(0,-0.05)(0.3,-0.3)
				\psbezier(-0.3, 0.3)(0, 0.05)(0, 0.05)(0.3, 0.3)
			}%
		\end{movie} \\
		Birth/death &&
		Saddle move
	\end{tabular}
	\caption{Movie diagrams for generator of $\cat{Cob}^4(k)$. Each diagram represents
		up to two generators, depending on the~direction in which the~movie is watched.}
	\label{fig:cob4-generators}
\end{figure}

\begin{conjecture}
	The~above defines a~\quot{functor} $\KhCom\colon\cat{Cob}^4(k)\to\Kom(\kChCobL(k))$ that assigns to a~tangle $T$ the~generalized Khovanov complex $\KhCom(T)$ and to a~tangle cobordism $W$ a~chain map $\KhCom(W)\colon\KhCom(T)\to\KhCom(T')$, defined up to a~global invertible scalar.
\end{conjecture}

\subsection*{\texorpdfstring{$\Diamond$}{Diamond}-change revisited}

\wrapfigure[r]{\begin{pspicture}(-0.8,-0.2)(0.8,0.8)\diagConnT{}{-}{}{-}\end{pspicture}}
The~choice we used to assign a~coefficient for a~$\Diamond$-change is not the~only one. We might as well assign $1$ to the~diagram with the~outer arrow pointing to the~right and $\permMM\permSS$ for the~other case, and $\iota$ would still be coherent with all relations between elementary changes of chronologies. The~new commutativity cocycle $\overline\psi$ has the~same values as $\psi$, except that
\begin{equation}\label{eq:dual-psi}
	\overline\psi\left(\begin{centerpict}(-0.8,-0.6)(0.8,0.6)\diagConnT{}{->}{}{->}\end{centerpict}\right) = \permMM\permSS
	\qquad\text{and}\qquad
		\overline\psi\left(\begin{centerpict}(-0.8,-0.6)(0.8,0.6)\diagConnT{}{->}{}{<-}\end{centerpict}\right) = 1.
\end{equation}
We shall now prove that the~corrected cube of resolutions does not depend on which commutativity cocycle we choose. Unfortunately, there is a~gap in the~original proof from \cite{ORS}, noticed by Cotton Seed: given a~sign assignment $\epsilon$ with $d\epsilon=\psi$ the~authors of \cite{ORS} constructed $\overline\epsilon$ with $d\overline\epsilon=\overline\psi$, but an~isomorphism of cubes $\KhCubeSigned{T}{\epsilon}\cong\KhCubeSigned{T}{\overline\epsilon}$ is missing. We found such an~isomorphism only when $T$ is a~link and the~cube $\KhCubeSigned{T}\epsilon$ is regarded as a~diagram in $\kChCob(0)$,\footnote{
	This step requires us to enumerate circles in each resolution, since the~disjoint union in $\kChCob(0)$ is not strictly symmetric. The~cube, however, is independent of these choices: different orders of circles are related by canonical isomorphisms, which in turn induce an~isomorphism of cubes.}
which is enough for the~odd theory, but leaves the~case of nested theories open.

\begin{proposition}\label{lem:kh-cube-indep-of-psi}
	Given a~link diagram $D$ choose sign assignments $\epsilon$ and $\overline\epsilon$ for the~cube $\KhCube{D}$ with respect to the~cocycles $\psi$ and $\overline\psi$ respectively. Then there is an~isomorphism of cubes $\KhCubeSigned{D}{\epsilon}\cong\KhCubeSigned{D}{\overline\epsilon}$, regarded as diagrams in $\kChCob(0)$.
\end{proposition}
\begin{proof}
	Instead of constructing $\overline\epsilon$ we shall alter the~diagram $D$ into $D'$, so that $\delta\epsilon=\overline\psi$ for $D'$. Color the~diagram $D$ black and white in a~checkerboard fashion. Given a~set of arrows orienting crossings, reverse every arrow between white regions:
	\begin{equation}
		\psset{unit=10mm}
		\begin{centerpict}(0,0)(1,1)
			\pspolygon[fillstyle=solid,fillcolor=lightgray,linestyle=none](0,0)(1,1)(1,0)(0,1)
			\psset{linewidth=0.5pt,border=2.5pt}%
			\psline(0,0)(1,1)\psline(1,0)(0,1)
			\psline[border=0.5pt]{->}(0.5,0.2)(0.5,0.8)
		\end{centerpict}
		\to/<->/
		\begin{centerpict}(0,0)(1,1)
			\pspolygon[fillstyle=solid,fillcolor=lightgray,linestyle=none](0,0)(1,1)(1,0)(0,1)
			\psset{linewidth=0.5pt,border=2.5pt}%
			\psline(0,0)(1,1)\psline(1,0)(0,1)
			\psline[border=0.5pt]{<-}(0.5,0.2)(0.5,0.8)
		\end{centerpict}
		\hskip 1cm
		\begin{centerpict}(0,0)(1,1)
			\pspolygon[fillstyle=solid,fillcolor=lightgray,linestyle=none](0,0)(1,1)(0,1)(1,0)
			\psset{linewidth=0.5pt,border=2.5pt}%
			\psline(0,0)(1,1)\psline(1,0)(0,1)
			\psline[border=0.5pt]{->}(0.5,0.2)(0.5,0.8)
		\end{centerpict}
		\to/<->/
		\begin{centerpict}(0,0)(1,1)
			\pspolygon[fillstyle=solid,fillcolor=lightgray,linestyle=none](0,0)(1,1)(0,1)(1,0)
			\psset{linewidth=0.5pt,border=2.5pt}%
			\psline(0,0)(1,1)\psline(1,0)(0,1)
			\psline[border=0.5pt]{->}(0.5,0.2)(0.5,0.8)
		\end{centerpict}
	\end{equation}
	to obtain a~new decorated diagram $D'$. This operation preserves all the~diagrams from Tab.~\ref{tab:cube-faces}, except the~two shown in \eqref{eq:dual-psi}, which are exchanged. Hence, $\delta\epsilon:=\overline\psi$ for $D'$. We construct an~isomorphism $s\colon\KhCubeGradedSigned{D}{\epsilon}\cong\KhCubeGradedSigned{D'}{\epsilon}$ as follows. The~coloring of $D$ induces a~coloring of its resolutions $D_\xi$ such that every circle is a~boundary of a~unique black region. Take the~boundary circles of a~black region and apply a~half-twist to them; the~component $s_\xi\colon D_\xi \to D'_\xi$ is a~composition of such half-twists for all black regions in $D_\xi$. It is an~isomorphism of cubes, since what it does is exactly to reverse the~arrows connecting white regions.
\end{proof}

\noindent
In fact, the~only condition for $\iota$ to be coherent with relations between changes of chronologies is that the~quotient of its values on the~two $\Diamond$-changes is equal to $\permMM\permSS$. Hence, we can set
\begin{equation}
	\iota\bigg(\pictConnT{}{->}{}{->}\bigg)=\permT\permMM
		\hskip 2cm
	\iota\bigg(\pictConnT{}{->}{}{<-}\bigg)=\permT\permSS
\end{equation}
where $\permT$ is an~additional generator. This new parameter is useless from the~point of view of Frobenius algebras: it will give only an~additional restriction, that $\permT\permMM-1$ and \mbox{$\permT\permSS-1$} annihilate $\mu\circ\Delta$. However, it may be used to produce odd versions of nested homology theories (the~two cobordisms related by a~$\Diamond$-change are diffeomorphic, but not isotopic), see \cites{StWeb, BelWag}.

\subsection*{Rotating arrows and \texorpdfstring{$\mathfrak{sl}(2)$}{sl(2)} foams}

In the~original construction of odd Khovanov homology, the~small arrow over a~crossing frames not only the~negative eigenspace $E^-(p)$ of a~saddle point $p$, but also its positive eigenspace $E^+(p)$ and the~latter is used to distinguish between the~two output circles of a~split. Because of the~convention that every arrow rotates clockwise when going up, one framing arrow is enough.

If we allow an~arrow to rotate in any direction, i.e.\ when we orient both $E^-(p)$ and $E^+(p)$ independently, we will create a~richer category with two versions of each generating cobordism. It is not difficult to find out chronological relations: the~coefficients assigned to changes do not depend on how the~arrows rotate, except $\times$-and $\Diamond$-changes, in which cases the~coefficients are multiplied by $\permSS$, if the~arrows rotate in different directions, see Tab.~\ref{tab:extended-coeffs}.

\begin{remark}
	This is not the~most general solution. For instance, one can assign different coefficients to changes permutating merges that are differently oriented. In the~most general case one obtains a~system of nine independent parameters.
\end{remark}

A~choice of how a~single arrow rotates introduces another datum to the~construction of the~generalized Khovanov complex. The~isomorphism class of the~complex does not depend on this~additional chain, which follows from the~commutativity of the~following square:
\begin{equation}\label{diag:change-rotation}
	\begin{diagps}(-3,-1.8)(3,2.6)
		\fnode[linestyle=none,framesize=1.3 0.5](-2, 1.5){domA}
			\pscircle(-2.4, 1.5){0.25}\pscircle(-1.6, 1.5){0.25}
		\fnode[linestyle=none,framesize=1.3 0.5](-2,-1.5){domB}
			\pscircle(-2.4,-1.5){0.25}\pscircle(-1.6,-1.5){0.25}
		\fnode[linestyle=none,framesize=0.5](2, 1.5){codA}
			\pscircle(2, 1.5){0.25}
		\fnode[linestyle=none,framesize=0.5](2,-1.5){codB}
			\pscircle(2,-1.5){0.25}
		\diagline{->}{domA}{codA}\naput{\psset{unit=0.7}\cobordism[2](-0.6,0.1)({a}M-R)}
		\diagline{->}{domB}{codB}\naput{\psset{unit=0.7}\cobordism[2](-0.6,0.1)({c}M-R)}
		\diagline{->}{domA}{domB}\nbput{%
			\psset{unit=0.7}%
			\rput[r](-2.2,0){$\scriptstyle\permMM$}
			\cobordism*[1](-2.1,-0.9)(sI)(I)
			\cobordism*[1](-0.5,-0.9)({c}slB)({a}M-R)}
		\diagline{->}{codA}{codB}\naput{\psset{unit=0.7}\cobordism[1](0.1,-0.6)(I)}
	\end{diagps}
\end{equation}
where the~left vertical cobordism is an~isomorphism in $\kChCob(0)$ and its inverse is given by the~same picture, but with different orientations of critical points:
\begin{equation}
	\psset{unit=0.8}
	\textcobordism*[1]({c}slB)({a}M-R)({a}slB)({c}M-R)
		\,\dblto^X\,
	\textcobordism*{%
		\COBshortBirth{a}(0.4,0)
		\COBshortRightBirth{c}[1,2](0,0.6)
		\COBcylinder(1.6,0)(1.6,1.2)
	}{\COBmergeFrRight{c}[3](0,1.2)}{\COBmergeFrRight{a}(0.4,2.4)}
		\,\dblto^1\,
	\textcobordism*[1]({a}lB)({a}M-R)
		\,\dblto^X\,
	\textcobordism*[1](I)(I)
\end{equation}
and similarly for the~other composition. The~vertical morphisms are homogeneous in degree $0$, which implies they commute with all other edge morphisms in the~cubes. Hence, \eqref{diag:change-rotation} induces an~isomorphism between complexes obtained from two diagrams of a~tangle, that differ only in the~way a~single arrow rotates.

\def\smallfrac#1#2{\tfrac{#1}{\raisebox{2pt}{$\scriptstyle #2$}}}
\begin{table}%
	\begin{tabular}{cccc}
		\pictConnX{a|c}{->}{a|c}{->} &
		\pictConnX{a|c}{->}{c|a}{->} &
		\pictConnX{a|c}{->}{a|c}{<-} &
		\pictConnX{a|c}{->}{c|a}{<-} \\
		$\permSS$ & $1$ & $\permMM$ & $\permMM\permSS$ \\ \\
		\pictConnT{\smallfrac ac}{->}{a|c}{->} &
		\pictConnT{\smallfrac ac}{->}{c|a}{->} &
		\pictConnT{\smallfrac ac}{->}{a|c}{<-} &
		\pictConnT{\smallfrac ac}{->}{c|a}{<-} \\
		$1$ & $\permSS$ & $\permMM\permSS$ & $\permMM$ \\
	\end{tabular}
	\vskip\baselineskip
	\caption{%
		Coefficients assigned to $\times$- and $\Diamond$-permutation, when each arrow can rotate either clockwise or anticlockwise. The~symbols \protect\raisebox{1pt}{$\scriptstyle a|c$} and $\protect\smallfrac ac$ stand for two alternative ways of rotating an~arrow and one has to make the~same choice (left/top or right/bottom) for both arrows.
	}
	\label{tab:extended-coeffs}
\end{table}

The~author was encouraged to investigate rotations of arrows by M.~Hempel, who computed several circular movies for the~odd theory and noticed, that if arrows over crossings with opposite signs rotate differently, movies consisting of Reidemeister II moves induce identity chain maps. This suggests a~connection with $\mathfrak{sl}(2)$ foams, i.e. singular cobordisms with two types of saddle points, one for positive and one for negative crossings.

\appendix

\section{Framed functions}\label{sec:framed-function}
Let $W$ be a~smooth compact manifold, possibly with boundary.
\begin{definition}\label{def:Igusa-function}
	An~\emph{Igusa function} is a~smooth function $f\colon W\to\mathbb{R}$, such that at every point $p\in W$ one of the~following conditions holds:
	\begin{enumerate}[label={IF\arabic{*}:},ref={IF\arabic{*}}]
		\item\label{GM-reg} $p$ is \emph{regular}, i.e. the~derivative $df_p$ does not vanish, or
		\item\label{GM-A1} $f$ has a~\emph{Morse singularity} (or $A_1$ singularity) at $p$, i.e.\ $df_p=0$ but the Hessian $\Hess_p(f)$ is nondegenerate, or
		\item\label{GM-A2} $f$ has a~\emph{birth-death singularity} (or $A_2$ singularity) at $p$, i.e.\ $df_p=0$ and $\Hess_p(f)$ has a~1-dimensional kernel $N(p)\subset T_pW$, but $d^3f_p$ is nonzero on $N(p)$.
	\end{enumerate}
\end{definition}

Morse and birth-death singularities of a~function $f$ have the~following local models:
\begin{align}
	f(x_1,\dots,x_n) &= f(p) - x_1^2 - \dots - x_k^2 + x_{k+1}^2 +\dots x_n^2,\\
	f(x_1,\dots,x_n) &= f(p) - x_1^2 - \dots - x_k^2 + x_{k+1}^2 +\dots x_{n-1}^2+x_n^3.
\end{align}
In the~latter case the~nullspace $N(p)$ of $Hess_p(f)$ is spanned by $\frac\partial{\partial x_n}$. The~number $k=\mu(p)$ is called the~\emph{index} of $p$.

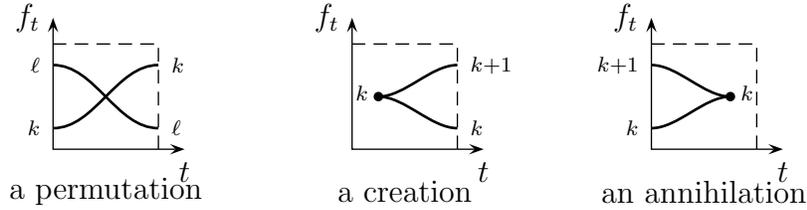
\begin{figure}
	\centering
	\psset{unit=7mm}%
\begin{pspicture}(-1,-1)(3.5,2.8)
	\locusCanvas(2,2)
	\psbezier[linewidth=1pt]{-}(0,0.4)(0.8,0.4)(1.2,1.6)(2,1.6)
	\uput[180](0,0.4){$\scriptstyle   k$}\uput[0](2,1.6){$\scriptstyle   k$}
	\psbezier[linewidth=1pt]{-}(0,1.6)(0.8,1.6)(1.2,0.4)(2,0.4)
	\uput[180](0,1.6){$\scriptstyle\ell$}\uput[0](2,0.4){$\scriptstyle\ell$}
	\rput[Bc](1,-0.8){a permutation}
\end{pspicture}\hskip 2em
\begin{pspicture}(-1,-1)(3.5,2.8)
	\locusCanvas(2,2)
	\psbezier[linewidth=1pt]{-}(2,0.4)(1.5,0.4)(1,1)(0.5,1)
	\psbezier[linewidth=1pt]{-}(2,1.6)(1.5,1.6)(1,1)(0.5,1)
	\psdot(0.5,1)
	\uput[170](0.55,1){$\scriptstyle k$}\uput[0](2,1.6){$\scriptstyle k+1$}\uput[0](2,0.4){$\scriptstyle k$}
	\rput[Bc](1,-0.8){a creation}
\end{pspicture}\hskip 2em
\begin{pspicture}(-1,-1)(3.5,2.8)
	\locusCanvas(2,2)
	\psbezier[linewidth=1pt]{-}(0,0.4)(0.5,0.4)(1,1)(1.5,1)
	\psbezier[linewidth=1pt]{-}(0,1.6)(0.5,1.6)(1,1)(1.5,1)
	\psdot(1.5,1)
	\uput[180](0,1.6){$\scriptstyle k+1$}\uput[180](0,0.4){$\scriptstyle k$}\uput[10](1.45,1){$\scriptstyle k$}
	\rput[Bc](1,-0.8){an annihilation}
\end{pspicture}
	\caption[Singular loci for elementary homotopies of Igusa functions]{Singular loci for elementary homotopies of Igusa functions. Cusps represent $A_2$-singularities, and labels are the~indicies of critical points.}%
	\label{fig:singular-loci}%
\end{figure}

Igusa functions arise naturally if one considers homotopies between smooth functions: a~generic function on $W$ is Morse (conditions \ref{GM-reg} and \ref{GM-A1}) and \emph{separative} (critical points lie on different levels), but a~space of such functions is not even connected. However, a~transversality argument implies a~generic homotopy $f_t$ is separative Morse except finitely many moments $0 < t_1 < \dots < t_k < 1$, at which either two critical points are permuted or a~birth-death singularity occurs \cite{Cerf}; we refer to them as \emph{events}. We can visualize them by drawing the~singular locus $\SingLocus{f} := \left\{(t,f_t(x))\ \vert\ x\in\mathrm{crit}(f_t)\right\}$, see Fig.~\ref{fig:singular-loci}.

Choose a~generic two-parameter family $f_{t,s}\colon W\to\R$ of Igusa functions, $t,s\in I$. The~path $t\mapsto f_{t,s}$ is a~generic homotopy of Igusa functions for all except finitely many $s\in I$, at which one of the~situations described below occurs, see \cites{IgusaHS,EliMish-Contr}.

\begin{enumerate}[label={Case \Roman{*}},leftmargin=!,labelwidth=3.5em,align=left]
	
	\item Two events can occur at the~same time $t_i$. For example, we have homotopies
	\begin{equation}\label{rel:far-away}\begin{split}
		\psset{unit=0.4cm}
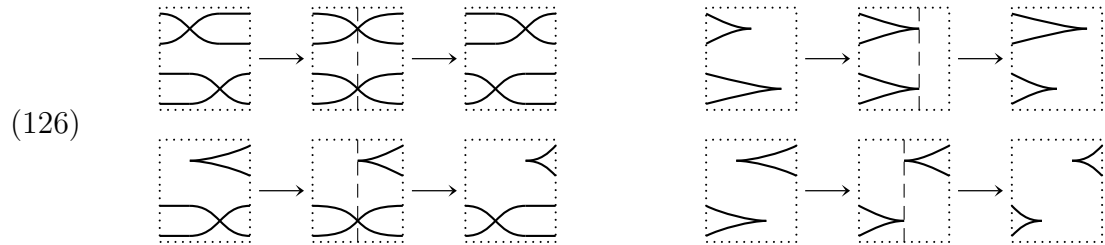
		\begin{centerpict}(0,-0.5)(3,3.5)
			\pspolygon[linestyle=dotted,dotsep=2pt](0,-0.2)(3,-0.2)(3,3.2)(0,3.2)
			\psline(0,0)(1,0)\psbezier(1,0)(2,0)(2,1)(3,1)
			\psline(0,1)(1,1)\psbezier(1,1)(2,1)(2,0)(3,0)
			\psbezier(0,2)(1,2)(1,3)(2,3)\psline(2,3)(3,3)
			\psbezier(0,3)(1,3)(1,2)(2,2)\psline(2,2)(3,2)
		\end{centerpict}
		\to
		\begin{centerpict}(0,-0.5)(3,3.5)
			\pspolygon[linestyle=dotted,dotsep=2pt](0,-0.2)(3,-0.2)(3,3.2)(0,3.2)
			\psline[linestyle=dashed,linewidth=0.5\pslinewidth](1.5,-0.2)(1.5,3.2)
			\psbezier(0,0)(2,0)(1,1)(3,1)
			\psbezier(0,1)(2,1)(1,0)(3,0)
			\psbezier(0,2)(2,2)(1,3)(3,3)
			\psbezier(0,3)(2,3)(1,2)(3,2)
		\end{centerpict}
		\to
		\begin{centerpict}(0,-0.5)(3,3.5)
			\pspolygon[linestyle=dotted,dotsep=2pt](0,-0.2)(3,-0.2)(3,3.2)(0,3.2)
			\psbezier(0,0)(1,0)(1,1)(2,1)\psline(2,1)(3,1)
			\psbezier(0,1)(1,1)(1,0)(2,0)\psline(2,0)(3,0)
			\psline(0,2)(1,2)\psbezier(1,2)(2,2)(2,3)(3,3)
			\psline(0,3)(1,3)\psbezier(1,3)(2,3)(2,2)(3,2)
		\end{centerpict}
		\hskip 2cm
		\begin{centerpict}(0,-0.5)(3,3.5)
			\pspolygon[linestyle=dotted,dotsep=2pt](0,-0.2)(3,-0.2)(3,3.2)(0,3.2)
			\psbezier(0,0)(1,0.25)(2,0.5)(2.5,0.5)
			\psbezier(0,1)(1,0.75)(2,0.5)(2.5,0.5)
			\psbezier(0,2)(0.5,2.25)(1,2.5)(1.5,2.5)
			\psbezier(0,3)(0.5,2.75)(1,2.5)(1.5,2.5)
		\end{centerpict}
		\to
		\begin{centerpict}(0,-0.5)(3,3.5)
			\pspolygon[linestyle=dotted,dotsep=2pt](0,-0.2)(3,-0.2)(3,3.2)(0,3.2)
			\psline[linestyle=dashed,linewidth=0.5\pslinewidth](2,-0.2)(2,3.2)
			\psbezier(0,0)(0.75,0.25)(1.5,0.5)(2,0.5)
			\psbezier(0,1)(0.75,0.75)(1.5,0.5)(2,0.5)
			\psbezier(0,2)(0.75,2.25)(1.5,2.5)(2,2.5)
			\psbezier(0,3)(0.75,2.75)(1.5,2.5)(2,2.5)
		\end{centerpict}
		\to
		\begin{centerpict}(0,-0.5)(3,3.5)
			\pspolygon[linestyle=dotted,dotsep=2pt](0,-0.2)(3,-0.2)(3,3.2)(0,3.2)
			\psbezier(0,0)(0.5,0.25)(1,0.5)(1.5,0.5)
			\psbezier(0,1)(0.5,0.75)(1,0.5)(1.5,0.5)
			\psbezier(0,2)(1,2.25)(2,2.5)(2.5,2.5)
			\psbezier(0,3)(1,2.75)(2,2.5)(2.5,2.5)
		\end{centerpict} \\
		\psset{unit=0.4cm}
		\begin{centerpict}(0,-0.5)(3,3.5)
			\pspolygon[linestyle=dotted,dotsep=2pt](0,-0.2)(3,-0.2)(3,3.2)(0,3.2)
			\psline(0,0)(1,0)\psbezier(1,0)(2,0)(2,1)(3,1)
			\psline(0,1)(1,1)\psbezier(1,1)(2,1)(2,0)(3,0)
			\psbezier(3,2)(2.5,2.25)(1.5,2.5)(1,2.5)
			\psbezier(3,3)(2.5,2.75)(1.5,2.5)(1,2.5)
		\end{centerpict}
		\to
		\begin{centerpict}(0,-0.5)(3,3.5)
			\pspolygon[linestyle=dotted,dotsep=2pt](0,-0.2)(3,-0.2)(3,3.2)(0,3.2)
			\psline[linestyle=dashed,linewidth=0.5\pslinewidth](1.5,-0.2)(1.5,3.2)
			\psbezier(0,0)(2,0)(1,1)(3,1)
			\psbezier(0,1)(2,1)(1,0)(3,0)
			\psbezier(3,2)(2.5,2.25)(2,2.5)(1.5,2.5)
			\psbezier(3,3)(2.5,2.75)(2,2.5)(1.5,2.5)
		\end{centerpict}
		\to
		\begin{centerpict}(0,-0.5)(3,3.5)
			\pspolygon[linestyle=dotted,dotsep=2pt](0,-0.2)(3,-0.2)(3,3.2)(0,3.2)
			\psbezier(0,0)(1,0)(1,1)(2,1)\psline(2,1)(3,1)
			\psbezier(0,1)(1,1)(1,0)(2,0)\psline(2,0)(3,0)
			\psbezier(3,2)(2.75,2.25)(2.5,2.5)(2,2.5)
			\psbezier(3,3)(2.75,2.75)(2.5,2.5)(2,2.5)
		\end{centerpict}
		\hskip 2cm
		\begin{centerpict}(0,-0.5)(3,3.5)
			\pspolygon[linestyle=dotted,dotsep=2pt](0,-0.2)(3,-0.2)(3,3.2)(0,3.2)
			\psbezier(0,0)(0.5,0.25)(1.5,0.5)(2,0.5)
			\psbezier(0,1)(0.5,0.75)(1.5,0.5)(2,0.5)
			\psbezier(3,2)(2.5,2.25)(1.5,2.5)(1,2.5)
			\psbezier(3,3)(2.5,2.75)(1.5,2.5)(1,2.5)
		\end{centerpict}
		\to
		\begin{centerpict}(0,-0.5)(3,3.5)
			\pspolygon[linestyle=dotted,dotsep=2pt](0,-0.2)(3,-0.2)(3,3.2)(0,3.2)
			\psline[linewidth=0.5\pslinewidth,linestyle=dashed](1.5,-0.2)(1.5,3.2)
			\psbezier(0,0)(0.5,0.25)(1,0.5)(1.5,0.5)
			\psbezier(0,1)(0.5,0.75)(1,0.5)(1.5,0.5)
			\psbezier(3,2)(2.5,2.25)(2,2.5)(1.5,2.5)
			\psbezier(3,3)(2.5,2.75)(2,2.5)(1.5,2.5)
		\end{centerpict}
		\to
		\begin{centerpict}(0,-0.5)(3,3.5)
			\pspolygon[linestyle=dotted,dotsep=2pt](0,-0.2)(3,-0.2)(3,3.2)(0,3.2)
			\psbezier(0,0)(0.25,0.25)(0.5,0.5)(1,0.5)
			\psbezier(0,1)(0.25,0.75)(0.5,0.5)(1,0.5)
			\psbezier(3,2)(2.75,2.25)(2.5,2.5)(2,2.5)
			\psbezier(3,3)(2.75,2.75)(2.5,2.5)(2,2.5)
		\end{centerpict}
	\end{split}\end{equation}
	where dashed lines indicate singular values of $t$. See also Fig.~\ref{diag:event-two} for singular loci of the~left two homotopies.
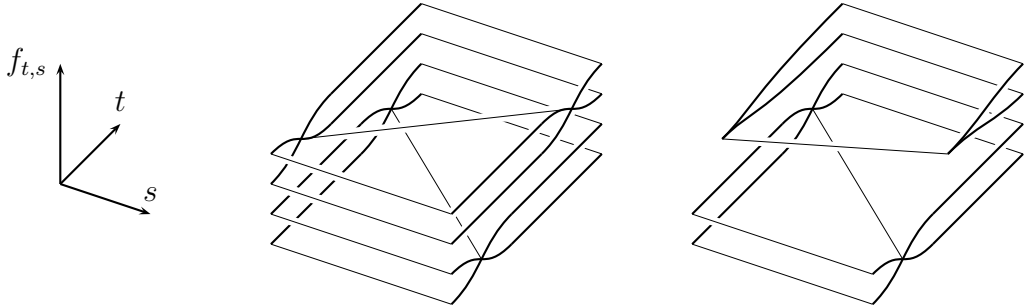
\begin{figure}
	\centering\psset{unit=0.4cm}
\begin{pspicture}(0,0)(35,10)
\rput(1,4){
	\psline{->}(0,0)(0, 4)\uput[l](0, 4){$f_{t,s}$}
	\psline{->}(0,0)(3,-1)\uput[u](3,-1){$s$}
	\psline{->}(0,0)(2, 2)\uput[u](2, 2){$t$}
}
\rput(8,0){
	\psline[linewidth=0.3\pslinewidth](5, 7)(11,5)											
	\psline[border=2\pslinewidth](8,3)(11,6)\psbezier(6,0)(7,1)(7,2)(8,3)	
	\psline[linewidth=0.3\pslinewidth](5, 8)(11,6)											
	\psline(0,2)(3,5)\psbezier(3,5)(4,6)(4,7)(5,8)			
	\psline[linewidth=0.3\pslinewidth](0,2)(6,0)												
	\psline[border=2\pslinewidth,linewidth=0.3\pslinewidth](0,3)(6,1)		
	\psline(0,3)(3,6)\psbezier(3,6)(4,7)(4,6)(5,7)			
	\psline(8,2)(11,5)\psbezier(6,1)(7,2)(7,1)(8,2)			
	\psline[linewidth=0.3\pslinewidth](4,6.5)(7,1.5)		
	\psline[linewidth=5\pslinewidth,linecolor=white](1,5.5)(10,6.5)		
	\psline[linewidth=5\pslinewidth,linecolor=white](0,4)(6,2)				
	\psline[linewidth=5\pslinewidth,linecolor=white](0,5)(6,3)				
	\psline[linewidth=0.3\pslinewidth](5, 9)(11,7)			
	\psline[border=2\pslinewidth](6,2)(9,5)\psbezier[border=2\pslinewidth](9,5)(10,6)(10,7)(11,8)		
	\psline(2,7)(5,10)\psbezier(0,4)(1,5)(1,6)(2,7)		
	\psline[border=2\pslinewidth,linewidth=0.3\pslinewidth](0,5)(6,3)			
	\psline(2,6)(5, 9)\psbezier(0,5)(1,6)(1,5)(2,6)		
	\psline[border=2\pslinewidth](6,3)(9,6)\psbezier(9,6)(10,7)(10,6)(11,7)		
	\psline[linewidth=0.3\pslinewidth](5,10)(11,8)		
	\psline[linewidth=0.3\pslinewidth](0,4)(6,2)			
	\psline[linewidth=0.3\pslinewidth](1,5.5)(10,6.5)		
}
\rput(22,0){
	\psline[linewidth=0.3\pslinewidth](5, 7)(11,5)											
	\psline[border=2\pslinewidth](8,3)(11,6)\psbezier(6,0)(7,1)(7,2)(8,3)	
	\psline[linewidth=0.3\pslinewidth](5, 8)(11,6)											
	\psline(0,2)(3,5)\psbezier(3,5)(4,6)(4,7)(5,8)			
	\psline[linewidth=0.3\pslinewidth](0,2)(6,0)												
	\psline[border=2\pslinewidth,linewidth=0.3\pslinewidth](0,3)(6,1)		
	\psline(0,3)(3,6)\psbezier(3,6)(4,7)(4,6)(5,7)			
	\psline(8,2)(11,5)\psbezier(6,1)(7,2)(7,1)(8,2)			
	\psline[linewidth=0.3\pslinewidth](4,6.5)(7,1.5)		
	\psline[border=2\pslinewidth,linewidth=0.3\pslinewidth](1,5.5)(8.5,5)		
	\psbezier[linecolor=white,linewidth=5\pslinewidth](8.5,5)(9.5,6)(10,6)(11,7)		
	\psline[linewidth=0.3\pslinewidth](5, 9)(11,7)						
	\psbezier[border=2\pslinewidth](8.5,5)(9.5,6)(10,7)(11,8)		
	\psbezier(8.5,5)(9.5,6)(10,6)(11,7)													
	\psline[linewidth=0.3\pslinewidth](5,10)(11,8)						
	\psbezier(1,5.5)(2,6.5)(2,7)(5,10)		
	\psbezier(1,5.5)(2,6.5)(2,6)(5, 9)		
}
\end{pspicture}
	\caption{Exampes of singular loci, when two events occurs  at the~same times.}\label{diag:event-two}%
\end{figure}
	
	\item A~non-transverse event occurs, i.e.\ the~singular set is not transverse to some level set $\{t=a\}$. Up to direction of the~change, there are three such homotopies
	\begin{equation}\label{rel:inverses}\begin{split}
		\psset{unit=0.6cm}
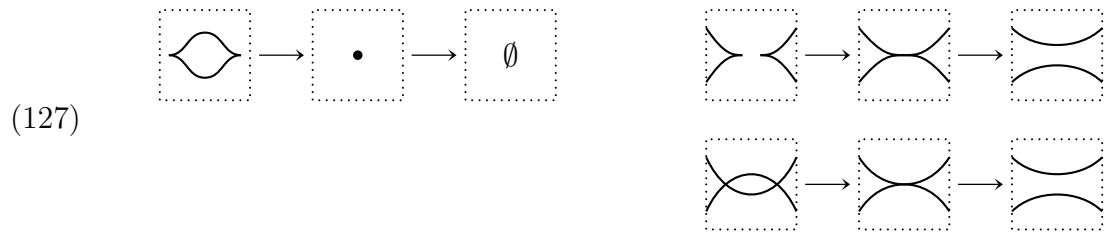
		\begin{centerpict}(-1,-1.3)(1,1.3)
			\pspolygon[linestyle=dotted,dotsep=2pt](-1,-1)(1,-1)(1,1)(-1,1)
			\psbezier(-0.8,0)(-0.4,0)(-0.4, 0.5)(0, 0.5)
			\psbezier(-0.8,0)(-0.4,0)(-0.4,-0.5)(0,-0.5)
			\psbezier( 0.8,0)( 0.4,0)( 0.4, 0.5)(0, 0.5)
			\psbezier( 0.8,0)( 0.4,0)( 0.4,-0.5)(0,-0.5)
		\end{centerpict}
		\to
		\begin{centerpict}(-1,-1.3)(1,1.3)
			\pspolygon[linestyle=dotted,dotsep=2pt](-1,-1)(1,-1)(1,1)(-1,1)
			\psdot(0,0)
		\end{centerpict}
		\to
		\begin{centerpict}(-1,-1.3)(1,1.3)
			\pspolygon[linestyle=dotted,dotsep=2pt](-1,-1)(1,-1)(1,1)(-1,1)
			\rput(0,0){$\emptyset$}
		\end{centerpict}
		\hskip 2cm
		\begin{centerpict}(-1,-1.3)(1,1.3)
			\pspolygon[linestyle=dotted,dotsep=2pt](-1,-1)(1,-1)(1,1)(-1,1)
			\psbezier(-1, 0.6)(-0.6,0)(-0.4,0)(-0.2,0)
			\psbezier(-1,-0.6)(-0.6,0)(-0.4,0)(-0.2,0)
			\psbezier( 1, 0.6)( 0.6,0)( 0.4,0)( 0.2,0)
			\psbezier( 1,-0.6)( 0.6,0)( 0.4,0)( 0.2,0)
		\end{centerpict}
		\to
		\begin{centerpict}(-1,-1.3)(1,1.3)
			\pspolygon[linestyle=dotted,dotsep=2pt](-1,-1)(1,-1)(1,1)(-1,1)
			\psbezier(-1, 0.6)(-0.5,0)(-0.3,0)(0,0)
			\psbezier(-1,-0.6)(-0.5,0)(-0.3,0)(0,0)
			\psbezier( 1, 0.6)( 0.5,0)( 0.3,0)(0,0)
			\psbezier( 1,-0.6)( 0.5,0)( 0.3,0)(0,0)
		\end{centerpict}
		\to
		\begin{centerpict}(-1,-1.3)(1,1.3)
			\pspolygon[linestyle=dotted,dotsep=2pt](-1,-1)(1,-1)(1,1)(-1,1)
			\psbezier(-1, 0.6)(-0.5, 0.1)(0.5, 0.1)(1, 0.6)
			\psbezier(-1,-0.6)(-0.5,-0.1)(0.5,-0.1)(1,-0.6)
		\end{centerpict} \\
		\psset{unit=0.6cm}
		\begin{centerpict}(-1,-1.3)(1,1.3)
			\pspolygon[linestyle=dotted,dotsep=2pt](-1,-1)(1,-1)(1,1)(-1,1)
			\psbezier(-1, 0.6)(-0.5,-0.5)(0.5,-0.5)(1, 0.6)
			\psbezier(-1,-0.6)(-0.5, 0.5)(0.5, 0.5)(1,-0.6)
		\end{centerpict}
		\to
		\begin{centerpict}(-1,-1.3)(1,1.3)
			\pspolygon[linestyle=dotted,dotsep=2pt](-1,-1)(1,-1)(1,1)(-1,1)
			\psbezier(-1, 0.6)(-0.5,-0.2)(0.5,-0.2)(1, 0.6)
			\psbezier(-1,-0.6)(-0.5, 0.2)(0.5, 0.2)(1,-0.6)
		\end{centerpict}
		\to
		\begin{centerpict}(-1,-1.3)(1,1.3)
			\pspolygon[linestyle=dotted,dotsep=2pt](-1,-1)(1,-1)(1,1)(-1,1)
			\psbezier(-1, 0.6)(-0.5, 0.1)(0.5, 0.1)(1, 0.6)
			\psbezier(-1,-0.6)(-0.5,-0.1)(0.5,-0.1)(1,-0.6)
		\end{centerpict}
	\end{split}\end{equation}
	and their singular loci are shown in Fig.~\ref{diag:event-nontranverse}.
\begin{figure}%
	\centering\psset{unit=0.4cm}
\begin{pspicture}(0,0)(35,10)
\rput(1,3){
	\psline{->}(0,0)(0, 4)\uput[l](0, 4){$f_{t,s}$}
	\psline{->}(0,0)(3,-1)\uput[u](3,-1){$s$}
	\psline{->}(0,0)(2, 2)\uput[u](2, 2){$t$}
}
\rput(6,2){
	\psbezier(0,2)(1,3)(1,2)(2,3)\psbezier(2,3)(3,4)(3,5)(4,6)
	\psbezier[border=2\pslinewidth,linewidth=0.3\pslinewidth](2.2,5.2)(5.5,4.1)(5,3.6)(5.9,3.3)
	\psbezier(0,2)(1,3)(1,4)(2,5)\psbezier(2,5)(3,6)(3,5)(4,6)
	\psbezier[linewidth=0.3\pslinewidth](1.8,2.8)(4,2.5)(5,3.6)(5.9,3.3)
	\psbezier[linewidth=0.3\pslinewidth](4,6)(6.1,5.5)(6.9,4.3)(5.9,3.3)
	\psbezier[linewidth=0.3\pslinewidth](0,2)(3,1)(4.9,2.3)(5.9,3.3)
}
\rput(13,0){
	\psbezier[linewidth=0.3\pslinewidth](1.75,4.75)(3.25,4.25)(5,4)(5.5,4.5)
	\psbezier[linewidth=0.3\pslinewidth](3.25,6.25)(4.75,5.75)(6,5)(5.5,4.5)
	\psbezier[linewidth=0.3\pslinewidth](5.5,4.5)(6.4,4.2)(7.5,4)(8.5,2.5)
	\psline[linewidth=0.3\pslinewidth](5,7)(11,5)			
	\psline[linewidth=0.3\pslinewidth](0,2)(6,0)			
	\psline(6,0)(11,5)													
	\psbezier(0,2)(0.75,3.75)(1,4)(1.75,4.75)		
	\psbezier(5,7)(4.25,7.00)(4,7)(3.25,6.25)		
	\psline[border=2\pslinewidth](6,2)(11,7)															
	\psline[border=2\pslinewidth,linewidth=0.3\pslinewidth](0,4)(6,2)			
	\psline[linewidth=0.3\pslinewidth](5,9)(11,7)													
	\psbezier(0,4)(0.75,4.00)(1,4)(1.75,4.75)				
	\psbezier(5,9)(4.25,7.25)(4,7)(3.25,6.25)				
	\psbezier[linewidth=0.3\pslinewidth](5.5,4.5)(6.4,4.2)(7.5,4)(8.5,4.5)
}
\rput(24,0){
	\psbezier[linewidth=0.3\pslinewidth](1.09136,4.09136)(4.09136,3.09136)(5.25,4.25)(5.5,4.5)
	\psbezier[linewidth=0.3\pslinewidth](3.90864,6.90864)(6.90864,5.90864)(5.75,4.75)(5.5,4.5)
	\psbezier(5,9)(4,7.5)(3.5,5.5)(2.5,4.5)		
	\psline[linewidth=0.3\pslinewidth](5,7)(11,5)			
	\psline[linewidth=0.3\pslinewidth](0,2)(6,0)			
	\psline(6,0)(11,5)													
	\psbezier[border=2\pslinewidth,linewidth=0.3\pslinewidth](2.5,6.5)(3.6,6)(5,5.5)(5.5,4.5)
	\psbezier(0,2)(1,3.5)(1.5,5.5)(2.5,6.5)		
	\psbezier[linewidth=0.3\pslinewidth](5.5,4.5)(6,3.5)(7.4,3)(8.5,2.5)
	\psbezier[linewidth=0.3\pslinewidth](2.5,4.5)(3.6,4)(5,4.2)(5.5,4.5)
	\psline[border=2\pslinewidth](6,2)(11,7)															
	\psline[border=2\pslinewidth,linewidth=0.3\pslinewidth](0,4)(6,2)			
	\psline[linewidth=0.3\pslinewidth](5,9)(11,7)													
	\psbezier(5,7)(4,6.5)(3.5,7.5)(2.5,6.5)		
	\psbezier(0,4)(1,4.5)(1.5,3.5)(2.5,4.5)		
	\psbezier[linewidth=0.3\pslinewidth](5.5,4.5)(6,4.8)(7.4,5)(8.5,4.5)
}
\end{pspicture}

%
%
	\caption{Singular loci of non-transverse events.}\label{diag:event-nontranverse}
\end{figure}
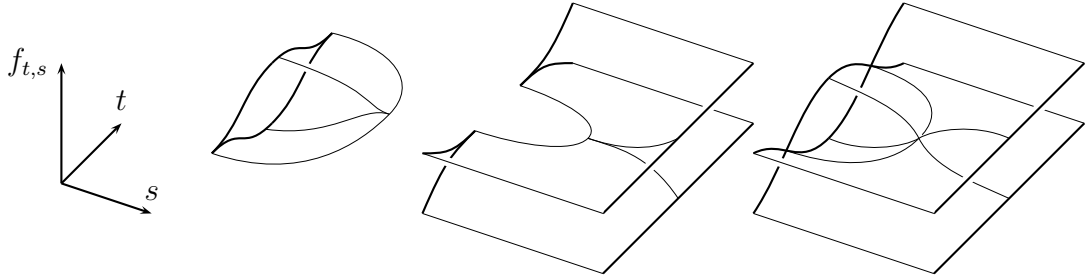

	\item Either three Morse singularities or an~$A_2$-singularity and a~Morse one meet at the~same critical level. There are three types of such homotopies
	\begin{equation}\label{rel:permutations}\begin{split}
		\psset{unit=0.6cm}
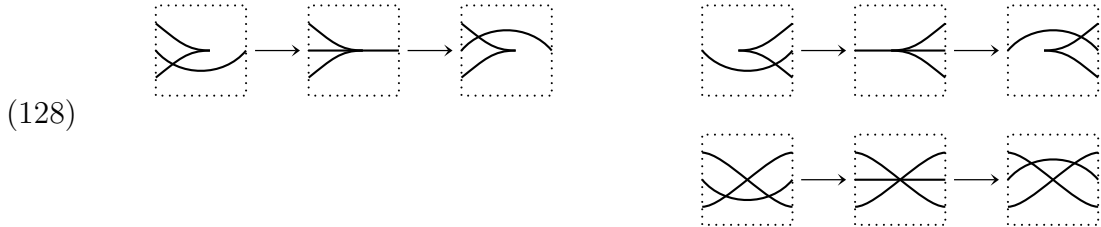
		\begin{centerpict}(-1,-1.3)(1,1.3)
			\pspolygon[linestyle=dotted,dotsep=2pt](-1,-1)(1,-1)(1,1)(-1,1)
			\psbezier(-1, 0.6)(-0.5, 0.2)(-0.3,0)(0.2,0)
			\psbezier(-1,-0.6)(-0.5,-0.2)(-0.3,0)(0.2,0)
			\psbezier(-1, 0)(-0.5,-0.6)(0.5,-0.6)(1,0)
		\end{centerpict}
		\to
		\begin{centerpict}(-1,-1.3)(1,1.3)
			\pspolygon[linestyle=dotted,dotsep=2pt](-1,-1)(1,-1)(1,1)(-1,1)
			\psbezier(-1, 0.6)(-0.5, 0.2)(-0.3,0)(0.2,0)
			\psbezier(-1,-0.6)(-0.5,-0.2)(-0.3,0)(0.2,0)
			\psline(-1, 0)(1,0)
		\end{centerpict}
		\to
		\begin{centerpict}(-1,-1.3)(1,1.3)
			\pspolygon[linestyle=dotted,dotsep=2pt](-1,-1)(1,-1)(1,1)(-1,1)
			\psbezier(-1, 0.6)(-0.5, 0.2)(-0.3,0)(0.2,0)
			\psbezier(-1,-0.6)(-0.5,-0.2)(-0.3,0)(0.2,0)
			\psbezier(-1, 0)(-0.5,0.6)(0.5,0.6)(1,0)
		\end{centerpict}
		\hskip 2cm
		\begin{centerpict}(-1,-1.3)(1,1.3)
			\pspolygon[linestyle=dotted,dotsep=2pt](-1,-1)(1,-1)(1,1)(-1,1)
			\psbezier(1, 0.6)(0.5, 0.2)(0.3,0)(-0.2,0)
			\psbezier(1,-0.6)(0.5,-0.2)(0.3,0)(-0.2,0)
			\psbezier(-1, 0)(-0.5,-0.6)(0.5,-0.6)(1,0)
		\end{centerpict}
		\to
		\begin{centerpict}(-1,-1.3)(1,1.3)
			\pspolygon[linestyle=dotted,dotsep=2pt](-1,-1)(1,-1)(1,1)(-1,1)
			\psbezier(1, 0.6)(0.5, 0.2)(0.3,0)(-0.2,0)
			\psbezier(1,-0.6)(0.5,-0.2)(0.3,0)(-0.2,0)
			\psline(-1, 0)(1,0)
		\end{centerpict}
		\to
		\begin{centerpict}(-1,-1.3)(1,1.3)
			\pspolygon[linestyle=dotted,dotsep=2pt](-1,-1)(1,-1)(1,1)(-1,1)
			\psbezier(1, 0.6)(0.5, 0.2)(0.3,0)(-0.2,0)
			\psbezier(1,-0.6)(0.5,-0.2)(0.3,0)(-0.2,0)
			\psbezier(-1, 0)(-0.5,0.6)(0.5,0.6)(1,0)
		\end{centerpict} \\
		\psset{unit=0.6cm}
		\begin{centerpict}(-1,-1.3)(1,1.3)
			\pspolygon[linestyle=dotted,dotsep=2pt](-1,-1)(1,-1)(1,1)(-1,1)
			\psbezier(-1, 0.6)(-0.5, 0.6)(0.5,-0.6)(1,-0.6)
			\psbezier(-1,-0.6)(-0.5,-0.6)(0.5, 0.6)(1, 0.6)
			\psbezier(-1, 0.0)(-0.5,-0.6)(0.5,-0.6)(1,0)
		\end{centerpict}
		\to
		\begin{centerpict}(-1,-1.3)(1,1.3)
			\pspolygon[linestyle=dotted,dotsep=2pt](-1,-1)(1,-1)(1,1)(-1,1)
			\psbezier(-1, 0.6)(-0.5, 0.6)(0.5,-0.6)(1,-0.6)
			\psbezier(-1,-0.6)(-0.5,-0.6)(0.5, 0.6)(1, 0.6)
			\psline(-1, 0)(1,0)
		\end{centerpict}
		\to
		\begin{centerpict}(-1,-1.3)(1,1.3)
			\pspolygon[linestyle=dotted,dotsep=2pt](-1,-1)(1,-1)(1,1)(-1,1)
			\psbezier(-1, 0.6)(-0.5, 0.6)(0.5,-0.6)(1,-0.6)
			\psbezier(-1,-0.6)(-0.5,-0.6)(0.5, 0.6)(1, 0.6)
			\psbezier(-1, 0)(-0.5,0.6)(0.5,0.6)(1,0)
		\end{centerpict}
	\end{split}\end{equation}
	with singular loci of two of them visualized in Fig.~\ref{diag:event-permutations} (the~case of an~annihilation is symmetric to the~one of a~creation).
\end{enumerate}
	
The~space of Igusa functions is not simply connected, which is manifested by the~lack of the~dove tail singularity in the~list above. Indeed, this singularity is modeled by a~biquadratic polynomial and as such it cannot appear. We introduce framing to obtain a~simply connected space.\footnote{
	Framed functions were introduced to overcome the~problem of lost information, when replacing a~manifold with a~Morse function: although a~Morse function decomposes $W$ into cells, one cannot build $W$ back, unless a~parametrization of each cell is given. This is the~additional information a~framing provides \cite{IgusaFMF}.
}
In fact, the~space of framed functions is contractible \cites{IgusaFMF,Lurie-Cobs, EliMish-Contr}, but we will not use this result in this paper. The~following definition comes from \cite{EliMish-Contr}.

Choose a~Riemannian metric on $W$ and a~critical point $p\in W$ of an~Igusa function $f\colon W\to\R$. We shall write $E^-(p)$ and $E^+(p)$ for the~negative and positive eigenspaces of the~Hessian of $f$ at the~point $p$, regarded as a~linear map $\Hess_p(f)\colon T_pW\to T_pW$.

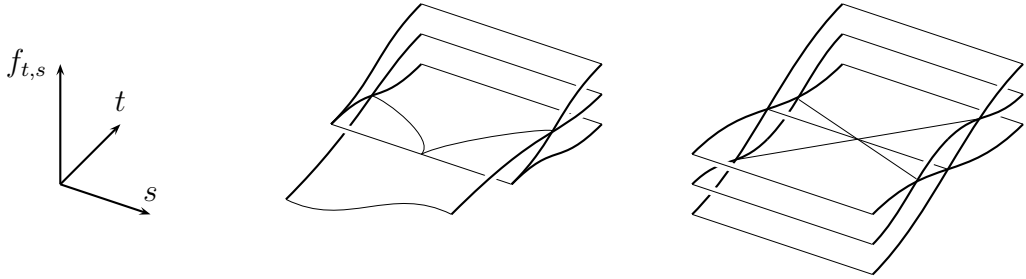
\begin{figure}
	\centering\psset{unit=0.4cm}
\begin{pspicture}(0,0)(35,10)
\rput(1,3){
	\psline{->}(0,0)(0, 4)\uput[l](0, 4){$f_{t,s}$}
	\psline{->}(0,0)(3,-1)\uput[u](3,-1){$s$}
	\psline{->}(0,0)(2, 2)\uput[u](2, 2){$t$}
}
\rput(8,0){
	\psbezier(0.5,2.5)(3,5)(3.5,6.5)(5,8)		
	\psline[border=2\pslinewidth,linewidth=0.3\pslinewidth](2,5)(8,3)		
	\psline[linewidth=0.3\pslinewidth](5,7)(11,5)		
	\psbezier(2,5)(3.5,6.5)(3.5,5.5)(5,7)			
	\psbezier[linecolor=white,linewidth=5\pslinewidth](6,2)(9,5)(9.5,4.5)(11,6)		
	\psline[linewidth=0.3\pslinewidth](5,8)(11,6)		
	\psbezier(2,5)(3.5,6.5)(3.5,7.5)(5,9)		
	\psbezier[border=2\pslinewidth](8,3)(9.5,4.5)(9.5,5.5)(11,7)	
	\psbezier(8,3)(9.5,4.5)(9.5,3.5)(11,5)	
	\psbezier(6,2)(9,5)(9.5,4.5)(11,6)			
	\psline[linewidth=0.3\pslinewidth](5,9)(11,7)			
	\psbezier[linewidth=0.3\pslinewidth](0.5,2.5)(2.5,1.5)(4,3)(6,2)	
	\psbezier[linewidth=0.3\pslinewidth](3.355,5.945)(3.855,5.8)(5.5,4.5)(5,4)
	\psbezier[linewidth=0.3\pslinewidth](9.33,4.73)(9.83,4.98)(6,4.5)(5,4)
}
\rput(22,0){
	\psline[linewidth=0.3\pslinewidth](5,7)(11,5)		
	\psbezier[linecolor=white,linewidth=5\pslinewidth](11,6)(10,5)(9.5,5.5)(8.5,4.5)	
	\psline[linewidth=0.3\pslinewidth](5,8)(11,6)		
	\psbezier[linecolor=white,linewidth=5\pslinewidth](6,0)(8,2)(9,5)(11,7)				
	\psline[linewidth=0.3\pslinewidth](5,9)(11,7)		
	\psbezier(5,8)(4,7)(3.5,5.5)(2.5,4.5) 
	\psline[border=2\pslinewidth,linewidth=0.3\pslinewidth](2.5,5.5)(8.5,3.5)		
	\psbezier(0,2)(2,4)(3,7)(5,9)					
	\psline[linewidth=0.3\pslinewidth](0,2)(6,0)		
	\psline[border=2\pslinewidth,linewidth=0.3\pslinewidth](0,3)(6,1)		
	\psbezier(0,3)(1,4)(1.5,3.5)(2.5,4.5)	
	\psline[border=2\pslinewidth,linewidth=0.3\pslinewidth](0,4)(6,2)		
	\psbezier(0,4)(2,6)(3,5)(5,7)					
	\psbezier[border=2\pslinewidth](6,1)(7,2)(7.5,3.5)(8.5,4.5)		
	\psbezier(11,6)(10,5)(9.5,5.5)(8.5,4.5)	
	\psbezier(6,0)(8,2)(9,5)(11,7)				
	\psbezier(6,2)(8,4)(9,3)(11,5)				
	\psline[linewidth=0.3\pslinewidth](1.45464,3.8202)(9.54536,5.1798)
	\psline[linewidth=0.3\pslinewidth](3.54536,5.8202)(7.45464,3.1798)
}
\end{pspicture}
	\caption{Singular loci of exceptional events from the~third group.}\label{diag:event-permutations}
\end{figure}

\begin{definition}\label{def:framing}
	Let $f\colon W\to\R$ be an~Igusa function. A~\emph{framing} on $f$ is a~choice of a~Riemannian metric on $W$ and an~orthonormal frame $v_1,\dots,v_{\mu(p)}$ of $E^-(p)$ at every critical point $p$. If $p$ is an~$A_2$-singularity, we add an~extra vector $v_{\mu(p)+1}\in N(p)$ in the~positive direction of $d^3\tau$.
\end{definition}

The~topology on the~space of framed functions $\FFun(W)$ was described indirectly in \cite{IgusaFMF} by constructing a~simplicial complex homotopy equivalent to this space. Here we only remind how homotopies look like, following \cite{EliMish-Contr}.

Choose a~smooth function $f\colon W\times I^m\to\R$ such that each \emph{slice} $f_{\underline t}\colon W\to\R$ for $\underline t\in I^m$ is an~Igusa function. Denote by $V\subset W\times I^m$ the~set of critical points of all slice functions $f_{\underline t}$ and let $\Sigma$ be the~subset of all $A_2$ points. Genericly, $V$ is an~$m$-dimensional submanifold of $W\times I^m$, $\Sigma$ has codimension $1$ in $V$, and $V$ is transverse to each slice $W\times \{\underline t\}$ at the~set $V-\Sigma$, see \cite{EliMish-Contr}. Let $V-\Sigma = V^0\cup\cdots\cup V^n$ and $\Sigma=\Sigma^0\cup\cdots\cup\Sigma^{n-1}$ be decompositions of $V-\Sigma$ and $\Sigma$ with respect to the~index. Then
\begin{itemize}
	\item $\Sigma^k$ is the~intersection of the~closures of $V^k$ and $V^{k+1}$, and
	\item for $z=(p,\underline t)\in V^k$ one has $T_pW = E^-(z)\oplus E^+(z)$, and
	\item for $z=(p,\underline t)\in\Sigma^k$ one has $T_pW = E^-(z)\oplus N(z)\oplus E^+(z)$,
\end{itemize}
where $E^{\pm}(z)$ stands for the~positive or negative eigenspace of $\Hess_p(f_{\underline t})$ and $N(z)$ is its nullspace. It follows that for $z_0\in\Sigma^k$ and $z\in V^k$
\begin{equation}\label{eq:limits-of-V-r}
	\lim_{\phantom0z\to<1em> z_0}E^+(z) = N(z_0)\oplus E^+(z_0)
			\quad\textrm{and}\quad
	\lim_{\phantom0z\to<1em> z_0}E^-(z) = E^-(z_0),
\end{equation}
whereas for $z_0\in\Sigma^k$ and $z\in V^{k+1}$
\begin{equation}\label{eq:limits-of-V-r+1}
	\lim_{\phantom0z\to<1em> z_0}E^-(z) = E^-(z_0)\oplus N(z_0)
			\quad\textrm{and}\quad
	\lim_{\phantom0z\to<1em> z_0}E^+(z) = E^+(z_0).
\end{equation}
A~framing on $f\colon W\times I^m\to I$ forms a~collection of sections $(v_1,\dots,v_n)$, where each $v_k$ is defined only over the~union $\Sigma^{k-1}\cup V^k\cup\cdots\cup\Sigma^{n-1}\cup V^n$, such that $v_k(z)\in N(z)$ for $z\in\Sigma^{k-1}$ and at $z\in V^k\cup\Sigma^k$ the~vectors $v_1(z), \dots, v_k(z)$ form an~orthonormal frame of $E^-(z)$. In particular, when we approach a~birth-death singularity, framings of canceling points agree with the~framing of the~limiting point, see Fig.~\ref{fig:framed-creation}. For more details see \cite{EliMish-Contr}.

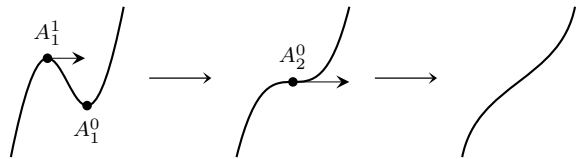
\begin{figure}
	\centering
	\begin{pspicture}(0,0)(7.5,2)
		\rput(6,0){\psbezier(0,0)(0.2,1)(1.3,1)(1.5,2)}
		\rput(3,0){%
				\psbezier(0,0)(0.5,2)(1.0,0)(1.5,2)
				\psline[linewidth=0.3pt,arrowsize=5pt,arrowlength=1]{->}(0.75,1)(1.5,1)
				\psdot(0.75,1)\uput[u](0.75,1){$\scriptstyle A_2^0$}}
		\rput(0,0){%
				\psbezier(0,0)(0.7,3.5)(0.8,-1.5)(1.5,2)
				\psline[linewidth=0.3pt,arrowsize=5pt,arrowlength=1]{->}(0.487,1.315)(0.987,1.315)
				\psdot(0.487,1.315)\uput[u](0.487,1.315){$\scriptstyle A_1^1$}
				\psdot(1.013,0.685)\uput[d](1.013,0.685){$\scriptstyle A_1^0$}}
		\rput(2.25,1){$\to<2em>$}
		\rput(5.25,1){$\to<2em>$}
	\end{pspicture}
	\caption{A~cancelation of framed $A_1$ points.}%
	\label{fig:framed-creation}%
\end{figure}

\begin{theorem}[cf. \cites{EliMish-Contr,Lurie-Cobs}]\label{thm:framed-contr}
	The~space of framed Igusa functions $\FFun(W)$ is contractible for any compact manifold $W$.
\end{theorem}

There is a~natural action of $SO(k)$ on the~set of all framings of a~critical point of index $k$. The~quotient by this action, one per each critical point, results in a~much smaller space of functions, which is still simply connected.

\begin{definition}\label{def:oriented-fun}
	An~\emph{orientation} of an~Igusa function is a~choice of an~orientation of the~negative eigenspace $E^-(p)$ at every critical point $p$. The~space of oriented Igusa functions on $W$ will be denoted by $\OFun(W)$.
\end{definition}

\begin{theorem}
	$\OFun(W)$ is simply connected for any compact manifold $W$.
\end{theorem}
\begin{proof}
	Consider the~canonical projection $\pi\colon\FFun(W)\to\OFun(W)$. It is easy to see that it has connected fibers (a~product of $SO(k)$'s). Hence, if we can show it has a~path-lifting property, then any~loop $\gamma$ can be lifted to a~loop up to reparametrization (lift $\gamma$ as a~path and connect its endpoints in a~fiber). Then a~contracting homotopy upstairs descends to a~contracting homotopy of $\gamma$.

	Pick a~path $\gamma\colon[0,1]\to\OFun(W)$. The~compactness of $[0,1]$ implies the~existence of a~sequence $0=t_0<t_1<\dots<t_k=1$ such that $\gamma|_{[t_{i-1},t_i]}$ looks like one of the~homotopies listed in Fig.~\ref{fig:singular-loci}. Since $\pi$ has connected fibers, it is enough to lift each of the~three homotopies.
	\begin{itemize}
	
		\item If $\gamma$ has only Morse singularities, for each critical point of $\gamma(0)$ choose any framing with a~given orientation and transport it along the~path.
	
		\item If $\gamma$ has a~birth singularity of index $k$ at $p$, pick any framing at this point agreeing with its orientation. Then transport it along the~path of points with index $k$ and for the~path of index $k+1$ add to the~framing the~additional vector coming from the~nullspace $N(p)$.
		
		\item For a~death singularity do the~same but with the~time reversed.
	
	\end{itemize}
	Hence, every path in $\OFun(W)$ lifts to $\FFun(W)$.
\end{proof}

\begin{remark}
	The~group $SO(k)$ is not simply connected, and there is a~choice for a~path connecting the~endpoints of the~lift. In particular, $\pi_2(\OFun(W))$ may be nontrivial. This is not a~problem for us, as we never go beyond $\pi_1(\OFun(W))$ in this paper.
\end{remark}

\section{2-categories}\label{sec:2-cats}
\subsection{Basic definitions}
This section provides basic definitions from the~theory of 2-categories \cites{Ben,Gray} and monoidal structures on them \cites{Baez-braided, KaprVoed-braided}. The~shortest way to~define a~2-category is to say that it is a~category enriched over $\cat{Cat}$. This means the~following:
\begin{itemize}

	\item for every two objects $A$ and $B$ there is a~category of morphisms $\Mor(A,B)$; morphism of this category are called \emph{$2$-morphisms}\footnote{
		We use the~double arrow notation for 2-morphisms, i.e.\ $\alpha\colon f\dblto g$, to distinguish them from the~regular ones.
	} and composition is denoted by $\star$,

	\item the~composition is given by functors $\circ_{A,B,C}\colon\Mor(B,C)\times\Mor(A,B)\to\Mor(A,C)$,

	\item the~identity morphisms are picked by functors $1\!\mathrm{l}_A\colon *\to\Mor(A,A)$, where the~category $*$ consists of a~single object $*$ and a~single morphism $\id_{*}$,

	\item the~unitarity and associativity axioms are replaced with three invertible $2$-morphisms\linebreak $\rho_f\colon f\circ \id_A \dblto f$, $\lambda_f\colon \id_B\circ f \dblto f$, and $\alpha_{f,g,h}\colon f\circ(g\circ h) \dblto (f\circ g)\circ h$ for any $f\in\Mor(A,B)$, $g\in\Mor(B,C)$, and $h\in\Mor(C,D)$, fitting into the~commutative diagrams
	\begin{gather}
		\begin{diagps}(0,0)(17em,20ex)\Dpentagon(0em,1ex)<17em,16ex>{=>`=>`=>`=>`=>}[%
			  f\circ(g \circ(h \circ k))` f\circ((g\circ h)\circ k)`
			 (f\circ g)\circ(h \circ k)` (f\circ (g \circ h))\circ k`
			((f\circ g)\circ h)\circ k;
			\id\circ\alpha`
			\alpha`\alpha`
			\alpha`\alpha\circ\id]
		\end{diagps}
		\\[2ex]
		\begin{diagps}(0,0)(10em,10ex)
			\Vtriangle<10em,8ex>{=>`=>`=>}[%
				f\circ(\id\circ g)`(f\circ\id)\circ g`f\circ g;%
				\alpha`\id\circ\lambda`\rho\circ\id%
			]
		\end{diagps}
	\end{gather}
	They are called the~MacLane's coherence conditions \cite{MacLane-cats}.
\end{itemize}
A 2-category is \emph{strict}, if all $\alpha$, $\rho$ and $\lambda$ are identities. Otherwise, it is \emph{weak}.

\begin{example}\label{ex:2cat-of-cats}
	Given two small categories $\cat{C}$ and $\cat{D}$ there is a~category $[\cat{C}\to<1em>\cat{D}]$ of functors from $\cat{C}$ to $\cat{D}$, where the~role of morphisms is played by natural transformations. Therefore, we have a~2-category of all small categories. This 2-category is strict, because composition of functors is associative.
\end{example}

\begin{example}\label{ex:2cat-R-Mod}
	Consider a~category $\Mod{R}$ of modules over a~fixed commutative ring $R$. We can extend it to a~2-category with 2-morphisms given by elements of $R$ as follows. Choose module homomorphisms $f,g\colon M\to N$ and $r\in R$. We write $r\colon f\dblto g$ if $g(m) = f(rm)$ for any $m\in M$. Both compositions of 2-morphisms are given as multiplication in $R$. The~2-category defined this way is again strict.
\end{example}

If we represent objects by points on a~plane and 1-morphisms by oriented edges, then 2-morphisms decorate regions. With this interpretation, a~picture of a~typical 2-morphisms looks as~follows:
\begin{equation}
	\begin{diagps}(-1.45,-0.9)(1.45,1.2)
		\node dom(-1.25,0)[A]
		\node cod( 1.25,0)[B]
		\carrow[angleA= 45,angleB= 135,arm=0.7,linearc=1.2]|a|{->}[dom`cod;\Rnode{f}{\scriptstyle f}]
		\carrow[angleA=-45,angleB=-135,arm=0.7,linearc=1.2]|b|{->}[dom`cod;\Rnode{g}{\scriptstyle g}]
		\arrow[nodesep=2.5ex]|a{npos=0.55,labelsep=1pt}|{=>}[f`g;\alpha]
	\end{diagps}
\end{equation}

There are two ways of composing 2-morphisms: a~\emph{vertical} composition, induced by the~internal composition in morphism categories $\Mor(A,B)$
\begin{equation}
	\begin{diagps}(-2.7,-1.2)(2.7,1.5)
		\node dom(0.0,0)[A]
		\node cod(2.5,0)[B]
		\carrow[angleA= 60,angleB= 120,arm=0.9,linearc=1.1]|a|{->}[dom`cod;\Rnode{f}{\scriptstyle f}]
		\arrow|c|{->}[dom`cod;\Rnode{m}{}]
		\carrow[angleA=-60,angleB=-120,arm=0.9,linearc=1.1]|b|{->}[dom`cod;\Rnode{g}{\scriptstyle g}]
		\arrow[nodesep=1.7ex]|a{npos=0.65,labelsep=1pt}|{=>}[f`m;\alpha]
		\arrow[nodesep=1.7ex]|a{npos=0.55,labelsep=1pt}|{=>}[m`g;\beta]
	\end{diagps}
	\qquad=\qquad
	\begin{diagps}(-0.2,-0.9)(2.7,1.2)
		\node dom(0.0,0)[A]
		\node cod(2.5,0)[B]
		\carrow[angleA= 45,angleB= 135,arm=0.7,linearc=1.2]|a|{->}[dom`cod;\Rnode{f}{\scriptstyle f}]
		\carrow[angleA=-45,angleB=-135,arm=0.7,linearc=1.2]|b|{->}[dom`cod;\Rnode{g}{\scriptstyle g}]
		\arrow[nodesep=2.5ex]|a{npos=0.55,labelsep=1pt}|{=>}[f`g;\beta\star\alpha]
	\end{diagps}
\end{equation}
and a~\emph{horizontal} composition, given by the~composition functors $\circ_{A,B,C}$
\begin{equation}
	\begin{diagps}(-0.2,-0.9)(5.2,1.2)
		\node A(0.0, 0)[A]
		\node B(2.5, 0)[B]
		\node C(5.0, 0)[C]
		\carrow[angleA= 45,angleB= 135,arm=0.7,linearc=1.2]|a|{->}[A`B;\Rnode{f0}{\scriptstyle f}]
		\carrow[angleA=-45,angleB=-135,arm=0.7,linearc=1.2]|b|{->}[A`B;\Rnode{g0}{\scriptstyle{g\phantom'}}]
		\carrow[angleA= 45,angleB= 135,arm=0.7,linearc=1.2]|a|{->}[B`C;\Rnode{f1}{\scriptstyle{f'}}]
		\carrow[angleA=-45,angleB=-135,arm=0.7,linearc=1.2]|b|{->}[B`C;\Rnode{g1}{\scriptstyle{g'}}]
		\arrow[nodesep=2.5ex]|a{npos=0.55,labelsep=1pt}|{=>}[f0`g0;\alpha]
		\arrow[nodesep=2.5ex]|a{npos=0.55,labelsep=1pt}|{=>}[f1`g1;\alpha]
	\end{diagps}
	\qquad=\qquad
	\begin{diagps}(-0.2,-0.9)(2.7,1.2)
		\node A(0.0, 0)[A]
		\node B(2.5, 0)[B]
		\carrow[angleA= 45,angleB= 135,arm=0.7,linearc=1.2]|a|{->}[A`B;\Rnode{f}{\scriptstyle{f'\circ f}}]
		\carrow[angleA=-45,angleB=-135,arm=0.7,linearc=1.2]|b|{->}[A`B;\Rnode{g}{\scriptstyle{g'\circ g}}]
		\arrow[nodesep=2.5ex]|a{npos=0.55,labelsep=1pt}|{=>}[f`g;\beta\circ\alpha]
	\end{diagps}
\end{equation}
Moreover, the~two ways of composing 2-morphisms are compatible, which means that the~diagram
\begin{equation}\label{diag:interchange-law}
	\begin{diagps}(-2.5,-5ex)(2.5,6.5ex)
		\node X(-2.5,0)[A]
		\node Y( 0.0,0)[B]
		\node Z( 2.5,0)[C]
		\carrow[angleA=60,angleB=120,arm=0.9,linearc=1.1]|c|{->}[X`Y;\Rnode{tl}{}]
		\arrow|c|{->}[X`Y;\Rnode{ml}{}]
		\carrow[angleA=-60,angleB=-120,arm=0.9,linearc=1.1]|c|{->}[X`Y;\Rnode{bl}{}]
		\carrow[angleA=60,angleB=120,arm=0.9,linearc=1.1]|c|{->}[Y`Z;\Rnode{tr}{}]
		\arrow|c|{->}[Y`Z;\Rnode{mr}{}]
		\carrow[angleA=-60,angleB=-120,arm=0.9,linearc=1.1]|c|{->}[Y`Z;\Rnode{br}{}]
		\arrow[nodesep=1.5ex]|a{npos=0.6,labelsep=1pt}|{=>}[tl`ml;\alpha]
		\arrow[nodesep=1.5ex]|a{npos=0.6,labelsep=1pt}|{=>}[ml`bl;\beta]
		\arrow[nodesep=1.5ex]|a{npos=0.6,labelsep=1pt}|{=>}[tr`mr;\alpha']
		\arrow[nodesep=1.5ex]|a{npos=0.6,labelsep=1pt}|{=>}[mr`br;\beta']
	\end{diagps}
\end{equation}
produces the~same 2-morphism no matter whether we first compose the 2-morphisms vertically or horizontally. In other words,
\begin{equation}\label{eq:interchange-law}
	(\beta'\star\alpha')\circ(\beta\star\alpha) = (\beta'\circ\beta)\star(\alpha'\circ\alpha).
\end{equation}
This property, called the~\emph{interchange law}, together with the~obvious associativity and unitarity axioms, is another way how to define a~2-category \cite{Ben}.

\begin{example}
	Chronological cobordisms form a~strict 2-category:
	\begin{itemize}
		\item objects are smooth (collared) oriented manifolds,
		\item morphisms are (collared) cobordisms with chronologies,
		\item 2-morphisms are homotopy classes of changes of chronologies.
	\end{itemize}
	The~vertical composition of 2-morphisms is given by concatenation of homotopies, whereas the~horizontal one by juxtaposition. A~routine check shows both operations are compatible, i.e. the~interchange law holds.
\end{example}

The~higher structure of 2-categories affects a~notion of a~functor: we no longer assume that it preserves identities nor compositions of morphisms. Instead, both properties should hold up to some 2-morphisms, which are part of the~data, subject to some coherence relations.\footnote{
	See \cite{Ben} for details. The~most general definition does not even assume invertibility of $\iota$ and $\varphi$, but we will never need such functors.
}

\begin{definition}\label{def:2-functor}
	A~functor $F\colon\cat{C}\to\cat{D}$ between 2-categories consists of a~function of objects $F_0\colon\Ob\cat{C}\to\Ob\cat{D}$, a~collection of functors $F_{A,B}\colon\Mor(A,B)\to\Mor(FA,FB)$, and $2$-morphisms $\iota_A\colon \id_{FA}\dblto F(\id_A)$ and $\varphi_{f,g}\colon F(f)\circ F(g)\dblto F(f\circ g)$ satisfying certain coherence relations. A~functor $F$ is \emph{strict}, if both $2$-morphisms are equalities.
\end{definition}

A~famous result states that every 2-category can be strictified: every 2-category is equivalent to some strict 2-category. Hence, we do not have to care about weak 2-categories. On the~other hand, this does not apply to functors: there are functors between strict 2-categories that cannot be replaced by strict ones. However, most functors used in this paper will be strict, with the~only exception being the~cubical functors \cite{GPS-coherence-3-cats}.

\begin{definition}\label{def:cubical}
	A~functor $F\colon\cat{C}_1\times\dots\times\cat{C}_r\longrightarrow\cat{D}$ between strict 2-categories\footnote{
		There is also a~more general notion of a~cubical functor between weak 2-categories.
	} is \emph{cubical} if the~following conditions hold:
	\begin{enumerate}
		\item $F(\id_{A_1},\dots,\id_{A_r}) = \id_{F(A_1,\dots,A_r)}$, and
		\item $F(f_1\circ g_1,\dots,f_r\circ g_r) = F(f_1,\ldots,f_r)\circ F(g_1,\ldots,g_r)$ if there is $k$ such that $f_i=\id$ and $g_j=\id$ for all $i>k>j$.
	\end{enumerate}
	In other words, $\iota$ is the~identity $2$-morphism and so is $\varphi$, unless we have to \quot{permute} nontrivial morphisms $f_i$ and $g_j$ with $i>j$.
\end{definition}

In the~case of a~cubical functor, the~coherence relations mentioned in Definition~\ref{def:2-functor} reduce to two commuting diagrams of 2-morphisms

\begin{gather}
	\label{eq:cubical-2-morphisms}
	\begin{diagps}(0,-1ex)(12em,15ex)
		\square(0,0)(12em,12ex){=>`=>`=>`=>}[%
			{F(\underline f)\circ F(\underline g)}`{F(\underline f')\circ F(\underline g')}`%
			{F(\underline f\circ \underline g)}`{F(\underline f'\circ\underline g')};%
			{F(\underline\alpha)\circ F(\underline\beta)}`\varphi`\varphi`{F(\underline\alpha\circ\underline\beta)}%
		]
	\end{diagps}
	\\[2ex]
	\label{eq:cubical-monoidal}
	\begin{diagps}(0,-1ex)(12em,15ex)
		\square(0,0)(12em,12ex){=>`=>`=>`=>}[%
			{F(\underline f)\circ F(\underline g)\circ F(\underline h)}`%
			{F(\underline f\circ \underline g)\circ F(\underline h)}`%
			{F(\underline f)\circ F(\underline g\circ\underline h)}`%
			{F(\underline f\circ\underline g\circ\underline h)};%
			\varphi\circ\id`\id\circ\varphi`\varphi`\varphi
		]
	\end{diagps}
\end{gather}
where we used a~shortcut notation $\underline f = (f_1,\dots,f_r)$ for morphisms in a~product of $2$-categories, and similarly for $2$-morphisms. The~latter condition has the~following interpretation when $r=2$: whenever we have three pairs of morphisms, passing from a~composition of values of $F$ on them to the~value of $F$ on their composition requires two \quot{transpositions} of \quot{inner} arguments and it can be done in two different ways. The~condition \eqref{eq:cubical-monoidal} says, it does not matter which way we choose.

\begin{example}\label{ex:right-is-cubical}
	The~\quot{right-then-left} disjoint sum $\rdsum$ is a~cubical functor, whereas the~\quot{left-then-right} one $\ldsum$ is cocubical (i.e. $\varphi$ in Definition~\ref{def:cubical} is identity if for some $k$ we have $f_i=\id$ and $g_j=\id$ for $i<k<j$).
\end{example}

\subsection{Gray products}

A~Gray monoidal structure on a~2-category is an~analogue of a~strict monoidal one for ordinary categories: there is a~more general definition of a~(weak) monoidal 2-category, but each such category is equivalent (in a~monoidal sense) to a~Gray-monoidal one \cite{GPS-coherence-3-cats}. Because of that it is sometimes called a~\emph{semi-strict} monoidal $2$-category \cites{Baez-braided,Lauda-open-string}.

\begin{definition}\label{def:Gray-monoidal}
	A~\emph{Gray monoidal structure} in a~strict 2-category $\cat{C}$ consists of an~associative cubical functor $\otimes\colon\cat{C}\times\cat{C}\to\cat{C}$ and a~unit object $I\in\cat{C}$ such that both $I\otimes(\blank)$ and $(\blank)\otimes I$ are identity $2$-functors.
\end{definition}

\begin{example}\label{ex:Gray-product-in-Mor-R}
	Consider a~(non-additive) subcategory $\Mod{R}^h\subset\Mod{R}$ of all $G$-graded $R$-modules and only homogeneous morphisms. The~graded tensor product, when restricted to this subcategory, is a~cubical functor: the~2-morphism $\varphi\colon(f'\otimes g')\circ(f\otimes g)\dblto(f'\circ f)\otimes(g'\circ g)$ is given as multiplication by $\lambda(\deg g',\deg f)$. This example shows that graded monoidal categories are very close to Gray categories.
\end{example}	

It is much harder to describe braiding in a~monoidal $2$-category: writing down all coherence conditions takes usually a~few pages \cites{Baez-braided,KaprVoed-braided}. Since we will never use this notion in such generality, we provide here a~very simplified version with all $2$-morphisms being identities. That is why we call it a~\emph{strict braiding}.

\begin{definition}\label{def:strict-symmetry}
	A~\emph{strict braiding} in a~Gray monoidal category $(\cat{C},\otimes,I)$ is a~collection of isomorphisms $\sigma_{A,B}\colon	A\otimes B\to B\otimes A$ such that each $\sigma_{A,\blank}$ and $\sigma_{\blank,B}$ is a~natural transformation and the~triangle below commutes
	\begin{equation}\label{diag:strict-symm-triangle}
		\begin{diagps}(0,-0.5ex)(10em,10.5ex)
			\Vtriangle<10em,8ex>[%
				A\otimes B\otimes C`A\otimes C\otimes B`C\otimes A\otimes B;
				\id\otimes\sigma_{B,C}`\sigma_{A\otimes B,C}`\sigma_{A,C}\otimes\id
			]
		\end{diagps}
	\end{equation}
	for any object $C$. If in addition $\sigma_{A,B}\circ\sigma_{B,A}=\id$, we call $\sigma$ a~\emph{strict symmetry}.
\end{definition}

A~natural transformation $\eta\colon F\to G$ in a~2-categorical setting must be coherent with 2-morphisms. This means the~following compositions of $2$-morphisms are equal
\begin{equation}
	\begin{diagps}(-0.5,0.2)(3.5,5.1)
		\node lt(0,4)[F(A)]		\node rt(3,4)[F(B)]
		\node lb(0,1)[G(A)]		\node rb(3,1)[G(B)]
		\carrow*[arcangle= 30]|a|{->}[lt`rt;\Rnode{Ff0}{\scriptstyle{F(f)}}]
		\carrow*[arcangle=-30]|b|{->}[lt`rt;\Rnode{Ff1}{\scriptstyle{F(f')}}]
		\arrow|b|{->}[lt`lb;\eta_X]	\arrow|a|{->}[rt`rb;\eta_Y]
		\arrow{->}[lb`rb;G(f')]
		\arrow[nodesep=2ex]|a{npos=0.55}|{=>}[Ff0`Ff1;F(\alpha)]
		\arrow[nodesep=5ex,offset=1ex]|a{labelsep=1ex}|{=>}[rt`lb;\id]
	\end{diagps}
	\qquad=\qquad
	\begin{diagps}(-0.5,0.2)(3.5,5.1)
		\node lt(0,4)[F(A)]		\node rt(3,4)[F(B)]
		\node lb(0,1)[G(A)]		\node rb(3,1)[G(B)]
		\arrow{->}[lt`rt;`F(f)]
		\arrow|b|{->}[lt`lb;\eta_X]	\arrow|a|{->}[rt`rb;\eta_Y]
		\carrow*[arcangle= 30]|a|{->}[lb`rb;\Rnode{Gf0}{\scriptstyle{G(f)}}]
		\carrow*[arcangle=-30]|b|{->}[lb`rb;\Rnode{Gf1}{\scriptstyle{G(f')}}]
		\arrow[nodesep=2ex]|a{npos=0.55}|{=>}[Gf0`Gf1;G(\alpha)]
		\arrow[nodesep=5ex,offset=-1ex]|b{labelsep=1ex}|{=>}[rt`lb;\id]
	\end{diagps}
\end{equation}
for any $2$-morphism $\alpha\colon f\dblto f'$.

\begin{example}
	The~category $\Mod{R}^h$ from Example~\ref{ex:Gray-product-in-Mor-R} is strictly braided, with the~braiding isomorphism $\sigma_{A,B}(a\otimes b) := \lambda(\deg a,\deg b)b\otimes a$.
\end{example}

\begin{example}
	The~2-category $\cat{ChCob}$ of chronological cobordisms is a~strictly symmetric Gray monoidal category, with a~product given by the~\quot{right-then-left} disjoint sum $\rdsum$, the~empty manifold $\emptyset$ as a~unit object, and a~permutation cylinder as a~symmetry. On the~other hand, the~$2$-category $\EmbChCob(0)$ of cobordisms embedded in $\Disk\times I$ is only strictly braided.
\end{example}

\end{document}